\documentclass[reqno,12pt,letterpaper]{amsart}
\usepackage{euler,epic,eepic,latexsym,amssymb,bbm,
amscd,amsfonts,floatflt,amsthm,enumerate,url}
\input xy
\xyoption{all}

\numberwithin{equation}{section}

\large

\newcommand{\assign}{:=}
\newcommand{\mathd}{\mathrm{d}}
\newcommand{\mathi}{\mathrm{i}}
\newcommand{\nin}{\not\in}
\newcommand{\tmem}[1]{{\em #1\/}}
\newcommand{\tmname}[1]{\textsc{#1}}
\newcommand{\tmop}[1]{\ensuremath{\operatorname{#1}}}
\newcommand{\tmscript}[1]{\text{\scriptsize{$#1$}}}
\newcommand{\tmstrong}[1]{\textbf{#1}}
\newcommand{\tmtextbf}[1]{{\bfseries{#1}}}
\newcommand{\tmtextit}[1]{{\itshape{#1}}}
\newcommand{\um}{-}
\newcommand{\upl}{+}
\newenvironment{enumeratenumeric}{\begin{enumerate}[1.] }{\end{enumerate}}

\newtheorem{corollary}{Corollary}[section]
\newtheorem{lemma}[corollary]{Lemma}
\newtheorem{proposition}[corollary]{Proposition}
\newtheorem{theorem}[corollary]{Theorem}
\newtheorem{conjecture}[corollary]{Conjecture}
\begin{document}

\title{On large deviations of additive functions}
\author{Maksym Radziwill}
\address{McGill University}
\curraddr{Stanford University}
\begin{abstract}

We prove that if two additive functions (from a certain class) take large values with roughly the same probability then they must be identical. The Kac-Kubilius model suggests that the distribution of values of a given additive function can be modeled by a sum of random variables. We show that the model is accurate (in a large deviation sense) when one is looking at values of the additive function around its mean, but fails, by ``a constant multiple", for large values of the additive function. We believe that this phenomenon arises, because the model breaks down for the values of the additive function on the ``large" primes.

In the second part of the paper, we are motivated by a question of Elliott, to understand how much the distribution of values of the additive function on primes determines, and is determined by, the distribution of values of the additive function on all of the integers. For example, our main theorem implies that a positive, strongly additive function is roughly Poisson distributed on the integers if and only if it is $1+o(1)$ or $o(1)$ on almost all primes.

\end{abstract}

\subjclass[2000]{Primary: 11N64 Secondary: 11N60, 11K65, 60F10}

\thanks{The author would like to acknowledge financial support from Universit\'e de Montr\'eal (summer 2008), McGill University (summer 2009) and NSERC}

\maketitle

\tableofcontents

\section{Introduction}

Let $g$ be a strongly additive function. According to Mark Kac, 
the distribution of the $g (n)$'s (with $n \leqslant x$ and
$x$ large) can be predicted by studying the random variable
\begin{equation}
  \sum_{p \leqslant x} g (p) X_p
\end{equation}
In $(1.1)$ the $X_p$'s are independent random variables with 
$\mathbbm{P}(X_p = 1) = 1 / p$ and $\mathbbm{P}(X_p = 0) = 1 - 1 / p$.
Thus, we expect the $g(n)$'s to cluster around the mean $\mu(g;x)$ of $(1.1)$
and within $O(\sigma(g;x))$. Here $\mu(g;x)$ and $\sigma^2(g;x)$ are 
respectively the mean and the variance of $(1.1)$, so that
\begin{eqnarray*}
  \mu (g ; x) = \sum_{p \leqslant x} \frac{g (p)}{p} & \tmop{and} & \sigma^2
  (g ; x) = \sum_{p \leqslant x} \frac{g (p)^2}{p} \cdot \left( 1 -
  \frac{1}{p} \right)
\end{eqnarray*}

%% Change here

The Erd\"os-Kac theorem (see {\cite{3}}, theorem 12.2), states that
\[ \mathcal{D}_g (x ; \Delta) \assign \frac{1}{x} \cdot \# \left\{ n \leqslant
   x : \frac{g (n) - \mu (g ; x)}{\sigma (g ; x)} \geqslant \Delta \right\}
   \sim \int_{\Delta}^{\infty} e^{- u^2 / 2} \cdot \frac{\mathd u}{\sqrt[]{2
   \pi}} \]
for $\Delta = O (1)$ and any reasonable additive function $g$. However it is
known (see {\cite{14}}) that this ``normal approximation'' fails, already when
$\Delta = \sigma (g ; x)^{1 / 3}$. In this range $\mathcal{D}_g (x ;
\Delta)$ is asymptotic to a constant $c < 1$ times the normal distribution. 
For larger values of $\Delta$ there is an ugly asymptotic (see \cite{14} 
or \cite{11})
comparing $\mathcal{D}_g(x;\Delta)$ to the normal law. Our first contribution
is the observation that the ugly asymptotic can be recast in a more natural
form. Namely, we have
\begin{equation}
  \frac{1}{x} \cdot \# \left\{ n \leqslant x : \frac{g (n) - \mu (g ;
  x)}{\sigma (g ; x)} \geqslant \Delta \right\} \sim \mathbbm{P} \left(
  \sum_{p \leqslant x} \left[ g (p) - \frac{1}{p} \right] X_p \geqslant \Delta
  \sigma (g ; x) \right)
\end{equation}
uniformly in $1 \leqslant \Delta \leqslant o (\sigma (g ; x))$, for instance
for strongly additive functions $g$ such that $0 \leqslant g (p) \leqslant O
(1)$ and $\sigma (g ; x) \rightarrow \infty$. Since $(1.2)$ is a natural 
extension of the Erd\"os-Kac, it is desirable to know if $(1.2)$ holds
for all additive functions for which the Erd\"os-Kac does (in the version of
 \cite{3}, theorem 12.2). 
We leave this open. Instead, we ask
in which range $(1.2)$ is no longer true, and by how much does $(1.2)$ fail
in that range? 
To answer this question we confine our attention to the class $\mathcal{C}$, 
defined below.

\

{\noindent}\tmtextbf{Definition. }\tmtextit{An additive function $g$ belongs
to $\mathcal{C}$ if and only if
\begin{itemize}
  \item The function $g$ is strongly additive\footnote{That is $g(p^k)=g(p)$ for all primes $p$ and integers  $k\geq 1$, and $g(mn)=g(m)+g(n)$ whenever $(m,n)=1$.} and strictly positive.
  \item Given any $A > 0$, we have for all $t \geqslant 0$ and $x \geqslant
  2$,
  \begin{equation}
    \frac{1}{\pi (x)} \sum_{\tmscript{\begin{array}{c}
      p \leqslant x\\
      g (p) \geqslant t
    \end{array}}} 1 = O_A \left( e^{- At} \right)
  \end{equation}
  \item There is a distribution function $\Psi (g ; t)$ with non-zero second
  moment, such that for all $k > 0$
  \begin{equation}
    \frac{1}{\pi (x)} \sum_{\tmscript{\begin{array}{c}
      p \leqslant x\\
      g (p) \leqslant t
    \end{array}}} 1 - \Psi (g ; t) = O_k \left( \frac{1}{\log^k x} \right)
    \text{ uniformly in $t \in \mathbbm{R}$}
  \end{equation}
\end{itemize}}{\hspace*{\fill}}{\medskip}

Additive functions belonging to $\mathcal{C}$ are particularly well-behaved.
Nonetheless, even for a $g \in \mathcal{C}$, the asymptotic $(1.2)$ doesn't
hold in the wider range $\Delta \asymp \sigma (g ; x)$. Indeed, we
prove that for any fixed $\delta > 0$, uniformly in $1 \leqslant \Delta
\leqslant \delta \sigma (g ; x)$,
\begin{equation}
  \frac{1}{x} \# \left\{ n \leqslant x : \frac{g (n) - \mu (g ; x)}{\sigma (g
  ; x)} \geqslant \Delta \right\} \sim \mathcal{A} \left( g ;
  \frac{\Delta}{\sigma} \right) \mathbbm{P} \left( \sum_{p \leqslant x} g (p)
  \left[ X_p - \frac{1}{p} \right] \geqslant \Delta \sigma \right)
\end{equation}
where $\mathcal{A} \left(g;z) \right.$ is an analytic function depending
only on $\Psi (g ; \cdot)$ and $\sigma$ stands for $\sigma (g ; x)$. Most
importantly $0 <\mathcal{A}(g ; x) \leqslant \mathcal{A}(g ; 0) = 1$ for
positive $x$, and $\mathcal{A}(g ; x)$ is a strictly decreasing function,
decaying to $0$ as $x \rightarrow \infty$. For example when $g (n) = \omega
(n)$, where $\omega (n)$ is the number of prime factors of $n$, we have
$\mathcal{A}(g ; z) = e^{- \gamma z} / \Gamma (1 + z)$ (where $\gamma$ is
the Euler-Mascheroni constant). The appearance of the
factor $\mathcal{A}(g ; \cdot)$, is largely due to the large prime factors,
and we state a precise conjecture explaining the phenomena, in the next
section. In order to prove $(1.5)$ we simply establish asymptotics for the
left and right hand side of $(1.5)$ and then compare them.

%% Change here

Our main result is a ``structure theorem'', classifying additive functions
in $\mathcal{C}$ in terms of the distribution of their large values.

%% Stop here

\begin{theorem}(The ``structure theorem'') 
  Let $f, g \in \mathcal{C}$. Suppose that $\sigma (f
; x) \sim \sigma (g ; x)$ and let $\sigma \assign \sigma (x)$ denote a
  function such that $\sigma (f ; x) \sim \sigma (x) \sim \sigma (g ; x)$. The
  asymptotic 
  \begin{equation}
    \frac{1}{x} \cdot \# \left\{ n \leqslant x : \frac{f (n) - \mu (f ;
    x)}{\sigma (f ; x)} \geqslant \Delta \right\} \sim \frac{1}{x} \cdot \#
    \left\{ n \leqslant x : \frac{g (n) - \mu (g ; x)}{\sigma (g ; x)}
    \geqslant \Delta \right\}
  \end{equation}
  holds uniformly, in the range,
  \begin{enumerate}
    \item $1 \leqslant \Delta \leqslant o (\sigma^{1 / 3})$ -- always (the
    distribution is normal)

    \item $1 \leqslant \Delta \leqslant o (\sigma^{\alpha})$ with an $\alpha
    \in (1 / 3 ; 1)$ if and only if
    \begin{eqnarray*}
      \int_{\mathbbm{R}} t^k \mathd \Psi (f ; t) & = & \int_{\mathbbm{R}} t^k
      \mathd \Psi (g ; t)
    \end{eqnarray*}
    for all $k = 3, 4, \ldots, \varrho (\alpha)$ where $\varrho (\alpha)
    \assign \left\lceil (1 + \alpha) / (1 - \alpha) \right\rceil$

    \item $1 \leqslant \Delta \leqslant o (\sigma)$ if and only if $\Psi (f ;
    t) = \Psi (g ; t)$

    \item $1 \leqslant \Delta \leqslant c \sigma$ with some fixed $c > 0$, if
    and only if $f = g$
  \end{enumerate}
\end{theorem}

{\noindent}\tmtextbf{Example. }Let $0 < \alpha, \beta < 1$ be two algebraic
irrationals. Let $f, g$ be two additive functions defined by 
letting $f (p^k) =\{\alpha p\}$ and $g (p^k) =\{\beta p\}$ at all
primes powers $p^k$. By Vinogradov's theorem (on the uniform distribution of
$\{\alpha p\}$, see {\cite{23}}, ch. 11), both $f, g \in \mathcal{C}$ and in 
fact $\Psi (f ; t) = t = \Psi (g ; t)$ for $0 \leqslant t \leqslant 1$. 
Thus by Theorem 1.1, $f, g$ are
similarly distributed (i.e $(1.6)$ holds) when $\Delta$ is in the range $1
\leqslant \Delta \leqslant o \left( \sigma \right)$ but not when $\Delta
\asymp \sigma$ unless $f = g$, that is $\alpha = \beta$.{\hspace*{\fill}}{\medskip}

The theorem highlights a certain ``discrete'' behaviour of additive functions
belonging to $\mathcal{C}$ : For example, if $(1.6)$ holds uniformly in $1
\leqslant \Delta \leqslant o (\sigma^{1 / 3 + \varepsilon})$ with any fixed
$\varepsilon > 0$, then $(1.6)$ holds uniformly in $1 \leqslant \Delta
\leqslant o (\sigma^{1 / 2})$. In fact, given any $\alpha \in (1 / 3 ; 1)$,
suppose that $(1.6)$ holds uniformly in $1 \leqslant \Delta \leqslant o
(\sigma^{\alpha})$, then for any $\delta > 0$ relation $(1.6)$ holds in
$\left. 1 \leqslant \Delta \leqslant o (\sigma^{\alpha + \delta} \right)$ as
long as $\varrho (\alpha + \delta) = \varrho (\alpha)$.

\
For a $f \in \mathcal{C}$ we have $\sigma^2(f;x) \sim c \log\log x$ for
some constant $c > 0$. Thus given $f,g \in \mathcal{C}$ we can always find
a constant $c > 0$ such that $\sigma(f;x) \sim \sigma(c \cdot g;x)$.
Keeping this observation in mind and applying Theorem 1.1, we obtain the following corollary.

\

{\noindent}\tmtextbf{Corollary. }\tmtextit{Let $f,g \in \mathcal{C}$. Suppose
that $(1.6)$ holds uniformly in $1 \leqslant \Delta \ll \sigma \asymp 
(\log\log x)^{1/2}$, then there is a constant $c > 0$ such that 
$f = c \cdot g$.}
{\hspace*{\fill}}{\medskip}

%% Rework - maybe erase what is below

(We mention another consequence of theorem 1.1 at the end of the introduction).

%% Change here

In the second part of the paper we focus on strongly additive $f$ 
such that $0 \leqslant f(p) \leqslant O(1)$ and 
$\sigma(f;x) \rightarrow \infty$. We investigate the relationship
 between the asymptotic behaviour of

%% Stop here

\begin{equation}
  \mathcal{D}_f \left( x ; \Delta \right) \text{ } \assign \text{\! }
  \frac{1}{x} \cdot \# \left\{ n \leqslant x : \frac{f (n) - \mu (f ;
  x)}{\sigma (f ; x)} \geqslant \Delta \right\}
\end{equation}
in the range $1 \leqslant \Delta \ll_{\varepsilon} \sigma (f ; x)^{1 -
\varepsilon}$ and the convergence properties of
\begin{equation}
  \text{\!} \frac{1}{\sigma^2(f;x)} \sum_{\tmscript{\begin{array}{c}
    p \leqslant x\\
    f(p) \leqslant t
  \end{array}}} \frac{f (p)^2}{p} \cdot \left( 1 - \frac{1}{p} \right)
\end{equation}
We prove roughly the following : If $(1.8)$ converges to a distribution
function $\Psi (\cdot)$ sufficiently fast, then $\mathcal{D}_f (x ; \Delta)$
behaves asymptotically like a sum of $\sigma^2 (f ; x)$ independent and
identically distributed random variables $X_1, X_2, \ldots$ with distribution
determined by
\begin{eqnarray}
  \mathbbm{E} \left[ e^{\tmop{it} X_1} \right] & = & \exp \left(
  \int_{\mathbbm{R}} \frac{e^{iut} - iut - 1}{u^2} d \Psi (u) \right) .
\end{eqnarray}
Note that the above forces $\mathbbm{E} \left[ X_1 \right] = 0$ and
$\tmop{Var} (X_1) = 1$. We do need to be more precise here since it is certainly possible that  $\sigma^2 (f ; x)$ is not an integer: So, 
when we write ``a sum of $\sigma^2 (f ; x)$ i.i.d random variables'', we really mean a
Levy process at time $t = \sigma^2 (f ; x)$, with initial distribution
determined by $(1.9)$.\footnote{Levy processes are defined below; for now we can think of it as a natural way to make continuous a count that is naturally discrete.}

%% Changed here - be careful

In the converse direction we prove that if $(1.7)$ behaves asymptotically
like a sum of $\sigma^2 (f ; x)$ i.i.d random variables
(distributed according to $(1.9)$ plus $\Psi(\alpha) - \Psi(0) = 1$
for some $\alpha > 0$) then
$(1.8)$ converges almost everywhere to the distribution function $\Psi (t)$. 

The original motivation for studying this question was to characterize
additive functions with a ``Poisson distribution'' on the integers. Namely, we
wanted to show that any strongly additive functions whose values on the integers are
``Poisson distributed'' must be $1 + o (1)$ or $o (1)$ on most primes. Here is
an example of what was achieved in this direction (the example is a particular
case of the theorems discussed previously): For convenience denote by
\[ \mathbbm{P} \tmop{oisson} \left( x ; \Delta \right) = \sum_{k \geqslant x +
   \Delta \sqrt[]{x}} e^{- x} \cdot \frac{x^k}{k!} \]
the tails of a Poisson distribution with parameter $x \geqslant 0$. As a
consequence {\footnote{Hal\'asz originally proved that $\mathcal{D}_f(x;\Delta) \sim \mathbbm{P}\tmop{oisson}(\mu(f;x);\Delta)$. This does not contradict our statement $(1.10)$ because 
when $f(p) \in \{0;1\}$ we have $\mu(f;x) = \sigma^2(f;x)+O(1)$ and in
particular $\mathbbm{P}\tmop{oisson}(\mu(f;x);\Delta) \sim \mathbbm{P}
\tmop{oisson}(\sigma^2(f;x);\Delta)$ in the range $1 \leqslant \Delta
\leqslant o(\sigma(f;x))$.}} of a well-known result of Hal\'asz {\cite{8}} if $f$ is strongly
additive, $f (p) \in \{0, 1\}$ and $\sigma (f ; x) \rightarrow \infty$, then
\begin{equation}
  \mathcal{D}_f (x ; \Delta) \sim \mathbbm{P} \tmop{oisson} \left( \sigma^2 (f
  ; x) ; \Delta \right)
\end{equation}
uniformly in $1 \leqslant \Delta \leqslant o (\sigma (f ; x))$. Conversely,
suppose that $f$ is strongly additive $0 \leqslant f (p) \leqslant O (1)$,
$\sigma (f ; x) \rightarrow \infty$ and that $(1.10)$ holds uniformly in $1
\leqslant \Delta \leqslant o (\sigma (f ; x))$. Then
\[ \frac{1}{\sigma^2 (f ; x)} \sum_{\tmscript{\begin{array}{c}
     p \leqslant x\\
     f (p) \leqslant t
   \end{array}}} \frac{f (p)^2}{p} \cdot \left( 1 - \frac{1}{p} \right)
   \longrightarrow \delta (t) \assign \left\{ \begin{array}{l}
     1 \text{ if } t \geqslant 1\\
     0 \text{ otherwise}
   \end{array} \right. \]
at all $t \in \mathbbm{R}$, with the possible exception of $t = 1$. Therefore,
for most primes $p$ we either have $f (p) = 1 + o (1)$ or $f (p) = o (1)$,
thus complementing Hal\'asz's result.

\

Let us mention, without giving a proof, one consequence of the above result.
Suppose that $f \geqslant 0$ is bounded 
on the primes, $\sigma(f;x) \longrightarrow \infty$ and 
$\sigma^2(f;x) = c \log\log x + O(1)$ for some constant $c > 0$ 
(the last assumption certainly holds if $f \in \mathcal{C}$). 
The following holds: 
If $\mathcal{D}_f(x;\Delta) \sim \mathbb{P}\tmop{oisson}(\log\log x;\Delta)$
uniformly in $1 \leqslant \Delta \leqslant o(\sigma(f;x))$ 
then $f(p) = \sqrt{c} + o(1)$ for all but $o(\pi(x))$ 
primes $\leqslant x$ \footnote{
If instead we assume $\sigma^2(f;x) = c \mu(f;x) + O(1)$ and 
$\mathcal{D}_{f}(x;\Delta) \sim \mathbb{P}\tmop{oisson}(\mu(f;x);\Delta)$
then necessarily $c = 1$ and $f(p) = 1 + o(1)$ or $f(p) = o(1)$ for 
most primes $p$.}. For a $f \in \mathcal{C}$ a more
precise result follows from theorem 1.1 (and the fact that
$\mathcal{D}_{\omega}(x;\Delta) \sim \mathbb{P}\tmop{oisson}(
\log\log x;\Delta)$ uniformly in $1 \leqslant \Delta \leqslant
o(\sigma(f;x))$ where $\omega(n)$ is the number of prime factors of $n$).

%%For an $f \in \mathcal{C}$ one obtains
%%a more precise result by using theorem 1.1 instead.

%%we
%%obtain a more precise result by using theorem 1.1. Indeed it suffice to take
%%$g = \omega$ in $(1.6)$, normalize $f$ appropriately and then use the fact
%%that $\mathcal{D}_{\omega}(x;\Delta) \sim \mathbb{P}\tmop{oisson}
%%(\log\log x;\Delta)$ uniformly in $1 \leqslant \Delta \leqslant 
%%o(\sigma(f;x))$ (as is known from the work of Selberg and Sathe). 

\section{Precise Statement of results.}

Building on earlier work by Kubilius ({\cite{12}}, p. 160) and Maciulis {\cite{14}} we establish Theorem 2.1.

\begin{theorem} 
  If $f \in \mathcal{C}$ and $\sigma = \sigma (f
  ; x)$  then 
  \begin{equation}
    \mathcal{D}_f (x ; \Delta) \text{\,} \sim \text{\! } \mathbbm{P} \left(
    \sum_{p \leqslant x} f (p) \left[ X_p - \frac{1}{p} \right] \geqslant
    \Delta \sigma \right)
  \end{equation}
  uniformly in $1 \leqslant \Delta
  \leqslant o \left( \sigma \right)$.
\end{theorem}

{\noindent}\tmtextbf{Remark. } One can prove $(2.1)$ for many other classes of 
$f$ (for example when $f \geqslant 0$ is a strongly additive function such that
 $\sigma(f;x) \rightarrow \infty$ and $0 \leqslant f(p) \leqslant O(1)$). 
We believe that $(2.1)$ holds in very broad generality, perhaps even for any 
$f$ satisfying the Erd\"os-Kac theorem (in the Kubilius-Shapiro version, see 
\cite{3}, theorem 12.2), though one may have to introduce some natural 
restrictions.
{\hspace*{\fill}}{\medskip}

As announced in the introduction,  the asymptotic relation $(2.1)$ fails when 
$\Delta \asymp \sigma (f ; x)$. This phenomenon is described in the next
Theorem.

\begin{theorem}
  Let $f \in \mathcal{C}$. Let $\sigma = \sigma (f ; x)$. For fixed $\delta >
  0$, uniformly in $1 \leqslant \Delta \leqslant \delta \sigma$,
  \[ \mathcal{D}_f (x ; \Delta) \sim \mathcal{A} \left( f ;
     \frac{\Delta}{\sigma} \right) \cdot \mathbbm{P} \left( \sum_{p \leqslant
     x} f (p) \cdot \left[ X_p - \frac{1}{p} \right] \geqslant \Delta \sigma
     \right) \]
  The function $\mathcal{A}(f ; x)$ is analytic (in a neighborhood of
  $\mathbb{R}^{+} \cup \{0\}$), strictly decreasing, and
  decays to $0$ as $x \rightarrow \infty$. Further $\mathcal{A}(f ; 0) = 1$.
  All those properties are the consequence of an explicit formula for
  $\mathcal{A}(f ; z)$ that we now describe. Denote by
  \[ \hat{\Psi} \left( f ; z \right) = \int_{\mathbbm{R}} e^{zt} d \Psi (f ;
     t) \]
  the Laplace transform of the distribution function $\Psi (f ; t)$. Let
  $\omega (z) = \omega (f ; z)$ be defined implicitly by $\hat{\Psi}' (f ;
  \omega (z)) = \hat{\Psi}' (f ; 0) + z \hat{\Psi}'' (f ; 0)$. The function
  $\omega (f ; z)$ thus defined is well-defined in a neighborhood of
  $\mathbbm{R}^{+} \cup \{0\}$ and analytic there. We have 
  \[ \mathcal{A} (f ; z) = \frac{e^{- \gamma ( \hat{\Psi} (f ; \omega (z)) - 1)}}{\Gamma
     ( \hat{\Psi} (f ; \omega (z))} \]
\end{theorem}

{\noindent}\tmtextbf{Example. }In the case of $f$ being the number of prime
factors of $n$, we find that $\hat{\Psi} (f ; z) = e^z$ and that $\omega (f ;
z) = \log (1 + z)$. Therefore 
\[ \mathcal{A}(f ; z) = \frac{e^{- \gamma z}}{\Gamma \left( 1 + z \right)} \]
Thus $\mathcal{A}(f ; x)$ decays very fast !{\hspace*{\fill}}{\medskip}

The function $\mathcal{A}(f ; \Delta /\sigma)$ stays essentially constant
throughout the range $\Delta = c \sigma + o (\sigma)$ (where $\sigma = \sigma (f ; x)$), since $A (f ; \Delta /\sigma) = A (f ; c) + o (1)$ by analyticity. 
In this respect when $\Delta \sim c \sigma$ the
quantity $\mathcal{D}_f (x ; \Delta)$ differs asymptotically from its
probabilistic counterpart only by a constant. 

%% Change here

We believe that the appearance of the function
$\mathcal{A}(f ; z)$ is essentially due to the large prime factors. To back up
our claim, let us look at what happens when one ignores the large prime 
factors.  We denote by $f (n ; y)$ the truncated additive
function
\[ f (n ; y) = \sum_{\tmscript{\begin{array}{c}
     p|n\\
     p \leqslant y
   \end{array}}} f (p) \]

The following conjecture was suggested by Kevin Ford.

\begin{conjecture}
  Suppose that $u \assign \log x / \log y \longrightarrow \infty$ and $u \leqslant \tmop{loglog} x$. Then 
  \[ \frac{1}{x} \cdot \# \left\{ n \leqslant x : \frac{f (n ; y) - \mu (f ;
     y)}{\sigma (f ; y)} \geqslant \Delta \right\} \sim \mathbbm{P} \left(
     \sum_{p \leqslant y} f (p) \left[ X_p - \frac{1}{p} \right] \geqslant
     \Delta \sigma (f ; y) \right) \]
  uniformly in $1 \leqslant \Delta \leqslant c \sigma (f ; y)$ for any fixed
  $c > 0$.
\end{conjecture}

In support of the conjecture we have the following simple proposition (which 
we deduce from Kubilius's theorem, in Barban-Vinogradov's version, 
{\cite{21}}, lemma 3.2, p. 122).

\begin{proposition}
  Suppose that $u = \log x / \log y \asymp \tmop{loglog} x$. Then
  \[ \frac{1}{x} \cdot \# \left\{ n \leqslant x : \frac{f (n ; y) - \mu (f ;
     y)}{\sigma (f ; y)} \geqslant \Delta \right\} \sim \mathbbm{P} \left(
     \sum_{p \leqslant y} f (p) \left[ X_p - \frac{1}{p} \right] \geqslant
     \Delta \sigma (f ; y) \right) \]
  uniformly in $1 \leqslant \Delta \leqslant c \sigma (f ; y)$ for any fixed
  $c > 0$.
\end{proposition}

%% Stop here

Theorem 2.2 seems to suggest the inequality
\begin{eqnarray}
  \mathcal{D}_f (x ; \Delta) & \leqslant & \left( 1 + o (1) \right) \cdot
  \mathbbm{P} \left( \sum_{p \leqslant x} f (p) \left[ X_p - \frac{1}{p}
  \right] \geqslant \Delta \sigma (f ; x) \right)
\end{eqnarray}
might be true in general. This is possibly true for fixed strongly additive
$f \geqslant 0$, uniformly in $\Delta \geqslant 0$. It is certainly false if
we drop the condition $f \geqslant 0$ and allow both $f$ and $\Delta$ to vary
uniformly. Ruzsa's paper (see \cite{16}) contains a  weaker version
of $(2.2)$ which is however valid uniformly in $f$ and $\Delta$. There is 
also a discussion of the ``optimal'' inequality in Tenenbaum's book
(see \cite{18}, p. 315). 
\

We now turn to the following question : Given an additive function $f$ what
is the relationship between the distribution of $f$ on the primes and the
distribution of $f$ on the integers ? An early result in that direction is
Kubilius's theorem, stated below (see \cite{3}, p. 12).

{\noindent}\tmtextbf{Theorem. }\tmtextit{Let $f$ be an additive function.
Let $\sigma = \sigma(f;x)$. Suppose that for every fixed $t \in (0;1)$ 
we have $\sigma(f;x) - \sigma(f;x^t) = o(\sigma(f;x))$. The following
equivalence holds: There is a distribution
function $\Psi(\cdot)$ such that
\[ \frac{1}{\sigma^2} \sum_{\tmscript{\begin{array}{c}
     p \leqslant x\\
     f (p) \leqslant t \sigma
   \end{array}}} \frac{f (p)^2}{p} \cdot \left( 1 - \frac{1}{p} \right) \]
converges weakly to $\Psi (t)$ if and only if, there is a distribution
function $F$ with mean $0$ and variance $1$ such that $\mathcal{D}_f (x ; t)$
converges weakly to $1 - F (t)$. The relationship between $\Psi$ and $F$ is
determined by
\[ \int_{\mathbbm{R}} e^{iut} \tmop{dF} (t) = \exp \left( \int_{\mathbbm{R}}
   \frac{e^{iut} - iut - 1}{u^2} d \Psi (u) \right) .
\]}{\hspace*{\fill}}{\medskip}

The striking feature of Kubilius's theorem is that from the statistical
behaviour of $f (\cdot)$ on the integers one is able to deduce the statistical
behaviour of $f (\cdot)$ on the primes. The simplest case in which the theorem
is applicable, is when $f$ is equal to the number of prime factors of $n$. In
this case
\begin{eqnarray}
  \frac{1}{\sigma^2 (f ; x)} \sum_{\tmscript{\begin{array}{c}
    p \leqslant x\\
    f (p) \leqslant t \sigma
  \end{array}}} \frac{f (p)^2}{p} & \longrightarrow & \delta_{0} (t) \assign
  \left\{ \begin{array}{l}
    1 \text{ if } t > 0\\
    0 \text{ if } t \leqslant 0
  \end{array} \right.
\end{eqnarray}
because all the $f (p)$'s belong to a bounded range. As a consequence the
limit law $1 - F (t) = \lim \mathcal{D}_f (x ; t)$ is normal. This is of course
nothing else than a variation on the Erd\"os-Kac theorem. However
$(2.3)$ holds, in fact, for all $f \in \mathcal{C}$. Thus, all that Kubilius's
theorem is saying about additive function $f \in \mathcal{C}$ is that the
limit law $1 - F (t) = \lim \mathcal{D}_f (x ; t)$ is normal. In
what follows we will be interested in obtaining more detailed information
about the convergence of
\begin{equation}
  \frac{1}{\sigma^2(f;x)} \sum_{\tmscript{\begin{array}{c}
    p \leqslant x\\
    f (p) \leqslant t
  \end{array}}} \frac{f (p)^2}{p} \cdot \left( 1 - \frac{1}{p} \right)
  \longrightarrow \Psi (t) \text{ a.e}
\end{equation}
from assumptions on the large deviation behaviour of $\mathcal{D}_f (x ;
\Delta)$ when $1 \leqslant \Delta \ll_{\varepsilon} \sigma^{1 - \varepsilon}$
and vice-versa.\, Note the difference in scale between $(2.3)$ and $(2.4)$
(that is, $t \sigma$ is replaced by $t$). To state the next two results let us
define a {\sl Levy process} (or rather a diluted version of that notion: 
we don't need any assumptions on the underlying probability space - the name
{\sl Levy process in law} would seem appropriate but it is already taken). We
allow ourselves a little sloppiness in the definition (the sloppiness comes
from working with an uncountable set of mutually independent random variables,
without discussing the existence of such a family).

{\noindent}\tmtextbf{Definition. }\tmtextit{Let $\Psi$ be a distribution
function. Denote by $\left\{ \mathcal{Z}_{\Psi} (u) : u > 0 \right\}$ an
indexed family of mutually independent random variables, with distribution
determined by
\begin{eqnarray}
  \mathbbm{E} \left[ e^{it\mathcal{Z}_{\Psi} (u)} \right] & = & \exp \left( u
  \cdot \int_{\mathbbm{R}} \frac{e^{itx} - itx - 1}{x^2} d \Psi (x) \right)
\end{eqnarray}
Note that for each $u > 0$ the random variable $\mathcal{Z}_{\Psi} (u)$ has
mean 0 and variance $u$. }{\hspace*{\fill}}{\medskip}

It is clear that the distribution of $\mathcal{Z}_{\Psi} (u)$ is known once
the distribution of $\mathcal{Z}_{\Psi} (1)$ is. Also, when $n$ is a positive
integers we can write $\mathcal{Z}_{\Psi}(n) \overset{\tmop{law}}{=} X_1 +
\ldots + X_n$ with $X_1, X_2, \ldots$ independent and identically distributed
random variables, each being distributed in exactly the same way as
$\mathcal{Z}_{\Psi} (1)$. Thus $\mathcal{Z}_{\Psi} (x)$ is a rather natural
``continuous'' generalization of the notion of a ``sum of $n$ independent and
identically distributed random variables''. Finally, let us note that in the
special case when $\Psi (t)$ has a jump of size $1$ at $t = 1$, the random
variable $\mathcal{Z}_{\Psi} (x)$ is a centered Poisson random variable with
parameter $x$. 

The content of the next Theorem is that a nice distribution on the primes
implies a nice distribution on the integers. Following Elliott (see {\cite{4}},
p. 50) we consider this a Theorem in the ``primes to integers'' direction.

\begin{theorem}
  Let $f$ be a strongly additive function. Suppose that $\sigma^2 = \sigma^2
  (f ; x) \rightarrow \infty$ and that $0 \leqslant f (p) \leqslant O (1)$ for
  all primes $p$. If there is a distribution function $\Psi (t)$ such that
  \[ \frac{1}{\sigma^2_{}} \sum_{\tmscript{\begin{array}{c}
       p \leqslant x\\
       f (p) \leqslant t
     \end{array}}} \frac{f (p)^2}{p} \cdot \left( 1 - \frac{1}{p} \right) -
     \Psi (t) \ll \frac{1}{\sigma^2} \]
  uniformly in $t \in \mathbbm{R}$ as $x \longrightarrow \infty$, then
  
  \[ \mathcal{D}_f (x ; \Delta) \sim \mathbbm{P} \left( \mathcal{Z}_{\Psi}
     \left( \sigma^2 \right) \geqslant \Delta \sigma \right) \]
uniformly in $1 \leqslant \Delta \leqslant o (\sigma)$,  as $x \rightarrow \infty$.
\end{theorem}

Theorem 2.5 is saying that assuming certain regularity conditions on the
primes, the distribution of an additive function on the integers mimics a sum
of $\sigma^2 (f ; x)$ random variables. In the converse
``integers to primes'' direction we have Theorem 2.6.

\begin{theorem}
  Let $f$ be a strongly additive function. Suppose that 
  $\sigma^2 = \sigma^2(f;x) \longrightarrow \infty$ and that
  $0 \leqslant f(p) \leqslant O(1)$ for all primes $p$. If we have
  \[ \mathcal{D}_f (x ; \Delta) \sim \mathbbm{P} \left( \mathcal{Z}_{\Psi}
     \left( \sigma^2 \right) \geqslant \Delta \sigma \right) \]
 uniformly in $1 \leqslant \Delta \ll_\varepsilon \sigma^{1 -
  \varepsilon}$, for some distribution function $\Psi$ of compact support on
$\mathbb{R}_{\geqslant 0}$ (that is $\Psi
  (\alpha) - \Psi (0) = 1$ for some $\alpha > 0$), then
  \[ \frac{1}{\sigma^2 (f ; x)} \sum_{\tmscript{\begin{array}{c}
       p \leqslant x\\
       f (p) \leqslant t
     \end{array}}} \frac{f (p)^2}{p} \cdot \left( 1 - \frac{1}{p} \right)_{}
  \]
  converges weakly to the distribution function $\Psi$ (i.e converges to 
  $\Psi(t)$ at all continuity points of $t$).
\end{theorem}

\noindent (This is not an ``integer to primes'' theorem in the
sense of {\cite{4}}, because of the $\sigma^2(f;x) \rightarrow \infty$ 
and $f (p) \leqslant O (1)$ assumption;
we believe both can be dropped without (too) much difficulty). The motivation for Theorem 2.6 and Theorem 2.5 comes from an open-ended question raised in Elliott's book {\cite{4}}: given the ``average'' behaviour of an additive function $f$ on the integers, how much can we say about its behaviour on the primes? (see p. 50 in {\cite{4}}).

By taking $\Psi (t)$ to have a jump of size $1$ at $t = 1$ in the previous
theorem, we obtain the following corollary.

\begin{corollary}
  Let
  \[ \mathbbm{P} \tmop{oisson} \left( x ; \Delta \right) = \sum_{k \geqslant x
     + \Delta \sqrt[]{x}} e^{- x} \cdot \frac{x^k}{k!} \]
  denote the tails of a Poisson distribution with parameter $x$. By a result
  of Hal\'asz {\cite{8}} for any strongly additive function $f$ such that 
  $f(p) \in \{0, 1\}$ and $\sigma^2 (f ; x) \rightarrow \infty$, we have
  $\mathcal{D}_f (x ; \Delta) \sim \mathbbm{P} \tmop{oisson} (\sigma^2 (f ; x)
  ; \Delta)$ uniformly in the range $1 \leqslant \Delta \leqslant o (\sigma (f
  ; x))$. Conversely, given a strongly additive function $f$ such that $0
  \leqslant f (p) \leqslant O (1)$ and $\sigma^2 (f ; x) \rightarrow \infty$,
  suppose that $\mathcal{D}_f (x ; \Delta) \sim \mathbbm{P} \tmop{oisson}
  (\sigma^2 (f ; x) ; \Delta)$ holds uniformly in the range $1 \leqslant
  \Delta \ll_{\varepsilon} \sigma (f ; x)^{1 - \varepsilon}$; then
  \[ \frac{1}{\sigma^2 (f ; x)} \sum_{\tmscript{\begin{array}{c}
       p \leqslant x\\
       f (p) \leqslant t
     \end{array}}} \frac{f (p)^2}{p} \cdot \left( 1 - \frac{1}{p} \right)
     \longrightarrow \delta (t) \assign \left\{ \begin{array}{l}
       1 \tmop{ if} t \geqslant 1\\
       0 \tmop{ else}
     \end{array} \right. \]
  at all $t \in \mathbbm{R}$, except possibly at $t = 1$. Thus for almost all
  primes $p$ we either have $f (p) = 1 + o (1)$ or $f (p) = o (1)$.
\end{corollary}

We now turn to the description of the technical Theorem 2.8 which we state in 
the introduction because of its central importance. A recurrent difficulty in 
the paper is that when we deal with the distribution of $f$ in the range
$\Delta \asymp \sigma$ we are forced to consider the following two cases 
separately:
\begin{enumeratenumeric}
  \item The values $f (p)$ do cluster on $\alpha \mathbbm{Z}$ for some $\alpha
  > 0$.

  \item The values $f (p)$ do not cluster on $\alpha \mathbbm{Z}$ for any
  $\alpha > 0$.
\end{enumeratenumeric}

{\noindent}\tmtextbf{Definition. }\tmtextit{Let $X$ be a random variable. We
say that $X$ is lattice distributed on $\alpha \mathbbm{Z}$ ($\alpha > 0$) if
$\mathbbm{P} \left( X \in \alpha \mathbbm{Z} \right) = 1$, and  
$\mathbbm{P} \left( X \in \beta \mathbbm{Z}) < 1
\right.$ for all $\beta > \alpha$. Analogously we say that a distribution
function is lattice distributed (resp. non-lattice distributed) when 
the underlying random variable is lattice (resp. non-lattice distributed).}
{\hspace*{\fill}}{\medskip}

In order to state Theorem 2.8 we introduce further notation. We denote by
\begin{equation}
  \hat{\Psi} \left( f ; z \right) \assign \int_{\mathbbm{R}} e^{zt} d \Psi (f
  ; t) = 1 + \sum_{k = 1}^{\infty} \int_{\mathbbm{R}} t^k d \Psi (f ; t) \cdot
  \frac{z^k}{k!} ,
\end{equation}
the two-sided Laplace transform of $\Psi (f ; t)$. By $(1.3)$ and $(1.4)$
we have $1 - \Psi(f;t) \ll_{A} e^{-At}$ and $\Psi(f;t) = 0$ for $t < 0$.
 Hence the Laplace transform 
$\hat{\Psi}$ is an entire function with Taylor expansion as in $(2.6)$.
Thus all moments of $\Psi (f ; t)$ exists, and in accordance with
$(2.6)$ the $k$-th moment of $\Psi (f ; \cdot)$ is$\int_{\mathbbm{R}} t^k d
\Psi (f ; t)$. We also define the function $\omega (f ; z)$ implicitly by,
\[ \hat{\Psi}' (f ; \omega (f ; z)) = \hat{\Psi}' (f ; 0) + z \cdot
   \hat{\Psi}'' (f ; 0) \]
Although this function does not appear in the statement of Theorem 2.8, it will
frequently be encountered in subsequent proofs.

\begin{theorem}
  Let $f \in \mathcal{C}$. We have,
  \begin{enumerate}
    \item Uniformly in $1 \leqslant \Delta \leqslant o \left( \sigma (f ;
    x)^{1 / 3} \right)$,
    \[ \mathcal{D}_f (x ; \Delta) \sim \frac{1}{\sqrt[]{2 \pi}}
       \int_{\Delta}^{\infty} e^{- u^2 / 2} \cdot \mathd u \]
    \item Given $\varepsilon > 0$, uniformly in $(\tmop{loglog}
    x)^{\varepsilon} \ll \Delta \leqslant o \left( \sigma (f ; x) \right)$,
    \[ \mathcal{D}_f (x ; \Delta) \sim S_f (x ; \Delta) \assign \frac{\left(
       \log x \right)^{\hat{\Psi} (f ; v) - 1 - v \hat{\Psi}' (f ; v)}}{v (2
       \pi \hat{\Psi}'' (f ; v) \tmop{loglog} x)^{1 / 2}} \]
    Here $v = v_f (x ; \Delta)$ is a parameter, defined as the unique positive
    solution to the equation
    \[ \hat{\Psi}' (f ; v) \tmop{loglog} x = 
	\hat{\Psi}' (f ; 0) \tmop{loglog} x + \Delta \cdot ( \hat{\Psi}''
       (f ; 0) \tmop{loglog} x)^{1 / 2} \]
    \item If $\Psi (f ; t)$ is not lattice distributed, then given $\delta,
    \varepsilon > 0$, uniformly in the range $(\tmop{loglog} x)^{\varepsilon}
    \ll \Delta \leqslant \delta \sigma (f ; x)$,
    \[ \mathcal{D}_f (x ; \Delta) \sim \frac{L (f ; v) e^{- vc (f)}}{\Gamma (
       \hat{\Psi} (f ; v))} \cdot S_f (x ; \Delta) \text{ }, \text{ } v = v_f
       (x ; \Delta) \]
    where $L (f ; z)$ is an (entire) function, defined by
    \begin{eqnarray*}
      L (f ; z) & = & \prod_p \left( 1 - \frac{1}{p} \right)^{\hat{\Psi} (f ;
      z)} \cdot \left( 1 + \frac{e^{zf (p)}}{p - 1} \right)
    \end{eqnarray*}
    and $c (f)$ is defined by $\mu (f ; x) = \hat{\Psi}' (f ; 0) \cdot
    \tmop{loglog} x + c (f) + o (1)$.

    \item If $\Psi (f ; t)$ is lattice distributed on $\mathbbm{Z}$, then
    given $\delta, \varepsilon > 0$, uniformly in the range $(\tmop{loglog}
    x)^{\varepsilon} \ll \Delta \leqslant \delta \sigma (f ; x)$,
    \[ \mathcal{D}_f (x ; \Delta) \sim \frac{L (\mathfrak{g}; v) e^{- vc
       (f)}}{\Gamma ( \hat{\Psi} (f ; v))} \cdot \mathcal{P}_{\mathfrak{h}}
       \left( \xi_f (x ; \Delta) ; v \right) \cdot S_f (x ; \Delta) \text{ , }
       v = v_f (x ; \Delta) \]
    where $\xi_f (x ; \Delta) = \mu (f ; x) + \Delta \sigma (f ; x)$,
    and $\mathfrak{g}, \mathfrak{h}$ are two additive functions defined by
    \begin{eqnarray*}
      \mathfrak{g}(p) = \left\{ \begin{array}{l}
       f (p) \tmop{ if } f (p) \in \mathbbm{Z}\\
       0 \tmop{ \ \ \ \ otherwise}
    \end{array} \right. & \tmop{and} & \mathfrak{h}(p) = \left\{
    \begin{array}{l}
       f (p) \tmop{ if } f (p) \nin \mathbbm{Z}\\
       0 \tmop{ \ \ \ \ otherwise}
    \end{array} \right.
    \end{eqnarray*}
    Finally, the function $\mathcal{P}_{\mathfrak{h}} (a ; v)$ is defined by,
    \[ \mathcal{P}_{\mathfrak{h}} \left( a ; v \right) = v \sum_{\ell \in
       \mathbbm{Z}} e^{v \left( \ell +\{a\} \right)} \cdot \mathbbm{P} \left(
       \sum_p \mathfrak{h}(p) X_p \geqslant \ell +\{a\} \right) \]
  \end{enumerate}
\end{theorem}

{\noindent}\tmtextbf{Remark. }In part $(4)$ of Theorem 2.8 the assumption 
``$\Psi (f ; t)$ lattice distributed on $\mathbbm{Z}$'' entails no loss of 
generality. If $\Psi (f ; t)$ is lattice distributed on $\alpha \mathbbm{Z}$ 
then $\Psi (f / \alpha ; t) = \Psi (f ; \alpha t)$ is lattice distributed 
on $\mathbbm{Z}$ and $f / \alpha$ is an additive function. 
{\hspace*{\fill}}{\medskip}

Let us make a few remarks about the asymptotics in Theorem 2.8. In the range 
$1 \leqslant \Delta \leqslant o (\sigma)$ the parameter 
$v \assign v_f (x ; \Delta)$ is $o (1)$. In addition $v$ admits
an convergent expansion of the form $\sum_k a_k (\Delta / \sigma)^{k}$ and
so does the function $A(f;z) = \sum_k b_k z^k$. On composing the two we obtain 
$(\log x)^{A(f;v)} = exp(\log\log x \sum c_k (\Delta / \sigma)^{k})
\sim exp(\sigma^2 \sum c_k (\Delta / \sigma)^k)$ for some coefficients $c_k$.
Thus for $\Delta \leqslant \sigma^{1-\varepsilon}$ only the first $\ll 1/\varepsilon$ terms will dictate the asymptotic behaviour of $(\log x)^{A(f;v)}$.  
To complete the picture,
\[ \frac{e^{\Delta^2 / 2}}{\sqrt[]{2 \pi}} \int_{\Delta}^{\infty} e^{- u^2 /
   2} \tmop{du} \sim \frac{1}{\sqrt{2 \pi} \Delta} \text{ for } \Delta \rightarrow \infty
\]
On the other hand when $\Delta \sim c \sigma$ for some fixed constant $c > 0$,
all the $c_k$ have a non-trivial contribution and the the parameter 
$v = v_f (x ; \Delta) = \kappa + o (1)$ for some $\kappa > 0$ depending on $c$.
In the range $\Delta \sim c \sigma$ both $L (f;v)$ and $e^{- vc(f)}$ are 
essentially constant, while $S_f (x ; \Delta)$ is about the size of 
$(\log x)^{\hat{\Psi} (f ; \kappa) - 1 - \kappa \hat{\Psi}' (f ; \kappa)+o(1)}$. Hopefully, these few remarks give a good picture of
the asymptotic behaviour in part $(1)$, $(2)$ and $(3)$ of theorem 2.8. 
Regarding part $(4)$ of theorem 2.8, since the $f (p)$ are
concentrated on $\mathbbm{Z}$, we have $\mathfrak{g}(p) = f (p)$ for ``almost
all'' prime $p$. Throughout the range $\Delta \asymp \sigma$ we have
$L(\mathfrak{g};v) \asymp 1$ and $\mathcal{P}_{\mathfrak{h}} 
(\xi_f (x ; \Delta) ; v) \asymp 1$ but otherwise the latter expression is 
highly irregular. In fact, because 
$\mathcal{P}_{\mathfrak{h}} (\xi_f (x ; \Delta) ; v)$
involves $\{\xi_f (x ; \Delta)\}=\{\mu (f ; x) + \Delta \sigma (f ; x)\}$ the
ratio
\[ \mathcal{D}_f (x; c \sigma (f;x)) \cdot S_f (x;c \sigma (f;x))^{- 1} \]
does not tend to a limit as $x \rightarrow \infty$, when $c$ is fixed. This is
also discussed in Balazard, Nicolas, Pomerance, and Tenenbaum's paper {\cite{1}}. Let us note in passing that the probabilistic interpretation we give for
$\mathcal{P}_{\mathfrak{h}} (\xi_f (x ; \Delta) ; v)$ might be of interest in 
connection with some of the question raised in {\cite{22}}. Compared to previous results, namely those of the Lithuanian school, the novelty in Theorem 2.8 is the bigger range $1 \leqslant \Delta \leqslant \delta \sigma (f ; x)$, although it is quite possible that the result was known, or at least anticipated, by the experts in the field. \\

{\noindent}\tmtextbf{Plan of the paper. } Section 4-8 and Section 9-10 can be taught of as separate. In sections 4.3-4.5 we establish rather general large deviations results. The lemmas in section 4.1 will allow to specialize these to cases of arithmetical interest. In section 5, we deduce theorem 2.8 from the lemmas in section 4. In section 6 we establish theorem 1.1 by using theorem 2.8. Finally we prove theorem 2.2 in section 7. Theorems 2.6 and 2.7 are proven
respectively in section 10 and section 9.

Regarding sections 4-8, the core ideas are scattered throughout the proof of proposition 4.10, the proof of proposition 4.17 and the entire section 6. Sections 9 and 10 are essentially ``stand-alone'' and the techniques used there differ from the ones appearing before. {\hspace*{\fill}}{\medskip}

{\newpage}

\section{Notation}

We summarize in the table below some of the recurrent notation. We let $f \in
\mathcal{C}$.
\begin{eqnarray*}
  \mathcal{D}_f (x ; \Delta) & \assign & \frac{1}{x} \cdot \# \left\{ n
  \leqslant x : \frac{f (n) - \mu (f ; x)}{\sigma (f ; x)} \geqslant \Delta
  \right\}\\
  L (f ; z) & \assign & \prod_p \left( 1 - \frac{1}{p} \right)^{\hat{\Psi} (f
  ; z)} \cdot \left( 1 + \frac{e^{zf (p)}}{p - 1} \right)\\
  \omega (f ; z) & \assign & \text{defined implicitly by $\hat{\Psi}' (f ;
  \omega (f ; z)) = \hat{\Psi}' (f ; 0) + z \cdot \hat{\Psi}'' (f ; 0)$}\\
  \sigma_{\Psi}^2 (f ; x) & \assign & \hat{\Psi}'' (f ; 0) \cdot \tmop{loglog}
  x \text{ } = \text{ } \int_{\mathbbm{R}} t^2 \mathd \Psi (f ; t) \cdot
  \tmop{loglog} x\\
  c (f) & \assign & \mu (f ; x) - \hat{\Psi}' (f ; 0) \cdot \tmop{loglog} x +
  o (1)\\
  v_f (x ; \Delta) & \assign & \omega (f ; \Delta / \sigma_{\Psi} (f ; x))\\
  A (f ; z) & \assign & \hat{\Psi} (f ; z) - 1 - z \hat{\Psi}' (f ; z)\\
  \mathcal{E}(f ; z) & \assign & A (f ; \omega (f ; z))\\
  S_f (x ; \Delta) & \assign & \frac{\left( \log x \right)^{\hat{\Psi} (f ; v)
  - v \hat{\Psi}' (f ; v) - 1}}{v (2 \pi \hat{\Psi}'' (f ; v) \tmop{loglog}
  x)^{1 / 2}} \text{ with } v = v_f (x ; \Delta)\\
  \mathfrak{h}_f (n) & \assign & \text{strongly additive function such that
  $\mathfrak{h}_f (p) = \left\{ \begin{array}{l}
    f (p) \text{ if } f (p) \nin \mathbbm{Z}\\
    0 \text{ \ \ \ \ otherwise}
  \end{array} \right.$}\\
  S(\mathfrak{h}) & \assign & \{p: \mathfrak{h}(p) \neq  0\} 
  = \{p : f(p) \nin \mathbbm{Z} \} \\
  \mathfrak{g}_f (n) & \assign & f (n) -\mathfrak{h}_f (n)\\
  \{X_p \} & \assign & \text{Independent Bernoulli random variable with
  $\mathbbm{P} \left( X_p = 1) = 1 / p \right.$} .\\
  X (\mathfrak{h}_f) & \assign & \sum_{p \leqslant x} \mathfrak{h}_f (p) X_p\\
  \mathcal{P}_{\mathfrak{h}_f} \left( a ; v \right) & \assign & v \sum_{k \in
  \mathbbm{Z}} e^{v (k +\{a\})}_{} \cdot \mathbbm{P} \left( X (\mathfrak{h}_f)
  \geqslant k +\{a\} \right)\\
  \xi_f (x ; \Delta) & \assign & \mu (f ; x) + \Delta \sigma (f ; x)\\
  B^2 (f ; x) & \assign & \sum_{p \leqslant x} \frac{f (p)^2}{p}\\
  \mathcal{D}_f^{\times} \left( x ; \Delta \right) & \assign & \frac{1}{x}
  \cdot \# \left\{ n \leqslant x : \frac{f (n) - \mu (f ; x)}{B (f ; x)}
  \geqslant \Delta \right\}
\end{eqnarray*}

Sometimes we will write $\log_k x$ to mean the $k$ times iterated logarithm.
When the context is clear we will drop the subscript $f$ from $\mathfrak{h}_f$
and $\mathfrak{g}_f$. In the same vein we usually abbreviate $v_f (x ;
\Delta)$ by $v$, and sometimes $\sigma_{\Psi} (f ; x)$ by $\sigma_{\Psi}$, 
although this is always mentioned when done.

{\newpage}

\section{Preliminary lemmata}

In the first two subsections we collect background information. The main
technical tools are developed in the subsequent sections: indeed, the three
general large deviations theorems corresponding to Proposition 4.9,
Proposition 4.10 and Proposition 4.17, form the technical backbone of this
paper.

\subsection{A mean-value theorem}

The object of this section is to prove the following mean value theorem.

\begin{proposition}
  Let $f \in \mathcal{C}$. Given $C > 0$, uniformly in $- C \leqslant \kappa
  \assign \tmop{Re} s \leqslant C$,
  \[ \frac{1}{x} \sum_{n \leqslant x} e^{sf (n)} = \frac{L (f ; s)}{\Gamma (
     \hat{\Psi} (f ; s))} \cdot (\log x)^{\hat{\Psi} (f ; s) - 1} + O_{A, C}
     \left( \mathcal{E}_A(x;s) \cdot (\log x)^{\hat{\Psi} (f ;
     \kappa) - 2} \right) \]
  where $\mathcal{E}_A(x;s) = 1 + |\tmop{Im} s|^{1/A} + |\tmop{Im} s|/\log x$.
  In particular, given $C > 0$, we have, uniformly in 
  $- C \leqslant \kappa \assign \tmop{Re} s \leqslant C$ and 
  $| \tmop{Im} s| \leqslant \log x$,
  \[ \frac{1}{x} \sum_{n \leqslant x} e^{sf (n)} = \frac{L (f ; s)}{\Gamma (
     \hat{\Psi} (f ; s))} \cdot (\log x)^{\hat{\Psi} (f ; s) - 1} + O_C \left(
     (\log x)^{\hat{\Psi} (f ; \kappa) - 3 / 2} \right) \]
\end{proposition}

{\noindent}\tmtextbf{Remark.} We will be mostly using the second formula.
{\hspace*{\fill}}{\medskip}

With more effort one can (probably) show that for $|\tmop{Im} s| \leqslant
\log x$ the error term is $(\log x)^{\tmop{Re} ( \hat{\Psi} (f ; s)) - 2}$, 
but this will not be needed. The lemma is proven by using the method 
of Levin and Fainleb (see {\cite{5}} for a survey article and {\cite{13}} 
for the paper we will follow). We include the proof only for completeness's 
sake. It is quite likely that a comparable result can be deduced directly 
from one of the lemma in Tenenbaum's book {\cite{18}} but maybe only for a 
more restrained class of additive functions.

First let us prove that $\hat{\Psi} (f ; z)$ is entire.

\begin{lemma}
  Let $f \in \mathcal{C}$. The function $\hat{\Psi} (f ; s)$ is entire. 
\end{lemma}

\begin{proof}
  Since $f \in \mathcal{C}$, by assumption $(1.3)$ and $(1.4)$
  \[ 1 - \Psi (f ; t) \leqslant c (A) \cdot e^{- At} \]
  for every fixed $A > 0$ and $c (A)$ a constant depending on $A$. Since in
  addition we require $f$ to be positive, $\Psi (f ; t) = 0$ when $t < 0$. It
  follows that
  \begin{eqnarray*}
    \int_{\mathbbm{R}} t^k \mathd \Psi (f ; t) & = & \int_0^{\infty} t^k
    \mathd \Psi (f ; t)\\
    & = & k \int_0^{\infty} t^{k - 1} \cdot (1 - \Psi (f ; t)) \mathd t\\
    & \leqslant & c (A) \cdot k \int_0^{\infty} t^{k - 1} e^{- At} \mathd t
    \text{ } = \text{ } c (A) \cdot k \cdot k! \cdot A^{- k}
  \end{eqnarray*}
  Therefore the series
  \[ 1 + \sum_{k \geqslant 1} \int_{\mathbbm{R}} t^k \mathd \Psi (f ; t) \cdot
     \frac{s^k}{k!} \]
  converges absolutely in $|s| < A / 2$. This allows us to interchange
  summation and integration and we obtain
  \begin{eqnarray*}
    1 + \sum_{k \geqslant 1} \int_{\mathbbm{R}} t^k \mathd \Psi (f ; t) \cdot
    \frac{s^k}{k!} & = & \int_{\mathbbm{R}} e^{st} \mathd \Psi (f ; t) \text{
    } \assign \text{ } \hat{\Psi} (f ; s)
  \end{eqnarray*}
  for $|s| < A / 2$. Since the series on the left is absolutely convergent for
  $|s| < A / 2$ and sums to $\hat{\Psi} (f ; s)$ it follows that $\hat{\Psi}
  (f ; s)$ is analytic in $|s| < A / 2$. But $A$ is arbitrary, therefore
  $\hat{\Psi} (f ; s)$ is entire. 
\end{proof}

%% Start cutting here

\begin{lemma}
  Let $f \in \mathcal{C}$. Given $A, C > 0$ we have, uniformly in $| \tmop{Re}
  s| \leqslant C$,
  \[ \sum_{p \leqslant x} e^{sf (p)} = \pi (x) \cdot \left[ \hat{\Psi} (f ; s)
     + O_{A, C} \left( \frac{1 + | \tmop{Im} s|}{(\log x)^{2 A}} \right)
     \right] \]
  In particular the estimate
  \[ \sum_{p \leqslant x} e^{sf (p)} = \pi (x) \cdot \left[ \hat{\Psi} (f ; s)
     + O_{A, C} \left( (\log x)^{- A} \right) \right] \]
  holds uniformly in $| \tmop{Re} s| \leqslant C$ and $| \tmop{Im} s|
  \leqslant (\log x)^A$ (hence also for $| \tmop{Im} s| \leqslant \tmop{loglog}
  x$).
\end{lemma}

\begin{proof}
  To simplify notation let $F (x ; t) = (1 / \pi (x)) \cdot \#\{p \leqslant x
  : f (p) \leqslant t\}$. Let $A > 0$ be an arbitrary, but fixed constant, and
  write $\xi \assign \tmop{loglog} x$. We have
  \begin{eqnarray}
    \sum_{p \leqslant x} e^{sf (p)} & = & \pi (x) \cdot \int_0^{\infty} e^{st}
    \mathd F (x ; t) \nonumber\\
    & = & \pi (x) \cdot \left[ \int_0^{A \xi} e^{st} \mathd F (x ; t) +
    \int_{A \xi}^{\infty} e^{st} \mathd F (x ; t) \right] 
  \end{eqnarray}
  Since $F (x ; t)$ is a distribution function the second integral is for
  $\tmop{Re} s \leqslant C$, bounded in modulus by $\int_{A \xi}^{\infty}
  e^{Ct} \mathd F (x ; t)$ which is $\ll (\log x)^{- 2 A}$ since $1 - F (x ;
  t) \ll_C e^{- (C + 2) t}$ by assumptions. We rewrite the first integral in
  $(4.1)$ as
  \begin{equation}
    \int_0^{A \xi} e^{st} \mathd F (x ; t) = \int_0^{A \xi} e^{st} \mathd \Psi
    (f ; t) + \int_0^{A \xi} e^{st} \mathd (F (x ; t) - \Psi (f ; t))
  \end{equation}
  Since $F (x ; t) - \Psi (f ; t) \ll_{C, A} (\log x)^{- A (C + 3)}$ (again by
  assumptions) the second integral in $(4.2)$ is bounded by $\ll (\log x)^{- 2
  A} + |s| \cdot (\log x)^{- 2 A}$ which is less than $\ll (1 + | \tmop{Im}
  s|) \cdot (\log x)^{- 2 A}$ because $| \tmop{Re} s| \leqslant C$. As for the
  first integral in $(4.2)$ we note that $1 - \Psi (f ; t) \ll e^{- (C + 2)
  t}$ hence $|\int_{A \xi}^{\infty} e^{st} \mathd \Psi (f ; t)|
  \leqslant \int_{A \xi}^{\infty} e^{Ct} \mathd \Psi(f;t) \ll (\log x)^{-2A}$
  which allows us to complete the tails. By $(4.2)$ and the above
  observations
  \[ \int_0^{A \xi} e^{st} \mathd F (x ; t) = \int_0^{\infty} e^{st} \mathd
     \Psi (f ; t) + O \left( \frac{1 + | \tmop{Im} s|}{(\log x)^{2 A}} \right)
     = \hat{\Psi} (f ; s) + O \left( \frac{1 + | \tmop{Im} s|}{(\log x)^{2 A}}
     \right) \]
  Plugging the above back into $(4.1)$ and recalling that the second integral
  in $(4.1)$ was bounded by $O ((\log x)^{- 2 A})$ we conclude that
  \[ \sum_{p \leqslant x} e^{sf (p)} = \pi (x) \cdot \left[ \hat{\Psi} (f ; s)
     + O_{A, C} \left( \frac{1 + | \tmop{Im} s|}{(\log x)^{2 A}} \right)
     \right] \]
  as desired.
\end{proof}

%% End cutting here

We now focus on $L (f ; z)$. We prove that $L (f ; z)$ is entire -- this is
used all over the place, but especially in the proof of the ``structure
theorem''.

\begin{lemma}
  Let $f \in \mathcal{C}$. The function $L (f ; z)$ is entire. Given
  $\kappa > 0$ there is an $x_0(\kappa)$ such that uniformly in $|\tmop{Re }
  z| \leqslant \kappa$, $|\tmop{Im } z| \leqslant \log \log x$ and
  $x \geqslant x_0(\kappa)$, 
  \[ \prod_{p \leqslant x} \left( 1 - \frac{1}{p} \right)^{\hat{\Psi} (f ; z)}
     \cdot \left( 1 + \frac{e^{zf (p)}}{p - 1} \right) = L (f ; z) \cdot
     \left( 1 + O \left( (\log x)^{- 1} \right) \right) \]
  Furthermore, uniformly in $| \tmop{Re} z| \leqslant \kappa$ we have $L (f ;
  z) = O_{\kappa, \varepsilon} (1 + | \tmop{Im} z|^{\varepsilon})$.
\end{lemma}

\begin{proof}
  First let us prove that $L (f ; s)$ is entire. Let $\kappa$ be given, and
  $\mathcal{B}$ a disk of radius $\kappa$ around $0$. By assumption $(1.3)$,
  $f (p) = o (\log p)$. Therefore there is a $C \assign C (\kappa) > 2$ such
  that for all $p \geqslant C$ we have $e^{sf (p)} \leqslant p^{1 / 3}$ for
  $\tmop{Re} s \leqslant \kappa$. In particular none of the terms $(1 + e^{sf
  (p)} / p - 1 / p)$ vanish when $\tmop{Re} s \leqslant \kappa$ and $p > C$. To
  show that $L (f ; s)$ is entire, it's enough to show that the products
  \begin{equation}
    \prod_{C \leqslant p \leqslant x} \left( 1 - \frac{1}{p}
    \right)^{\hat{\Psi} (f ; s)} \cdot \left( 1 + \frac{e^{sf (p)}}{p - 1}
    \right)
  \end{equation}
  converge uniformly in $s \in \mathcal{B}$. Equivalently since none of the
  terms in $(4.3)$ vanish when $\tmop{Re} s \leqslant \kappa$ (hence in $s \in
  \mathcal{B}$), it's enough to show that the tails
  \begin{equation}
    \sum_{p > x} \left[ \hat{\Psi} (f ; s) \cdot \log \left( 1 - \frac{1}{p}
    \right)^{} + \log \left( 1 + \frac{e^{sf (p)}}{p - 1} \right) \right]
    \rightarrow 0
  \end{equation}
  uniformly in $s \in \mathcal{B}$ as $x \rightarrow \infty$. Since $|e^{sf
  (p)} | \leqslant p^{1 / 3}$ and $| \hat{\Psi} (f ; s) | \leqslant \hat{\Psi}
  (f ; \kappa)$ for $s \in \mathcal{B}$, it follows from a Taylor expansion
  that
  \begin{eqnarray}
    &  & \sum_{p > x} \left[ \hat{\Psi} (f ; s) \cdot \log \left( 1 -
    \frac{1}{p} \right) + \log \left( 1 + \frac{e^{sf (p)}}{p - 1} \right)
    \right]^{} \nonumber\\
    & = & \sum_{p > x} \left[ \frac{e^{sf (p)}}{p} - \frac{\hat{\Psi} (f ;
    s)}{p} \right] + O \left( \sum_{p > x} \frac{\hat{\Psi} (f ; \kappa)}{p^2}
    + \sum_{p > x} \frac{e^{\kappa f (p)}}{p^2} \right) 
  \end{eqnarray}
  By lemma $4.3$ and an integration by parts the error term in $(4.5)$ is
  $O_{\kappa} (1 / x)$. We can assume that $x \geqslant e^{e^{\kappa}}$. Let
  $F (s ; x) = (1 / x) \sum_{n \leqslant x} e^{sf (p)}$. We have
  \begin{eqnarray}
    &  & \sum_{p > x} \left[ \frac{e^{sf (p)}}{p} - \frac{\hat{\Psi} (f ;
    s)}{p} \right] \text{ } = \text{ } \int_x^{\infty} \frac{1}{t} \mathd
    \left[ F (s ; x) - \hat{\Psi} (f ; s) \pi (t) \right] \\
    & = & - \frac{F (s ; x) - \hat{\Psi} (f ; s) \pi (x)}{x} +
    \int_x^{\infty} \frac{1}{t^2} \cdot \left[ F (s ; t) - \hat{\Psi} (f ; s)
    \pi (t) \right] \mathd t \nonumber
  \end{eqnarray}
  Since $t \geqslant x \geqslant e^{e^{\kappa}}$ by lemma 4.3 we have $F (s ;
  t) - \hat{\Psi} (f ; s) \pi (t) = O_{\kappa, A} \left( t (\log t)^{- A}
  \right)$ uniformly in $| \tmop{Im} s| \leqslant \kappa$ and $| \tmop{Re} s|
  \leqslant \kappa$ (hence uniformly in $s \in \mathcal{B}$). It follows that
  $(4.6)$ is bounded by
  \[ (\log x)^{- A} + \int_x^{\infty} \frac{t (\log t)^{- A}}{t^2} \mathd t
     \ll (\log x)^{- A + 1} \]
  uniformly in $s \in \mathcal{B}$. Therefore $(4.4)$ holds uniformly in $s
  \in \mathcal{B}$, and thus
  \[ \sum_{p \geqslant C} \left[ \hat{\Psi} (f ; s) \cdot \log \left( 1 -
     \frac{1}{p} \right) + \log \left( 1 + \frac{e^{sf (p)}}{p - 1} \right)
     \right] \]
  is analytic in $\mathcal{B}$. Exponentiating and multiplying by a product
  over the primes $p \leqslant C \leqslant x$ (obviously analytic) we conclude
  that
  \begin{eqnarray*}
    L (f ; s) & \assign & \prod_p \left( 1 - \frac{1}{p} \right)^{\hat{\Psi}
    (f ; s)} \cdot \left( 1 + \frac{e^{sf (p)}}{p - 1} \right)
  \end{eqnarray*}
  is analytic in $\mathcal{B}$. Since $\mathcal{B}$ was a ball with an
  arbitrary radius, it follows that the function $L (f ; s)$ is entire. In
  fact we proved more. We established that the tails in $(4.4)$ are $\ll (\log
  x)^{- A + 1}$. Therefore, for all $x$ large enough (how large $x$ we have to
  choose depends only on how big $| \tmop{Re} s|$ we allow)
  \begin{equation}
    \prod_{p \leqslant x} \left( 1 - \frac{1}{p} \right)^{\hat{\Psi} (f ; s)}
    \cdot \left( 1 + \frac{e^{sf (p)}}{p - 1} \right) = L (f ; s) \cdot \left(
    1 + O_A \left( \left( \log x \right)^{- A + 1} \right) \right)
  \end{equation}
  uniformly in $s \in \mathcal{B}$. In fact in $(4.6)$, $F (s ; t) -
  \hat{\Psi} (f ; s) \pi (t) = O_A (t (\log t)^{- A})$ does hold uniformly in
  the range $| \tmop{Re} s| \leqslant \kappa$ and $| \tmop{Im} s| \leqslant
  \tmop{loglog} x$ for $t \geqslant x$, by lemma 4.3. Therefore the tails 
  $(4.4)$ are
  $O_{\kappa, A} ((\log x)^{- A + 1})$ uniformly in $| \tmop{Re} s| \leqslant
  \kappa$ and $| \tmop{Im} s| \leqslant \tmop{loglog} x$ and it follows that
  $(4.7)$ holds in that range. This gives the second claim of the lemma. Now
  it remains to prove that $L (f ; s) = O_{\kappa, \varepsilon} (1 + |
  \tmop{Im} s|^{\varepsilon})$ uniformly in $| \tmop{Re} s| \leqslant \kappa$.
  Let as usual $C \assign C (\kappa) > 0$ be chosen so that $(1 + e^{sf (p)}
  / (p - 1))$ does not vanish in the half-plane $\tmop{Re} s \leqslant \kappa$
  for $p > C$. We want to give a bound for $L (f ; s)$ that holds uniformly in
  $| \tmop{Re} s| \leqslant \kappa$ and $| \tmop{Im} s| \leqslant T$. Without
  loss of generality $T \geqslant 1$. Uniformly in $| \tmop{Re} s| \leqslant
  \kappa$
  \begin{eqnarray}
    &  & \sum_{p \geqslant C} \left[ \hat{\Psi} (f ; s) \cdot \log \left( 1 -
    \frac{1}{p} \right) + \log \left( 1 + \frac{e^{sf (p)}}{p - 1} \right)
    \right] \nonumber\\
    & = & \sum_{p \geqslant 3 / 2} \left[ \frac{e^{sf (p)}}{p} -
    \frac{\hat{\Psi} (f ; s)}{p} \right] + O_{\kappa} (1) \nonumber\\
    & = & \int_{3 / 2}^{\infty} \frac{F (s ; t) - \hat{\Psi} (f ; s) \pi
    (t)}{t} + O_{\kappa} (1) 
  \end{eqnarray}
  with $F (s ; t) \assign \sum_{n \leqslant t} e^{sf (p)}$ as usual.
  Note that by lemma 4.3, for any given $A > 0$ we have uniformly in $|
  \tmop{Re} s| \leqslant \kappa$ and $| \tmop{Im} s| \leqslant T$
  \begin{equation}
    F (s ; t) - \hat{\Psi} (f ; s) \pi (t) = O_{\kappa, A} \left(t (\log t)^{-
    A} \right) \text{ when } t \geqslant \exp \left( T^{1 / A} \right)
  \end{equation}
  We split the integral in $(4.8)$ into two parts. The part over $3 / 2
  \leqslant t \leqslant \exp (T^{1 / A})$ and the remaining part over $t
  \geqslant \exp (T^{1 / A})$. Note that $|F (s ; t) | \leqslant F (\kappa ;
  t)$. Furthermore by lemma 4.3, $F (\kappa ; t) \ll \hat{\Psi} (f ; \kappa)
  \pi (t)$. Using these observations the integral over the $3 / 2 \leqslant t
  \leqslant \exp (T^{1 / A})$ part is bounded by
  \begin{equation}
    \int_{3 / 2}^{e^{T^{1 / A}}} \frac{1}{t^2}  \left[ F (\kappa ; t) \upl
    \hat{\Psi} (f ; \kappa) \pi (t) \right] \mathd t \ll \hat{\Psi} (f ;
    \kappa) \sum_{p \leqslant e^{T^{1 / A}}} \frac{1}{p} \ll \frac{\hat{\Psi}
    (f ; \kappa)}{A} \log (1 + T)
  \end{equation}
  by making $A$ large enough we can make the
  integral above $\leqslant \varepsilon \log (1 + T)$ for any given
  $\varepsilon > 0$. The remaining integral over $t \geqslant \exp (T^{1 /
  A})$ is bounded using $(4.9)$. Indeed we find that
  \begin{equation}
    \int_{e^{T^{1 / A}}}^{\infty} \frac{1}{t^2} \cdot \left[ F (s ; t) -
    \hat{\Psi} (f ; s) \pi (t) \right] \mathd t \ll_{\kappa, A} \int_{e^{T^{1
    / A}}}^{\infty} \frac{t \cdot (\log t)^{- A}}{t^2} \mathd t 
    \ll_{\kappa, A} T^{- 1 + 1 / A}
  \end{equation}
  Of course we can assume that $A \geqslant 2$. By $(4.10)$ and $(4.11)$ we
  conclude that the integral in $(4.8)$ is $\leqslant \varepsilon \log (1 + T)
  + O_{\kappa} (1)$ uniformly in $| \tmop{Re} s| \leqslant \kappa$ and $|
  \tmop{Im} s| \leqslant T$. Exponentiating $(4.8)$ it follows that uniformly
  in $| \tmop{Re} s| \leqslant \kappa$ and $| \tmop{Im} s| \leqslant T$,
  \[ \prod_{C \leqslant p} \left( 1 - \frac{1}{p} \right)^{\hat{\Psi} (f ; s)}
     \cdot \left( 1 + \frac{e^{sf (p)}}{p - 1} \right) = O_{\kappa,
     \varepsilon} (1 + T^{\varepsilon}) \]
  Multiplying on both sides by $\prod_{p < C} (1 - 1 / p)^{\hat{\Psi} (f ; s)}
  \cdot (1 + e^{sf (p)} / (p - 1))$ does not change the bound. Thus $L (f ; s)
  = O_{\kappa, \varepsilon} (1 + T^{\varepsilon})$ uniformly in $| \tmop{Re}
  s| \leqslant \kappa, | \tmop{Im} s| \leqslant T$ in particular $L (f ; s) =
  O_{\kappa, \varepsilon} (1 + | \tmop{Im} s|^{\varepsilon})$ uniformly in $|
  \tmop{Re} s| \leqslant \kappa$. The claim follows.
\end{proof}

Finally we need an elementary lemma on sums of multiplicative functions.
The following lemma appears on page 308 of Tenenbaum's book {\cite{18}}.

\begin{lemma}
  Let $g \geqslant 0$ be a multiplicative function, such that for some $A$ and
  $B$,
  \begin{eqnarray*}
    \sum_{p \leqslant x} g (p) \log p & \leqslant & Ax\\
    \sum_p \sum_{v \geqslant 2} \frac{g (p^v)}{p^v} \cdot \log p^v_{} &
    \leqslant & B
  \end{eqnarray*}
  Then, for $x > 1$,
  \[ \sum_{n \leqslant x} g (n) \leqslant (A + B + 1) \cdot \frac{x}{\log x}
     \sum_{n \leqslant x} \frac{g (n)}{n} \]
\end{lemma}

\begin{corollary}
  Let $f \in \mathcal{C}$. Given $C > 0$, uniformly in $0 \leqslant \kappa
  \leqslant C$,
  \begin{eqnarray*}
    \sum_{n \leqslant x} e^{\kappa f (n)} & = & O_C \left( x \cdot (\log
    x)^{\hat{\Psi} (f ; \kappa) - 1} \right)\\
    \sum_{n \leqslant x} \frac{e^{\kappa f (n)}}{n} & = & O_C \left( (\log
    x)^{\hat{\Psi} (f ; \kappa)} \right)
  \end{eqnarray*}
\end{corollary}

\begin{proof}
  In lemma 4.5 we choose $g (n) \assign e^{\kappa f (n)}$. By lemma 4.3 there
  is an $A \assign A (C)$ such that
  \[ \sum_{p \leqslant x} e^{\kappa f (p)} \cdot \log p \leqslant \sum_{p
     \leqslant x} e^{Cf (p)}_{} \cdot \log p \leqslant A (C) \cdot x \]
  for all $x > 1$. Also note that
  \begin{eqnarray*}
    \sum_p \sum_{v \geqslant 2} \frac{e^{\kappa f (p^v)}}{p^v} \cdot \log p^v
    & \ll & \sum_p e^{Cf (p)} \cdot \frac{\log p}{p^2}
  \end{eqnarray*}
  and by lemma 4.3 the above sum converges. Hence the second assumption of the
  lemma holds, for some $B \assign B (C)$ large enough. Thus by lemma 4.5,
  \begin{equation}
    \sum_{n \leqslant x} e^{\kappa f (n)}_{} = O_C \left( \frac{x}{\log x}
    \cdot \sum_{n \leqslant x} \frac{e^{\kappa f (n)}}{n} \right) = O \left(
    \frac{x}{\log x} \cdot \prod_{p \leqslant x} \left( 1 + \frac{e^{\kappa f
    (p)}}{p - 1} \right) \right)
  \end{equation}
  By lemma 4.3 and an integration by parts $\sum_{p \leqslant x} e^{\kappa f
  (p)} \cdot (p - 1)^{- 1} = \hat{\Psi} (f ; \kappa) \tmop{loglog} x + O_C
  (1)$. Therefore the product in $(4.12)$ is bounded by $(\log x)^{\hat{\Psi}
  (f ; \kappa)}$. Hence the mean-value $M (x) \assign \sum_{n \leqslant x}
  e^{\kappa f (n)} \ll x (\log x)^{\hat{\Psi} (f ; \kappa) - 1}$ and also
  \begin{eqnarray*}
    \sum_{n \leqslant x} \frac{e^{\kappa f (n)}}{n} & = & \frac{M (x)}{x} +
    \int_1^x \frac{M (t)}{t^2} \cdot \mathd t\\
    & \ll & (\log x)^{\hat{\Psi} (f ; \kappa) - 1} + \int_1^x (\log
    t)^{\hat{\Psi} (f ; \kappa) - 1} \cdot t^{- 1} \mathd t \text{ } \ll
    \text{ } (\log x)^{\hat{\Psi} (f ; \kappa)}
  \end{eqnarray*}
  as desired. \ 
\end{proof}

We are now ready to prove Proposition 4.1.

%% New proof here

\begin{proof}[Proof of Proposition 4.1]
  Let $P_k$ denote the product of the first $k$ primes. The plan of our proof
  is the following. First we estimate
  \begin{equation}
    m_k (x ; z) \assign \sum_{\tmscript{\begin{array}{c}
      n \leqslant x\\
      (n, P_k) = 1
    \end{array}}} \frac{e^{zf (n)}}{n}
  \end{equation}
  uniformly in $| \tmop{Re} z| \leqslant C$ and with $k = k (C) > 0$ chosen
  suitably. Then we relate $(4.13)$ to the mean value $M_k (x ; z) \assign
  \sum_{n \leqslant x, (n, P_k) = 1} e^{zf (n)}$. By a simple convolution
  argument we subsequently obtain the desired asymptotic for $M (x ; z)
  \assign \sum_{n \leqslant x} e^{zf (n)}$. Denote by $\Lambda_f (z ; n)$ the
  ``generalized van Mangoldt function'' defined by
  \begin{equation}
    e^{zf (n)} \cdot \log n = \sum_{d|n} e^{zf (d)} \cdot \Lambda_f (z ; n /
    d)
  \end{equation}
  Looking at the Dirichlet series for $\Lambda_f (z ; n)$ we conclude that
  $\Lambda_f (z ; n)$ vanishes when $n$ is not a prime power. On the other
  hand when $n = p^{\alpha}$ is a prime power (see \cite{5}, lemma 1.1.2)
  \[ \Lambda_f (z ; p^{\alpha}) = \log p^{\alpha} \cdot \sum_{m \leqslant
     \alpha} \frac{\left( - 1 \right)^{m - 1}}{m} \cdot e^{zmf (p)} \cdot
     \left(\begin{array}{c}
       \alpha - 1\\
       m - 1
     \end{array}\right) \]
  Therefore
  \begin{eqnarray}
    \sum_{\tmscript{\begin{array}{c}
      n \leqslant x\\
      (n, P_k) = 1
    \end{array}}} \Lambda_f (z ; n) & = & \sum_{\tmscript{\begin{array}{c}
      p^{\alpha} \leqslant x\\
      p > k
    \end{array}}} \log p^{\alpha} \sum_{m \leqslant \alpha} \frac{\left( - 1
    \right)^{m - 1}}{m} \cdot e^{zmf (p)} \cdot \left(\begin{array}{c}
      \alpha - 1\\
      m - 1
    \end{array}\right) \nonumber\\
    & = & \sum_{m \leqslant \log x / \log k} \frac{\left( - 1 \right)^{m -
    1}}{m} \cdot \sum_{\tmscript{\begin{array}{c}
      p^{\alpha} \leqslant x\\
      \alpha \geqslant m\\
      p > k
    \end{array}}} \log p^{\alpha} \cdot e^{zmf (p)} \cdot
    \left(\begin{array}{c}
      \alpha - 1\\
      m - 1
    \end{array}\right) 
  \end{eqnarray}
  We split the above sum into two. The terms with $m = 1$ contribute
  \begin{equation}
    \sum_{\tmscript{\begin{array}{c}
      p^{\alpha} \leqslant x\\
      p > k
    \end{array}}} e^{zf (p)} \cdot \log p^{\alpha} = \hat{\Psi} (f ; z) \cdot
    x + O_{A, C} \left( x \cdot \frac{1 + | \tmop{Im} z|}{(\log x)^{3 A}}
    \right)
  \end{equation}
  by lemma 4.3 and an integration by parts (using the prime number theorem
  with a $O_B(x (\log x)^{-B})$ error term. The terms $m \geqslant 2$
  contribute
  \begin{equation}
    \sum_{2 \leqslant m \leqslant \log x / \log k} \frac{\left( - 1 \right)^{m
    - 1}}{m} \sum_{\tmscript{\begin{array}{c}
      p^{\alpha} \leqslant x\\
      \alpha \geqslant m\\
      p > k
    \end{array}}} \alpha \log p \cdot e^{zmf (p)} \cdot \left(\begin{array}{c}
      \alpha - 1\\
      m - 1
    \end{array}\right)
  \end{equation}
  Since $f (p) = o (\log p)$ (because of assumption $(1.3)$) we can choose 
  $k$ large enough so as to have $f(p) \leqslant (1 / 4 C) \log p$ 
  for all $p > k$. With this choice of $k
  \assign k (C)$, for $\tmop{Re} z \leqslant C$, the sum in $(4.17)$ is
  bounded in modulus by
  \begin{eqnarray*}
    & \ll & \sum_{2 \leqslant m \leqslant \log x} \frac{1}{m}
    \sum_{\tmscript{\begin{array}{c}
      p^{\alpha} \leqslant x\\
      \alpha \geqslant m\\
      p > k
    \end{array}}} (\log x)^2 \cdot \exp \left( Cm \cdot \frac{\log p}{4 C}
    \right) \cdot 2^{\alpha}\\
    & \ll & x^{\log 2 / \log k} \cdot \sum_{2 \leqslant m \leqslant \log x}
    \frac{1}{m} \cdot \left( \log x \right)^3 \cdot x^{1 / 4} \cdot x^{1 / m}
    \text{ } \ll \text{ } x^{3 / 4 + \log 2 / \log k} \cdot (\log x)^4
  \end{eqnarray*}
  To obtain the second bound we use $p \leqslant x^{1/m}$ to get 
  $exp(C m \log p / 4C) \leqslant x^{1/4}$ and then the bound 
  $\sum_{p^{\alpha} \leqslant x, \alpha \geqslant m} 1 \ll x^{1/m} \log x$.
  Making $k$ larger if necessary we see that the sum in $(4.17)$ is bounded by
  $x^{1 - \varepsilon}$ for some small but fixed $\varepsilon > 0$. Our bound
  for $(4.17)$ together with $(4.16)$ allows us to conclude that
  \begin{equation}
    \sum_{\tmscript{\begin{array}{c}
      n \leqslant x\\
      (n, P_k) = 1
    \end{array}}} \Lambda_f (z ; n) = \hat{\Psi} (f ; z) \cdot x + O_{A, C}
    \left( x \cdot \frac{1 + | \tmop{Im} z|}{(\log 2 x)^{3 A}} \right)
  \end{equation}
  Upon integrating by parts (and making $A$ larger if necessary) we obtain
  \begin{equation}
    \sum_{\tmscript{\begin{array}{c}
      n \leqslant x\\
      (n, P_k) = 1
    \end{array}}} \frac{\Lambda_f (z ; n)}{n} = \hat{\Psi} (f ; z) \cdot \log
    x + A_0 (f ; z) + O_{A, C} \left( \frac{1 + | \tmop{Im} z|}{(\log 2 x)^{3
    A}} \right)
  \end{equation}
  uniformly in $| \tmop{Re} z| \leqslant C$ where $A_0 (f ; z) \assign \int_{1
  }^{\infty} [G (z ; t) - \hat{\Psi} (f ; z) t] t^{- 2} \mathd t$ is 
  analytic in $|\tmop{Re} z| \leqslant C$ and where
  $G (z ; t) \assign \sum_{n \leqslant x, (n, P_k)} \Lambda_f (z ; n)$. Using
  equation $(4.16)$ and repeating the same proof as in lemma 4.4 we find that
  $A_0 (f ; z) = O_{A, C} (1 + | \tmop{Im} z|^{1 / A})$ uniformly in the range
  $| \tmop{Re} z| \leqslant C$. Following Levin and Fainleb we express
  $\sum_{n \leqslant x, (n, P_k) = 1} e^{zf (n)} \cdot \log n \cdot n^{- 1}$
  in two different ways. On the one hand, integrating by parts we get
  \begin{equation}
    \sum_{\tmscript{\begin{array}{c}
      n \leqslant x\\
      (n, P_k) = 1
    \end{array}}} \frac{e^{zf (n)} \cdot \log n}{n} = m_k (x ; z) \cdot \log x
    - \int_2^x \frac{m_k (u ; z)}{u} \mathd u
  \end{equation}
  where $m_k (x ; z) \assign \sum_{n \leqslant x, (n, P_k) = 1} e^{zf (n)}
  \cdot n^{- 1}$. On the other by $(4.14)$ and $(4.19)$,
  \begin{eqnarray}
    &  & \sum_{\tmscript{\begin{array}{c}
      n \leqslant x\\
      (n, P_k) = 1
    \end{array}}} \frac{e^{zf (n)} \cdot \log n}{n} =
    \sum_{\tmscript{\begin{array}{c}
      d \leqslant x\\
      (d, P_k) = 1
    \end{array}}} \frac{e^{zf (d)}}{d} \sum_{\tmscript{\begin{array}{c}
      n \leqslant x / d\\
      (n, P_k) = 1
    \end{array}}} \frac{\Lambda_f (z ; n)}{n} \nonumber\\
    & = & \sum_{\tmscript{\begin{array}{c}
      d \leqslant x\\
      (d, P_k) = 1
    \end{array}}} \frac{e^{zf (d)}}{d} \cdot \left[ \hat{\Psi} (f ; z) \cdot
    \left( \log x - \log d \right) + A_0 (f ; z) + O_{A, C} \left( \frac{1 + |
    \tmop{Im} z|}{(\log 2 x / d)^{3 A}} \right) \right] \nonumber\\
    & = & \hat{\Psi} (f ; z) \int_2^x \frac{m_k (u ; z)}{u} \mathd u + A_0 (f
    ; z) m_k (x ; z) + O_{A, C} \left( \frac{1 + | \tmop{Im} z|}{(\log x)^{2
    A}} \right) 
  \end{eqnarray}
  In the error term we bound $\sum e^{\kappa f (d)} \cdot d^{- 1/2} 
  \cdot d^{-1/2} \cdot (\log 2 x / d)^{- 3 A}$ by using Cauchy-Schwarz's 
  inequality and Corollary 4.6
  (also, we assume without loss of generality that $A$ is chosen sufficiently
  large, $A \geqslant 4 \hat{\Psi} (f ; 2C) + 4$ will do). \ Comparing
  $(4.20)$ with $(4.21)$ we conclude that
  \[ m_k (x ; z) \log x - (1 + \hat{\Psi} (f ; z)) \int_2^x \frac{m_k (u ;
     z)}{u} \mathd u = A_0 (f ; z) m_k (x ; z) + O \left( \frac{1 + |
     \tmop{Im} z|}{(\log x)^{2 A}} \right) \]
  uniformly in $| \tmop{Re} z| \leqslant C$. Recall that $A$ is taken large
  enough, $A \geqslant 4 \hat{\Psi} (f ; 2C) + 4$. Dividing by $x (\log
  x)^{\hat{\Psi} (f ; z) + 2}$ on both sides and integrating from $2$ to $x$
  we obtain
  \begin{eqnarray}
    &  & \int_2^x \frac{m_k (u ; z) \mathd u}{u (\log u)^{\hat{\Psi} (f ; z)
    + 1}} - \int_2^x \frac{1 + \hat{\Psi} (f ; z)}{u (\log u)^{\hat{\Psi} (f ;
    z) + 2}} \int_2^u \frac{m_k (v ; z)}{v} \mathd v \mathd u \\
    & = & A_0 (f ; z) \int_2^x \frac{m_k (x ; u) \mathd u}{u (\log
    u)^{\hat{\Psi} (f ; z) + 2}} + A_1 (f ; z) + O \left( \frac{1 + |
    \tmop{Im} z|}{(\log x)^{A + \hat{\Psi} (f ; C) + 1}} \right) \nonumber
  \end{eqnarray}
  with both $A_0 (f ; z)$ and $A_1 (f ; z)$ analytic in $|\tmop{Re} z|
  \leqslant C$. In fact by a proof
  similar to the one in lemma 4.4 we find that $A_1 (f ; z) \ll_{A, C} 1 + |
  \tmop{Im} z|^{1 / A}$. Upon interchanging integrals the second term in
  $(4.22)$ can be re-written as
  \begin{eqnarray*}
    &  & \left( 1 + \hat{\Psi} (f ; z) \right) \int_2^x \frac{m_k (v ; z)}{v}
    \int_v^x \frac{\mathd u \mathd v}{u (\log u)^{\hat{\Psi} (f ; z) + 2}}\\
    & = & \int_2^x \frac{m_k (v ; z) \mathd v}{v (\log v)^{\hat{\Psi} (f ; z)
    + 1}} - \int_2^x \frac{m_k (v ; z) \mathd v}{v (\log x)^{\hat{\Psi} (f ;
    z) + 1}}
  \end{eqnarray*}
  Therefore $(4.22)$ simplifies to
  \begin{eqnarray*}
    \int_2^x \frac{m_k (u ; z)}{u} \mathd u & = & A_0 (f ; z) \int_2^x
    \frac{m_k (u ; z)}{u (\log u)^{\hat{\Psi} (f ; z) + 2}} \cdot (\log
    x)^{\hat{\Psi} (f ; z) + 1}\\
    &  & + A_1 (f ; z) \cdot (\log x)^{\hat{\Psi} (f ; z) + 1} + O_{A, C}
    \left( \frac{1 + | \tmop{Im} z|}{(\log x)^A} \right)
  \end{eqnarray*}
  Plugging the above relation into the equation right above $(4.22)$
  yields
  \begin{eqnarray*}
    m_k (x ; z) \cdot \log x & = & (1 + \hat{\Psi} (f ; z)) A_0 (f ; z)
    \int_2^x \frac{m_k (u ; z) \mathd u}{u (\log u)^{\hat{\Psi} (f ; z) + 2}}
    \cdot (\log x)^{\hat{\Psi} (f ; z) + 1}\\
    &  & + (1 + \hat{\Psi} (f ; z)) A_1 (f ; z) \cdot \left( \log x
    \right)^{\hat{\Psi} (f ; z) + 1} + O \left( \frac{1 + | \tmop{Im}
    z|}{(\log x)^A} \right)\\
    &  & + A_0 (f ; z) \cdot m_k (x ; z) + O \left( \frac{1 + | \tmop{Im}
    z|}{(\log x)^{2 A}} \right)
  \end{eqnarray*}
  because $|\hat{\Psi}(f;z)| \leqslant \hat{\Psi}(f;C)$. 
  We could iterate to obtain an asymptotic expansion. We
  choose not to do so. Instead we note the bound $|m_k (x ; z) | \leqslant m_k
  (x ; \kappa) \ll (\log x)^{\hat{\Psi} (f ; \kappa)}$ ($\kappa
  \assign \tmop{Re} z$) coming from from Corollary
  4.6. Recall also that $A_0 (f ; z) \ll_C 1 + | \tmop{Im} z|^{1 / A}$ and
  that $\hat{\Psi} (f ; z) \ll_C 1$. With these two bounds at hand our
  previous equality becomes
  \begin{eqnarray*}
    m_k (x ; z) & = & (1 + \hat{\Psi} (f ; z)) A_1 (f ; z) \cdot (\log
    x)^{\hat{\Psi} (f ; z)} + O \left( \mathcal{E}_A(x;z) \cdot
    (\log x)^{\hat{\Psi} (f ; \kappa) - 1} \right)
  \end{eqnarray*}
  uniformly in $| \tmop{Re} z| \leqslant C$ and where $\mathcal{E}_A(z;x) =
  1 + |\tmop{Im} z|^{1/A} + |\tmop{Im} z| \cdot (\log x)^{-1}$. 
  We now evaluate $M_k(x;z)
  \assign \sum_{n \leqslant x, (n, P_k) = 1} e^{zf (n)}$. Using the definition
  of $\Lambda_f (z ; n)$, equation $(4.18)$, corollary 4.6, and the previous
  line, we get
  \begin{eqnarray*}
    W_k (x ; z) & = & \sum_{\tmscript{\begin{array}{c}
      n \leqslant x\\
      (n, P_k) = 1
    \end{array}}} e^{zf (n)} \cdot \log n \text{ } = \text{ }
    \sum_{\tmscript{\begin{array}{c}
      d \leqslant x\\
      (d, P_k) = 1
    \end{array}}} e^{zf (d)} \sum_{\tmscript{\begin{array}{c}
      n \leqslant x / d\\
      (n, P_k) = 1
    \end{array}}} \Lambda_f (z ; n)\\
    & = & \sum_{\tmscript{\begin{array}{c}
      d \leqslant x\\
      (d, P_k) = 1
    \end{array}}} e^{zf (d)} \cdot \left[ \hat{\Psi} (f ; z) (x / d) + O_{A,
    C} \left( \frac{x}{d} \cdot \frac{1 + | \tmop{Im} z|}{(\log 2 x / d)^{2
    A}} \right) \right]\\
    & = & \hat{\Psi} (f ; z) \cdot x m_k (x ; z) + O_{A, C} \left( x \cdot
    \frac{1 + | \tmop{Im} z|}{(\log x)^{A}} \right)\\
    & = & A_2 (f ; z) \cdot x (\log x)^{\hat{\Psi} (f ; z)} + O \left( 
    \mathcal{E}_A(x;z) \cdot x (\log x)^{\hat{\Psi} (f ; \kappa) - 1}
    \right)
  \end{eqnarray*}
  uniformly in $| \tmop{Re} z| \leqslant C$ and where $A_2 (f ; z) \assign (1
  + \hat{\Psi} (f ; z)) \hat{\Psi} (f ; z) A_1 (f ; z)$. In the second line
  above, we bound $\sum e^{\kappa f (d)} \cdot d^{- 1} \cdot (\log 2 x / d)^{-
  2 A}$ by applying Cauchy-Schwarz's inequality and using Corollary 4.6 
  (also recall that $A \geqslant 4\hat{\Psi}(f;2C)+4$).
  Integrating by parts our previous result we conclude that the mean value
  $M_k (x ; z)$ equals to
  \[ M_k (x ; z) \assign \int_2^x \frac{\mathd W_k (t ; z)}{\log t} =
     \frac{W_k (x ; z)}{\log x} + \int_2^x \frac{W_k (t ; z)}{t (\log t)^2}
     \mathd t \]
  Corollary 4.6 yields the bound $|W_k (t ; z) | \leqslant W_k (t ; \kappa) =
  O_C (t \cdot (\log t)^{\hat{\Psi} (f ; \kappa)})$ where as usual $\kappa
  \assign \tmop{Re} z$. It follows that the second integral in the above
  equation is bounded by $x \cdot (\log x)^{\hat{\Psi} (f ; \kappa) - 2}$. We
  conclude that
  \[ M_k (x ; z) = A_2 (f ; z) \cdot x (\log x)^{\hat{\Psi} (f ; z) - 1} +
     O_{A, C} \left( \mathcal{E}_A(x;z) \cdot x (\log x)^{\hat{\Psi}
     (f ; \kappa) - 2} \right) \]
  It remains to estimate $M (x ; z) = \sum_{n \leqslant x} e^{zf (n)}$. At
  this point recall that the function $A_1 (f ; z) = O_C (1 + | \tmop{Im}
  z|^{1 / A})$ and that $A_2 (f ; z) = \hat{\Psi} (f ; z) (1 + \hat{\Psi} (f ;
  z)) A_1 (f ; z)$ hence the same bound holds for $A_2 (f ; z)$. Let $g (z ;
  n)$ be a multiplicative function defined by $g (z ; p^{\ell}) = \exp (zf
  (p^{\ell}))$ when $p \leqslant k$ and $g (z ; p^{\ell}) = 0$ otherwise. We have
  \[ M (x ; z) = \sum_{d \leqslant x} g (z ; d)
     \sum_{\tmscript{\begin{array}{c}
       n \leqslant x / d\\
       (n, P_k) = 1
     \end{array}}} e^{zf (n)} = \sum_{d \leqslant x} g (z ; d) M_k (x / d ; z)
  \]
  Using our estimate for $M_k (x / d ; z)$ this simplifies to
  \[ M (x ; z) = A_3 (f ; z) \cdot x (\log x)^{\hat{\Psi} (f ; z) - 1} +
     O_{A, C} \left( \mathcal{E}_A(x;z) \cdot x (\log x)^{\hat{\Psi}
     (f ; \kappa) - 2} \right) \]
  where $A_3 (f ; z) = \prod_{p \leqslant k} (1 + e^{zf (p)} \cdot (p - 1)^{- 1}) A_2
  (f ; z)$ is analytic. It remains to show that $A_3 (f ; z) = L (f ; z) /
  \Gamma ( \hat{\Psi} (f ; z))$. Here, we use an abelian argument. Consider
  the two-variable function.
  \[ L_f (s ; z) \assign \prod_p \left( 1 - \frac{1}{p^s} \right)^{\hat{\Psi}
     (f ; z)} \cdot \left( 1 + \frac{e^{zf (p)}}{p^s - 1} \right) \]
  Mimicking the proof of lemma 4.4 it is not too hard to prove that $L_f (s ;
  \kappa)$ is uniformly bounded when $1 \leqslant s \leqslant 2$ and 
  $\kappa \in [0
  ; \delta]$ for some $\delta > 0$. In addition by {\cite{5}} (corollary to 
  lemma 1.1.7) for fixed $\kappa \geqslant 0$ the function
  $L_f(s;\kappa)$ is right continuous at $s = 1$, when $s$ is going through
  the reals. Thus $L_f(s;\kappa) \rightarrow L_f(1;\kappa) = L(f;\kappa)$ 
  for fixed $\kappa$
  and as $s \rightarrow 1^+$. Furthermore we have the factorization
  \begin{eqnarray}
    &  & L_f (s ; \kappa) \zeta (s)^{\hat{\Psi} (f ; \kappa)} = \sum_{n
    \geqslant 1} \frac{e^{\kappa f (n)}}{n^s} = s \int_1^{\infty} M (t ;
    \kappa) t^{- s - 1} \mathd t \\
    & = & A_3 (f ; \kappa) \cdot s \int_1^{\infty} (\log t)^{\hat{\Psi} (f ;
    \kappa) - 1} \cdot t^{- s} \mathd t + O_{\delta} \left( \int_1^{\infty}
    (\log t)^{\hat{\Psi} (f ; \kappa) - 2} \cdot t^{- s} \mathd t \right)
    \nonumber
  \end{eqnarray}
  By a change of variable $u \assign \log t$ the first integral becomes
  \[ \int_1^{\infty} (\log t)^{\hat{\Psi} (f ; \kappa) - 1} \cdot t^{- s}
     \mathd t = \int_0^{\infty} e^{- t (s - 1)} \cdot t^{\hat{\Psi} (f ;
     \kappa) - 1} \mathd t = \frac{\Gamma ( \hat{\Psi} (f ; \kappa))}{(s -
     1)^{\hat{\Psi} (f ; \kappa)}} \]
  Therefore $(4.23)$ can be re-written as
  \[ L_f (s ; \kappa) \cdot \zeta (s)^{\hat{\Psi} (f ; \kappa)} = A_3 (f ;
     \kappa) \Gamma ( \hat{\Psi} (f ; \kappa)) s (s - 1)^{- \hat{\Psi} (f ;
     \kappa)} + O ((s - 1)^{- \hat{\Psi} (f ; \kappa) + 1}) \]
  Choose $s = 1 + 1 / \log x$ and fix $\kappa$. By our earlier remark $L_f (s
  ; \kappa) = L (f ; \kappa) + o (1)$. Furthermore $\zeta (s) \sim 1 / (s -
  1)$. Therefore the previous equation turns into
  \[ L (f ; \kappa) - A_3 (f ; \kappa) \Gamma ( \hat{\Psi} (f ; \kappa)) = o
     (1) \]
  It follows that $A_3 (f ; \kappa) = L (f ; \kappa) / \Gamma ( \hat{\Psi} (f
  ; \kappa))$. Since both functions are analytic in $| \tmop{Re} z| \leqslant
  C$ and coincide on a compact interval we get $A_3 (f ; z) = L (f ; z) /
  \Gamma ( \hat{\Psi} (f ; z))$ for all $| \tmop{Re} z| \leqslant C$. It now
  follows that
  \[ \frac{1}{x} \sum_{n \leqslant x} e^{zf (n)} = \frac{L (f ; z)}{\Gamma (
     \hat{\Psi} (f ; z))} \cdot (\log x)^{\hat{\Psi} (f ; z) - 1} + O_{A, C}
     \left( \mathcal{E}_A(x;z) \cdot (\log
     x)^{\hat{\Psi} (f ; \kappa) - 2} \right) \]
  uniformly in $| \tmop{Re} z| \leqslant C$ which is the desired claim. 
\end{proof}

%% End new proof

\subsection{Two simple estimates for $v_f (x ; \Delta)$}

In the next lemma we collect a few useful facts about $v_f (x ; \Delta)$.
First we prove that $v_f (x ; \Delta)$ is essentially $\Delta / \sigma_{\Psi}
(f ; x)$.

\begin{lemma}
  Let $f \in \mathcal{C}$. Given $\delta > 0$ uniformly in $1 \leqslant \Delta
  \leqslant \delta \sigma_{\Psi} (f ; x)$,
  \[ v_f (x ; \Delta) \text{ } \asymp_{\delta} \text{ } \Delta / \sigma_{\Psi}
     (f ; x) \]
  Furthermore $v_f (x ; \Delta) \sim \Delta / \sigma_{\Psi} (f ; x)$ in the $1
  \leqslant \Delta \leqslant o (\sigma_{\Psi} (f ; x))$ range. Finally the 
  function $\omega(f;z)$ is analytic in a neighborhood of $\mathbb{R}^{+} \cup
  \{0\}$ 
\end{lemma}

\begin{proof}
  Consider the function $\omega (f ; z)$ defined implicitly by
  \[ \hat{\Psi}' (f ; \omega (f ; z)) = \hat{\Psi}' (f ; 0) + z \cdot
     \hat{\Psi}'' (f ; 0) \]
  Note that by definition $v = v_f (x ; \Delta) \assign \omega (f ; \Delta /
  \sigma_{\Psi} (f ; x))$. Since $\hat{\Psi}'' (f ; x) \neq 0$ for all $x
  \geqslant 0$, by Lagrange's inversion the function $\omega (f ; z)$ is
  analytic in a neighborhood of $\mathbbm{R}^+ \cup \{0\}$. Therefore
  \begin{eqnarray}
    v_f (x ; \Delta) & = & \omega (f ; \Delta / \sigma_{\Psi} (f ; x)) =
    \Delta / \sigma_{\Psi} (f ; x) + O \left( (\Delta / \sigma_{\Psi} (f ;
    x)^2 \right) 
  \end{eqnarray}
  Therefore for $\Delta \leqslant c \sigma_{\Psi} (f ; x)$ and $c$ small
  enough $v_f (x ; \Delta) \asymp \Delta / \sigma_{\Psi} (f ; x)$. In the
  remaining range $c \leqslant \Delta / \sigma_{\Psi} (f ; x) \leqslant
  \delta$ it is clear that $v_f (x ; \Delta) \asymp 1 \asymp \Delta /
  \sigma_{\Psi} (f ; x)$: indeed, $v_f (x ; \Delta) = \omega (f ; \Delta /
  \sigma_{\Psi} (f ; x))$, the function $\omega (f ; x)$ is positive and
  continuous for $x \geqslant 0$ and $\Delta / \sigma_{\Psi} (f ; x)$ belongs
  to a bounded interval. Also, the second assertion of the lemma follows
  immediately from $(4.24)$. 
\end{proof}

\begin{lemma}
  Let $f \in \mathcal{C}$. As usual let $\xi_f (x ; \Delta) \assign \mu (f ;
  x) + \Delta \sigma (f ; x)$. For any given $\delta > 0$, we have uniformly
  in $1 \leqslant \Delta \leqslant \delta \sigma (f ; x)$,
  \begin{eqnarray*}
    \xi_f (x ; \Delta) & = & \hat{\Psi}' (f ; v_f (x ; \Delta)) \cdot
    \tmop{loglog} x + c (f) + O_{\delta} \left( \frac{1}{\sqrt[]{\tmop{loglog}
    x}} \right)
  \end{eqnarray*}
\end{lemma}

\begin{proof}
  Integrating by parts the result of lemma 4.3 gives an estimate for the
  average $\sum_{p \leqslant x} e^{sf (p)} / p$. Differentiating using
  Cauchy's formula and setting $s = 0$ we find that
  \begin{eqnarray*}
    \mu (f ; x) & = & \hat{\Psi}' (f ; 0) \cdot \tmop{loglog} x + c (f) + O
    \left( \frac{1}{\sqrt[]{\log x}} \right)\\
    \sigma^2 (f ; x) & = & \hat{\Psi}'' (f ; 0) \cdot \tmop{loglog} x + O
    \left( 1) \right.
  \end{eqnarray*}
  By definition of $v_f (x ; \Delta)$ we have
  \begin{eqnarray*}
    \hat{\Psi}' (f ; v_f (x ; \Delta)) \cdot \tmop{loglog} x & = & \hat{\Psi}'
    (f ; 0) \cdot \tmop{loglog} x + \Delta ( \hat{\Psi}'' (f ; 0)
    \tmop{loglog} x)^{1 / 2}\\
    & = & \mu (f ; x) - c (f) + \Delta \sigma (f ; x) + O \left(
    \frac{\Delta}{\tmop{loglog} x} \right)\\
    & = & \mu (f ; x) - c (f) + \Delta \sigma (f ; x) + O_{\delta} \left(
    \frac{1}{\sqrt[]{\tmop{loglog} x}} \right)
  \end{eqnarray*}
  and the claim follows. 
\end{proof}

\subsection{Large deviations when $1 \leqslant \Delta = o ((\tmop{loglog}
x)^{1 / 6})$}

The following is a consequence of a result of Hwang {\cite{11}} (see the
statement of the main result in 1.1 and then Corollary 3).

\begin{proposition}
  Let $f \in \mathcal{C}$. Let $\Omega (f ; x)$ be a sequence of random
  variables, such that
  \[ \mathbbm{E} \left[ e^{s \Omega (f ; x)} \right] =\mathcal{A}(s) \cdot
     \left( \log x \right)^{\hat{\Psi} (f ; s) - 1} \cdot \left( 1 + o_{x
     \rightarrow \infty} (1) \right) \]
  uniformly in $|s| \leqslant \varepsilon$ for some $\varepsilon > 0$
  sufficiently small and with $\mathcal{A}(s)$ analytic and non-zero in a
  neighborhood of $s = 0$. Then, uniformly in $1 \leqslant \Delta \leqslant o
  \left( \sigma (f ; x)^{1 / 3} \right)$,
  \[ \mathbbm{P} \left( \frac{\Omega (f ; x) - \mu (f ; x)}{\sigma (f ; x)}
     \geqslant \Delta \right) \sim \int_{\Delta}^{\infty} e^{- u^2 / 2} \cdot
     \frac{\mathd u}{\sqrt[]{2 \pi}} \]
\end{proposition}

For all interesting $\Omega (f ; x)$ we will be able to determine asymptotics
for
\[ \mathbbm{P} \left( \frac{\Omega (f ; x) - \mu (f ; x)}{\sigma (f ; x)}
   \geqslant \Delta \right) \]
when $\Delta$ is in the range $(\tmop{loglog} x)^{\varepsilon} \ll \Delta
\leqslant c \sigma (f ; x)$. Hwang's lemma will be used to complement these
results -- that is, handle the (easy) range $1 \leqslant \Delta \leqslant o
((\tmop{loglog} x)^{1 / 6})$. Let us note that Maciulis {\cite{14}} proved a
result similar to proposition 4.9, but much earlier. The drawback of his
result is that it is harder to use because of the many parameters introduced
in the statement.

\subsection{Large deviations: $(\tmop{loglog} x)^{\varepsilon} \ll \Delta \ll
\sigma (f ; x)$ and $\Psi (f ; t)$ non-lattice}

The object of this section is to prove the following (general) lemma.

\begin{proposition}
  Let $f \in \mathcal{C}$. Suppose that $\Psi (f ; t)$ is not lattice
  distributed. Let $\Omega (f ; x)$ be a sequence of random variables such
  that, for any given $C > 0$, uniformly in $0 \leqslant \kappa \assign
  \tmop{Re} s \leqslant C$ and $| \tmop{Im} s| \leqslant \tmop{loglog} x$,
  \[ \mathbbm{E} \left[ e^{s \Omega (f ; x)} \right] =\mathcal{A}(s) \cdot
     \left( \log x \right)^{\hat{\Psi} (f ; s) - 1} + O_C \left( (\log
     x)^{\hat{\Psi} (f ; \kappa) - 3 / 2} \right) \]
  Here $\mathcal{A}(s)$ is analytic in $\tmop{Re} s \geqslant 0$ and 
  non-vanishing on $\mathbb{R}^{+} \cup \{0\}$. Assume that
  $\mathcal{A}(s) \ll_C (1 + | \tmop{Im} s|^{1 / 8})$ holds throughout $0
  \leqslant \tmop{Re} s \leqslant C$. Then, given $\delta, \varepsilon > 0$,
  uniformly in $(\tmop{loglog} x)^{\varepsilon} \ll \Delta \leqslant \delta
  \sigma (f ; x)$,
  \[ \mathbbm{P}(\Omega (f ; x) \geqslant \mu (f ; x) + \Delta \sigma (f ; x))
     \sim \mathcal{A}(v) \cdot \frac{\left( \log x \right)^{\hat{\Psi} (f ; v)
     - 1 - v \hat{\Psi}' (f ; v)}}{v (2 \pi \hat{\Psi}'' (f ; v) \tmop{loglog}
     x)^{1 / 2}} \cdot e^{- vc (f)} \]
  where $v \assign v_f (x ; \Delta)$ is the unique positive solution to the
  equation
  \[ \hat{\Psi}' (f ; v) \cdot \tmop{loglog} x = \hat{\Psi}' (f ; 0) \cdot
     \tmop{loglog} x + \Delta ( \hat{\Psi}'' (f ; 0) \tmop{loglog} x)^{1 / 2}
  \]
  and $c (f)$ as in the statement of Theorem 2.8 (or see section 3).
\end{proposition}

It is possible to prove proposition 4.10 using the method of ``associated
distribution'' due to Cramer {\cite{2}}. \ The method presented here is more
concise, and avoids some of the redundancy inherent in Cramer's method.
One of the peculiarity of our method is that it seems to require an
asymptotic for $\mathbbm{E}[e^{s \Omega (f ; x)}]$ in the range $| \tmop{Re}
s| \leqslant C$ and $| \tmop{Im} s| \leqslant \psi (x)$ for some $\psi (x)
\rightarrow \infty$, whereas Cramer's methods needs only an assumption on the
range $|s| \leqslant C$, for $C$ big enough.

Our proof relies on the following six lemmata. The first lemma is
``well-known''. A proof can be found in Petrov's book {\cite{15}} (or in
Ess\'een's thesis {\cite{7}}, theorem 5, p. 26).

\begin{lemma}
  A distribution function $F (t)$ is not lattice distributed if and only if
  for all $t \neq 0$ the Fourier transform $\phi (t) = \int_{\mathbbm{R}}
  e^{\tmop{it} u} \mathd F (u)$ has modulus $< 1$.
\end{lemma}

Lemma 4.11 admits the following consequence.

\begin{lemma}
  Let $f \in \mathcal{C}$. Suppose that $\Psi (f ; t)$ is not lattice
  distributed. Then
  \[ \phi (t) = e^{\hat{\Psi} (f ; \tmop{it}) - 1} \]
  is the Fourier transform of a non-lattice distribution function.
  Furthermore, for any $w \geqslant 0$ and $t \in \mathbbm{R}$, we have
  \[ | \exp ( \hat{\Psi} (f ; w + \tmop{it}) - \hat{\Psi} (f ; w)) | \leqslant
     | \phi (t) | \]
\end{lemma}

\begin{proof}
  For a distribution function $F$ denote by $F^{\ast n}$ the $n$-fold
  convolution of $F$ with itself. Consider the distribution function
  \[ D (f ; t) = \frac{1}{e} \sum_{k \geqslant 0} \Psi^{\ast k} \left( f ; t
     \right) \cdot \frac{1}{k!} \]
  The Fourier transform of $D (f ; t)$ is given by
  \begin{eqnarray*}
    \int_{\mathbbm{R}} e^{\tmop{it} u} \mathd D (f ; u) & = & \frac{1}{e}
    \sum_{k \geqslant 0} \frac{1}{k!} \int_{\mathbbm{R}} e^{\tmop{it} u}
    \mathd \Psi^{\ast k} \left( f ; u \right)\\
    & = & \frac{1}{e} \sum_{k \geqslant 0} \frac{1}{k!} \cdot \hat{\Psi}
    \left( f ; \tmop{it} \right)^k \text{ } = \text{ } e^{\hat{\Psi} (f ;
    \tmop{it}) - 1}
  \end{eqnarray*}
  This proves existence. Furthermore, since $\Psi (f ; t)$ is not lattice
  distributed, by Lemma 4.11, we have $| \hat{\Psi} (f ; \tmop{it}) | < 1$ for
  all $t \neq 0$. Therefore $|e^{\hat{\Psi} (f ; \tmop{it}) - 1} | < 1$ for
  all $t \neq 0$. Hence by Lemma 4.11, $e^{\hat{\Psi} (f ; \tmop{it}) - 1}$ is
  the Fourier transform of a non-lattice distribution function. Finally, for
  the last statement of this lemma, let us note that
  \begin{eqnarray*}
    \tmop{Re} \left( \hat{\Psi} (f ; w + \tmop{it}) - \hat{\Psi} (f ; w)
    \right) & = & \int_0^{\infty} e^{w u} \cdot \left( \cos (t u) - 1 \right)
    \mathd \Psi (f ; u)\\
    & \leqslant & \int_0^{\infty} (\cos (t u) - 1) \mathd \Psi (f ; u) =
    \tmop{Re} \left( \hat{\Psi} (f ; \tmop{it}) - 1 \right)
  \end{eqnarray*}
  Note that $\Psi (f ; u) = 0$ for $u < 0$, this is why we are allowed to
  ``forget'' about integrating over $- \infty < u \leqslant 0$. 
\end{proof}

The next lemma is taken from Ess\'een's thesis {\cite{7}} (see Lemma 1 on page
49).

\begin{lemma}
  Let $F (t)$ be a distribution function and denote by $\phi (t)$ it's Fourier
  transform $\int_{\mathbbm{R}} e^{\tmop{it} u} \mathd F (u)$. If $F (t)$ is
  not lattice-distributed, then, for any $c > 0$ there is a $\lambda (x)
  \rightarrow \infty$ and a $\xi (x) \rightarrow \infty$ such that
  \begin{eqnarray*}
    \int_c^{\lambda (x)} | \phi (t) |^x \cdot \frac{\mathd t}{t} & \ll &
    \frac{1}{\xi (x) \cdot \sqrt[]{x}}
  \end{eqnarray*}
\end{lemma}

From lemma 4.13 and lemma 4.12 we obtain the following useful estimate.

\begin{lemma}
  Let $f \in \mathcal{C}$. Suppose that $\Psi (f ; t)$ is not lattice
  distributed. Then, for any $c > 0$ there is a $\lambda (x) \rightarrow
  \infty$ and a $\xi (x) \rightarrow \infty$ such that uniformly in 
  $w \geqslant 0$,
  \begin{eqnarray*}
    \int_c^{\lambda (x)} \left| \left( \log x \right)^{\hat{\Psi} (f ; w +
    \tmop{it}) - \hat{\Psi} (f ; w)} \right| \cdot \frac{\mathd t}{t} & \ll &
    \frac{1}{\xi (x) \cdot \sqrt[]{\tmop{loglog} x}}
  \end{eqnarray*}
\end{lemma}

\begin{proof}
  Since $\Psi (f ; t)$ is not lattice distributed, by lemma 4.12 the function
  $\phi (t) = e^{\hat{\Psi} (f ; \mathi t) - 1}$ is the Fourier transform of a
  non-lattice distribution function. Therefore by lemma 4.13, given any $c >
  0$ there is a $\lambda (x) \rightarrow \infty$ and a $\xi (x) \rightarrow
  \infty$ such that
  \begin{eqnarray}
    \int_c^{\lambda (x)} \left| e^{\hat{\Psi} (f ; \tmop{it}) - 1}
    \right|^{\tmop{loglog} x} \cdot \frac{\mathd t}{t} & \ll & \frac{1}{\xi
    (x) \cdot \sqrt[]{\tmop{loglog} x}} 
  \end{eqnarray}
  By lemma 4.12 we have for all $w \geqslant 0$,
  \[ \int_c^{\lambda (x)} \left| \left( \log x \right)^{\hat{\Psi} (f ; w +
     \tmop{it}) - \hat{\Psi} (f ; w)} \right| \cdot \frac{\mathd t}{t}
     \leqslant \int_c^{\lambda (x)} \left| e^{\hat{\Psi} (f ; \tmop{it}) - 1}
     \right|^{\tmop{loglog} x} \cdot \frac{\mathd t}{t} \]
  This together with $(4.25)$ gives the claim. 
\end{proof}

We need one more lemma from Ess\'een's thesis {\cite{7}} (see theorem 6 on page
27).

\begin{lemma}
  Let $F (t)$ be a distribution function. Suppose that $F (t)$ is not
  degenerate (that is $F (t)$ does not have a jump of mass 1). Denote by $\phi
  (t)$ the Fourier transform $\int_{\mathbbm{R}} e^{\mathi t u} \mathd F (u)$
  of the distribution function $F (\cdot)$. There is a $c_0$ and a $c_1$ such
  that for any interval $I$ of size less than $c_0$
  \[ \underset{u \in I}{\tmop{meas}} \left( | \phi (u) |^2 \geqslant 1 -
     \delta \right) \leqslant c_1 \cdot \sqrt[]{\delta} \]
  The constants $c_1$ and $c_0$ depend at most on the distribution function
  $F$. 
\end{lemma}

Finally, we need one last lemma that will allow us to smooth out
$\mathbbm{P}(\Omega (f ; x) \geqslant t)$ (the smoothing will be negligible
because $\Psi (f ; t)$ is not lattice distributed). The lemma is essentially
what appears in Tenenbaum {\cite{19}} (first formula in section 3 and first
formula in section 4 of his paper).

\begin{lemma}
  Let $Y (x)$ be a sequence of random variables. Suppose that each $Y (x)$ has
  an entire moment generating function and define
  \[ \Phi_Y (x ; t) (z) =\mathbbm{E} \left[ e^{zY (x)} \right] \cdot e^{- zt}
  \]
  Let $C > 0$ be given. Then for all $\kappa > 0$ and $M, T > 0$, we have
  \begin{eqnarray}
    \mathbbm{P}(Y (x) \geqslant t) & = & \frac{1}{2 \pi i} \int_{\kappa -
    iM}^{\kappa + iM} \Phi_Y (x ; t) (z) \cdot \frac{T \mathd z}{z (z + T)}
    \nonumber\\
    &  & + O \left( \tmop{Err} \right) + O \left( \frac{Te^{\kappa / T}}{M}
    \cdot \Phi_Y (x ; t) (\kappa) \right) 
  \end{eqnarray}
  and the error term Err is given by
  \[ \tmop{Err} = \frac{1}{2 \pi i} \int_{\kappa - iM}^{\kappa + iM} \Phi_Y (x
     ; t) (z) \cdot \frac{e^{z / T} \cdot T \mathd z}{(z + T) (z + 2 T)} \]
\end{lemma}

\begin{proof}
  First we establish the above when $M = \infty$. This case follows from the
  inequalities appearing in Tenenbaum's paper. Let $y^+ = \max (y ; 0)$.
  Following Tenenbaum {\cite{19}} (see first equation in section 3)
  we have
  \begin{eqnarray*}
    1 - e^{- Ty^+} & = & \frac{1}{2 \pi i} \int_{\kappa - i \infty}^{\kappa +
    i \infty} e^{zy} \cdot \frac{T \mathd z}{z (z + T)}
  \end{eqnarray*}
  for all $y \in \mathbbm{R}$ and $\kappa, T \geqslant 0$. Let $\chi (\cdot)$
  denote the characteristic function of $[0 ; \infty)$. Again according to
  Tenenbaum's paper (see beginning of section 4), we have the inequality
  \begin{eqnarray*}
    0 \text{ } \leqslant \text{\, } \chi (y) - \left( 1 - e^{- Ty^+} \right) &
    \leqslant & \frac{e^2}{e - 1} \cdot \left( e^{- T \left( y + 1 / T)^+
    \right.} - e^{- 2 T (y + 1 / T)^+} \right)\\
    & = & \frac{e^2}{e - 1} \cdot \frac{1}{2 \pi i} \int_{\kappa - i
    \infty}^{\kappa + i \infty} e^{zy} \cdot \frac{e^{z / T} \cdot T \mathd
    z}{(z + T) (z + 2 T)}
  \end{eqnarray*}
  It follows that for $u, t \in \mathbbm{R}$,
  \begin{eqnarray*}
    0 & \leqslant & \chi (u - t) - \frac{1}{2 \pi i} \int_{\kappa - i
    \infty}^{\kappa + i \infty} e^{zu} \cdot e^{- zt} \cdot \frac{T \mathd
    z}{z (z + T)}\\
    & \leqslant & \frac{K}{2 \pi i} \int_{\kappa - i \infty}^{\kappa + i
    \infty} e^{zu} \cdot e^{- zt} \cdot \frac{e^{z / T} \cdot T \mathd z}{(z +
    T) (z + 2 T)}
  \end{eqnarray*}
  with $K = e^2 / (e-1)$. Integrating the above inequality over $u$, 
  with respect to the measure
  $\mathd \mathbbm{P}(Y (x) \leqslant u)$ and applying Fubini's theorem we
  obtain the claim, in the case $M = \infty$. To obtain the general case, note
  that |$\Phi_Y (x ; t) (z) | \leqslant \Phi_Y (x ; t) (\kappa)$ for 
  $\tmop{Re} z = \kappa$ and let
  $\mathcal{R}(\kappa, M) \assign \{\kappa + \mathi t : |t| \geqslant M\}$. By
  the previous inequality for $\Phi_Y$,
  \[ \left| \int_{\mathcal{R}(\kappa, M)} \Phi_Y (x ; t) (z) \cdot \frac{T
     \mathd z}{z (z + T)} \right| \leqslant \Phi_Y (x ; t) (\kappa) \cdot 2
     \int_M^{\infty} \frac{T \mathd t}{t^2} \leqslant \frac{2 T}{M} \cdot
     \Phi_Y (x ; t) (\kappa)_{} \]
  Therefore
  \[ \frac{1}{2 \pi i} \int_{\kappa - \i\infty}^{\kappa + \i\infty}
     \Phi_Y(x;t)(z) \frac{T \mathd z}{z(z+T)} = 
     \frac{1}{2 \pi i} \int_{\kappa - iM}^{\kappa + iM} \Phi_Y (x ; t) (z)
     \cdot \frac{T \mathd z}{z (z + T)} + O \left( \Phi_Y (x ; t) (\kappa)
     \cdot \frac{T}{M} \right) \]
  We truncate the integral appearing in the term $\tmop{Err}$ in a similar
  fashion. In this case the truncation contributes $O (Te^{\kappa / T} / M
  \cdot \Phi_Y (x ; t) (\kappa))$. Having truncated our integrals we obtained
  the ``general'' case of our lemma.
\end{proof}

We are now in position to prove proposition 4.10.

%% Cut here

\begin{proof}[Proof of Proposition 4.10]
  Let's keep the notation $\Phi_{\Omega} (x ; t) =\mathbbm{E}[e^{z \Omega (f ;
  x)}] \cdot e^{- zt}$ introduced in lemma 4.16 and abbreviate $\mu \assign
  \mu (f ; x)$, $\sigma \assign \sigma (f ; x)$. Throughout we set $z \assign
  v + \mathi t = v_f (x ; \Delta) + \mathi t$ with $t \in \mathbbm{R}$ and we
  abbreviate $v \assign v_f (x ; \Delta)$. Note that by lemma 4.7 there is a
  $C = C (\delta) > 0$ such that $0 \leqslant v \leqslant C$ when $\Delta$ is
  in the range $1 \leqslant \Delta \leqslant \delta \sigma (f ; x)$. 
  (We allow our error term to depend on $C$). Also by lemma 4.8, 
  $e^{- v (\mu + \Delta \sigma)} \asymp (\log x)^{- v
  \hat{\Psi}' (f ; v)}$. Thus, by assumptions and this estimate
  \begin{eqnarray}
    &  & \Phi_{\Omega} (x ; \mu + \Delta \sigma) (z) =\mathcal{A}(z) (\log
    x)^{\hat{\Psi} (f ; z) - 1} e^{- z (\mu + \Delta \sigma)} + O ((\log
    x)^{\hat{\Psi} (f ; v) - 3 / 2} e^{- v (\mu + \Delta \sigma)}) \nonumber\\
    & = & \mathcal{A}(z) (\log x)^{\hat{\Psi} (f ; z) - 1} \cdot e^{- z (\mu
    + \Delta \sigma)} + O \left( (\log x)^{A (f ; v) - 1 / 2} \right) 
  \end{eqnarray}
  for $0 \leqslant v \assign \tmop{Re} z \leqslant C$ and $| \tmop{Im} z|
  \leqslant \tmop{loglog} x$ and where $A (f ; v) \assign \hat{\Psi} (f ; v) - 1 -
  v \hat{\Psi}' (f ; v)$. We insert $(4.27)$ into $(4.26)$ of the previous
  lemma. In there we set $\kappa \assign v_f (x ; \Delta)$, $Y (x) \assign
  \Omega (f ; x)$, $M \assign \tmop{loglog} x$ and $T \assign \sqrt[]{\lambda
  (x)} \longrightarrow \infty$. The function $\lambda (x)$ is $\ll
  \tmop{logloglog} x$ and tends to infinity as $x \rightarrow \infty$. It will
  be specified explicitly later on. We get
  \begin{eqnarray}
    &  & \mathbbm{P}(\Omega (f ; x) \geqslant \mu + \Delta \sigma) \assign
    \frac{1}{2 \pi i} \int_{v - iM}^{v + iM} \mathcal{A}(z) (\log
    x)^{\hat{\Psi} (f ; z) - 1} e^{- z (\mu + \Delta \sigma)} \cdot \frac{T
    \mathd z}{z (z + T)} \nonumber\\
    &  & + O \left( \frac{1}{2 \pi i} \int_{v - iM}^{v + iM} \mathcal{A}(z)
    (\log x)^{\hat{\Psi} (f ; z) - 1} e^{- z (\mu + \Delta \sigma)} \cdot
    \frac{e^{z / T} \cdot T \mathd z}{(z + T) (z + 2 T)} \right) \\
    &  & + O \left( \int_{v - iM}^{v + iM} (\log x)^{A (f ; v) - 1 / 2} \cdot
    \frac{e^{v / T} \cdot T |\mathd z|}{|z| \cdot |z + T|} \right) + O \left(
    \frac{T}{M} \cdot \Phi_{\Omega} (x ; \mu + \Delta \sigma) (v) \right)
    \nonumber
  \end{eqnarray}
  At the outset note that the very last error term is negligible. Indeed, by
  $(4.27)$ and the boundedness of $\mathcal{A}(v)$ in $0 \leqslant v \leqslant
  C$ (the function $\mathcal{A}(\cdot)$ is continuous!), we have
  $\Phi_{\Omega} (x ; \mu + \Delta \sigma) (v) \ll (\log x)^{A (f ; v)}$.
  Therefore $T / M \cdot \Phi_{\Omega} (x ; \mu + \Delta \sigma) (v) \ll
  (\tmop{logloglog} x / \tmop{loglog} x) \cdot (\log x)^{A (f ; v)}$ which is
  negligible compared to the expected size of the main term.
  
  The integral over $T \cdot | \mathd z| / |z| |z + T|$ contributes less than
  $v^{- 1} + T \ll v^{- 1} (1 + T)$. Thus the second error term in
  $(4.28)$ is $\ll (\log x)^{A (f ; v) - 1 / 2} v^{- 1} (1 + T)$. Since $T \ll
  \tmop{logloglog} x$ this error term is negligible compared to the expected
  size of the main term.
  
  Once we evaluate the main term in $(4.28)$ it will be clear how to bound
  the first error term in $(4.28)$. Therefore let's focus on estimating
  \begin{equation}
    \frac{1}{2 \pi i} \int_{v - iM}^{v + iM} \mathcal{A}(z) (\log
    x)^{\hat{\Psi} (f ; z) - 1} e^{- z (\mu + \Delta \sigma)} \cdot \frac{T
    \mathd z}{z \cdot (z + T)}
  \end{equation}
  This corresponds to the main term for $\mathbbm{P}(\Omega (f ; x) \geqslant
  \mu (f ; x) + \Delta \sigma (f ; x))$. We split $(4.29)$ into a part over
  $\mathcal{M} \assign \{v + \mathi t : |t| \leqslant \eta (x) \cdot
  (\tmop{loglog} x)^{- 1 / 2} \}$ where $\eta (x) = \tmop{logloglog} x$ and a
  part over $\mathcal{R}=\{v + \mathi t : |t| \leqslant M\}-\mathcal{M}$. The
  part over $\mathcal{M}$ will furnish the main term and the part over
  $\mathcal{R}$ will be negligible.
  
  \  
  
  {\tmem{1. Asymptotic for $(4.29)$ restricted to $z = v + \mathi t \in
  \mathcal{M}$.}}
  
  \
  
  By lemma 4.8 for $z \in \mathcal{M}= \left\{ v + \mathi t : |t| \leqslant
  \eta (x) \cdot (\tmop{loglog} x)^{- 1 / 2} \right\}$,
  \[ e^{- z (\mu + \Delta \sigma)} = (\log x)^{- z \hat{\Psi}' (f ; v)} \cdot
     e^{- zc (f)} \cdot (1 + O ((\tmop{loglog} x)^{- 1 / 2})) \]
  Therefore, for $z \in \mathcal{M}$,
  \begin{eqnarray*}
    &  & \mathcal{A}(z) (\log x)^{\hat{\Psi} (f ; z) - 1} \cdot e^{- z (\mu +
    \Delta \sigma)}\\
    & = & \mathcal{A}(z) e^{- zc (f)} \cdot (\log x)^{\hat{\Psi} (f ; z) - 1
    - z \hat{\Psi}' (f ; v)} \cdot \left( 1 + O \left( (\tmop{loglog} x)^{- 1
    / 2} \right) \right)\\
    & = & \mathcal{A}(z) e^{- zc (f)} \cdot (\log x)^{\hat{\Psi} (f ; z) - 1
    - z \hat{\Psi}' (f ; v)} + O \left( (\log x)^{A (f ; v)} \cdot
    (\tmop{loglog} x)^{- 1 / 2} \right)
  \end{eqnarray*}
  In the third line we use the fact that $| \hat{\Psi} (f ; z) | \leqslant
  \hat{\Psi} (f ; v)$ and that $\mathcal{A}(z)$ is analytic hence bounded in
  the (bounded) region $0 \leqslant v \assign \tmop{Re} z \leqslant C$, $|
  \tmop{Im} z| \leqslant 2$. Plugging the above estimate into $(4.29)$ 
  (restricted to $z \in \mathcal{M}$) yields
  \begin{eqnarray}
    &  & \frac{1}{2 \pi i} \int_{\mathcal{M}} \mathcal{A}(z) (\log
    x)^{\hat{\Psi} (f ; z) - 1} \cdot e^{- z (\mu + \Delta \sigma)} \cdot
    \frac{T \mathd z}{z (z + T)} \nonumber\\
    & = & \frac{1}{2 \pi i} \int_{\mathcal{M}} \mathcal{A}(z) e^{- zc (f)}
    \cdot (\log x)^{\hat{\Psi} (f ; z) - 1 - z \hat{\Psi}' (f ; v)} \cdot
    \frac{T \mathd z}{z (z + T)} \\
    &  & + O \left( \frac{(\log x)^{A (f ; v)}}{\sqrt[]{\tmop{loglog} x}}
    \int_{\mathcal{M}} \frac{T \cdot | \mathd z|}{|z| \cdot |z + T|} \right)
    \nonumber
  \end{eqnarray}
  and the error term is bounded by $O ((\log x)^{A (f ; v)} \cdot v^{- 1}
  \cdot \eta (x) (\tmop{loglog} x)^{- 1})$ (note that $\mathcal{M}$ is
  in length $\ll \eta(x)/(\log\log x)^{1/2}$) which is negligible when compared
  to the expected size of the main term. We parametrize the integral in
  $(4.30)$ and perform a series of Taylor expansions. Recall that $z \assign v
  + \mathi t$ by convention, that $0 \leqslant v = v_f (x ; \Delta) \leqslant
  C$ whenever $1 \leqslant \Delta \leqslant \delta \sigma (f ; x)$ and that
  when $z \in \mathcal{M}$ then $|t| \leqslant \eta (x) \cdot (\tmop{loglog}
  x)^{- 1 / 2}$. With this in mind, for $z = v + \mathi t \in \mathcal{M}$,
  \begin{eqnarray*}
    \mathcal{A}(z) e^{- zc (f)} & = & \mathcal{A}(v) e^{- vc (f)} + O_C \left(
    |t| \right) \text{ } = \text{ } \mathcal{A}(v) e^{- vc (f)} + O_C \left(
    \eta (x) \cdot (\tmop{loglog} x)^{- 1 / 2} \right)\\
    & = & \mathcal{A}(v) e^{- vc (f)} \cdot \left( 1 + O_C \left( \eta (x)
    (\tmop{loglog} x)^{- 1 / 2} \right) \right)
  \end{eqnarray*}
  where the last line is justified by 1) the non-vanishing of $\mathcal{A}(x)
  e^{- xc (f)}$ on the positive real axis 2) the fact that $v \asymp \Delta /
  \sigma (f ; x)$ is bounded throughout $1 \leqslant \Delta \leqslant c \sigma
  (f ; x)$ ($0 \leqslant v \leqslant C$). We will not mention any further, 
  the dependence on $C$ in implicit
  constants. Proceeding as in the previous equation, we find
  \begin{eqnarray*}
    \hat{\Psi} (f ; z) - \hat{\Psi} (f ; v) - \tmop{it} \hat{\Psi}' (f ; v) &
    = & - \left( t^2 / 2 \right) \hat{\Psi}'' (f ; v) + O \left( \eta (x)^3
    \cdot (\tmop{loglog} x)^{- 3 / 2} \right)
  \end{eqnarray*}
  for $z = v + \mathi t \in \mathcal{M}$. Upon multiplying by $\tmop{loglog}
  x$ and exponentiating, we obtain
  \[ \left( \log x \right)^{\hat{\Psi} (f ; z) - \hat{\Psi} (f ; v) -
     \tmop{it} \hat{\Psi}' (f ; v)} = e^{- (t^2 / 2) \hat{\Psi}'' (f ; v)
     \tmop{loglog} x} \cdot \left( 1 + O \left( \eta (x)^3 (\tmop{loglog}
     x)^{- 1 / 2} \right) \right) \]
  By lemma 4.7 for $z = v + \mathi t \in \mathcal{M}$ we have $|t / v| \ll
  \eta (x) (\tmop{loglog} x)^{- 1 / 2} \cdot v^{- 1} \asymp \eta (x) \cdot
  \Delta^{- 1}$. Since $\eta (x) = \tmop{logloglog} x = o (\Delta)$ it follows
  that $|t / v| = o (1)$ when $v + \mathi t \in \mathcal{M}$. Since in
  addition we will choose $T \rightarrow \infty$, we have for $z = v + \mathi
  t \in \mathcal{M}$,
  \begin{eqnarray*}
    \frac{T}{z \cdot (z + T)} & = & \frac{1}{v + \tmop{it}} \cdot \frac{1}{1 +
    (v + \tmop{it}) / T}\\
    & = & \frac{1}{v} \cdot \left( 1 + O \left( t / v \right)^{} \right)
    \cdot \left( 1 + O \left( 1 / T \right)^{} \right)\\
    & = & \frac{1}{v} \cdot \left( 1 + O \left( 1 / T + \eta (x)
    (\tmop{loglog} x)^{- 1 / 2} \cdot v^{- 1} \right) \right)
  \end{eqnarray*}
  Collecting together the previous estimates, we conclude that for $z = v +
  \mathi t \in \mathcal{M}$,
  \begin{eqnarray*}
    &  & \mathcal{A}(z) e^{- zc (f)} \cdot \left( \log x \right)^{\hat{\Psi}
    (f ; z) - \hat{\Psi} (f ; v) - \tmop{it} \hat{\Psi}' (f ; v)} \cdot
    \frac{T}{z (z + T)}\\
    & = & (\mathcal{A}(v) e^{- vc (f)} / v) \cdot e^{- (t^2 / 2) \hat{\Psi}''
    (f ; v) \tmop{loglog} x} \cdot \left( 1 + O \left( 1 / T + \eta (x)^3
    \cdot (\tmop{loglog} x)^{- 1 / 2} \cdot v^{- 1} \right) \right)
  \end{eqnarray*}
  Since $v \asymp \Delta \cdot (\tmop{loglog} x)^{- 1 / 2}$ the error term
  simplifies to $\mathcal{E} \assign 1 / T + \eta (x)^3 \cdot \Delta^{- 1}$.
  Let $\xi \assign \eta (x) \cdot (\tmop{loglog} x)^{- 1 / 2}$. Parametrizing
  the integral in $(4.30)$ and using the previous asymptotic we find that 
  the integral in $(4.30)$ equals (where as usual $z \assign v + \mathi t$)
  \begin{eqnarray*}
    &  & \frac{1}{2 \pi} \int_{- \xi}^{\xi} \mathcal{A}(z) e^{- zc (f)}
    \cdot (\log x)^{\hat{\Psi} (f ; z) - 1 - z \hat{\Psi}' (f ; v)} \cdot
    \frac{T \mathd t}{z (z + T)}\\
    & = & (\log x)^{A (f ; v)} \cdot \frac{1}{2 \pi} \int_{- \xi}^{\xi}
    \mathcal{A}(z) e^{- zc (f)} \cdot (\log x)^{\hat{\Psi} (f ; z) -
    \hat{\Psi} (f ; v) - \mathi t \hat{\Psi}' (f ; v)} \cdot \frac{T \mathd
    t}{z (z + T)}\\
    & = & (\log x)^{A (f ; v)} \cdot (1 / v)\mathcal{A}(v) e^{- vc (f)}
    \int_{- \xi}^{\xi} e^{- (t^2 / 2) \hat{\Psi}'' (f ; v) \cdot \tmop{loglog}
    x} \cdot \frac{\mathd t}{2 \pi} \cdot \left( 1 + O (\mathcal{E}) \right)\\
    & = & (\log x)^{A (f ; v)} \cdot \frac{\mathcal{A}(v) e^{- vc (f)}}{v (
    \hat{\Psi}'' (f ; v) \tmop{loglog} x)^{1 / 2}} \int_{- \eta (x)}^{\eta
    (x)} e^{- u^2 / 2} \cdot \frac{\mathd u}{2 \pi} \cdot \left( 1 + O \left(
    \mathcal{E} \right) \right)\\
    & = & (\log x)^{A (f ; v)} \cdot \frac{\mathcal{A}(v) e^{- vc (f)}}{v (2
    \pi \hat{\Psi}'' (f ; v) \tmop{loglog} x)^{1 / 2}} \cdot \left( 1 + O
    \left( e^{- \eta (x)^2 / 2} +\mathcal{E} \right) \right)
  \end{eqnarray*}
  Since $\mathcal{E} \assign 1 / T + \eta (x)^3 / \Delta \gg 1 /
  \sqrt[]{\tmop{loglog} x}$ and $\eta (x) \gg \tmop{logloglog} x$ the term
  $e^{- \eta^2 / 2}$ is absorbed into $O (\mathcal{E})$. The
  integral we just evaluated furnishes the main term in $(4.30)$. We conclude
  that
  \begin{eqnarray}
    &  & \frac{1}{2 \pi i} \int_{\mathcal{M}} \mathcal{A}(z) (\log
    x)^{\hat{\Psi} (f ; z) - 1} \cdot e^{- z (\mu + \Delta \sigma)} \cdot
    \frac{T \mathd z}{z (z + T)} \nonumber\\
    & = & \mathcal{A}(v) e^{- vc (f)} \cdot \frac{(\log x)^{\hat{\Psi} (f ;
    v) - 1 - v \hat{\Psi}' (f ; v)}}{v (2 \pi \hat{\Psi}'' (f ; v)
    \tmop{loglog} x)^{1 / 2}} \cdot \left( 1 + O \left( \frac{1}{T} +
    \frac{\eta (x)^3}{\Delta} \right) \right) 
  \end{eqnarray}
  
  \
  
  {\tmem{2. Bound for $(4.29)$ restricted to $z = v + \mathi t \in
  \mathcal{R}$.}}
  
  \  
  
  Recall our convention that $z \assign v + \mathi t$ and $v \assign v_f (x ;
  \Delta)$. By lemma 4.8, and the bound $\mathcal{A}(z) \ll 1 + |z|^{1 / 8}$ 
  we have
  \begin{eqnarray*}
    \mathcal{A}(z) (\log x)^{\hat{\Psi} (f ; z) - 1} \cdot e^{- z (\mu +
    \Delta \sigma)} & \ll & (1 + |z|^{1 / 8}) \cdot (\log x)^{\tmop{Re} (
    \hat{\Psi} (f ; z) - 1 - v \hat{\Psi}' (f ; v))}
  \end{eqnarray*}
  uniformly in $0 \leqslant v \assign \tmop{Re} z \leqslant C$ and $|
  \tmop{Im} z| \leqslant M \assign \tmop{loglog} x$. Therefore
  \begin{eqnarray}
    &  & \frac{1}{2 \pi i} \int_{\mathcal{R}} \mathcal{A}(z) (\log
    x)^{\hat{\Psi} (f ; z) - 1} \cdot e^{- z (\mu + \Delta \sigma)} \cdot
    \frac{T \mathd z}{z (z + T)} \\
    & \ll & (\log x)^{A (f ; v)} \cdot \int_{\mathcal{R}} (1 + |z|^{1 / 8})
    \cdot (\log x)^{\tmop{Re} ( \hat{\Psi} (f ; z) - \hat{\Psi} (f ; v))}
    \cdot \frac{T \cdot | \mathd z|}{|z| \cdot |z + T|} \nonumber
  \end{eqnarray}
  and it remains to bound the second integral, above. To do so we consider its
  behaviour in three ranges, $\mathcal{R}_1 =\{v + \mathi t : \eta (x) \cdot
  (\tmop{loglog} x)^{- 1 / 2} \leqslant |t| \leqslant c\}$ with $c > 0$ small
  enough, $\mathcal{R}_2 =\{v + \mathi t : c \leqslant |t| \leqslant \lambda
  (x)\}$ and $\mathcal{R}_3 =\{v + \mathi t : \lambda (x) \leqslant |t|
  \leqslant M\}$. We will fix $c$ and $\lambda(x)$ as we proceed
  through the proof. Recall also that $T = \lambda(x)^{1/2}$ and that 
  $\eta(x) = \log\log\log x$. 
  
  \
  
  {\tmem{2.1. The range $\mathcal{R}_1 =\{v + \mathi t : \eta (x) \cdot
  (\tmop{loglog} x)^{- 1 / 2} \leqslant |t| \leqslant c\}$}}
  
  \  

  Recall from lemma 4.12 that $\tmop{Re} ( \hat{\Psi} (f ; z) - \hat{\Psi} (f
  ; v)) \leqslant \tmop{Re} ( \hat{\Psi} (f ; \mathi t) - 1)$. Choose $c
  \leqslant 1$ small enough so as to ensure that for $z = v + \mathi t$,
  \[ \tmop{Re} ( \hat{\Psi} (f ; z) - \hat{\Psi} (f ; v)) \leqslant \tmop{Re}
     ( \hat{\Psi} (f ; \mathi t) - 1) \leqslant - \kappa t^2 / 2 \]
  for some $\kappa = \kappa (c) > 0$. The existence of such a $\kappa$, for
  $c$ small enough, is guaranteed by a Taylor expansion. Once $c \leqslant 1$
  is chosen sufficiently small, we can bound
  \begin{eqnarray*}
    &  & \int_{\mathcal{R}_1} (1 + |z|^{1 / 8}) \cdot (\log x)^{\tmop{Re} (
    \hat{\Psi} (f ; z) - \hat{\Psi} (f ; v))} \cdot \frac{T| \mathd z|}{|z|
    \cdot |z + T|}\\
    & \ll & \int_{\mathcal{R}_1} (\log x)^{- \kappa t^2 / 2} \cdot \frac{T
    \mathd t}{v (v + T)} \text{ } \ll \text{ } (1/v) 
    \exp (- \kappa \eta (x)^2 / 2)
  \end{eqnarray*}
  The last line comes from $|t| \geqslant \eta (x) \cdot (\tmop{loglog} x)^{-
  1 / 2}$. Since $\eta (x) = \tmop{logloglog} x$ it follows that $(4.32)$
  restricted to the range $\mathcal{R}_1$ is $\ll (\log x)^{A (f ; v)} \cdot
  (1/v) (\tmop{loglog} x)^{- 1}$ and this is as negligible as we want 
  it to be.
  
  \  
  
  {\tmem{2.2. The range $\mathcal{R}_2 =\{v + \mathi t : c \leqslant |t|
  \leqslant \lambda (x)\}$}}
  
  \
  
  Let $c > 0$ denote the constant that we fixed in the previous point. By
  lemma 4.14 there is a $\lambda_0 (x) \rightarrow \infty$ and a $\xi (x)
  \rightarrow \infty$ such that
  \[ \int_c^{\lambda_0 (x)} \left| e^{\hat{\Psi} (f ; z) - \hat{\Psi} (f ; v)}
     \right|^{\tmop{loglog} x} \cdot \frac{\mathd t}{t} \ll \frac{1}{\xi (x)
     \sqrt[]{\tmop{loglog} x}} \text{ , } z \assign v + \mathi t \]
  Let $\lambda (x) \assign \min (\lambda_0 (x), \xi (x), 1 + \tmop{logloglog}
  x)$. Note that $\lambda (x) \rightarrow \infty$. By the above equation and
  $\lambda (x) \leqslant \lambda_0 (x)$ we get
  \begin{equation}
    \int_c^{\lambda (x)} \left| e^{\hat{\Psi} (f ; z) - \hat{\Psi} (f ; v)}
    \right|^{\tmop{loglog} x} \cdot \frac{\mathd t}{t} \ll \frac{1}{\xi (x)
    \sqrt[]{\tmop{loglog} x}}
  \end{equation}
  We let $T \assign \sqrt[]{\lambda (x)}$. With this choice of $T$, we have $T
  \rightarrow \infty$ and $T \ll \tmop{logloglog} x$, and this is the only
  information about $T$ that we assumed {\tmem{a priori}}. By $(4.33)$ we have
  \begin{eqnarray*}
    &  & \int_{\mathcal{R}_2} (1 + |z|^{1 / 8}) \cdot (\log x)^{\tmop{Re} (
    \hat{\Psi} (f ; z) - \hat{\Psi} (f ; v))} \cdot \frac{T| \mathd z|}{|z|
    \cdot |z + T|}\\
    & \ll & \lambda (x)^{1 / 8} \cdot \int_{\mathcal{R}_2} \left|
    e^{\hat{\Psi} (f ; z) - \hat{\Psi} (f ; v)} \right|^{\tmop{loglog} x}
    \cdot \frac{\mathd t}{t} \text{ } \ll \text{ } \frac{\lambda (x)^{1 /
    8}}{\xi (x) \sqrt[]{\tmop{loglog} x}}
  \end{eqnarray*}
  Since $\lambda (x) \leqslant \xi (x)$ the above is bounded by $\xi (x)^{- 7
  / 8} \cdot (\tmop{loglog} x)^{- 1 / 2}$. It follows that $(4.32)$
  resitricted to $z \in \mathcal{R}_2$ is bounded by $(\log x)^{A (f ; v)}
  \cdot \xi (x)^{- 7 / 8} \cdot (\tmop{loglog} x)^{- 1 / 2}$ and again this is
  sufficiently negligible, for our purpose.
  
  \  
  
  {\tmem{2.3. The range $\mathcal{R}_3 \assign \{v + \mathi t : \lambda (x)
  \leqslant |t| \leqslant M\}$.}}
  
  \
  
  Since $0 \leqslant v \leqslant C$ and $z = v + \mathi t$ we have $|z| \ll
  |t|$. It follows that
  \begin{eqnarray}
    &  & \int_{\mathcal{R}_3} (1 + |z|^{1 / 8}) \cdot (\log x)^{\tmop{Re} (
    \hat{\Psi} (f ; z) - \hat{\Psi} (f ; v))} \cdot \frac{T \cdot | \mathd
    z|}{|z| \cdot |z + T|} \nonumber\\
    & \ll & \int_{\lambda (x)}^M t^{1 / 8} \cdot \left| e^{\hat{\Psi} (f ; z)
    - \hat{\Psi} (f ; v)} \right|^{\tmop{loglog} x} \cdot \frac{T \mathd t}{t
    \cdot (t + T)} \nonumber\\
    & \ll & \sum_{\ell \geqslant \lfloor\lambda (x)\rfloor} 
    \frac{T}{\ell^{7 / 8} \cdot
    (\ell + T)} \int_{\ell}^{\ell + 1} \left| e^{\hat{\Psi} (f ; z) -
    \hat{\Psi} (f ; v)} \right|^{\tmop{loglog} x} \mathd t 
  \end{eqnarray}
  We now show that the integral over $\ell \leqslant t \leqslant \ell + 1$ is
  $\ll (\tmop{loglog} x)^{- 1 / 2}$ uniformly in $\ell \geqslant 0$. Let
  $\phi_{\kappa} (t) \assign |e^{\hat{\Psi} (f ; \kappa + \mathi t) -
  \hat{\Psi} (f ; \kappa)} |$ and $\xi \assign \tmop{loglog} x$, so $| \phi_v
  (t) |^{\xi} = |e^{\hat{\Psi} (f ; z) - \hat{\Psi} (f ; v)} |^{\tmop{loglog}
  x}$ ($v \assign \tmop{Re} z = v_f(x;\Delta)$). 
  By lemma 4.12 we have $| \phi_{\kappa} (t) | \leqslant | \phi_0 (t) |$
  for all $\kappa > 0$ and $t \in \mathbbm{R}$. Furthermore by lemma 4.15
  there is a $c_0$ and a $c_1$ such that $\tmop{meas} (\{u \in I : | \phi_0
  (u) |^2 \geqslant 1 - \delta\}) \leqslant c_1 \cdot \sqrt[]{\delta}$ for all
  intervals $I$ of length $\leqslant c_0$. In particular $\tmop{meas} (\{u \in
  [\ell ; \ell + 1] : | \phi_0 (u) |^2 \geqslant 1 - \delta\}) \leqslant K
  \cdot \sqrt[]{\delta}$ where $K \assign (1 / c_0 + 1) \cdot c_1$. Using
  these two observations we conclude that
  \begin{eqnarray*}
    \int_{\ell}^{\ell + 1} | \phi_v (t) |^{\xi} \cdot \mathd t & \leqslant &
    \int_{\ell}^{\ell + 1} | \phi_0 (t) |^{\xi} \cdot \mathd t \text{ } =
    - \text{ } \int_0^1 t^{\xi / 2} \cdot \mathd \left( \underset{u \in [\ell ;
    \ell + 1]}{\tmop{meas}} \left( | \phi_0 (u) |^2 \geqslant t \right)
    \right)\\
    & = & (\xi / 2) \cdot \int_0^1 t^{\xi / 2 - 1} \cdot \underset{u \in
    [\ell ; \ell + 1]}{\tmop{meas}} \left( | \phi_0 (u) |^2 \geqslant t
    \right) \mathd t\\
    & \leqslant & (\xi / 2) \cdot K \int_0^1 t^{\xi / 2 - 1} \cdot \sqrt[]{1
    - t} \mathd t \text{ } \ll \text{ } \xi^{- 1 / 2}
  \end{eqnarray*}
  We evaluate the last integral by noticing that the integrand is essentially
  constant on intervals $[1 - (A + 1) / \xi ; 1 - A / \xi]$. The long chain of
  inequalities proves that
  \[ \int_{\ell}^{\ell + 1} \left| e^{\hat{\Psi} (f ; z) - \hat{\Psi} (f ; v)}
     \right|^{\tmop{loglog} x} \cdot \mathd t = \int_{\ell}^{\ell + 1} |
     \phi_v (t) |^{\xi} \mathd t \ll \xi^{- 1 / 2} = (\tmop{loglog} x)^{- 1 /
     2} \]
  as desired. Hence the sum in $(4.34)$ is bounded by
  \begin{eqnarray*}
    & \ll & \frac{1}{\sqrt[]{\tmop{loglog} x}} \sum_{\ell \geqslant 
    \lfloor \lambda
    (x) \rfloor} \frac{T}{\ell^{7 / 8} \cdot (\ell + T)} \text{ } \ll \text{ }
    \frac{1}{\lambda (x)^{3 / 8}} \cdot \frac{1}{\sqrt[]{\tmop{loglog} x}}
  \end{eqnarray*}
  (Recall that $T = \lambda (x)^{1 / 2}$). It follows that
  the integral in $(4.32)$ restricted to $z \in \mathcal{R}_3$ is bounded by
  $(\log x)^{A (f ; v)} \cdot (\tmop{loglog} x)^{- 1 / 2} \cdot \lambda (x)^{-
  3 / 8}$, which is sufficiently negligible for our purpose.
  
  \
  
  {\tmem{2.4. Final bound for $(4.32)$.}}
  
  \
  
  Collecting the previous bounds from 2.1, 2.2, and 2.3, we conclude that
  \begin{eqnarray*}
    &  & \int_{\mathcal{R}} \mathcal{A}(z) (\log x)^{\hat{\Psi} (f ; z) - 1}
    \cdot e^{- z (\mu + \Delta \sigma)} \cdot \frac{T \cdot \mathd z}{z (z +
    T)}\\
    & \ll & (\log x)^{A (f ; v)} \cdot \left( \frac{1}{v \tmop{loglog} x} +
    \frac{\xi (x)^{- 7 / 8} + \lambda (x)^{- 3 / 8}}{(\tmop{loglog} x)^{1 /
    2}} \right)
  \end{eqnarray*}
  which is negligible compared to the estimate we obtained in $(4.31)$, 
  because $\mathcal{A}(v)e^{-v c(f)} \asymp 1$ and
  $\hat{\Psi''}(f;v) \asymp 1$. The estimate 
  $\mathcal{A}(v)e^{-v c(f)} \asymp 1$
  follows from the continuity and non-vanishing of $\mathcal{A}(x)$ on the 
  positive real line, and the fact that the parameter $v$ is confined to 
  a bounded interval $0 \leqslant v \leqslant C$. 
  
  \
  
  {\tmem{3. Conclusion.}}
  
  \
  
  Comparing the bound we obtained in 2.4 with $(4.31)$ it follows that
  \begin{eqnarray}
    &  & \frac{1}{2 \pi i} \int_{v - iM}^{v + iM} \mathcal{A}(z) (\log
    x)^{\hat{\Psi} (f ; z) - 1} e^{- z (\mu + \Delta \sigma)} \cdot \frac{T
    \mathd z}{z (z + T)} \\
    & = & \mathcal{A}(v) e^{- vc (f)} \cdot \frac{(\log x)^{\hat{\Psi} (f ;
    v) - 1 - v \hat{\Psi}' (f ; v)}}{v (2 \pi \hat{\Psi}'' (f ; v)
    \tmop{loglog} x)^{1 / 2}} \cdot (1 + o (1)) \nonumber
  \end{eqnarray}
  uniformly throughout $1 \leqslant \Delta \leqslant \delta \sigma (f ; x)$.
  In the same way as we estimated the above integral we estimate the integral
  appearing in the first error term in $(4.28)$. Because of the additional $z
  + T$ in the denominator this integral will be negligible compared to
  $(4.35)$. Since the other error terms in $(4.28)$ are negligible compared
  to $(4.35)$ we finally conclude that
  \[ \mathbbm{P} \left( \Omega (f ; x) \geqslant \mu + \Delta \sigma \right)
     \sim \mathcal{A}(v) e^{- vc (f)} \cdot \frac{\left. (\log x
     \right)^{\hat{\Psi} (f ; v) - 1 - v \hat{\Psi}' (f ; v)}}{v (2 \pi
     \hat{\Psi}'' (f ; v) \tmop{loglog} x)^{1 / 2}} \]
  uniformly in $1 \leqslant \Delta \leqslant \delta \sigma (f ; x)$, as
  desired. 
\end{proof}

%% Stop cut. 

\subsection{Large deviations: $(\tmop{loglog} x)^{\varepsilon} \ll \Delta \ll
\sigma (f ; x)$ and $\Psi (f ; t)$ is lattice distributed}

We may assume by rescaling that $\Psi (f ; t)$ is lattice distributed on
$\mathbbm{Z}$. Throughout this section we write $f =\mathfrak{g}+\mathfrak{h}$
with $\mathfrak{g}, \mathfrak{h}$ two strongly additive functions defined by
\begin{eqnarray*}
  \mathfrak{g}(p) = \left\{ \begin{array}{l}
    f (p) \text{ if } f (p) \in \mathbbm{Z}\\
    0 \text{ \ \ \ \ otherwise}
  \end{array} \right. & \tmop{and} & \mathfrak{h}(p) = \left\{
  \begin{array}{l}
    f (p) \text{ if } f (p) \nin \mathbbm{Z}\\
    0 \text{ \ \ \ \ otherwise}
  \end{array} \right.
\end{eqnarray*}
The goal is to prove the following ``general'' proposition.

\begin{proposition}
  Let $f \in \mathcal{C}$. Suppose that $\Psi (f ; t)$ is lattice distributed
  on $\mathbbm{Z}$. Consider the random variable $\Omega (f ; x) \assign
  \sum_{p \leqslant x} f (p) Z_p$ where the $Z_p \in \{0 ; 1\}$ are random
  variables, not necessarily independent, over a common probability space
  $(\Omega_x, \mathcal{F}_x, \mathbb{P}_x)$ which we allow to depend on $x$. We denote by
  $\mathbbm{P}_{\mathcal{F}_x}$ and $\mathbbm{E}_{\mathcal{F}_x}$ the
  probability measure and the expectation in that probability space. Suppose
  that \
  \begin{enumeratenumeric}
    \item Uniformly in $0 \leqslant \kappa \assign \tmop{Re} s \leqslant C, |
    \tmop{Im} s| \leqslant \log\log x$ and uniformly in strongly additive 
    function
    $\mathfrak{H}$ such that $0 \leqslant \mathfrak{H}(p) \leqslant
    \left\lceil \mathfrak{h}(p) \right\rceil$,
    \[ \mathbbm{E}_{\mathcal{F}_x} \left[ e^{s \Omega (\mathfrak{g}; x) + s
       \Omega (\mathfrak{H}; x)} \right] =\mathbbm{E}_{\mathcal{F}_x} \left[
       e^{s \Omega (\mathfrak{g}; x)} \right] \cdot \prod_{p \leqslant x}
       \left( 1 + \frac{e^{s\mathfrak{H}(p)} - 1}{p} \right) + O_C \left(
       \mathcal{E}(x ; \kappa)^{^{}} \right) \]
    with an error term $\mathcal{E}(x ; \kappa) \assign (\log x)^{\hat{\Psi}
    (f ; \kappa) - 3 / 2}$.
    
    \item Given $C > 0$, we have, uniformly in $0 \leqslant \kappa \assign
    \tmop{Re} s \leqslant C$, $| \tmop{Im} s| \leqslant 2 \pi$,
    \[ \mathbbm{E}_{\mathcal{F}_x} \left[ e^{s \Omega (\mathfrak{g}; x)}
       \right] =\mathcal{A}(s) \cdot (\log x)^{\hat{\Psi} (f ; s) - 1} + O
       \left( (\log x)^{\hat{\Psi} (f ; \kappa) - 3 / 2} \right) \]
    where $\mathcal{A}(s)$ is an analytic function in $\tmop{Re} s \geqslant
    0$, which we assume to be non-zero on the positive real axis.
  \end{enumeratenumeric}
  Then, for any given $c > 0$, uniformly in $1 \leqslant \Delta \leqslant c
  \sigma (f ; x)$,
  \[ \mathbbm{P}_{\mathcal{F}_x} \left( \sum_{p \leqslant x} f (p) \left[ Z_p
     - \frac{1}{p} \right] \geqslant \Delta \sigma (f ; x) \right) \sim
     \mathcal{A}(v) e^{- vc (f)} \cdot \mathcal{P}_{\mathfrak{h}} \left( \xi_f
     (x ; \Delta) ; v \right) \cdot S_f (x ; \Delta) \]
  with $v \assign v_f (x ; \Delta)$ and the rest of the notation defined in
  the table of section 3.
\end{proposition}

In the most important case, when $\Omega_x = [1 ; x]$ and the random variables
$Z_p (n)$ are the indicator functions of the event $p|n$, proposition 4.17 can
be proved by following the method of {\cite{1}}. The proof there is more
natural, but unfortunately doesn't adapt to a more general situation, in
particular to the case when the $Z_p$ are independent random variables.

We need a substantial amount of preparation before we can prove the lemma. We
subdivide this section in three subsections. In 4.5.1 we gather information
about the additive function $\mathfrak{h}$. In 4.5.2 we evaluate a certain
``saddle-point'' integral. In 4.5.3 we prove proposition 4.17.

\subsubsection{Preliminary lemma on $\mathfrak{h}$}

Denote by $S (\mathfrak{h})$ the set of primes for which $\mathfrak{h}(p)
\neq 0$. Recall that $\mathfrak{h}(p)$ is equal to $f(p)$ whenever
$f(p) \nin \mathbb{Z}$ and equal to $0$ otherwise. Since $f(p) > 0$ (by
definition of the class $\mathcal{C}$) it follows that $\mathfrak{h}(p)$
vanishes exactly when $f(p) \in \mathbb{Z}$. Hence the set $S(\mathfrak{h})
 = \{p: \mathfrak{h}(p) \neq 0\}$ is in fact equal to 
the set $\{p: f(p) \nin \mathbb{Z}\}$.

\begin{lemma}
  Let $f \in \mathcal{C}$. Suppose that $\Psi (f ; t)$ is lattice distributed
  on $\mathbbm{Z}$. We have
  \[ |S (\mathfrak{h}) \cap [1 ; x] | \ll_A x \cdot (\log x)^{- A} \]
\end{lemma}

\begin{proof}
  By our remark above $S (\mathfrak{h}) = \{p : f(p) \nin \mathbbm{Z}\}$.
  By assumptions $(1.4)$ we have for arbitrary $a \in \mathbbm{Z}$
  \begin{eqnarray}
    \frac{1}{\pi (x)} \sum_{\tmscript{\begin{array}{c}
      p \leqslant x\\
      f (p) \leqslant a
    \end{array}}} 1 & = & \Psi (f ; a) + O_A \left( (\log x)^{- A - 1} \right)
 \end{eqnarray}
  Further since $\Psi (f ; t)$ is a distribution function it is right
  continuous, so
  \begin{eqnarray}
    \frac{1}{\pi (x)} \sum_{\tmscript{\begin{array}{c}
      p \leqslant x\\
      f (p) < a + 1
    \end{array}}} 1 & = & \lim_{t \uparrow a + 1} \Psi (f ; t) + O_A \left(
    (\log x)^{- A - 1} \right) 
  \end{eqnarray}
  Since $\Psi (f ; t)$ is lattice distributed on $\mathbbm{Z}$ it is constant
  on the interval $[a ; a + 1)$. Therefore the right hand side of $(4.36)$ and
  $(4.37)$ are equal. Hence subtracting $(4.36)$ from $(4.37)$ yields
  \begin{eqnarray}
    \frac{1}{\pi (x)} \sum_{\tmscript{\begin{array}{c}
      p \leqslant x\\
      a < f (p) < a + 1
    \end{array}}} 1 & = & O_A \left( \frac{1}{(\log x)^{A + 1}} \right) 
  \end{eqnarray}
  uniformly in $a \in \mathbbm{Z}$. By assumption $(1.3)$ there are only $O
  (1)$ primes $p \leqslant x$ such that $f (p) \geqslant \log x$. Therefore
  \begin{eqnarray*}
    \frac{1}{\pi (x)} \sum_{\tmscript{\begin{array}{c}
      p \leqslant x\\
      f (p) \nin \mathbbm{Z}
    \end{array}}} 1 & \leqslant & \sum_{0 \leqslant a \leqslant \log x}
    \frac{1}{\pi (x)} \sum_{\tmscript{\begin{array}{c}
      p \leqslant x\\
      a < f (p) < a + 1
    \end{array}}} 1 + \frac{1}{\pi (x)} \sum_{\tmscript{\begin{array}{c}
      p \leqslant x\\
      f (p) \geqslant \log x
    \end{array}}} 1
  \end{eqnarray*}
  by $(4.38)$ the above sum is $O_A (\log x \cdot (\log x)^{- A - 1})$ as
  desired. 
\end{proof}

As a consequence of the lemma $\prod_{p \in S (\mathfrak{h})} (1 + 1 / p)$
converges and $\Psi(f;t) = \Psi(\mathfrak{g};t)$. In fact we proved a little
bit more.

\ 

{\noindent}\tmtextbf{Corollary. }\tmtextit{Let $f \in \mathcal{C}$. Suppose
that $\Psi (f ; t)$ is lattice distributed on $\mathbbm{Z}$. Let $A > 0$ be
given. The sum
\[ \sum_{p|n \Rightarrow p \in S (\mathfrak{h})} \frac{(\log n)^A}{n} \]
converges.}{\hspace*{\fill}}{\medskip}

\begin{proof}
  For an integer $n$ with prime factorization $n = p^{\alpha_1}_1 \cdot \ldots
  \cdot p_k^{\alpha_k}$ we have the inequality $\log n \leqslant \prod_{\ell
  \leqslant k} (\alpha_{\ell} \cdot \log p_{\ell} + 1)$. Therefore
  \[ \sum_{p|n \Rightarrow p \in S (\mathfrak{h})} \frac{(\log n)^A}{n}
     \leqslant \prod_{p \in S (\mathfrak{h})} \left( 1 + \sum_{\alpha
     \geqslant 1} \frac{(\alpha \log p + 1)^A}{p^{\alpha}} \right) \text{}
     \ll \prod_{p \in S (\mathfrak{h})} \left( 1 + K \cdot \frac{(\log
     p)^A}{p} \right) \]
  for some constant $K > 0$. By lemma 4.18, $\sum_{p \in S (\mathfrak{h})}
  (\log p)^A \cdot p^{- 1} < + \infty$ therefore the product on the right is
  finite.
\end{proof}
 
We need more than mere convergence of the product 
$\prod_{p \in S(\mathfrak{h})} (1 - 1/p)$. 

\begin{lemma}
  Let $f \in \mathcal{C}$. Suppose that $\Psi (f ; t)$ is lattice distributed
  on $\mathbbm{Z}$. The function
  \[ G (\mathfrak{h}; s) \assign \prod_{p \in S (\mathfrak{h})} \left( 1 +
     \frac{e^{s\mathfrak{h}(p)}}{p - 1} \right) \cdot \left( 1 - \frac{1}{p}
     \right) \]
  is entire and the product converges for all $s \in \mathbbm{C}$.
  Furthermore, given $\delta > 0$, there is a $x_0 (\delta)$ such that
  uniformly in $\tmop{Re} s \leqslant \delta$ and $x \geqslant x_0 (\delta)$,
  \begin{eqnarray}
    \prod_{\tmscript{\begin{array}{c}
      p \leqslant x\\
      p \in S (\mathfrak{h})
    \end{array}}} \left( 1 + \frac{e^{s\mathfrak{h}(p)}}{p - 1} \right) \cdot
    \left( 1 - \frac{1}{p} \right) & = & G (\mathfrak{h}; s) \cdot \left( 1 +
    O_{\delta} \left( (\log x)^{- 1 / 2} \right) \right) 
  \end{eqnarray}
\end{lemma}

{\noindent}\tmtextbf{Remark. }Note that $G (\mathfrak{h}; s)$ is the moment
generating function of the random variable $X (\mathfrak{h}) \assign \sum_p
\mathfrak{h}(p) X_p$. Indeed,
\[ \mathbbm{E} \left[ e^{sX (\mathfrak{h})} \right] = \prod_p \left( 1 -
   \frac{1}{p} + \frac{e^{s\mathfrak{h}(p)}}{p} \right) = \prod_{p \in S
   (\mathfrak{h})} \left( 1 + \frac{e^{s\mathfrak{h}(p)}}{p - 1} \right) \cdot
   \left( 1 - \frac{1}{p} \right) \]
In particular $\mathbbm{E} \left[ e^{\kappa X (\mathfrak{h})} \right]$ is
finite for any fixed $\kappa > 0$.{\hspace*{\fill}}{\medskip}

\begin{proof}
  We are going to show that $\sum_{p > x, \mathfrak{h}(p) \neq 0} \log \left(
  1 + (e^{s\mathfrak{h}(p)} - 1) / p \right) \ll (\log x)^{- 1 / 2}$ uniformly
  in $\tmop{Re} s \leqslant \delta$, for all $x$ large enough (we need to take
  $x$ large enough to prevent $1 + (e^{s\mathfrak{h}(p)} - 1) / p$ from
  vanishing when $\tmop{Re} s \leqslant \delta$ and $p > x$). This bound
  admits two consequences. First of all, it implies that the partial products
  \begin{equation}
    \prod_{p \leqslant x} \left( 1 + \frac{e^{s\mathfrak{h}(p)}}{p - 1}
    \right) \cdot \left( 1 - \frac{1}{p} \right) =
    \prod_{\tmscript{\begin{array}{c}
      p \leqslant x\\
      \mathfrak{h}(p) \neq 0
    \end{array}}} \left( 1 + \frac{e^{s\mathfrak{h}(p)} - 1}{p} \right)
  \end{equation}
  converge uniformly on compact subsets of $\mathbbm{C}$. Hence $G
  (\mathfrak{h}; s)$ is an entire function. Secondly, since $(4.40)$ converges
  to $G (\mathfrak{h}; s)$ and its tails are $1 + O ((\log x)^{- 1 / 2})$ we
  obtain $(4.39)$. Thus it remains to bound the sum of $\log (1 +
  (e^{s\mathfrak{h}(p)} - 1) / p)$ over $p > x$. Assume without loss of
  generality that $\delta \geqslant 2$. By assumption $(1.3)$, $f (p) = o
  (\log p)$. In particular $\mathfrak{h}(p) = o (\log p)$ and thus
  $|e^{s\mathfrak{h}(p)} | \leqslant e^{\delta \mathfrak{h}(p)} = e^{o (\log
  p)}$ uniformly in $\tmop{Re} s \leqslant \delta$. Hence
  $e^{s\mathfrak{h}(p)} / p = o (1)$ and so $\log (1 + (e^{s\mathfrak{h}(p)} -
  1) / p) \ll e^{\delta \mathfrak{h}(p)} / p$. Using this inequality and
  breaking up our sum into ``dyadic'' intervals, we obtain, uniformly in
  $\tmop{Re} s \leqslant \delta$,
  \begin{eqnarray}
    &  & \sum_{\tmscript{\begin{array}{c}
      p > x\\
      \mathfrak{h}(p) \neq 0
    \end{array}}} \log \left( 1 + \frac{e^{s\mathfrak{h}(p)} - 1}{p} \right)
    \ll \sum_{\tmscript{\begin{array}{c}
      p > x\\
      \mathfrak{h}(p) \neq 0
    \end{array}}} \frac{e^{\delta \mathfrak{h}(p)}}{p} \leqslant \sum_{k
    \geqslant \log x} e^{- k} \sum_{\tmscript{\begin{array}{c}
      e^k \leqslant p \leqslant e^{k + 1}\\
      \mathfrak{h}(p) \neq 0
    \end{array}}} e^{\delta \mathfrak{h}(p)} \nonumber\\
    & = & \sum_{k \geqslant \log x} e^{- k} \cdot \biggl[
    \sum_{\tmscript{\begin{array}{c}
      e^k \leqslant p \leqslant e^{k + 1}\\
      0 <\mathfrak{h}(p) \leqslant \tmop{loglog} p
    \end{array}}} e^{\delta \mathfrak{h}(p)} + \sum_{A \geqslant 1}
    \sum_{\tmscript{\begin{array}{c}
      e^k \leqslant p \leqslant e^{k + 1}\\
      A \leqslant \mathfrak{h}(p) / \tmop{loglog} p \leqslant A + 1
    \end{array}}} e^{\delta \mathfrak{h}(p)} \biggr] 
  \end{eqnarray}

  Bounding the sum over $0 <\mathfrak{h}(p) \leqslant \tmop{loglog} p$ boils
  down to using the previous lemma. Note that under the condition $p \leqslant
  e^{k + 1}$ and $\mathfrak{h}(p) \leqslant \tmop{loglog} p$ we have
  $e^{\delta \mathfrak{h}(p)} \leqslant (k + 1)^{\delta}$. Therefore
  \begin{eqnarray*}
    \sum_{\tmscript{\begin{array}{c}
      e^k \leqslant p \leqslant e^{k + 1}\\
      0 <\mathfrak{h}(p) \leqslant \tmop{loglog} p
    \end{array}}} e^{\delta \mathfrak{h}(p)} & \leqslant & (k + 1)^{\delta}
    \sum_{\tmscript{\begin{array}{c}
      p \leqslant e^{k + 1}\\
      \mathfrak{h}(p) \neq 0
    \end{array}}} 1 = O_C \left( k^{\delta} \cdot e^k \cdot k^{- C} \right)^{}
  \end{eqnarray*}
  where in the last inequality we used Lemma 4.18. We chose $C = 2 \delta$,
  and conclude that the sum over $0 <\mathfrak{h}(p) \leqslant \tmop{loglog}
  p$ in $(4.41)$ is bounded by $e^k \cdot k^{- \delta}$.
  
  To bound the double sum over $A \geqslant 1$ and $A \leqslant
  \mathfrak{h}(p) / \tmop{loglog} p \leqslant A + 1$ in $(4.41)$ we will use
  assumption $(1.3)$. Note that under the conditions $e^k \leqslant p
  \leqslant e^{k + 1}$ and $A \leqslant \mathfrak{h}(p) / \tmop{loglog} p
  \leqslant A + 1$ we have $e^{\delta \mathfrak{h}(p)} \leqslant (k +
  1)^{\delta (A + 1)}$. It follows that
  \begin{eqnarray}
    \sum_{A \geqslant 1} \sum_{\tmscript{\begin{array}{c}
      e^k \leqslant p \leqslant e^{k + 1}\\
      A \leqslant \mathfrak{h}(p) / \tmop{loglog} p \leqslant A + 1
    \end{array}}} e^{\delta \mathfrak{h}(p)} & \leqslant & \sum_{A \geqslant
    1} \left( k + 1 \right)^{\left. \delta (A + 1 \right)}
    \sum_{\tmscript{\begin{array}{c}
      e^k \leqslant p \leqslant e^{k + 1}\\
      A \tmop{loglog} p \leqslant \mathfrak{h}(p)
    \end{array}}} 1_{} 
  \end{eqnarray}
  Regarding the innermost sum we proceed as follows: since $e^k \leqslant p$
  we overestimate a little by replacing $A \tmop{loglog} p \leqslant
  \mathfrak{h}(p)$ with $A \log k \leqslant \mathfrak{h}(p)$. Furthermore
  since $A \log k \leqslant \mathfrak{h}(p)$ implies $A \log k \leqslant f
  (p)$ we overestimate even more by replacing $A \log k \leqslant
  \mathfrak{h}(p)$ with $A \log k \leqslant f (p)$. From there, it follows
  that the sum in $(4.42)$ is bounded by
  \begin{equation}
    \leqslant \sum_{A \geqslant 1} (k + 1)^{\delta (A + 1)}
    \sum_{\tmscript{\begin{array}{c}
      p \leqslant e^{k + 1}\\
      A \log k \leqslant f (p)
    \end{array}}} 1 \ll_B \sum_{A \geqslant 1} k^{\delta (A + 1)} \cdot e^k
    e^{- B (A \log k)}
  \end{equation}
  where in the last line we used the assumption $(1.3)$. In our upper bound we
  choose $B = 3 \delta$ and then $(4.43)$ becomes $\ll e^k \sum_{A \geqslant
  1} k^{\delta (A + 1) - 3 \delta A} \ll e^k \cdot k^{- \delta}$, $(A
  \geqslant 1)$. Collecting $(4.42)$ and $(4.43)$ it follows that the double
  sum over $A \geqslant 1$ and $A \leqslant \mathfrak{h}(p) / \tmop{loglog} p
  \leqslant A + 1$ in $(4.41)$, is bounded by $e^k \cdot k^{- \delta}$.
  
  Putting together our bounds, we conclude that the whole sum in $(4.41)$ is
  less than $\ll \sum_{k \geqslant \log x} e^{- k} \cdot \left[ e^k k^{-
  \delta} + e^k k^{- \delta} \right] \ll (\log x)^{- \delta + 1} \ll (\log
  x)^{- 1 / 2}$ since we assumed $\delta \geqslant 2$. It follows that
  \begin{eqnarray*}
    \sum_{p > x} \log \left( 1 + \frac{e^{s\mathfrak{h}(p)} - 1}{p} \right) &
    \ll & (\log x)^{- 1 / 2}
  \end{eqnarray*}
  uniformly in $\tmop{Re} s \leqslant \delta$ (where $\delta \geqslant 2$
  without loss of generality). By the remarks made at the beginning of the
  lemma, the claim follows. 
\end{proof}

An important consequence of lemma 4.19 and lemma 4.18 is that $L
(\mathfrak{g}; s)$ is entire.

\begin{lemma}
  Let $f \in \mathcal{C}$. Suppose that $\Psi (f ; t)$ is lattice distributed
  on $\mathbbm{Z}$. Then the function $L (\mathfrak{g}; z)$ is entire.
  Furthermore given $C > 0$, there is a $x_0 (C)$ such that uniformly in $|
  \tmop{Re} s| \leqslant C$, $| \tmop{Im} s| \leqslant 2 \pi$ and $x \geqslant
  x_0 (C)$,
  \[ \prod_{p \leqslant x} \left( 1 - \frac{1}{p} \right)^{\hat{\Psi} (f ; s)}
     \left( 1 + \frac{e^{s\mathfrak{g}(p)}}{p - 1} \right) = L (\mathfrak{g};
     s) \cdot \left( 1 + O_C \left( (\log x)^{- 1 / 2} \right) \right) \]
\end{lemma}

\begin{proof}
  Because $\Psi(f;t)$ is lattice distributed on $\mathbb{Z}$ it has jumps
  on the integers. Therefore $\hat{\Psi}(f;z) = \sum_{k \geqslant 0} 
  \lambda_k e^{z k}$ with $\lambda_k \geqslant 0$. In particular 
  $\hat{\Psi}(f;v+\mathi t)$ is $2 \pi$-periodic in the $t$ variable. Hence
  \[ L(\mathfrak{g};v + \mathi t) = 
     \prod_{p} \left( 1 - \frac{1}{p} \right)^{\hat{\Psi} (f ; v + \mathi t)}
     \left( 1 + \frac{e^{(v + \mathi t)\mathfrak{g}(p)}}{p - 1} \right) \]
  is $2 \pi$ periodic in the $t$ variable (if the above equation is not clear
  recall that $\Psi(f;t) = \Psi(\mathfrak{g};t)$ by lemma 4.18, hence
  $\hat{\Psi}(f;z) = \hat{\Psi}(\mathfrak{g};z)$).
  Therefore, to prove that $L (\mathfrak{g}; s)$ is entire, it's
  enough to prove that $L (\mathfrak{g}; s)$ is analytic in $| \tmop{Im} s|
  \leqslant 2 \pi$. Given $C > 0$, consider $s$ in the region $\left.
  \mathcal{D}(C) \assign \{s : | \tmop{Re} s| \leqslant C \text{ and } |
  \tmop{Im} s| \leqslant 2 \pi \right\}$. Note that
  \begin{eqnarray*}
    &  & \prod_{p \leqslant x} \left( 1 - \frac{1}{p} \right)^{\hat{\Psi} (f
    ; s)} \cdot \left( 1 + \frac{e^{sf (p)}}{p - 1} \right)\\
    & = & \prod_{p \leqslant x} \left( 1 - \frac{1}{p} \right)^{\hat{\Psi} (f
    ; s)} \cdot \left( 1 + \frac{e^{s\mathfrak{g}(p)}}{p - 1} \right) 
    \prod_{\tmscript{\begin{array}{c}
      p \leqslant x
    \end{array}}} \left( 1 - \frac{1}{p} \right) \cdot \left( 1 +
    \frac{e^{s\mathfrak{h}(p)}}{p - 1} \right)
  \end{eqnarray*}
  By lemma 4.4 the first product equals $L (f ; s) \cdot (1 + O_C ((\log x)^{-
  1 / 2})$. By lemma 4.19 the last product equals to $G (\mathfrak{h}; s)
  \cdot (1 + O_C ((\log x)^{- 1 / 2})$ (keeping the notation of lemma 4.19).
  Both approximations hold uniformly in $s \in \mathcal{D}(C)$ and $x
  \geqslant x_0 (C)$ with $x_0 (C)$ large enough. Dividing by $G
  (\mathfrak{h}; s)$ on both sides we conclude that uniformly in $s$ such that
  $| \tmop{Re} s| \leqslant C, | \tmop{Im} s| \leqslant 2 \pi$ and $G
  (\mathfrak{h}; s) \neq 0$,
  \begin{eqnarray}
    \prod_{p \leqslant x} \left( 1 - \frac{1}{p} \right)^{\hat{\Psi} (f ; s)}
    \cdot \left( 1 + \frac{e^{s\mathfrak{g}(p)}}{p - 1} \right) & = & \frac{L
    (f ; s)}{G (\mathfrak{h}; s)} \cdot \left( 1 + O_C ((\log x)^{- 1 / 2}
    \right) 
  \end{eqnarray}
  Note that $L (f ; s) / G (\mathfrak{h}; s)$ is in fact analytic in
  $\mathcal{D}(C)$, because if $G (\mathfrak{h}; s)$ vanishes then $L (f ; s)$
  vanishes to the same order. Thus, by continuity $(4.44)$ extends to all of
  $\mathcal{D}(C)$. Since that region is bounded, the function $L (f ; s) / G
  (\mathfrak{h}; s)$, being analytic, is bounded there. Hence $(4.44)$
  guarantees that $\prod_{p \leqslant x} (1 \um 1 / p)^{\hat{\Psi} (f ; s)} (1
  \upl e^{s\mathfrak{g}(p)} (p - 1)^{- 1})$ converges uniformly in $|
  \tmop{Re} s| \leqslant C$, $| \tmop{Im} s| \leqslant 2 \pi$. Thus $L
  (\mathfrak{g}; s)$ is analytic in $| \tmop{Im} s| \leqslant 2 \pi$. Since $L
  (\mathfrak{g}; v + \mathi t)$ is $2 \pi$ periodic in the $t$ variable, it
  follows that $L (\mathfrak{g}; s)$ is entire. In addition, by $(4.44)$ we
  must have $L (\mathfrak{g}; s) = L (f ; s) / G (\mathfrak{h}; s)$ and the
  second assertion of the lemma follows from $(4.44)$.
\end{proof}

A further consequence of lemma 4.19 is that $X (\mathfrak{h}) \assign \sum_p
\mathfrak{h}(p) X_p$ has an entire moment generating function
\begin{eqnarray*}
  \mathbbm{E} \left[ e^{sX (\mathfrak{h})} \right] & = & \prod_p \left( 1 +
  \frac{e^{s\mathfrak{h}(p)}}{p - 1} \right) \cdot \left( 1 - \frac{1}{p}
  \right)
\end{eqnarray*}
Thus all moments of $X (\mathfrak{h})$ are finite, and in particular the
variance of $X (\mathfrak{h})$ is finite. Hence by Kolmogorov's three series
theorem $X (\mathfrak{h}) = \sum_p \mathfrak{h}(p) X_p$ converges almost
surely. In the next lemma we give an explicit expression for $\mathbbm{P}(X
(\mathfrak{h}) \geqslant t)$.

\begin{lemma}
  Let $f \in \mathcal{C}$. Suppose that $\Psi (f ; t)$ is lattice distributed
  on $\mathbbm{Z}$. We have
  \[ \left. \mathbbm{P} \left( X (\mathfrak{h} \right) \leqslant t \right) =
     \prod_{p \in S (\mathfrak{h})} \left( 1 - \frac{1}{p} \right) 
     \sum_{\tmscript{\begin{array}{c}
       n \geqslant 1\\
       p|n \Rightarrow p \in S (\mathfrak{h})\\
       \mathfrak{h}(n) \leqslant t
     \end{array}}} \frac{1}{n} \]
\end{lemma}

\begin{proof}
  Since $X_p$ is a Bernoulli random variable with $\mathbbm{P} \left( X_p = 1)
  = 1 / p \right.$,
  \[ \mathbbm{E} \left[ e^{sX_p} \right] = 1 - \frac{1}{p} + \frac{e^s}{p} =
     \left( 1 - \frac{1}{p} \right) \cdot \left( 1 + \frac{e^s}{p - 1} \right)
  \]
  Since in addition the $X_p$'s are independent, and $X (\mathfrak{h}) =
  \sum_p \mathfrak{h}(p) X_p$,
  \begin{eqnarray*}
    \mathbbm{E} \left[ e^{sX (\mathfrak{h})} \right] & = & \prod_p \mathbbm{E}
    \left[ e^{s\mathfrak{h}(p) X_p} \right] \text{ } = \text{ } \prod_p \left(
    1 - \frac{1}{p} \right) \cdot \left( 1 + \frac{e^{s\mathfrak{h}(p)}}{p -
    1} \right)
  \end{eqnarray*}
  Note if $\mathfrak{h}(p) = 0$ for a prime $p$ then the corresponding term in
  the above product is 1. Thus we can restrict the product to those primes $p$
  for which $\mathfrak{h}_{} (p) \neq 0$ or equivalently to the prime $p \in S
  (\mathfrak{h})$. Now let
  \[ F (t) = \prod_{p \in S (\mathfrak{h})} \left( 1 - \frac{1}{p} \right)
     \sum_{\tmscript{\begin{array}{c}
       n \geqslant 1\\
       p|n \Rightarrow p \in S (\mathfrak{h})\\
       \mathfrak{h}(n) \leqslant t
     \end{array}}} \frac{1}{n} \]
  We compute the Laplace transform $\int_{\mathbbm{R}} e^{st} \mathd F (t)$ of
  $F (\cdot)$,
  \begin{eqnarray*}
    \int_{\mathbbm{R}} e^{st} \cdot \mathd F (t) & = & \prod_{p \in S
    (\mathfrak{h})} \left( 1 - \frac{1}{p} \right) \int_{\mathbbm{R}} e^{st}
    \cdot \mathd \sum_{\tmscript{\begin{array}{c}
      n \geqslant 1\\
      p|n \Rightarrow p \in S (\mathfrak{h})\\
      \mathfrak{h}(n) \leqslant t
    \end{array}}} \frac{1}{n}\\
    & = & \prod_{p \in S (\mathfrak{h})} \left( 1 - \frac{1}{p} \right)
    \sum_{\tmscript{\begin{array}{c}
      n \geqslant 1\\
      p|n \Rightarrow p \in S (\mathfrak{h})
    \end{array}}} \frac{e^{s\mathfrak{h}(n)}}{n}\\
    & = & \prod_{\left. p \in S (\mathfrak{h} \right)} \left( 1 - \frac{1}{p}
    \right) \cdot \left( 1 + \frac{e^{s\mathfrak{h}(p)}}{p - 1} \right)\\
    & = & \mathbbm{E} \left[ e^{sX (\mathfrak{h})} \right] =
    \int_{\mathbbm{R}} e^{st} \mathd \mathbbm{P} \left( X (\mathfrak{h})
    \leqslant t \right)
  \end{eqnarray*}
  By uniqueness of Laplace transforms $F (t) =\mathbbm{P} \left( X
  (\mathfrak{h}) \leqslant t \right)$ as desired.
\end{proof}

By the discussion preceding the above lemma, we know that
\[ \sum_{p \leqslant x} \mathfrak{h}(p) X_p \text{ } \longrightarrow \text{ }
   \sum_p \mathfrak{h}(p) X_p \]
almost surely. Thus the convergence also holds in distribution. In the next
lemma we investigate the speed of convergence in more detail.

\begin{lemma}
  Let $f \in \mathcal{C}$. Suppose that $\Psi (f ; t)$ is lattice distributed
  on {\tmname{$\mathbbm{Z}$}}. Let
  \begin{eqnarray*}
    V_{\mathfrak{h}} (x ; t) & = & \mathbbm{P} \left( \sum_{p \leqslant x}
    \mathfrak{h}(p) X_p \geqslant t \right)
  \end{eqnarray*}
  Then $V_{\mathfrak{h}} (x ; t) = V_{\mathfrak{h}} (\infty ; t) + O
  (V_{\mathfrak{h}} (\infty ; t)^{1 / 4} \cdot (\log x)^{- 1 / 2})$ uniformly
  in $t \in \mathbbm{R}$.
\end{lemma}

\begin{proof}
  Let $S (\mathfrak{h}) = \left\{ p : \mathfrak{h}(p) \neq 0 \right\}$.
  Proceeding as in Lemma 4.21 we find that
  \begin{eqnarray}
    V_{\mathfrak{h}} (x ; t) & = & \prod_{\tmscript{\begin{array}{c}
      p \in S (\mathfrak{h})\\
      p \leqslant x
    \end{array}}} \left( 1 - \frac{1}{p} \right) 
    \sum_{\tmscript{\begin{array}{c}
      n \geqslant 1\\
      p|n \Rightarrow p \in S (\mathfrak{h}), p \leqslant x\\
      \mathfrak{h}(n) \geqslant t
    \end{array}}} \frac{1}{n} 
  \end{eqnarray}
  We first complete the product over $p \leqslant x$ to a product over all
  primes in $S (\mathfrak{h})$. First of all $\prod_{p > x, p \in S} (1 - 1 /
  p)^{- 1} \geqslant 1$. Upon expanding the Euler product we find
  \[ 1 \leqslant \prod_{\tmscript{\begin{array}{c}
       p \in S (\mathfrak{h})\\
       p > x
     \end{array}}} \left( 1 - \frac{1}{p} \right)^{- 1} \leqslant 1 +
     \sum_{\tmscript{\begin{array}{c}
       n \geqslant x\\
       p|n \Rightarrow p \in S (\mathfrak{h})
     \end{array}}} \frac{1}{n} \leqslant 1 + \frac{1}{\log x}
     \sum_{\tmscript{\begin{array}{c}
       n \geqslant 1\\
       p|n \Rightarrow p \in S (\mathfrak{h})
     \end{array}}} \frac{\log n}{n} \]
  The rightmost sum converges by the corollary to lemma 4.18. 
  Thus $\prod_{p > x, p \in S
  (\mathfrak{h})} (1 - 1 / p)$ equals to $1 + O (1 / \log x)$. Hence $(4.45)$
  becomes
  \[ V_{\mathfrak{h}} (x ; t) = \prod_{p \in S} \left( 1 - \frac{1}{p} \right)
     \sum_{\tmscript{\begin{array}{c}
       n \geqslant 1\\
       p|n \Rightarrow p \in S, p \leqslant x\\
       \mathfrak{h}(n) \geqslant t
     \end{array}}} \frac{1}{n} + O \left( \frac{V_f (x ; t)}{\log x} \right)
  \]
  Forgetting about $p \leqslant x$ in the above formula, we obtain, 
  $V_{\mathfrak{h}} (x ; t) \leqslant V_{\mathfrak{h}}
  (\infty ; t) + O (V_{\mathfrak{h}} (x ; t) / \log x)$. Iterating this
  inequality gives $V_{\mathfrak{h}} (x ; t) \leqslant V_{\mathfrak{h}}
  (\infty ; t) + O (V_{\mathfrak{h}} (\infty ; t) / \log x)$. To obtain a
  lower bound for $V_{\mathfrak{h}} (x ; t)$ we bound $V_{\mathfrak{h}}
  (\infty ; t) - V_{\mathfrak{h}} (x ; t)$ from above. By the previous
  equation
  \begin{eqnarray}
    V_{\mathfrak{h}} (\infty ; t) - V_{\mathfrak{h}} (x ; t) & = & \prod_{p
    \in S} \left( 1 - \frac{1}{p} \right) \sum_{\tmscript{\begin{array}{c}
      p|n \Rightarrow p \in S (\mathfrak{h})\\
      \exists p|n : p > x\\
      \mathfrak{h}(n) \geqslant t
    \end{array}}} \frac{1}{n} + O \left( \frac{V_f (x ; t)}{\log x} \right) 
  \end{eqnarray}
  We overestimate the above sum by replacing the condition $\exists p|n : p >
  x$ with $n > x$. Then we apply Cauchy-Schwarz, singling out $n > x$ in one
  term and the remaining condition in the second term. We select weights so as
  to make the sum over $n > x$ convergent. In more detail, we bound $(4.46)$
  by
  \begin{equation}
    \biggl( \sum_{n > x} \frac{1}{n \left( \log n \right)^2} \biggr)^{1/2}
    \cdot \biggl( \sum_{\tmscript{\begin{array}{c}
      m \geqslant 1\\
      p|m \Rightarrow p \in S (\mathfrak{h})\\
      \mathfrak{h}(m) \geqslant t
    \end{array}}} \frac{(\log m)^2}{m} \biggr)^{1 / 2}
  \end{equation}
  The sum on the left is $\ll 1 / (\log x)$. To bound the sum on the right we
  apply once again Cauchy-Schwarz, obtaining the following bound
  \begin{equation}
    \biggl( \sum_{\tmscript{\begin{array}{c}
      m \geqslant 1\\
      p|m \Rightarrow p \in S (\mathfrak{h})\\
      \mathfrak{h}(m) \geqslant t
    \end{array}}} \frac{1}{m} \biggr)^{1 / 2} \cdot \biggl(
    \sum_{\tmscript{\begin{array}{c}
      m \geqslant 1\\
      p|m \Rightarrow p \in S (\mathfrak{h})
    \end{array}}} \frac{(\log m)^4}{m} \biggr)^{1/2}
  \end{equation}
  By the corollary to lemma 4.18 the sum over $m \geqslant 1$ is $O(1)$. 
  By lemma 4.21 the sum on the left is $C \mathbb{P}(X(\mathfrak{h}) 
  \geqslant t)$ for some constant $C > 0$. Thus the above is bounded by
  $\ll \mathbb{P}(X(\mathfrak{h}) \geqslant t)^{1/2}$.  
  By $(4.46)$, $(4.47)$ and $(4.48)$, $V_{\mathfrak{h}} (\infty ; t)
  - V_{\mathfrak{h}} (x ; t) \leqslant O ((\log x)^{- 1 / 2} V_{\mathfrak{h}}
  (\infty ; t)^{1 / 4})$. On the other hand $0 \leqslant V_{\mathfrak{h}} (x ;
  t) \leqslant V_{\mathfrak{h}} (\infty ; t) + O (V_{\mathfrak{h}} (\infty ;
  t) / \log x)$. The lemma follows.
\end{proof}

The next result is a rather technical corollary to the above lemma. It shows
that we can modify the random variable $\sum_{p \leqslant x} \mathfrak{h}(p)
X_p$ on the primes $p > \xi (x)$ $(\xi (x) \rightarrow \infty)$ without
destroying uniform convergence (in distribution) to $\sum_p \mathfrak{h}(p)
X_p$.

\begin{corollary}
  Let $f \in \mathcal{C}$. Suppose that $\Psi (f ; t)$ is lattice distributed
  on $\mathbbm{Z}$. Let $y \rightarrow \infty$ as $x \rightarrow \infty$ but 
  with $y \leqslant x$. Let $\mathfrak{H}$ be a strongly additive function 
  defined by
  \[ \mathfrak{H}(p) = \left\{ \begin{array}{l}
       \left\lceil \mathfrak{h}(p) \right\rceil  \text{ if } p \geqslant y\\
       \mathfrak{h}(p) \text{ \ \ \ otherwise}
     \end{array} \right. \]
  We have, uniformly in $t \in \mathbbm{R}$,
  \[ \mathbbm{P} \left( \sum_{p \leqslant x} \mathfrak{H}(p) X_p \geqslant t
     \right) =\mathbbm{P} \left( \sum_p \mathfrak{h}(p) X_p \geqslant t
     \right) + O_A \left( \frac{e^{- At}}{(\log y)^{1 / 4}} \right) \]
  for any given $A > 0$. 
\end{corollary}

\begin{proof}
  We retain the notation $V_{\mathfrak{h}} (x ; t)$ from the previous lemma.
  Since $\mathfrak{H} \geqslant \mathfrak{h}$ and $x \geqslant y$,
  \begin{eqnarray}
    \mathbbm{P} \left( \sum_{p \leqslant x} \mathfrak{H}(p) X_p \geqslant t
    \right) & \geqslant & \mathbbm{P} \left( \sum_{p \leqslant y}
    \mathfrak{h}(p) X_p \geqslant t \right) = V_{\mathfrak{h}} (\infty ; t)
    \upl O \left( \frac{V_{\mathfrak{h}} (\infty ; t)^{1 / 2}}{(\log y)^{1 /
    4}} \right) 
  \end{eqnarray}
  where in the second equality we used the previous lemma. Further by lemma
  4.19 $X (\mathfrak{h}) \assign \sum_p \mathfrak{h}(p) X_p$ has an entire
  moment generating function. Hence, by Chernoff's bound $V_{\mathfrak{h}} (
  \infty; t) \leqslant \mathbbm{E}[e^{AX (\mathfrak{h})}] e^{- At} = 
  O_A (e^{- At})$ for any given $A > 0$. We conclude that the error term 
  in $(4.49)$ is bounded by $O_A (e^{- At} \cdot (\log y)^{- 1 / 4})$. To 
  derive the upper bound let us note that lemma 4.21 and lemma 4.22 also 
  holds for the additive function $\mathfrak{H}$ (the important observation 
  here is that $\mathfrak{H}(p)$ vanishes exactly when $\mathfrak{h}(p)$ so 
  $\mathfrak{H}$ is a ``small'' additive function). Therefore
  \begin{eqnarray}
    \mathbbm{P} \left( \sum_{p \leqslant x} \mathfrak{H}(p) X_p \geqslant t
    \right) & \leqslant & V_{\mathfrak{H}} (\infty ; t) = V_{\mathfrak{H}} (y
    ; t) + O \left( \frac{V_{\mathfrak{H}} (\infty ; t)^{1 / 2}}{(\log y)^{1 /
    4}} \right) 
  \end{eqnarray}
  Now $X (\mathfrak{H}) \assign \sum_{p} \mathfrak{H}(p) X_p$ also
  has an entire moment generating function. Hence by Chernoff's bound
  $V_{\mathfrak{H}} (\infty ; t) = O_A (e^{- At})$ for any given $A > 0$. It
  follows that the error term in $(4.50)$ is $O_A (e^{- At} \cdot (\log y)^{-
  1 / 4})$. By definition of $\mathfrak{H}$ we have the equality
  $V_{\mathfrak{H}} (y ; t) = V_{\mathfrak{h}} (y ; t)$. By lemma 4.22 and the
  bound $V_{\mathfrak{h}} (\infty ; t) \ll_A e^{- At}$,
  \[ V_{\mathfrak{h}} (y ; t) = V_{\mathfrak{h}} (\infty ; t) + O_A \left(
     e^{- At} \cdot (\log y)^{- 1 / 4}) \right. \]
  It follows that $V_{\mathfrak{H}} (y ; t) = V_{\mathfrak{h}} (\infty ; t) +
  O_A (e^{- At} \cdot (\log y)^{- 1 / 4})$. On combining this equality with
  $(4.50)$ we obtain the desired upper bound
  \[ \mathbbm{P} \left( \sum_{p \leqslant x} \mathfrak{H}(p) X_p \geqslant t
     \right) \leqslant V_{\mathfrak{h}} (\infty ; t) + O_A \left( \frac{e^{-
     At}}{(\log y)^{1 / 4}} \right) \]
  We also established a lower bound of the same quality, hence the lemma
  follows. 
\end{proof}

\subsubsection{Computing a ``saddle-point integral''}

The goal of this section is to prove the following lemma.

\begin{lemma}
  Let $f \in \mathcal{C}$. Let $\mathcal{A}(s)$ be analytic in $\tmop{Re} s
  \geqslant 0$, and suppose that $\mathcal{A}(x)$ does not vanish for $x
  \geqslant 0$. Suppose that $\Psi (f ; t)$ is lattice distributed on
  $\mathbbm{Z}$. Let $s \assign v + \mathi t$ with both $v, t$ real and $v
  \assign v_f (x ; \Delta)$. Given $\delta, \varepsilon > 0$, we have
  \begin{eqnarray}
    &  & \frac{1}{2 \pi} \int_{- \pi}^{\pi} \mathcal{A}(s) \cdot (1 /
    s)\mathcal{P}_{\mathfrak{h}} (\xi_f (x ; \Delta) ; s) \cdot \left( \log x
    \right)^{\hat{\Psi} (f ; s) - 1} \cdot e^{- s \xi_f (x ; \Delta)} \cdot
    \mathd t \\
    & = & \mathcal{A}(v) (1 / v)\mathcal{P}_{\mathfrak{h}} (\xi_f (x ;
    \Delta) ; v) \cdot \frac{\left( \log x \right)^{\hat{\Psi} (f ; v) - 1 - v
    \hat{\Psi}' (f ; v)}}{(2 \pi \hat{\Psi}'' (f ; v) \tmop{loglog} x)^{1 /
    2}} \cdot e^{- vc (f)} \cdot (1 + o (1)) \nonumber
  \end{eqnarray}
  uniformly for $\Delta$ in the range $(\tmop{loglog} x)^{\varepsilon} \ll
  \Delta \leqslant \delta \sigma_{\Psi} (f ; x)$.
\end{lemma}

First we need to show that $\mathcal{P}_{\mathfrak{h}} (a ; s)$ behaves ``as
an analytic function''.

\begin{lemma}
  Let $f \in \mathcal{C}$. Suppose that $\Psi (f ; t)$ is lattice distributed
  on $\mathbbm{Z}$. Given $C > 0$, uniformly in $| \delta | \leqslant \pi$, $0
  \leqslant v \leqslant C$ and $0 \leqslant a \leqslant 1$,
  \[ \mathcal{P}_{\mathfrak{h}} \left( a ; v + \delta \right)
     =\mathcal{P}_{\mathfrak{h}} \left( a ; v \right) + O_C \left( \delta
     \right) \]
  Here $\delta$ is allowed to be a complex number. Furthermore $v / (e^v - 1)
  \leqslant \mathcal{P}_{\mathfrak{h}} (a ; v) = O_C (1)$ uniformly in $0
  \leqslant a \leqslant 1$ and $0 \leqslant v \leqslant C$. 
\end{lemma}

{\noindent}\tmtextbf{Remark. }The restriction $0 \leqslant a \leqslant 1$ 
is unnecessary because $\mathcal{P}_{\mathfrak{h}}(a;v)$ is $1$-periodic in
the $a$ variable. {\hspace*{\fill}}{\medskip}

\begin{proof}
  As usual write $X (\mathfrak{h}) \assign \sum_p \mathfrak{h}(p) X_p$. By
  definition we have
  \begin{eqnarray*}
    \mathcal{P}_{\mathfrak{h}} \left( a ; v + \delta \right) & = & (v +
    \delta) \sum_{\ell \in \mathbbm{Z}} e^{(v + \delta) \cdot \left( \ell + a
    \right)} \cdot \mathbbm{P} \left( X \left( \mathfrak{h} \right) \geqslant
    \ell + a \right)
  \end{eqnarray*}
  We split the sum at $\ell < 0$ and $\ell \geqslant 0$ and handle separately
  the two ranges. When $\ell < 0$ we have $\mathbbm{P} \left( X (\mathfrak{h})
  \geqslant \ell + a \right) = 1$. Note also that $z / (e^z - 1)$ is analytic
  in the region $\{v + \mathi t : v \in \mathbbm{R}, |t| \leqslant \pi\}$.
  Therefore $(v + \delta) / (e^{(v + \delta)} - 1) = v / (e^v - 1) + O
  (\delta)$ (the implicit constant depends on $C$, we won't bother making that
  dependence explicit). With those two remarks in mind, the sum over $\ell <
  0$ contributes
  \begin{eqnarray}
    \left( v + \delta \right)  \sum_{\ell > 0} e^{(- \ell + a) \cdot \left( v
    + \delta \right)} & = & \frac{v + \delta}{e^{v + \delta} - 1} \cdot e^{a
    \cdot \left( v + \delta \right)} \nonumber\\
    & = & \left( \frac{v}{e^v - 1} + O \left( \delta \right) \right) \cdot
    e^{av} \cdot \left( 1 + O \left( \delta \right) \right) \nonumber\\
    & = & \frac{v \cdot e^{av}}{e^v - 1} + O \left( \delta \right) \text{\, }
    = \text{ } v \sum_{\ell > 0} e^{(- \ell + a) \cdot v} + O \left( \delta
    \right) 
  \end{eqnarray}
  We split the sum over $\ell \geqslant 0$,
  \begin{equation}
    (v + \delta) \sum_{\ell \geqslant 0} e^{\left( v + \delta) \cdot \left(
    \ell + a \right) \right.} \cdot \mathbbm{P} \left( X \left( \mathfrak{h}
    \right) \geqslant \ell + a \right)
  \end{equation}
  into $0 \leqslant \ell < | \delta |^{- 1}$ and $\ell > | \delta |^{- 1}$.
  When $0 \leqslant \ell \leqslant | \delta |^{- 1}$ we have
  \begin{eqnarray}
    (v + \delta) e^{(v + \delta) (\ell + a)} & = & ve^{(v + \delta) (\ell +
    a)} + O \left( \delta e^{(v + \pi) (\ell + 1)} \right) \nonumber\\
    & = & ve^{v (\ell + a)} \cdot (1 + O (\delta \ell)) + O \left( \delta
    e^{(v + \pi) (\ell + 1)} \right) \nonumber\\
    & = & ve^{v (\ell + a)} + O \left( \delta \ell e^{v (\ell + 1)} + \delta
    e^{(v \upl \pi) (\ell + 1)} \right) \nonumber\\
    & = & ve^{v (\ell + a)} + O \left( \delta \ell e^{(v + \pi) (\ell + 1)}
    \right) 
  \end{eqnarray}
  Splitting the sum $(4.53)$ into $0 \leqslant \ell < | \delta |^{- 1}$ and
  $\ell > | \delta |^{- 1}$, and using $(4.54)$, we obtain that $(4.53)$
  equals to
  \begin{eqnarray}
    &  & (v + \delta) \sum_{0 \leqslant \ell \leqslant | \delta |^{- 1}}
    e^{(v + \delta) (\ell + a)} \cdot \mathbbm{P} \left( X (\mathfrak{h})
    \geqslant \ell + a \right) + O \left( \sum_{\ell \geqslant | \delta |^{-
    1}} e^{(v + \pi) (\ell + 1)} \mathbbm{P} \left( X (\mathfrak{h}) \geqslant
    \ell \right) \right) \nonumber\\
    & = & v \sum_{0 \leqslant \ell \leqslant | \delta |^{- 1}} e^{v (\ell +
    a)} \cdot \mathbbm{P} \left( X (\mathfrak{h}) \geqslant \ell + a \right) +
    O \left( \sum_{\ell \geqslant 0} \delta \ell e^{(v + \pi) (\ell + 1)}
    \mathbbm{P} \left( X (\mathfrak{h}) \geqslant \ell \right) \right)
    \nonumber\\
    & = & v \sum_{\ell \geqslant 0} e^{v (\ell + a)} \cdot \mathbbm{P}(X
    (\mathfrak{h}) \geqslant \ell + a) + O \left( \sum_{\ell \geqslant 0}
    \delta \ell e^{(v + \pi) (\ell + 1)} \mathbbm{P} \left( X (\mathfrak{h})
    \geqslant \ell \right) \right) 
  \end{eqnarray}
  By lemma 4.19 $X (\mathfrak{h})$ has an entire moment generating function.
  Therefore for each fixed $A > 0$, we have $\mathbbm{P}(X (\mathfrak{h})
  \geqslant t) \leqslant \mathbbm{E}[e^{AX (\mathfrak{h})}] e^{- At} = O_A
  (e^{- At})$. In particular we have $\mathbbm{P}(X (\mathfrak{h}) \geqslant
  \ell) = O_C (e^{- (C + 2 + \pi) (\ell + 1)}) = O_C (e^{- (v + 2 + \pi) (\ell
  + 1)})^{}$. Thus the error term in $(4.55)$ is $O_C (\delta)$. Adding up the
  estimate $(4.55)$ and $(4.52)$, the first assertion of the lemma follows.
  
  The lower bound in the second assertion follows from
  \begin{eqnarray*}
    \mathcal{P}_{\mathfrak{h}} \left( a ; v) \right. & \geqslant & v
    \sum_{\ell \leqslant 0} e^{v (\ell + a)} \text{ } = \text{ }
    \frac{ve^{av}}{e^v - 1} \text{ } \geqslant \text{ } \frac{v}{e^v - 1}
  \end{eqnarray*}
  For the upper bound, recall that $\mathbbm{P}(X (\mathfrak{h}) \geqslant
  \ell) = O_C (e^{- (C + 1) (\ell + 1)})$. Therefore
  \begin{eqnarray*}
    \mathcal{P}_{\mathfrak{h}} (a ; v) & \leqslant & v \sum_{\ell < 0} e^{v
    (\ell + a)} + v \sum_{\ell \geqslant 0} e^{C \cdot (\ell + 1)} \cdot
    \mathbbm{P}(X (\mathfrak{h}) \geqslant \ell) \ll_C \frac{v}{e^v - 1} + v =
    O_C (1)
  \end{eqnarray*}
  The lemma is now proven.
\end{proof}

\begin{lemma}
  Let $f \in \mathcal{C}$. Suppose that $\Psi (f ; t)$ is lattice distributed
  on $\mathbbm{Z}$. Given $\varepsilon > 0$ there is a $\delta > 0$ such that
  \[ | \exp ( \hat{\Psi} (f ; v + \mathi t) - \hat{\Psi} (f ; v)) | \leqslant
     1 - \delta \]
  for all $\pi \geqslant |t| \geqslant \varepsilon$ and $v \geqslant 0$. 
\end{lemma}

\begin{proof}
  For any $v, t \in \mathbbm{R}$ we have
  \begin{eqnarray}
    \tmop{Re} \left( \hat{\Psi} \left( f ; v + \mathi t \right) \um \hat{\Psi}
    \left( f ; v \right) \right) & = & \tmop{Re} \left( \int_{\mathbbm{R}}
    e^{vu} \cdot \left( e^{\mathi tu} \um 1 \right) \mathd \Psi (f ; u)
    \right) \nonumber\\
    & = & \int_{\mathbbm{R}} e^{vu} \cdot \left( \cos (tu) \um 1 \right)
    \mathd \Psi (f ; u) \nonumber\\
    & \leqslant & \int_{\mathbbm{R}} \left( \cos (tu) \um 1 \right) \mathd
    \Psi (f ; u) \text{ } = \text{ } \tmop{Re} \left( \hat{\Psi} \left( f ;
    \mathi t \right) \um 1 \right) 
  \end{eqnarray}
  Therefore $| \exp ( \hat{\Psi} (f ; v + \mathi t) - \hat{\Psi} (f ; \mathi
  t)) | \leqslant | \exp ( \hat{\Psi} (f ; \mathi t) - 1) |$ and it's enough
  to show that given $\varepsilon > 0$ there is a $\delta > 0$ such that $|
  \exp ( \hat{\Psi} (f ; \mathi t) - 1) | \leqslant 1 - \delta$ for all $\pi
  \geqslant |t| \geqslant \varepsilon$. Since $\Psi (f ; t)$ is lattice
  distributed it has jumps at the integers $0, 1, 2, \ldots$ (there are no
  jumps at the negative integers because $f \geqslant 0$). Denote the size of
  each jump by $\lambda_0, \lambda_1, \ldots$ . Thus
  \[ \hat{\Psi} (f ; z) = \int_{\mathbbm{R}} e^{zt} \mathd \Psi (f ; t) =
     \sum_{k \geqslant 0} \lambda_k \cdot e^{zk} \]
  If $\Psi (f ; t)$ has all its mass concentrated at one integer $k$, then $k
  = 1$ and $\hat{\Psi} (f; z) = e^z$. In this case the bound $|
  \exp ( \hat{\Psi} (f ; \mathi t) - 1) | = \exp (\cos (t) - 1) \leqslant 1 -
  \delta$ for $\pi \geqslant |t| \geqslant \varepsilon$ is trivial. In the
  remaining case there are at least two $k, \ell$ for which $\lambda_k > 0$ and
  $\lambda_{\ell} > 0$. Without loss of generality we can assume that $(k,
  \ell) = 1$. Otherwise all the $k$ for which $\lambda_k > 0$ would be
  divisible by a common prime $p$; thus $\Psi (f ; t)$ would not be lattice
  distributed on $\mathbbm{Z}$ (but on $p\mathbbm{Z}$). Thus for two such $k,
  \ell$, we have
  \[ \tmop{Re} ( \hat{\Psi} (f ; \mathi t) - 1) = \sum_{r \geqslant 0}
     \lambda_r \cdot (\cos (rt) - 1) \leqslant \lambda_k \cdot (\cos (kt) - 1)
     + \lambda_{\ell} \cdot (\cos (\ell t) - 1) \leqslant 0 \]
  We claim that for $|t| \leqslant \pi$ the upper bound is attained only at $t
  = 0$. Indeed suppose that $\lambda_k (\cos (kt) - 1) + \lambda_{\ell} (\cos
  (\ell t) - 1) = 0$. Then simultaneously $\cos (kt) = 1$ and $\cos (\ell t) =
  1$. Hence $t = 2 \pi \alpha / \ell$ and $t = 2 \pi \beta / k$ for some
  integer $| \alpha | \leqslant \ell / 2$ and some integer $| \beta |
  \leqslant k / 2$, because $|t| \leqslant \pi$. In particular $2 \pi \alpha /
  \ell = 2 \pi \beta / k$. If $t \neq 0$ then $\alpha \neq 0$ and $\beta \neq
  0$, hence, $k \alpha / \ell = \beta \in \mathbbm{Z}$ which is impossible
  because $(k, \ell) = 1$ and $| \alpha | < \ell$. Thus $\tmop{Re} (
  \hat{\Psi} (f ; \mathi t) - 1) < 0$ for all $0 < |t| \leqslant \pi$. Hence
  we have $| \exp ( \hat{\Psi} (f ; \mathi t) - 1) | < 1$ for all $0 < |t|
  \leqslant \pi$. By continuity, given $\varepsilon > 0$, there is a $\delta >
  0$ such that $| \exp ( \hat{\Psi} (f ; \mathi t) - 1) | \leqslant 1 -
  \delta$ for all $\varepsilon \leqslant |t| \leqslant \pi$. Since
  \[ | \exp ( \hat{\Psi} (f ; v + \mathi t) - \hat{\Psi} (f ; v) | \leqslant |
     \exp ( \hat{\Psi} (f ; \mathi t) - 1) | \]
  for all $v \in \mathbbm{R}$, the lemma follows. 
\end{proof}

We are ready to prove the ``general lemma''.

\begin{proof}[Proof of Lemma 4.24]
  The function $\omega (f ; t)$ is continuous, hence the parameter $v \assign
  v_f (x ; \Delta) = \omega (f ; \Delta / \sigma_{\Psi} (f ; x))$ is bounded
  throughout $1 \leqslant \Delta \leqslant \delta \sigma_{\Psi} (f ; x)$, the
  bound depending only on $\delta$. To evaluate $(4.51)$ we proceed with the
  saddle-point method. As is usually done we split the integral into two
  ranges. The ``tiny'' range $|t| \leqslant \lambda (x) (\tmop{loglog} x)^{- 1
  / 2}$ denoted by $\mathcal{M}$, because this range will contribute the main
  term, and the remaining range $|t| \geqslant \lambda (x) (\tmop{loglog}
  x)^{- 1 / 2}$ denoted $\mathcal{R}$. Here we choose $\lambda (x)$ to be any
  function such that $(\tmop{logloglog} x)^4 \ll \lambda (x) \ll
  (\tmop{logloglog} x)^5$. Let us confine attention to how the integrand in
  $(4.51)$ behaves when $t \in \mathcal{M}$. First of all, when $t \in
  \mathcal{M}$, upon expanding $\hat{\Psi} (f ; v + \tmop{it})$ into a Taylor
  series we obtain
  \begin{eqnarray*}
    \hat{\Psi} (f ; v + \tmop{it}) & = & \hat{\Psi} (f ; v) + \tmop{it}
    \hat{\Psi}' (f ; v) - \frac{t^2}{2} \cdot \hat{\Psi}'' (f ; v) +
    O_{\delta} \left( \lambda (x)^3 (\log_2 x)^{- 3 / 2} \right)
  \end{eqnarray*}
  The $\delta$ in the error term comes from the bound $0 \leqslant v =
  O_{\delta} (1)$ on $v$. We will not indicate the dependence on $\delta$ in
  our error terms. Using the above expansion and Lemma 4.8 we conclude that
  \begin{eqnarray}
    &  & \left( \log x \right)^{\hat{\Psi} (f ; v + \mathi t)} \cdot e^{-
    \tmop{it} \xi_f (x ; \Delta)} \nonumber\\
    & = & \left( \log x \right)^{\hat{\Psi} (f ; v) + \mathi t \hat{\Psi}' (f
    ; v) - t^2 / 2 \cdot \hat{\Psi}'' (f ; v)} e^{- \mathi t \xi_f (x ;
    \Delta)}  \left( 1 + O (\lambda (x)^3 (\log_2 x)^{- 1 / 2}) \right)
    \nonumber\\
    & = & \left( \log x \right)^{\hat{\Psi} (f ; v) + \mathi t \hat{\Psi}' (f
    ; v) - t^2 / 2 \cdot \hat{\Psi}'' (f ; v)} \left( \log x \right)^{- \mathi
    t \hat{\Psi}' (f ; v)} e^{- \mathi tc (f)}  \left( 1 \upl O (\lambda (x)^3
    (\log_2 x)^{- 1 / 2} \right) \nonumber\\
    & = & \left( \log x \right)^{\hat{\Psi} (f ; v) - (t^2 / 2) \cdot
    \hat{\Psi}'' (f ; v)} e^{- \mathi tc (f)}  \left( 1 + O \left( \lambda
    (x)^3 (\log_2 x)^{- 1 / 2} \right) \right) \nonumber\\
    & = & \left( \log x \right)^{\hat{\Psi} (f ; v) - (t^2 / 2) \cdot
    \hat{\Psi}'' (f ; v)} \cdot \left( 1 + O \left( \lambda (x)^3 (\log_2
    x)^{- 1 / 2} \right) \right) 
  \end{eqnarray}
  where in the last line we used the expansion $e^{\mathi tc (f)} = 1 + O (t
  \cdot c (f))$ together with the bound $|t| \leqslant \lambda (x) (\log_2
  x)^{- 1 / 2}$. Note that $|t / v| \leqslant | \lambda (x) \cdot
  (\tmop{loglog} x)^{- 1 / 2} \cdot v^{- 1} | = o (1)$ because $v \asymp
  \Delta \cdot (\tmop{loglog} x)^{- 1 / 2}$ and $\lambda (x) = o (\Delta)$. By
  lemma 4.25, when $t \in \mathcal{M}$,
  \begin{eqnarray}
    &  & (1 / (v + \mathi t)) \cdot \mathcal{P}_{\mathfrak{h}} \left( \xi_f
    (x ; \Delta) ; v + \mathi t) \right. \nonumber\\
    & = & v^{- 1} \cdot (1 + O (t \cdot v^{- 1})) \cdot
    (\mathcal{P}_{\mathfrak{h}} (\xi_f (x ; \Delta) ; v) + O (t)) \nonumber\\
    & = & v^{- 1} \cdot \mathcal{P}_{\mathfrak{h}} (\xi_f (x ; \Delta) ; v)
    \cdot (1 + O_{\delta} (t)) \cdot (1 + O (t \cdot v^{- 1})) \nonumber\\
    & = & v^{- 1} \cdot \mathcal{P}_{\mathfrak{h}} (\xi_f (x ; \Delta) ; v)
    \cdot (1 + O (\lambda (x) \cdot (\tmop{loglog} x)^{- 1 / 2} \cdot v^{- 1})
  \end{eqnarray}
  The third line is justified by the bound $\mathcal{P}_{\mathfrak{h}} (a ; v)
  \gg_{\delta} 1$ which follows from the inequality
  $\mathcal{P}_{\mathfrak{h}} (a ; v) \geqslant v / (e^v - 1)$ of lemma 4.25
  and $v = O_{\delta} (1)$. Finally by analyticity of $\mathcal{A}(z)$, for $t
  \in \mathcal{M}$, we have
  \begin{eqnarray}
    \mathcal{A}(v + \tmop{it}) & = & \mathcal{A}(v) + O (t) \nonumber\\
    & = & \mathcal{A}(v) + O \left( \lambda (x) \cdot (\tmop{loglog} x)^{- 1
    / 2} \right) \nonumber\\
    & = & \left. \mathcal{A}(v) \cdot (1 + O \left( \lambda (x) \cdot
    (\tmop{loglog} x)^{- 1 / 2} \right) \right) 
  \end{eqnarray}
  where the last line is justified by the non-vanishing of $\mathcal{A}(v)$
  (because $\mathcal{A}(x) \neq 0$ for $x \geqslant 0$ and $0 \leqslant v
  \leqslant O_{\delta} (1)$ we have $\mathcal{A}(v) \gg_{\delta} 1$ by
  continuity of $\mathcal{A}(\cdot)$). From the equations $(4.57)$, $(4.58)$,
  $(4.59)$ and lemma 4.8, we conclude that
  \begin{eqnarray}
    &  & \mathcal{A}(v + \mathi t) \cdot (1 / (v + \mathi
    t))\mathcal{P}_{\mathfrak{h}} (\xi_f (x ; \Delta) ; v + \mathi t) \cdot
    \left( \log x \right)^{\hat{\Psi} (f ; v + \mathi t) - 1} \cdot e^{- (v +
    \mathi t) \xi_f (x ; \Delta)} \nonumber\\
    & = & \mathcal{A}(v) \cdot (1 / v)\mathcal{P}_{\mathfrak{h}} (\xi_f (x ;
    \Delta) ; v) \cdot (\log x)^{\hat{\Psi} (f ; v) - 1 - (t^2 / 2)
    \hat{\Psi}'' (f ; v)} \cdot e^{- v \xi_f (x ; \Delta)} \cdot (1 \upl O
    (\mathcal{E})) \nonumber\\
    & = & \mathcal{A}(v) \cdot (1 / v)\mathcal{P}_{\mathfrak{h}} (\xi_f (x ;
    \Delta) ; v) \cdot \left( \log x \right)^{A (f ; v) \um (t^2 / 2)
    \hat{\Psi}'' (f ; v)} \cdot e^{- vc (f)} \cdot \left( 1 \upl O \left(
    \mathcal{E} \right) \right) 
  \end{eqnarray}
  where $A (f ; v) = \hat{\Psi} (f ; v) - 1 - v \hat{\Psi}' (f ; v)$ and
  $\mathcal{E} \assign \mathcal{E} \left( x ; v \right) \assign \lambda (x)^3
  (\tmop{loglog} x)^{- 1 / 2} \cdot v^{- 1}$. In view of the above relation to
  estimate $(4.51)$ over $t \in \mathcal{M}$, it remains to note that
  \begin{eqnarray}
    &  & \int_{\mathcal{M}} \left( \log x \right)^{\left. - (t^2 / 2 \right)
    \cdot \hat{\Psi}'' (f ; v)} \cdot \frac{\mathd t}{2 \pi} \nonumber\\
    & = & \int_{\mathcal{M}} \exp \left( - \frac{t^2}{2} \cdot \hat{\Psi}''
    (f ; v) \tmop{loglog} x \right) \cdot \frac{\mathd t}{2 \pi} \nonumber\\
    & = & \frac{1}{( \hat{\Psi}'' (f ; v) \tmop{loglog} x)^{1 / 2}} \int_{-
    \lambda (x)}^{\lambda (x)} e^{- t^2 / 2} \cdot \frac{\mathd t}{2 \pi}
    \nonumber\\
    & = & \frac{1}{( \hat{\Psi}'' (f ; v) \tmop{loglog} x)^{1 / 2}} \cdot
    \left( \int_{\mathbbm{R}} e^{- t^2 / 2} \cdot \frac{\mathd t}{2 \pi} + o
    (1) \right) \nonumber\\
    & = & \frac{1}{(2 \pi \hat{\Psi}'' (f ; v) \tmop{loglog} x)^{1 / 2}}
    \cdot \left( 1 + o (1) \right) 
  \end{eqnarray}
  Together from $(4.60)$ and $(4.61)$ we conclude that the integral $(4.51)$
  restricted to $t \in \mathcal{M}$ is equal to
  \[ \mathcal{A}(v) (1 / v)\mathcal{P}_{\mathfrak{h}} (\xi_f (x ; \Delta) ; v)
     \cdot \frac{(\log x)^{\hat{\Psi} (f ; v) - 1 - v \hat{\Psi}' (f ; v)}}{(2
     \pi \hat{\Psi}'' (f ; v) \tmop{loglog} x)^{1 / 2}} \cdot e^{- vc (f)}
     \cdot \left( 1 + o (1) + O \left( \mathcal{E}(x ; v \right)^{^{}} \right)
  \]
  where $\mathcal{E}(x ; v) \assign \lambda (x)^3 (\tmop{loglog} x)^{- 1 / 2}
  v^{- 1}$. As we noticed $\lambda (x)^3 (\tmop{loglog} x)^{- 1 / 2} v^{- 1} =
  o (1)$ because $\lambda (x)^3 = o (\Delta)$ and $v \asymp \Delta /
  (\tmop{loglog} x)^{1 / 2}$. Therefore the above formula furnishes the
  desired main term. It remains to bound the integral $(4.51)$ restricted to
  $t \in \mathcal{R}$. By lemma $4.25$ we have $\mathcal{P}_{\mathfrak{h}} (a
  ; v + \mathi t) = \mathcal{P}_{\mathfrak{h}}(a;v) + O(|t|) \ll 1$ uniformly in $1 \leqslant \Delta \leqslant c \sigma
  (f ; x)$ because $v \asymp \Delta / \sigma (f ; x)$ belongs to a bounded
  range. Further since $\mathcal{A}(\cdot)$ is analytic we have $\mathcal{A}(v
  + \mathi t) \ll 1$ because $v + \mathi t$ lies in a bounded domain. Thus 
  (writing $s = v + \mathi t$)
  \begin{eqnarray}
    &  & \int_{\mathcal{R}} \mathcal{A}(s) (1 / s)\mathcal{P}_{\mathfrak{h}}
    (\xi_f (x ; \Delta) ; s) \cdot (\log x)^{\hat{\Psi} (f ; s) - 1} \cdot
    e^{- s \xi_f (x ; \Delta)} \cdot \mathd t \nonumber\\
    & \ll & (1 / v) \cdot (\log x)^{\Psi (f ; v) - 1} \cdot e^{- v \xi_f (x ;
    \Delta)} \cdot \int_{\mathcal{R}} (\log x)^{\tmop{Re} ( \hat{\Psi} (f ; v
    + \mathi t) - \hat{\Psi} (f ; v))} \cdot \mathd t 
  \end{eqnarray}
  Given $\varepsilon > 0$, by lemma 4.26 there is a $\delta > 0$ such that $|
  \exp ( \hat{\Psi} (f ; v + \mathi t) - \hat{\Psi} (f ; v)) | \leqslant 1 -
  \delta$ for all $\pi \geqslant |t| \geqslant \varepsilon$. Exponentiating we
  get $(\log x)^{\tmop{Re} ( \hat{\Psi} (f ; v + \mathi t) - \hat{\Psi} (f ;
  v)} \leqslant (\log x)^{\log (1 - \delta)}$ which is sufficient to bound the
  part $| \pi | \geqslant |t| \geqslant \varepsilon$ of the integral in
  $(4.62)$. We are left with bounding the remaining range $\lambda (x) \cdot
  (\tmop{loglog} x)^{- 1 / 2} \leqslant |t| \leqslant \varepsilon$. \ As in
  the proof of Lemma 4.26, we observe that since $\Psi (f ; t)$ is lattice
  distributed, we have
  \begin{eqnarray*}
    \hat{\Psi} \left( f ; z) \right. & = & \sum_{k \geqslant 0} \lambda_k
    \cdot e^{zk}
  \end{eqnarray*}
  Hence $\tmop{Re} ( \hat{\Psi} (f ; v + \mathi t) - \hat{\Psi} (f ; v)) =
  \sum_{k \geqslant 0} \lambda_k \cdot e^{vk} \cdot (\cos (kt) - 1) \leqslant
  \lambda_{\ell} (\cos (\ell t) - 1)$ where $\ell > 0$ is the first integer
  for which $\lambda_{\ell} > 0$. In the range $|t| \leqslant \varepsilon$ we
  have $\cos (\ell t) - 1 \leqslant - ct^2$ provided that $\varepsilon$ is
  sufficiently small (of course $c$ depends on $\varepsilon$ and $\ell$). We
  reduce $\varepsilon$ if necessary and fix it, once it's small enough (note
  that our bound over $\pi \geqslant |t| \geqslant \varepsilon$ is negligible
  as long as $\varepsilon$ is fixed). Thus $\tmop{Re} (
  \hat{\Psi} (f ; v + \mathi t) - \hat{\Psi} (f ; v)) \leqslant - ct^2$ for
  $|t| \leqslant \varepsilon$. Hence
  \begin{eqnarray*}
    &  & \int_{\mathcal{R} \cap \left\{ |t| \leqslant \varepsilon \right\}}
    \left( \log x \right)^{\tmop{Re} \left( \hat{\Psi} (f ; v + \tmop{it}) -
    \hat{\Psi} (f ; v) \right)} \tmop{dt}\\
    \leqslant &  & 2 \int_{\lambda (x) \cdot (\tmop{loglog} x)^{- 1 /
    2}}^{\varepsilon} \exp \left( - ct^2 \cdot \tmop{loglog} x \right)
    \tmop{dt} \leqslant 2 \cdot e^{- c \lambda (x)^2}
  \end{eqnarray*}
  Thus by $(4.62)$ and our earlier remarks, the integral in $(4.51)$
  restricted to $t \in \mathcal{R}$, turns out to be bounded by
  \[ (1 / v) \cdot \left( \log x \right)^{\hat{\Psi} \left( f ; v \right) - 1}
     \cdot e^{- v \xi_f (x ; \Delta)} \cdot \left[ \exp (- c \cdot \lambda
     (x)^2) + \left( \log x \right)^{\log (1 - \delta)} \right] \]
  which is negligible because $\lambda (x) \gg \tmop{logloglog} x$ and 
  $\xi_f(x;\Delta) = \hat{\Psi}'(f;v) \log\log x + O(1)$ by lemma 4.8. 
  Hence the
  integral $(4.51)$ restricted to $\mathcal{R}$ is negligible. This, together
  with the asymptotic for ``$(4.51)$ restricted to $\mathcal{M}$" finishes the
  proof of the lemma. 
\end{proof}

\subsubsection{Proof of Proposition 4.17}

We are now ready to prove Proposition 4.17.

\begin{proof}[Proof of Proposition 4.17]
Let $y \assign y(x) \leqslant (1/8)\log\log\log x$ be a parameter 
growing to infinity so slow so as to have
$1 - (\log\log x)^{-1/8} \geqslant 
|\{\mathfrak{h}(p)\} - \{\mathfrak{h}(q)\}| \geqslant 
(\log\log x)^{-1/8}$ for any two prime $p,q \leqslant y$ with 
$\{\mathfrak{h}(p)\}
\neq \{\mathfrak{h}(q)\}$. (Here $\{\mathfrak{h}(p)\}$ denotes the
fractional part of $\mathfrak{h}(p)$). Let $\mathfrak{H}$ be a strongly
  additive function defined by
  \[ \mathfrak{H}(p) = \left\{ \begin{array}{l}
       \left\lceil \mathfrak{h}(p) \right\rceil \text{ if } p \geqslant y\\
       \mathfrak{h}(p) \text{ \ \ \ otherwise}
     \end{array} \right. \]
  as in corollary 4.23. For an additive function $g$ we let $\Omega (g ; x) =
  \sum_{p \leqslant x} g (p) Z_p$. Since $\Omega (f ; x) = \Omega
  (\mathfrak{g}; x) + \Omega (\mathfrak{h}; x)$ and 
$\mathfrak{H} \geqslant \mathfrak{h} \geqslant 0,
  Z_p \geqslant 0$ we have
  \[ \Omega (\mathfrak{g}; x) + \Omega (\mathfrak{h}; y) \leqslant \Omega (f ;
     x) \leqslant \Omega (\mathfrak{g}; x) + \Omega (\mathfrak{H}; x) \]
  Therefore, for all $t \in \mathbbm{R}$,
  \[ \mathbbm{P}_{\mathcal{F}_x} \left( \Omega (\mathfrak{g}; x) + \Omega
     (\mathfrak{h}; y) \geqslant t) \leqslant \mathbbm{P}_{\mathcal{F}_x}
     \left( \Omega (f ; x) \geqslant t) \leqslant \mathbbm{P}_{\mathcal{F}_x}
     \left( \Omega (\mathfrak{g}; x) + \Omega (\mathfrak{H}; x) \geqslant t)
     \right. \right. \right. \]
  Now set $t \assign \xi_f (x ; \Delta) = \mu (f ; x) + \Delta \sigma (f ;
  x)$. Our goal is to show that both the upper and lower bound are
  asymptotic to the (asymptotic) expression given in proposition 4.17. We will
  carry out the proof only for the upper bound, because the proof for the 
  lower bound is almost identical. Let $v = v_f (x ; \Delta)$. By lemma
  4.7 the parameter $v$ is bounded when $1 \leqslant \Delta \leqslant c \sigma
  (f ; x)$. Denote by $\delta > 0$ a real such that $0 \leqslant v \leqslant
  \delta$ uniformly in $1 \leqslant \Delta \leqslant c \sigma (f ; x)$.
  Finally set $s \assign v + \mathi t = v_f (x ; \Delta) + \mathi t$. By
  assumptions
  \begin{equation}
    \mathbbm{E}_{\mathcal{F}_x} \left[ e^{s \Omega (\mathfrak{g}; x) + s
    \Omega (\mathfrak{H}; x)} \right] =\mathbbm{E}_{\mathcal{F}_x} \left[ e^{s
    \Omega (\mathfrak{g}; x)} \right] \cdot \prod_{p \leqslant x} \left( 1 -
    \frac{1}{p} + \frac{e^{s\mathfrak{H}(p)}}{p} \right) + O \left(
    \mathcal{E}(x ; v) \right)
  \end{equation}
  where $\mathcal{E}(x;v) = (\log x)^{\hat{\Psi}(f;v)-3/2}$. 
  Note that because $Z_p \in \{0 ; 1\}$ and $\Omega (\mathfrak{g}; x) \in
  \mathbbm{Z}$ all the values taken by $\Omega (\mathfrak{g}; x) + \Omega
  (\mathfrak{H}; x)$ lie in the set $\mathbbm{N}+\mathcal{D}_{\mathfrak{h}}
  (y)$ where $\mathcal{D}_{\mathfrak{h}} (y)$ is the set of fractional parts
  $\{ \sum_{p \leqslant y} \mathfrak{H}(p) \varepsilon_p \} = 
   \{ \sum_{p \leqslant y} \mathfrak{h}(p) \varepsilon_p \}$ , 
  $\varepsilon_i \in \{0 ; 1\}$ 
  (if it seems strange that in the fractional part we consider
  only the primes $p \leqslant y$ then recall that by definition
  $\mathfrak{H}(p) \in \mathbbm{Z}$ for $p > y$). In particular
  $|\mathcal{D}_{\mathfrak{h}} (y) | \leqslant 2^{\pi (y)} \leqslant 2^y
  \ll (\log\log x)^{1/8}$ (because $y \leqslant (1/8)\log\log\log x$). 
  Therefore we can write
  \begin{equation}
    \mathbbm{E}_{\mathcal{F}_x} \left[ e^{s \Omega (\mathfrak{g}; x) + s
    \Omega (\mathfrak{H}; x)} \right] = \sum_{\omega \in
    \mathbbm{N}+\mathcal{D}_{\mathfrak{h}} (y)} \mathbbm{P}_{\mathcal{F}_x}
    \left( \Omega (\mathfrak{g}; x) + \Omega (\mathfrak{H}; x) = \omega
    \right) \cdot e^{s \omega}
  \end{equation}
  And this sum converges because it's finite. In the same vein the main term
  in $(4.63)$ is the Laplace transform of the distribution function
  \begin{equation}
    F (x ; t) = \sum_{k \in \mathbbm{Z}} \mathbbm{P}_{\mathcal{F}_x} \left(
    \Omega (\mathfrak{g}; x) = k \right) \cdot \mathbbm{P}_{\mathcal{I}}
    \left( \sum_{p \leqslant x} \mathfrak{H}(p) X_p \leqslant t - k \right)
  \end{equation}
  Here the $X_p$ are independent Bernoulli random variable over some
  probability space $(\Theta ; \mathcal{I})$. Their distribution is given by
  $\mathbbm{P}(X_p = 1) = 1 / p$ and $\mathbbm{P}(X_p = 0) = 1 - 1 / p$. All
  the ``jumps'' of $F (x ; t)$ are contained in the set
  $\mathbbm{N}+\mathcal{D}_{\mathfrak{h}} (y)$. Thus the main term in $(4.63)$
  admits an expansion $\sum_{\omega \in \mathbb{N} + \mathcal{D}_\mathfrak{h}}
  \delta_\omega \cdot e^{s \omega}$ similar to the one in $(4.64)$. 
  Consider the ``kernel''
  \[ \mathfrak{K}(s ; t) = \sum_{\tmscript{\begin{array}{c}
       \omega \in \mathbbm{N}+\mathcal{D}_{\mathfrak{h}} (y)\\
       \omega \geqslant t
     \end{array}}} e^{- s \omega} \text{ } = \text{ }
     \sum_{\tmscript{\begin{array}{c}
       k \in \mathbbm{Z}\\
       k \geqslant t - 2
     \end{array}}} e^{- sk} \sum_{\tmscript{\begin{array}{c}
       d \in \mathcal{D}_{\mathfrak{h}} (y)\\
       k + d \geqslant t
     \end{array}}} e^{- sd} \]
  which is in modulus bounded by $\ll (1/v) e^{-vt} \cdot
  |\mathcal{D}_{\mathfrak{h}}
  (y) | \ll (1 / v) e^{- vt} \cdot (\log\log x)^{1/8}$. Let
  $\xi \assign \xi_f (x ; \Delta)$. We will keep this abbreviation in use.
  Multiplying the left hand side of $(4.63)$ by $\mathfrak{K}(s;\xi)$
  and integrating with 
  $(2 \log\log x)^{-1} \int_{- \log\log x}^{\log\log x} \dots
  \mathd t$ we obtain
  \begin{equation}
  \sum_{\omega \in \mathbb{N} + \mathcal{D}_{\mathfrak{h}}(y)} 
  \mathbb{P}_{\mathcal{F}_x}
  (\Omega(\mathfrak{g};x)+\Omega(\mathfrak{H};x) = \omega)
  \sum_{\omega' \in \mathbb{N} + \mathcal{D}_{\mathfrak{h}}(y), 
  \omega' \geqslant \xi} \frac{e^{v (\omega - \omega')}}{2 \log\log x}
  \int_{-\log\log x}^{\log\log x} e^{\mathi t (\omega - \omega')} \mathd t
  \nonumber
  \end{equation}
  If $\omega \neq \omega'$ (with $\omega,\omega' \in \mathbb{N} +
  \mathcal{D}_{\mathfrak{h}}(y)$) then by our choice of $y$ we have $
  1 - (\log\log x)^{-1/8} \geqslant |\{\omega\}
  - \{\omega'\}| \geqslant (\log\log x)^{-1/8}$. It follows that 
  $(2 \log\log x)^{-1} \int_{-\log\log x}^{\log\log x} e^{\mathi t
  (\omega - \omega')} \mathd t = \mathbb{I}_{\omega = \omega'} +
  O((\log\log x)^{-7/8})$. (Here $\mathbb{I}_{\omega = \omega'}$ 
  is the indicator function of $\omega = \omega'$).
  Hence the previous equation simplifies to
  \begin{equation}
  \mathbb{P}_{\mathcal{F}_x}
  (\Omega(\mathfrak{g};x) + \Omega(\mathfrak{H};x) \geqslant \xi)
  + O(\mathbb{E}_{\mathcal{F}_x} \left[ e^{v (\Omega(\mathfrak{g};x) + 
  \Omega(\mathfrak{H};x))} \right] (1/v) e^{-v \xi} \cdot 
  |\mathcal{D}_{\mathfrak{h}}(y)| \cdot (\log\log x)^{-7/8})
  \nonumber
  \end{equation}
  and by our assumptions and lemma 4.8 the error term is $\ll (1/v)
  (\log x)^{A(f;v)} \cdot (\log\log x)^{-3/4}$, where as usual
  $A(f;v) = \hat{\Psi}(f;v) - v\hat{\Psi}'(f;v) - 1$. 
  Similarly, multiplying the right hand side of $(4.63)$ by
  $\mathfrak{K}(s;\xi)$ and
  then integrating over $(2 \log\log x) \int_{-\log\log x}^{\log\log x}
  \dots \mathd t$ gives
  \begin{equation}
  1 - F(x; \xi) + O((1/v)(\log x)^{A(f;v)} \cdot (\log\log x)^{-3/4})
  \nonumber
  \end{equation}
  (Where $F(x;t)$ is defined by $(4.65)$).
  Therefore multiplying both sides of $(4.63)$ by $\mathfrak{K}(s;\xi)$
  and integrating as we've done before, we obtain the equality
  \begin{equation}
  \mathbbm{P}_{\mathcal{F}_x} \left( \Omega (\mathfrak{g}; x) + \Omega
    (\mathfrak{H}; x) \geqslant \xi \right) = 1 - F(x; \xi)
    + O((1/v) (\log x)^{A(f;v)} \cdot (\log\log x)^{-3/4})
  \end{equation}   
  Note that the error term is negligible compared to the (expected) size of 
  the main. Thus in view of $(4.66)$ and our earlier remark it remains 
  to estimate $1 - F(x; \xi)$. 
  To ease notation let $\Omega_X (\mathfrak{H}; x)
  = \sum_{p \leqslant x} \mathfrak{H}(p) X_p$ where the $X_p$ are independent
  Bernoulli random variables over the probability space $(\Theta ;
  \mathcal{I})$. Rewriting $(4.65)$ and using Cauchy's formula, we obtain
  \begin{eqnarray}
    1 - F (x ; \xi) & = & \sum_{k \in \mathbbm{Z}} \mathbbm{P}_{\mathcal{F}_x}
    \left( \Omega (\mathfrak{g}; x) = \lfloor \xi \rfloor - k \right) \cdot
    \mathbbm{P}_{\mathcal{I}} \left( \Omega_X (\mathfrak{H}; x) \geqslant
    \{\xi\}+ k \right) \nonumber\\
    & = & \sum_{k \in \mathbbm{Z}} \left[ \int_{- \pi}^{\pi}
    \mathbbm{E}_{\mathcal{F}_x} \left[ e^{s \Omega (\mathfrak{g}; x)} \right]
    e^{- s \lfloor \xi \rfloor + sk} \cdot \frac{\mathd t}{2 \pi} \right]
    \cdot \mathbbm{P}_{\mathcal{I}} \left( \Omega_X (\mathfrak{H}; x)
    \geqslant \{\xi\}+ k \right) \nonumber\\
    & = & \int_{- \pi}^{\pi} \mathbbm{E}_{\mathcal{F}_x} \left[ e^{s \Omega
    (\mathfrak{g}; x)} \right] e^{- s \lfloor \xi \rfloor} \cdot \sum_{k \in
    \mathbbm{Z}} e^{sk} \mathbbm{P}_{\mathcal{I}} \left( \Omega_X
    (\mathfrak{H}; x) \geqslant \{\xi\}+ k \right) \cdot \frac{\mathd t}{2
    \pi} 
  \end{eqnarray}
  We massage the above expression. Let $X (\mathfrak{h}) \assign \sum_p
  \mathfrak{h}(p) X_p$. By Corollary 4.23,
  \begin{eqnarray}
    &  & \sum_{k \geqslant 0} e^{sk} \cdot \mathbbm{P}_{\mathcal{I}} \left(
    \Omega_X \left( \mathfrak{H}; x \right) \geqslant \{\xi\}+ k \right)
    \nonumber\\
    & = & \sum_{k \geqslant 0} e^{sk} \cdot \left[ \mathbbm{P}_{\mathcal{I}}
    \left( X (\mathfrak{h}) \geqslant \{\xi\}+ k \right) + O_{\delta} \left(
    \frac{e^{- (2 \delta + 1) k}}{(\log y)^{1 / 4}} \right) \right] 
  \end{eqnarray}
  Since $\tmop{Re} s = v \leqslant \delta$ the error term simplifies to $O
  ((\log y)^{- 1 / 4})$. Also, note that for $k < 0$ we trivially have
  $\mathbbm{P}_{\mathcal{I}} (\Omega_X (\mathfrak{H}; x) \geqslant \{\xi\}+ k)
  = 1 =\mathbbm{P}_{\mathcal{I}} (X (\mathfrak{h}) \geqslant \{\xi\}+ k)$.
  Thus the identity $\sum_{k < 0} e^{sk} \mathbbm{P}_{\mathcal{I}} (\Omega_X
  (\mathfrak{H}; x) \geqslant \{\xi\}+ k) = \sum_{k < 0} e^{sk}
  \mathbbm{P}_{\mathcal{I}} (X (\mathfrak{h}) \geqslant \{\xi\}+ k)$ holds.
  Adding this identity to $(4.68)$, we obtain
  \begin{eqnarray*}
    \sum_{k \in \mathbbm{Z}} e^{sk} \mathbbm{P}_{\mathcal{I}} \left( \Omega_X
    (\mathfrak{H}; x) \geqslant \{\xi\}+ k \right) & = & \sum_{k \in
    \mathbbm{Z}} e^{sk} \mathbbm{P}_{\mathcal{I}} \left( X (\mathfrak{h})
    \geqslant \{\xi\}+ k \right) + O_{\delta} \left( (\log y)^{- 1 / 4}
    \right)\\
    & = & e^{- s\{\xi\}} \cdot (1 / s)\mathcal{P}_{\mathfrak{h}} (\xi ; s) +
    O_{\delta} \left( (\log y)^{- 1 / 4} \right)
  \end{eqnarray*}
  Inserting the above into $(4.67)$ yields
  \begin{equation}
    1 - F (x ; \xi) = \int_{- \pi}^{\pi} \mathbbm{E}_{\mathcal{F}_x} \left[
    e^{s \Omega (\mathfrak{g}; x)} \right] \frac{\mathcal{P}_{\mathfrak{h}}
    (\xi ; s) \mathd t}{e^{s \xi} \cdot 2 \pi s} \upl O \left( \frac{\int_{-
    \pi}^{\pi} |\mathbbm{E}_{\mathcal{F}_x} [e^{s \Omega (\mathfrak{g}; x)}] |
    \mathd t}{e^{v \xi} \cdot (\log y)^{1 / 4}} \right)
  \end{equation}
  Since $0 \leqslant v \leqslant \delta$ we have $\mathbbm{E}_{\mathcal{F}_x}
  \left[ e^{s \Omega (\mathfrak{g}; x)} \right] \ll_{\delta} (\log
  x)^{\tmop{Re} ( \hat{\Psi} (f ; s)) - 1} + (\log x)^{\hat{\Psi} (f ; v) - 3
  / 2}$, by assumptions, because $\mathcal{A}(s)$ is analytic hence bounded in
  the (bounded) region $0 \leqslant \tmop{Re} s \leqslant \delta$ and $|
  \tmop{Im} s| \leqslant 2 \pi$. Therefore the error term in $(4.69)$ is
  bounded by
  \begin{equation}
    \frac{(\log x)^{\hat{\Psi} (f ; v) - 1}}{(\log y)^{1 / 4}} \cdot e^{- v
    \xi} \int_{- \pi}^{\pi} (\log x)^{\tmop{Re} ( \hat{\Psi} (f ; s) -
    \hat{\Psi} (f ; v))} \cdot \mathd t + (\log x)^{\hat{\Psi} (f ; v) - 3 /
    2} \cdot e^{- v \xi}
  \end{equation}
  Since $\Psi (f ; t)$ is lattice distributed we have $\hat{\Psi} (f ; s) =
  \sum_{k \geqslant 0} \lambda_k e^{zk}$ with $\lambda_k \geqslant 0$ not all
  zero. Thus $\tmop{Re} ( \hat{\Psi} (f ; v + \mathi t) - \hat{\Psi} (f ; v))
  \leqslant \lambda_k (\cos (kt) - 1)$ for some $k$ with $\lambda_k > 0$.
  Hence the integral in $(4.70)$ is $\leqslant \int (\log x)^{\lambda_k (\cos
  (kt) - 1)} \mathd t \ll (\tmop{loglog} x)^{- 1 / 2}$. Thus, $(4.70)$ is
  bounded by $(\log x)^{\hat{\Psi} (f ; v) - 1} e^{- v \xi} (\tmop{loglog}
  x)^{- 1 / 2} \cdot (\log y)^{- 1 / 4}$ which is $\ll 
  (\log x)^{A(f;v)} (\log\log x)^{-1/2} (\log y)^{-1/4}$ by lemma 4.8. 
  Furthermore we have
  \begin{eqnarray}
    &  & \int_{- \pi}^{\pi} \mathbbm{E}_{\mathcal{F}_x} \left[ e^{s \Omega
    (\mathfrak{g}; x)} \right] e^{- s \xi} \cdot
    \frac{\mathcal{P}_{\mathfrak{h}} (\xi ; s)}{s} \cdot \frac{\mathd t}{2
    \pi} \\
    & = & \int_{- \pi}^{\pi} \mathcal{A}(s) \left( \log x \right)^{\hat{\Psi}
    (f ; s) - 1} e^{- s \xi} \cdot \frac{\mathcal{P}_{\mathfrak{h}} (\xi ;
    s)}{s} \cdot \frac{\mathd t}{2 \pi} + O_{\delta} \left( (1 / v) (\log
    x)^{\hat{\Psi} (f ; v) - 3 / 2} e^{- v \xi} \right) \nonumber
  \end{eqnarray}
  because $\mathbbm{E}_{\mathcal{F}_x} \left[ e^{s \Omega (\mathfrak{g}; x)}
  \right] =\mathcal{A}(s) (\log x)^{\hat{\Psi} (f ; s) - 1} + O ((\log
  x)^{\hat{\Psi} (f ; v) - 3 / 2})$ by assumptions,
  $|\mathcal{P}_{\mathfrak{h}} (a ; s) | = O_{\delta} (1)$ by lemma 4.25 and
  $\mathcal{A}(s) = O_{\delta} (1)$ because $s = v + \mathi t$ lies in a
  bounded domain and $\mathcal{A}(s)$ is an analytic function. By lemma 4.8
  the error term in $(4.71)$ is $\ll (1 / v) (\log x)^{A(f;v) - 1/2}$. 
  Collecting $(4.66)$, $(4.69)$ and $(4.71)$ gives
  \begin{equation}
    \mathbbm{P}_{\mathcal{F}_x} \left( \Omega (\mathfrak{g}; x) + \Omega
    (\mathfrak{H}; x) \geqslant \xi \right) = \int_{- \pi}^{\pi}
    \mathcal{A}(s) \left( \log x \right)^{\hat{\Psi} (f ; s) - 1} e^{- s \xi} 
    \frac{\mathcal{P}_{\mathfrak{h}} (\xi ; s) \mathd t}{2 \pi s} + O \left(
    \tmop{Err} \right)
  \end{equation}
  where $\tmop{Err} \assign (1 / v) (\log x)^{A(f;v)} 
  \cdot (\tmop{loglog} x)^{- 1 / 2} \cdot
  (\tmop{log} y)^{- 1 / 4}$ (and $A(f;v) = \hat{\Psi}(f;v) - v
  \hat{\Psi}'(f;v) - 1$). By lemma 4.24 the integral in $(4.72)$ is
  asymptotic to
  \begin{equation}
    \mathcal{A}(v) (1 / v)\mathcal{P}_{\mathfrak{h}} (\xi_f (x ; \Delta) ; v)
    \cdot \frac{\left( \log x \right)^{\hat{\Psi} (f ; v) - 1 - v \hat{\Psi}'
    (f ; v)}}{(2 \pi \hat{\Psi}'' (f ; v) \tmop{loglog} x)^{1 / 2}} \cdot e^{-
    vc (f)} \cdot (1 + o (1))
  \end{equation}
  uniformly in $(\tmop{loglog} x)^{\varepsilon} \ll \Delta \ll \sigma (f ;
  x)$. Since $0 \leqslant v \leqslant \delta$ is bounded we have
  $\mathcal{A}(v) \gg_{\delta} 1$ because $\mathcal{A}(\cdot)$ is continuous
  and non-zero on the positive real axis, and $\mathcal{P}_{\mathfrak{h}} (\xi
  ; v) \gg_{\delta} 1$ by lemma 4.25. Thus
  $\mathcal{A}(v)\mathcal{P}_{\mathfrak{h}} (\xi ; v) \gg_{\delta} 1$. Since
  in addition
  $\log y \longrightarrow \infty$ the error term $\tmop{Err}$ in $(4.72)$
  is negligible compared to $(4.73)$. 
  By $(4.72)$ and $(4.73)$, it follows that $\mathbbm{P}_{\mathcal{F}_x}
  (\Omega (\mathfrak{g}; x) + \Omega (\mathfrak{H}; x) \geqslant \xi)$ is
  asymptotic to $(4.73)$. Because
  \[ \mathbbm{P}_{\mathcal{F}_x} \left( \Omega (f ; x) \geqslant \xi)
     \leqslant \mathbbm{P}_{\mathcal{F}_x} \left( \Omega (\mathfrak{g}; x) +
     \Omega (\mathfrak{H}; x) \geqslant \xi) \right. \right. \]
  this gives an upper bound for $\mathbbm{P}_{\mathcal{F}_x} (\Omega (f ; x)
  \geqslant \xi)$ that is ``asymptotically'' correct. In the same way as above
  we establish that $\mathbbm{P}_{\mathcal{F}_x} (\Omega (\mathfrak{g}; x) +
  \Omega (\mathfrak{h}; y) \geqslant \xi)$ is asymptotic to $(4.73)$. Since
  $\mathbbm{P}_{\mathcal{F}_x} (\Omega (\mathfrak{g}; x) + \Omega
  (\mathfrak{h}; y) \geqslant \xi) \leqslant \mathbbm{P}_{\mathcal{F}_x}
  (\Omega (f ; x) \geqslant \xi)$ this gives a lower bound for
  $\mathbbm{P}_{\mathcal{F}_x} (\Omega (f ; x) \geqslant \xi)$ that is
  ``asymptotically'' correct. The proposition follows. 
\end{proof}

\section{An asymptotic for $\mathcal{D}_f (x ; \Delta)$}

The object of this section is to prove Theorem 2.8. This theorem turns out to
be consequence of the three general propositions from the previous section. We
break down the proof into three parts, corresponding to the case when
$(\tmop{loglog} x)^{\varepsilon} \ll \Delta \ll \sigma (f ; x)$ and $\Psi (f ;
t)$ is not lattice distributed, the case when $(\tmop{loglog} x)^{\varepsilon}
\ll \Delta \ll \sigma (f ; x)$ and $\Psi (f ; t)$ is lattice distributed on
$\alpha \mathbbm{Z}$, and the remaining case when $\Delta$ is in the range $1
\leqslant \Delta = o (\sigma (f ; x))^{}$. \ \

\subsection{$\mathcal{D}_f (x ; \Delta)$ when $\Psi (f ; t)$ is not lattice
distributed}

\begin{lemma}
  Let $f \in \mathcal{C}$. For any given $\kappa > 0$ the function $1 / \Gamma
  ( \hat{\Psi} (f ; z))$ is uniformly bounded in $\tmop{Re} z \leqslant
  \kappa$.
\end{lemma}

\begin{proof}
  We have $| \hat{\Psi} (f ; z) | \leqslant \hat{\Psi} (f ; \kappa)$ for
  $\tmop{Re} z \leqslant \kappa$. The function $1 / \Gamma (z)$ is entire,
  hence bounded for $|z| \leqslant \hat{\Psi} (f ; \kappa)$. It follows that
  $1 / \Gamma ( \hat{\Psi} (f ; z))$ is bounded for $\tmop{Re} z \leqslant
  \kappa$. 
\end{proof}

\begin{proof}[Proof of Part (3) of Theorem 2.8]
  Consider a random variable $\Omega (f ; x)$ with distribution function
  \begin{eqnarray}
    \mathbbm{P} \left( \Omega (f ; x) \leqslant t \right) & = & (1 / \lfloor x
    \rfloor) \sum_{\tmscript{\begin{array}{c}
      n \leqslant x\\
      f (n) \leqslant t
    \end{array}}} 1_{} 
  \end{eqnarray}
  Since $f \in \mathcal{C}$, by the mean-value theorem of Proposition 4.1
  \begin{eqnarray}
    \mathbbm{E} \left[ e^{s \Omega (f ; x)} \right] & = & \frac{1}{\lfloor x
    \rfloor} \sum_{n \leqslant x} e^{sf (n)} \nonumber\\
    & = & \frac{L (f ; s)}{\Gamma ( \hat{\Psi} (f ; s))} \cdot (\log
    x)^{\hat{\Psi} (f ; s) - 1} + O \left( (\log x)^{\hat{\Psi} (f ; \kappa) -
    3 / 2} \right) 
  \end{eqnarray}
  uniformly in $0 \leqslant \kappa \assign \tmop{Re} s \leqslant C$, $|
  \tmop{Im} s| \leqslant \tmop{loglog} x$, for any given $C > 0$. By lemma 4.4
  the function $L (f ; s)$ is entire, and bounded by $L (f ; s) = O_{C,
  \varepsilon} (1 + | \tmop{Im} s|^{\varepsilon})$ uniformly in $0 \leqslant
  \tmop{Re} s \leqslant C$. Furthermore, from the product representation for
  $L (f ; z)$, it is clear that $L (f ; x)$ does not vanish for any $x
  \geqslant 0$. The same properties hold true for $1 / \Gamma ( \hat{\Psi} (f
  ; s))$. Indeed, by lemma 4.2 the function $\hat{\Psi} (f ; s)$ is entire,
  hence $1 / \Gamma ( \hat{\Psi} (f ; s))$ is. All the zeroes of $1 / \Gamma
  (s)$ are located in $\tmop{Re} s \leqslant 0$. Hence $1 / \Gamma (
  \hat{\Psi} (f ; x))$ does not vanish, because $\hat{\Psi} (f ; x) \geqslant
  \hat{\Psi} (f ; 0) > 0$ for $x \geqslant 0$. Finally by Lemma 5.1 the
  function $1 / \Gamma ( \hat{\Psi} (f ; s))$ is uniformly bounded in
  $\tmop{Re} s \leqslant C$, for any given $C > 0$. It follows that the
  product
  \[ \mathcal{A}(s) \assign \frac{L (f ; s)}{\Gamma ( \hat{\Psi} (f ; s))} \]
  is entire, non-vanishing on the positive real line, and $\mathcal{A}(s) =
  O_{C, \varepsilon} (1 + | \tmop{Im} s|^{\varepsilon})$ uniformly in $0
  \leqslant \tmop{Re} s \leqslant C$, for any given $C > 0$. In addition
  $(5.2)$ holds. Hence our second ``general result'' -- proposition 4.10 --
  applies, and we obtain that uniformly in $(\tmop{loglog} x)^{\varepsilon}
  \ll \Delta \leqslant c \sigma (f ; x)$,
  \[ \mathbbm{P} \left( \frac{\Omega (f ; x) \um \mu (f ; x)}{\sigma (f ; x)}
     \geqslant \Delta \right) \sim \frac{L (f ; v)}{\Gamma ( \hat{\Psi} (f ;
     v))} \cdot \frac{\left( \log x \right)^{\hat{\Psi} (f ; v) - 1 - v
     \hat{\Psi}' (f ; v)} e^{- vc (f)}}{v (2 \pi \hat{\Psi}'' (f ; v)
     \tmop{loglog} x)^{1 / 2}} \text{ }, \text{} v = v_f (x ; \Delta) \]
  By $(5.1)$ the term on the left hand side equals to $\mathcal{D}_f (x ;
  \Delta)$. The result follows. \ 
\end{proof}

\subsection{$\mathcal{D}_f (x ; \Delta)$ when $\Psi (f ; t)$ is lattice
distributed on $\mathbbm{Z}$}

As usual when $\Psi (f ; t)$ is lattice distributed we introduce the strongly
additive functions $\mathfrak{g}$ and $\mathfrak{h}$ defined by
\begin{eqnarray*}
  \mathfrak{g}(p) = \left\{ \begin{array}{l}
    f (p) \text{ if } f (p) \in \mathbbm{Z}\\
    0 \text{ \ \ \ \ otherwise}
  \end{array} \right. & \tmop{and} & \mathfrak{h}(p) = \left\{
  \begin{array}{l}
    f (p) \text{ if } f (p) \nin \mathbbm{Z}\\
    0 \text{ \ \ \ \ otherwise}
  \end{array} \right.
\end{eqnarray*}
Of course $f =\mathfrak{g}+\mathfrak{h}$. The next lemma is proved by a
rather standard convolution argument (note that by lemma 4.4, we already
have an asymptotic for $\sum_{n \leq x} e^{s f(n)}$).

\begin{lemma}
  Let $f \in \mathcal{C}$. Suppose that $\Psi (f ; t)$ is lattice distributed
  on $\mathbbm{Z}$. Given $C > 0$, uniformly in $0 \leqslant \kappa \assign
  \tmop{Re} s \leqslant C, | \tmop{Im} s| \leqslant \log\log x$ and strongly
  additive functions $\mathfrak{H}(\cdot)$ such that $0 \leqslant
  \mathfrak{H}(p) \leqslant \left\lceil \mathfrak{h}(p) \right\rceil$,
  \[ \frac{1}{x} \sum_{n \leqslant x} e^{s\mathfrak{g}(n) + s\mathfrak{H}(n)}
     = \frac{1}{x} \sum_{n \leqslant x} e^{s\mathfrak{g}(n)} \cdot \prod_{p
     \leqslant x} \left( 1 + \frac{e^{s\mathfrak{H}(p)} - 1}{p} \right) + O_C
     \left( (\log x)^{\hat{\Psi} (f ; \kappa) - 3 / 2} \right) \]
  Furthermore, for any given $C > 0$, uniformly in $0 \leqslant \kappa \assign
  \tmop{Re} s \leqslant C$, $| \tmop{Im} s| \leqslant 2 \pi$,
  \[ \frac{1}{x} \sum_{n \leqslant x} e^{s\mathfrak{g}(n)} = \frac{L
     (\mathfrak{g}; s)}{\Gamma ( \hat{\Psi} (f ; s))} \cdot \left( \log x
     \right)^{\hat{\Psi} (f ; s) - 1} + O_C \left( (\log x)^{\hat{\Psi} (f ;
     \kappa) - 3 / 2} \right) \]
  
\end{lemma}

\begin{proof}
  Let $S (\mathfrak{h}) = \left\{ p : \mathfrak{h}(p) \neq 0\} \right.$. Using
  the definition of $\mathfrak{g}$ and $\mathfrak{h}$ we find
  \begin{eqnarray*}
    \sum_{n \geqslant 1} \frac{e^{zf (n)}}{n^s} & = & \prod_{p \nin S
    (\mathfrak{h})} \left( 1 + \frac{e^{z\mathfrak{g}(p)}}{p^s - 1} \right)
    \cdot \prod_{p \in S (\mathfrak{h})} \left( 1 +
    \frac{e^{z\mathfrak{h}(p)}}{p^s - 1} \right)\\
    & = & \sum_{n \geqslant 1} \frac{e^{z\mathfrak{g}(n)}}{n^s} \cdot
    \prod_{p \in S (\mathfrak{h})} \left( 1 + \frac{e^{z\mathfrak{h}(p)} -
    1}{p^s} \right)
  \end{eqnarray*}
  Note that $\mathfrak{H}(p)$ vanishes when $\mathfrak{h}(p)$ does. Hence
  $\mathfrak{H}(p) = 0$ when $p \nin S (\mathfrak{h})$. Therefore we can write
  \begin{eqnarray}
    \sum_{n \geqslant 1} \frac{e^{z\mathfrak{g}(n) + z\mathfrak{H}(n)}}{n^s} &
    = & \prod_{p \nin S (\mathfrak{h})} \left( 1 +
    \frac{e^{z\mathfrak{g}(p)}}{p^s - 1} \right) \prod_{p \in S
    (\mathfrak{h})} \left( 1 + \frac{e^{z\mathfrak{H}(p)}}{p^s - 1} \right)
    \nonumber\\
    & = & \sum_{n \geqslant 1} \frac{e^{z\mathfrak{g}(n)}}{n^s} \prod_{p \in
    S (\mathfrak{h})} \left( 1 + \frac{e^{z\mathfrak{H}(p)} - 1}{p^s} \right)
    \nonumber\\
    & = & \sum_{n \geqslant 1} \frac{e^{zf (n)}}{n^s} \prod_{p \in S
    (\mathfrak{h})} \frac{1 + (e^{z\mathfrak{H}(p)} - 1) \cdot p^{- s}}{1 +
    (e^{z\mathfrak{h}(p)} - 1) \cdot p^{- s}} \nonumber\\
    & = & \sum_{n \geqslant 1} \frac{e^{zf (n)}}{n^s} \cdot \sum_{n \geqslant
    1} \frac{g (z ; n)}{n^s} 
  \end{eqnarray}
  Here the function $g (z ; n)$ is multiplicative, and given explicitly by $g
  (z ; p^{\alpha}) = (- 1)^{\alpha} \cdot (e^{z\mathfrak{h}(p)} - 1)^{\alpha -
  1} \cdot (e^{z\mathfrak{h}(p)} - e^{z\mathfrak{H}(p)})$. To proceed we need
  to make a few simple remarks about $g (z ; n)$. Since $0 \leqslant
  \mathfrak{H}(p) \leqslant \mathfrak{h}(p) + 1$ we have for $0 \leqslant
  \kappa \assign \tmop{Re} z$,
  \[ |g (z ; p^{\alpha}) | \leqslant \left( 2 e^{\kappa \mathfrak{h}(p)}
     \right)^{\alpha - 1} \cdot 2 e^{\kappa (\mathfrak{h}(p) + 1)} \leqslant
     (2 e^{\kappa})^{\alpha} \cdot e^{\kappa \mathfrak{h}(p) \alpha} \]
  Hence $|g (z ; n) | \leqslant (2 e^{\kappa})^{\Omega (n)} \cdot e^{\kappa h
  (n)}$ where $h (n)$ is an additive function defined by $h (p^{\alpha})
  =\mathfrak{h}(p) \alpha$. In particular $|g (z ; n) | \leqslant (2
  e^C)^{\Omega (n)} \cdot e^{Ch (n)}$ in the half-plane $\tmop{Re} z \leqslant
  C$. Note also that $g (z ; p^{\alpha}) = 0$ whenever $p \nin S
  (\mathfrak{h})$. Therefore $g (z ; n) = 0$ unless all the prime factors of
  $n$ belong to $S (\mathfrak{h})$. We are now ready to start the proof of the
  lemma. Because of $(5.3)$,
  \[ \sum_{n \leqslant x} e^{z\mathfrak{g}(n) + z\mathfrak{H}(n)} = \sum_{d
     \leqslant x} g (z ; d) \cdot \sum_{n \leqslant x / d} e^{zf (n)} \]
  To evaluate the above sum we use proposition 4.1. By proposition 4.1, for 
  any fixed $C > 0$, the above sum equals to
  \begin{equation}
    \frac{L (f ; z)}{\Gamma ( \hat{\Psi} (f ; z))} \sum_{d \leqslant x}
    \frac{g (z ; d)}{d} \left( \log \frac{x}{d} \right)^{\hat{\Psi} (f ; z) -
    1} + O \left( \sum_{d \leqslant x} \frac{|g (z ; d) |}{d} \cdot \left(
    \log x \right)^{\hat{\Psi} (f ; \kappa) - 3 / 2} + 
    \sum_{d \geqslant \sqrt{x}} \frac{|g(z;d)|}{d} \right).
  \end{equation}
  uniformly in $0 \leqslant \kappa \assign \tmop{Re} z \leqslant C$ and $|
  \tmop{Im} z| \leqslant \log\log x$. We'll see in a second (see discussion
  after equation $(5.6)$) that 
  $\sum_{d \geqslant \sqrt{x}} |g(z;d)| \cdot d^{-1} \ll (\log x)^{-1}$.
  As for the remaining sum in the error term, we bound 
  $\sum_{d \leqslant x} |g(z;d)|
  \cdot d^{- 1}$ by an Euler product, and inside the Euler product we bound
  $|g (z ; p^{\alpha}) |$ by $(2 e^C)^{\alpha} \cdot (e^{C\mathfrak{h}(p)
  \alpha})$. Note that the Euler product will be taken over the primes $p \in
  S (\mathfrak{h})$ because $g (z ; d) = 0$ unless all the primes factors of
  $d$ are in $S (\mathfrak{h})$. Thus
  \[ \sum_{d \leqslant x} \frac{|g (z ; d) |}{d} \leqslant \prod_{p \in S
     (\mathfrak{h})} \left( 1 + \sum_{\alpha \geqslant 1} \frac{(2
     e^C)^{\alpha} \cdot e^{C\mathfrak{h}(p) \alpha}}{p^{\alpha}} \right) \]
  Since $\mathfrak{h}(p) = o (\log p)$ (to see this: by $(1.3)$ $f (p) = o
  (\log p)$ hence $\mathfrak{h}(p) = o (\log p)$) there is an constant $K
  \assign K (C) > 0$ such that the above product is bounded by $\prod_{p \in S
  (\mathfrak{h})} (1 + K \cdot e^{C\mathfrak{h}(p)} \cdot p^{- 1})$. This last
  product is finite by lemma 4.19. Hence the
  error term in $(5.4)$ is $\ll (\log x)^{-1} + 
  (\log x)^{\hat{\Psi}(f;\kappa)-3/2} \ll (\log x)^{\hat{\Psi}(f;\kappa)-3/2}$.
  It remains to estimate the main term in $(5.4)$. First we rewrite the main
  term as
  \begin{equation}
    \frac{L (f ; z)}{\Gamma ( \hat{\Psi} (f ; z))} \cdot \left( \log x
    \right)^{\hat{\Psi} (f ; z) - 1} \sum_{d \leqslant x} \frac{g (z ; d)}{d}
    \cdot \left( 1 - \frac{\log d}{\log x} \right)^{\hat{\Psi} (f ; z) - 1}
  \end{equation}
  We split the sum over $d \leqslant x$ into two ranges. The range $1
  \leqslant d \leqslant y \assign \exp \left( ({\log x})^{1/4} \right)$ and the
  remaining range $d \geqslant y$ on which we simply bound by $\sum_{d > y} |g
  (z ; d) | \cdot d^{- 1}$. In the range $d \leqslant y$ we use $(1 - \log d /
  \log x)^{\hat{\Psi} (f ; z) - 1} = 1 + O ((\log x)^{-3/4})$, which is
  valid because $| \hat{\Psi} (f ; z) | \leqslant \hat{\Psi} (f ; C)$ and
  $\log d \ll (\log x)^{1/4}$. Thus
  \begin{eqnarray}
    &  & \sum_{d \leqslant x} \frac{g (z ; d)}{d} \cdot \left( 1 - \frac{\log
    d}{\log x} \right)^{\hat{\Psi} (f ; z) - 1} \\
    & = & \sum_{d \leqslant y} \frac{g (z ; d)}{d} \cdot \left( 1 + O \left(
    \frac{1}{(\log x)^{3/4}} \right) \right) + O \left( \sum_{d \geqslant y}
    \frac{|g (z ; d) |}{d} \right) \nonumber\\
    & = & \prod_p \left( 1 + \sum_{\alpha \geqslant 1} \frac{g (z ;
    p^{\alpha})}{p^{\alpha}} \right) + O \left( \frac{1}{(\log x)^{3/4}}
    \sum_{d \leqslant y} \frac{|g (z ; d) |}{d} \right) + O \left( \sum_{d
    \geqslant y} \frac{|g (z ; d) |}{d} \right) \nonumber
  \end{eqnarray}
  We bound the second error term in the exactly the same way as before,
  getting a bound of $O ((\log x)^{-3/4})$. The third error term requires
  a different approach. Recall that $g (z ; n)$ vanishes if not all the prime
  factors of $n$ are in $S (\mathfrak{h})$. Therefore to the sum $\sum_{d
  \geqslant y} |g (z ; d) | \cdot d^{- 1}$ we can add the condition $p|d
  \Rightarrow p \in S (\mathfrak{h})$ without altering its value. Furthermore
  using the inequality $|g (z ; n) | \leqslant (2 e^C)^{\Omega (n)} \cdot
  e^{Ch (n)}$ and then applying Cauchy-Schwarz's inequality we obtain
  \begin{eqnarray*}
    \sum_{\tmscript{\begin{array}{c}
      d \geqslant y\\
      p|d \Rightarrow p \in S (\mathfrak{h})
    \end{array}}} \frac{|g (z ; d) |}{d} & \leqslant & \biggl( \sum_{p|n
    \Rightarrow p \in S (\mathfrak{h})} \frac{\left( 4 e^{2 C} \right)^{\Omega
    (n)}}{n} \biggr)^{1 / 2} \cdot \biggl( \sum_{\tmscript{\begin{array}{c}
      n \geqslant y\\
      p|n \Rightarrow p \in S (\mathfrak{h})
    \end{array}}} \frac{e^{2 Ch (n)}}{n} \biggr)^{1 / 2}
  \end{eqnarray*}
  The first sum warps into an Euler product which is finite by lemma 4.18 (and
  some elementary bounding). To the second sum we apply once again a
  Cauchy-Schwarz inequality, thus obtaining the upper bound
  \[ \leqslant \biggl( \sum_{p|n \Rightarrow p \in S (\mathfrak{h})} \frac{e^{4
     Ch (n)}}{n} \biggr)^{1 / 2} \cdot \biggl( \sum_{\tmscript{\begin{array}{c}
       p|n \Rightarrow p \in S (\mathfrak{h})\\
       n \geqslant y
     \end{array}}} \frac{1}{n} \biggr)^{1 / 2} \]
  Again the first sum can be rewritten as a (finite) Euler product. The
  second sum is bounded by 
  $\ll (\log y)^{- A} \cdot \sum_{p|n \Rightarrow p \in S
  (\mathfrak{h})} (\log n)^A \cdot n^{- 1} \ll (\log y)^{-A}$ where the 
  second bound comes from the corollary to lemma 4.18. It now follows
  that $\sum_{d \geqslant y} |g (z ; d) | \cdot d^{- 1} \ll (\log
  y)^{- A}$. Inserting this estimate in $(5.6)$ yields
  \begin{eqnarray*}
    \sum_{d \leqslant x} \frac{g (z ; d)}{d} \cdot \left( 1 - \frac{\log
    d}{\log x} \right)^{\hat{\Psi} (f ; z) - 1} & = & \prod_p \left( 1 +
    \sum_{\alpha \geqslant 1} \frac{g (z ; p^{\alpha})}{p^{\alpha}} \right) +
    O \left( \frac{1}{(\log x)^{3/4}} \right)\\
    & = & \prod_p \frac{1 + (e^{z\mathfrak{H}(p)} - 1) \cdot p^{- 1}}{1 +
    (e^{z\mathfrak{h}(p)} - 1) \cdot p^{- 1}} + O \left( \frac{1}{(\log
    x)^{3/4}} \right)
  \end{eqnarray*}
  In the second line we simply use the definition of $g (z ; p^{\alpha})$
  (see $(5.3)$). By
  lemma 4.4 and lemma 5.1 we have $L (f ; z) / \Gamma ( \hat{\Psi} (f ;z))
  \ll_{C} 1 + |\tmop{Im } z| \ll_C \log\log x$ 
  uniformly in $0 \leqslant \tmop{Re } z \leqslant C$, 
  $|\tmop{Im } z| \leqslant \log\log x$.   
  Multiplying both sides of the above equation by 
  $L (f ; z) / \Gamma ( \hat{\Psi} (f ; z)) (\log x)^{\hat{\Psi}(f;z)-1}$
  gives an asymptotic for $(5.5)$. In turn an asymptotic for $(5.5)$ allows us
  to evaluate $(5.4)$ (because $(5.5)$ is the main term for $(5.4)$). Since
  $(5.4)$ is equal to $(1 / x) \sum_{n \leqslant x} e^{z\mathfrak{g}(n) +
  z\mathfrak{H}(n)}$ we conclude that
  \begin{eqnarray*}
    \frac{1}{x} \sum_{n \leqslant x} e^{z\mathfrak{g}(n) + z\mathfrak{H}(n)} &
    = & \frac{L (f ; z)}{\Gamma ( \hat{\Psi} (f ; z))} \prod_p \frac{1 +
    (e^{z\mathfrak{H}(p)} - 1) p^{- 1}}{1 + (e^{z\mathfrak{h}(p)} - 1) p^{-
    1}} \cdot \left( \log x \right)^{\hat{\Psi} (f ; z) - 1} + O \left(
    \mathcal{E}(x ; \kappa) \right)\\
    & = & \frac{L (\mathfrak{g}; z)}{\Gamma ( \hat{\Psi} (f ; z))} \prod_p
    \left( 1 - \frac{1}{p} + \frac{e^{z\mathfrak{H}(p)}}{p} \right) \cdot
    \left( \log x \right)^{\hat{\Psi} (f ; z) - 1} + O \left( \mathcal{E}(x ;
    \kappa) \right)
  \end{eqnarray*}
  where $\mathcal{E}(x ; \kappa) \assign \left( \log x \right)^{\hat{\Psi} (f
  ; \kappa) - 3 / 2}$. The second line follows from the definition of $L(f;z)$
  and the fact that $\hat{\Psi}(f;z) = \hat{\Psi}(\mathfrak{g};z)$ when
  $\Psi(f;t)$ is lattice distributed on $\mathbb{Z}$.
  \ The above formula holds uniformly in strongly additive functions
  $\mathfrak{H}$ such that $0 \leqslant \mathfrak{H}(p) \leqslant \left\lceil
  \mathfrak{h}(p) \right\rceil$. Choosing $\mathfrak{H}= 0$ yields the second
  claim of the lemma. Choosing $\mathfrak{H}$ arbitrary (with the restriction
  $\mathfrak{H}(p) = 0$ for $p > x$) and comparing the resulting asymptotic
  with an asymptotic for $\sum_{n \leqslant x} e^{z\mathfrak{g}(n)}$ we obtain
  the first claim of the lemma. 
\end{proof}

\begin{proof}[Proof of Part (4) of Theorem 2.8]
  Let $\Omega_x \assign [1, x] \cap \mathbbm{N}$ and $\mathcal{F}_x
  =\mathcal{P}(\Omega_x)$, where $\mathcal{P}(\Omega_x)$ is the power-set of
  $\Omega_x$. Then $(\Omega_x, \mathcal{F}_x)$ equipped with the measure
  $\mathbbm{P}_{\mathcal{F}_x} (A) = (1 / \lfloor x \rfloor) \cdot \tmop{Card}
  (A)$ forms a probability space. Define the random variables
  \[ Z_p (n) \assign \left\{ \begin{array}{l}
       1 \text{ if } p|n\\
       0 \text{ otherwise}
     \end{array} \right. \]
  so that
  \begin{equation}
    \mathbbm{P}_{\mathcal{F}_x} \left( \sum_{p \leqslant x} f (p) Z_p (n)
    \geqslant t \right) = \frac{1}{\lfloor x \rfloor} \cdot \# \left\{ n
    \leqslant x : f (n) \geqslant t^{^{^{}}} \right\}
  \end{equation}
  By lemma 5.2, for any given $C > 0$, we have uniformly in $0 \leqslant
  \kappa \assign \tmop{Re} s \leqslant C$, $|\tmop{Im} s| \leqslant \log\log x$
  and uniformly in strongly additive $\mathfrak{H}(\cdot)$ such that $0
  \leqslant \mathfrak{H}(p) \leqslant \left\lceil \mathfrak{h}(p)
  \right\rceil$,
  \begin{eqnarray}
    &  & \mathbbm{E}_{\mathcal{F}_x} \left[ e^{s \Omega (\mathfrak{g}; x) + s
    \Omega (\mathfrak{H}; x)} \right] \text{ } = \text{ } \frac{1}{\lfloor x
    \rfloor} \sum_{n \leqslant x} e^{s\mathfrak{g}(n) + s\mathfrak{H}(n)}
    \nonumber\\
    & = & \left( \frac{1}{\lfloor x \rfloor} \sum_{n \leqslant x}
    e^{s\mathfrak{g}(n)} \right) \prod_{p \leqslant x} \left( 1 - \frac{1}{p}
    + \frac{e^{s\mathfrak{H}(p)}}{p} \right) + O \left( \mathcal{E}(x ;
    \kappa)^{^{}} \right) \nonumber\\
    & = & \mathbbm{E}_{\mathcal{F}_x} \left[ e^{s \Omega (\mathfrak{g}; x)}
    \right] \cdot \prod_{p \leqslant x} \left( 1 - \frac{1}{p} +
    \frac{e^{s\mathfrak{H}(p)}}{p} \right) + O \left( \mathcal{E}(x ;
    \kappa)^{^{}} \right) 
  \end{eqnarray}
  with $\mathcal{E}(x ; \kappa) \assign (\log x)^{\hat{\Psi} (f ; \kappa) - 3
  / 2}$. By the same lemma
  \begin{equation}
    \mathbbm{E}_{\mathcal{F}_x} \left[ e^{s \Omega (\mathfrak{g}; x)} \right]
    = \frac{L (\mathfrak{g}; s)}{\Gamma ( \hat{\Psi} (f ; s))} \cdot \left(
    \log x \right)^{\hat{\Psi} (f ; s) - 1} + O \left( \mathcal{E}(x ; \kappa)
    \right)
  \end{equation}
  uniformly in $0 \leqslant \kappa \assign \tmop{Re} s \leqslant C$ and $|
  \tmop{Im} s| \leqslant 2 \pi$. The function $G (s) \assign L (\mathfrak{g};
  s) / \Gamma ( \hat{\Psi} (f ; s))$ is entire by lemma 4.20 and lemma 4.2. In
  addition $G (x) \neq 0$ for $x \geqslant 0$ -- on the one hand it is clear
  that $L(\mathfrak{g};x) \neq 0$ for $x \geqslant 0$, 
  just by looking at its product
  representation; on the other hand $\hat{\Psi} (f ; x) \geqslant \hat{\Psi}
  (f ; 0) = 1$ for $x \geqslant 0$, hence $1 / \Gamma ( \hat{\Psi} (f ; x))
  \neq 0$ for $x \geqslant 0$, because $1 / \Gamma (z)$ vanishes only in the
  $\tmop{Re} z \leqslant 0$ half-plane. Thus by $(5.8)$, $(5.9)$ and the two
  properties of $G (s)$ we just mentioned, the assumptions of proposition 4.17
  are satisfied. Applying proposition 4.17 we obtain the desired asymptotic
  for $(5.7)$ when $t \assign \xi_f (x ; \Delta) = \mu (f ; x) + \Delta \sigma
  (f ; x)$. \ 
\end{proof}

\subsection{$\mathcal{D}_f (x ; \Delta)$ when $1 \leqslant \Delta \leqslant o
(\sigma (f ; x))$}

The desired asymptotic for $\mathcal{D}_f (x ; \Delta)$ (the one indicated in
part 2 of theorem 2.8) follows in the range $(\tmop{loglog} x)^{\varepsilon}
\ll \Delta \leqslant o (\sigma (f ; x))$ from part 3 and part 4 of theorem
2.8. There is some care needed in adapting those asymptotics to the desired
form. Also the case when $\Psi (f ; t)$ is lattice distributed on $\alpha
\mathbbm{Z}$ ($\alpha \neq 1$) requires a little bit of additional work. The
two lemmata below are a preparation to handle this case.

\begin{lemma}
  Let $f \in \mathcal{C}$. For real $\alpha > 0$ define $v_{\alpha} \assign
  v_{f / \alpha} (x ; \Delta)$ and $v \assign v_f(x;\Delta)$. 
  We have $\hat{\Psi}^{(k)}
  (f / \alpha ; z) = (1 / \alpha)^k \cdot \hat{\Psi}^{(k)} (f ; z / \alpha)$ 
  and $v_{\alpha} / \alpha = v$. In particular $\hat{\Psi} (f / \alpha ;
  v_{\alpha}) = \hat{\Psi} (f ; v)$.
\end{lemma}

\begin{proof}
  Note that $\Psi (f / \alpha ; t) = \Psi (f ; \alpha t)$. Therefore
  $\hat{\Psi} (f / \alpha ; z) = \hat{\Psi} (f ; z / \alpha)$. Differentiating
  we obtain $\hat{\Psi}^{(k)} (f / \alpha ; z) = (1 / \alpha)^k \cdot
  \hat{\Psi}^{(k)} (f ; z / \alpha)$. 
  It remains to prove that $v_{\alpha} / \alpha = v$. By definition
  \[ \hat{\Psi}' (f / \alpha ; v_{\alpha}) = \hat{\Psi}' (f / \alpha ; 0) +
     \frac{\Delta}{\sigma_{\Psi} (f / \alpha ; x)} \cdot \hat{\Psi}'' (f /
     \alpha ; 0) \]
  Note that $\sigma_{\Psi} (f / \alpha ; x) = (1 / \alpha) \sigma_{\Psi} (f ;
  x)$. Thus the above formula transforms into
  \[ (1 / \alpha) \hat{\Psi}' (f ; v_{\alpha} / \alpha) = (1 / \alpha)
     \hat{\Psi}' (f ; 0) + (1 / \alpha) \cdot \frac{\Delta}{\sigma_{\Psi} (f ;
     x)} \cdot \hat{\Psi}'' (f ; 0) \]
  By definition of $v$, the right hand side equals to $(1 / \alpha)
  \hat{\Psi}' (f ; v)$. Thus we obtain $\Psi' (f ; v_{\alpha} / \alpha) =
  \hat{\Psi}' (f ; v)$. The function $\hat{\Psi}' (f ; x)$ is strictly
  increasing for $x > 0$. It follows that $v_{\alpha} / \alpha = v$ as
  desired.
\end{proof}

\begin{lemma}
  Let $\alpha > 0$ be given and $f \in \mathcal{C}$. For all $x, \Delta
  \geqslant 1$,
  \[ S_f (x ; \Delta) = S_{f / \alpha} (x ; \Delta) \]
\end{lemma}

\begin{proof}
  Let $v_{\alpha} \assign v_{f / \alpha} (x ; \Delta)$ and $v \assign v_1$. By
  the previous lemma
  \[ \hat{\Psi} (f / \alpha ; v_{\alpha}) - v_{\alpha} \cdot \hat{\Psi}' (f /
     \alpha ; v_{\alpha}) = \hat{\Psi} (f ; v) - v \hat{\Psi}' (f ; v) \]
  Furthermore $v_{\alpha} \cdot \hat{\Psi}'' (f / \alpha ; v_{\alpha})^{1 / 2}
  = v \cdot \hat{\Psi}'' (f ; v)^{1 / 2}$. Therefore
  \[ \frac{(\log x)^{\hat{\Psi} (f / \alpha ; v_{\alpha}) - v_{\alpha} 
     \hat{\Psi}' (f / \alpha ; v_{\alpha}) - 1}}{v_{\alpha} \cdot (2 \pi
     \hat{\Psi}'' (f / \alpha ; v_{\alpha}) \tmop{loglog} x)^{1 / 2}} =
     \frac{(\log x)^{\hat{\Psi} (f ; v) - v \hat{\Psi}' (f ; v) - 1}}{v \cdot
     (2 \pi \hat{\Psi}'' (f ; v) \tmop{loglog} x)^{1 / 2}} \]
  The right hand side equals to $S_f (x ; \Delta)$, while the left hand side
  to $S_{f / \alpha} (x ; \Delta)$. It follows that $S_f (x ; \Delta) = S_{f /
  \alpha} (x ; \Delta)$ as desired.
\end{proof}

\begin{proof}[Proof of Part (1) and Part (2) of Theorem 2.8]
  When $\Delta \leqslant o (\sigma (f ; x))$ then by lemma 4.7 $v = v_f (x ;
  \Delta) \asymp \Delta / \sigma_{\Psi} (f ; x) = o (1)$. We claim that
  \begin{eqnarray}
    L (f ; v) e^{- vc (f)} / \Gamma ( \hat{\Psi} (f ; v)) & = & 1 + o (1) \\
    L (\mathfrak{g}; v) e^{- vc (f)} / \Gamma ( \hat{\Psi} (f ; v)) & = & 1 +
    o (1) \\
    \mathcal{P}_{\mathfrak{h}} (a ; v) & = & 1 + o (1) \text{ uniformly in $0
    \leqslant a \leqslant 1$} 
  \end{eqnarray}
  Let $G (z) \assign L (f ; z) e^{- zc (f)} / \Gamma ( \hat{\Psi} (f ; z))$.
  The function $G (z)$ is entire by lemma 4.4 and lemma 4.2. Therefore $G (v)
  = G (0) + O (v) = 1 + o (1)$. The same proof goes for $(5.11)$. Recall that
  \[ \mathcal{P}_{\mathfrak{h}} (a ; v) = \frac{v}{e^v - 1} + v \sum_{k
     \geqslant 0} e^{v (k + a)}_{} \cdot \mathbbm{P} \left( X (\mathfrak{h})
     \geqslant k + a \right) \]
  We have $v / (e^v - 1) = 1 + O (v)$. Also the sum on the right is $O (1)$
  throughout $0 \leqslant v \leqslant 1 / 2$ (because by lemma 4.19,
  $\mathbbm{E}[e^{X (\mathfrak{h})}] < \infty$ hence $\mathbbm{P}(X
  (\mathfrak{h}) \geqslant k) \leqslant e^{- k} \cdot \mathbbm{E}[e^{X
  (\mathfrak{h})}]$). Thus $\mathcal{P}_{\mathfrak{h}} (a ; v) = 1 + O (v)$
  uniformly throughout $0 \leqslant v \leqslant 1 / 2$. Now, if $\Psi (f ; t)$
  is not lattice distributed and $(\tmop{loglog} x)^{\varepsilon} \ll \Delta
  \leqslant o (\sigma (f ; x))$ then by $(5.10)$ and part 3 of theorem 2.8
  \[ \mathcal{D}_f (x ; \Delta) \sim \frac{L (f ; v) \cdot e^{- vc
     (f)}}{\Gamma ( \hat{\Psi} (f ; v))} \cdot S_f (x ; \Delta) \sim S_f (x ;
     \Delta) \]
  If $\Psi (f ; t)$ is lattice distributed on $\mathbbm{Z}$ and
  $(\tmop{loglog} x)^{\varepsilon} \ll \Delta \leqslant o (\sigma (f ; x))$
  then by $(5.11)$, $(5.12)$ and part 4 of theorem 2.8,
  \[ \mathcal{D}_f (x ; \Delta) \sim \frac{L (\mathfrak{g}; v) e^{- vc
     (f)}}{\Gamma ( \hat{\Psi} (f ; v))} \cdot \mathcal{P}_{\mathfrak{h}}
     (\xi_f (x ; \Delta) ; v) \cdot S_f (x ; \Delta) \sim S_f (x ; \Delta) \]
  Now consider the case when $\Psi (f ; t)$ is lattice distributed on $\alpha
  \mathbbm{Z}$ ($\alpha \neq 1$) and $\Delta$ is in the range $(\tmop{loglog}
  x)^{\varepsilon} \ll \Delta \leqslant o (\sigma (f ; x))$. Let $v_{\alpha}
  \assign v_{f / \alpha} (x ; \Delta)$. We reduce this case to the previous
  one. Note that $\mathcal{D}_f (x ; \Delta) =\mathcal{D}_{f / \alpha} (x ;
  \Delta)$ and that $\Psi (f / \alpha ; t)$ is lattice distributed on
  $\mathbbm{Z}$. Therefore, using part 4 of theorem 2.8,
  \[ \mathcal{D}_f (x ; \Delta) =\mathcal{D}_{f / \alpha} (x ; \Delta) \sim
     \frac{L (\mathfrak{g}_{f / \alpha} ; v_{\alpha}) e^{- v_{\alpha} c (f /
     \alpha)}}{\Gamma ( \hat{\Psi} (f / \alpha ; v_{\alpha}))} \cdot
     \mathcal{P}_{\mathfrak{h}_{f / \alpha}} (\xi_{f / \alpha} (x ; \Delta) ;
     v_{\alpha}) \cdot S_{f / \alpha} (x ; \Delta) \]
  By lemma 5.3, $v_{\alpha} \assign v_{f / \alpha} (x ; \Delta) = \alpha v = o
  (1)$. Thus the terms on the left to $S_{f / \alpha} (x ; \Delta)$ are $1 + o
  (1)$. It follows that the right hand side in the above equation is
  asymptotic to $S_{f / \alpha} (x ; \Delta)$. But by lemma 5.4, $S_{f /
  \alpha} (x ; \Delta) = S_f (x ; \Delta)$. Hence $\mathcal{D}_f (x ; \Delta)
  \sim S_f (x ; \Delta)$ as desired. It remains to show that $\mathcal{D}_f (x
  ; \Delta) \sim (1 / \sqrt[]{2 \pi}) \int_{\Delta}^{\infty} e^{- u^2 / 2}
  \cdot \mathd u$ when $\Delta$ is in the range $\Delta \leqslant o (\sigma (f
  ; x)^{1 / 3}) = o ((\tmop{loglog} x)^{1 / 6})$. This is a consequence of
  proposition 4.9. Indeed, let the random variable $\Omega (f ; x)$ be defined
  by $\mathbbm{P}(\Omega (f ; x) \leqslant t) = (1 / \lfloor x \rfloor)\#\{n
  \leqslant x : f (n) \leqslant t\}$. Then, by proposition 4.1,
  \[ \mathbbm{E} \left[ e^{s \Omega (f ; x)} \right] = \frac{1}{\lfloor x
     \rfloor} \sum_{n \leqslant x} e^{sf (n)} = \frac{L (f ; s)}{\Gamma (
     \hat{\Psi} (f ; s))} \cdot \left( \log x \right)^{\hat{\Psi} (f ; s) - 1}
     + O \left( (\log x)^{\hat{\Psi} (f ; \kappa) - 3 / 2} \right) \]
  uniformly in $|s| \leqslant \varepsilon$ for any given $\varepsilon > 0$.
  Since $L (f ; s) / \Gamma ( \hat{\Psi} (f ; s))$ is entire (by lemma 4.4 and
  lemma 4.2) and non-zero at $s = 0$ proposition 4.9 is applicable. It follows
  that
  \begin{eqnarray*}
    \mathbbm{P} \left( \frac{\Omega (f ; x) - \mu (f ; x)}{\sigma (f ; x)}
    \geqslant \Delta \right) & \sim & \frac{1}{\sqrt[]{2 \pi}}
    \int_{\Delta}^{\infty} e^{- u^2 / 2} \cdot \mathd u
  \end{eqnarray*}
  uniformly in $1 \leqslant \Delta \leqslant o (\sigma (f ; x)^{1 / 3})$.
  Since the left hand is equal to $\mathcal{D}_f (x ; \Delta)$ we are done. 
\end{proof}

\section{The ``structure theorem''}

We break down the proof of Theorem 1.1 into three parts corresponding to the
range $1 \leqslant \Delta \leqslant o (\sigma^{\alpha})$, $1 \leqslant \Delta
\leqslant o (\sigma)$ and $1 \leqslant \Delta \ll \sigma$. Throughout (just as
in the statement of theorem 1.1) $\sigma \assign \sigma (x)$ stands for a
function such that $\sigma (f ; x) \sim \sigma (x) \sim \sigma (g ; x)$.

\subsection{The $1 \leqslant \Delta \leqslant o (\sigma (x)^{\alpha})$ range}

We now prove Part (1) and Part (2) of Theorem 1.1. 
The rough idea of the proof is this: 
We show that $\mathcal{D}_f (x ; \Delta) \sim \mathcal{D}_g (x ;
\Delta)$ holds in the range $1 \leqslant \Delta \leqslant o (\sigma^{\alpha})$
if and only if the first $\varrho (\alpha) \assign \left\lceil (1 + \alpha) /
(1 - \alpha) \right\rceil$ coefficients of some power series agree. Then we
relate the equality of those coefficients to the equality of moments $\int t^k
\mathd \Psi (f ; t) = \int t^k \mathd \Psi (g ; t)$ for $k = 3, 4, \ldots,
\varrho (\alpha)$. \\Let us also note at the outset that the function we will be dealing with,
namely $\omega(f;z)$ and $A(f;z)$ are respectively analytic in a 
neighborhood of $\mathbb{R}^{+} \cup \{0\}$ (lemma 4.7) and entire (lemma 4.2).
 
\begin{lemma}
  Let $f \in \mathcal{C}$. Given $\varepsilon > 0$, uniformly in
  $(\tmop{loglog} x)^{\varepsilon} \ll \Delta \leqslant o (\sigma (f ; x))$,
  \[ \mathcal{D}_f (x ; \Delta) \sim (1 / \sqrt[]{2 \pi} \Delta) \cdot \left(
     \log x \right)^{\mathcal{E}(f ; \Delta / \sigma_{\Psi} (f ; x))} \]
  where $\mathcal{E}(f ; z) \assign A (f ; \omega (f ; z))$. The functions $A
  (f ; z)$ and $\omega (f ; z)$ are defined in section 3. \ 
\end{lemma}

\begin{proof}
  Let $v \assign v_f (x ; \Delta)$. By lemma 4.7, $v \sim \Delta /
  \sigma_{\Psi} (f ; x)$ when $\Delta \leqslant o (\sigma (f ; x))$, and in
  particular $v = o (1)$. Thus $\hat{\Psi}'' (f ; v) = \hat{\Psi}'' (f ; 0) +
  o (1)$ and
  \begin{equation}
    v (2 \pi \hat{\Psi}'' (f ; v) \tmop{loglog} x)^{1 / 2} \sim (\Delta /
    \sigma_{\Psi} (f ; x)) \cdot (2 \pi \cdot \hat{\Psi}'' (f ; 0)
    \tmop{loglog} x)^{1 / 2} = \sqrt[]{2 \pi} \Delta
  \end{equation}
  the last equality comes from $\sigma_{\Psi} (f ; x)^2 = \hat{\Psi}'' (f ; 0)
  \tmop{loglog} x$. By definition of $v$ and $\omega (f ; \cdot)$ we have $v =
  \omega (f ; \Delta / \sigma_{\Psi} (f ; x))$, and so
  \begin{equation}
    \hat{\Psi} (f ; v) - 1 - v \hat{\Psi}' (f ; v) = A (f ; v) =\mathcal{E}(f
    ; \Delta / \sigma_{\Psi} (f ; x))
  \end{equation}
  By part 2 of theorem 2.8, uniformly in $(\tmop{loglog} x)^{\varepsilon} \ll
  \Delta \leqslant o (\sigma (f ; x))$,
  \begin{eqnarray*}
    \mathcal{D}_f (x ; \Delta) & \sim & \frac{\left( \log x
    \right)^{\hat{\Psi} (f ; v) - 1 - v \hat{\Psi}' (f ; v)}}{v (2 \pi
    \hat{\Psi}'' (f ; v) \tmop{loglog} x)^{1 / 2}} \text{ , } v \assign v_f (x
    ; \Delta)
  \end{eqnarray*}
  By $(6.1)$, $(6.2)$ the right hand side is asymptotic to $( \sqrt[]{2 \pi}
  \Delta)^{- 1} (\log x)^{\mathcal{E}(f ; \Delta / \sigma_{\Psi} (f ; x))}$
\end{proof}

We now relate the asymptotic behaviour of $\mathcal{D}_f (x ; \Delta)$ to 
the coefficients of $\mathcal{E}(f ; z) = \sum_{k \geqslant 0} a_k z^k$.

\begin{lemma}
  Let $f, g \in \mathcal{C}$. Let $\varepsilon > 0$ be given. Suppose that
  $\sigma_{\Psi} (f ; x) = \sigma_{\Psi} (g ; x)$ and denote by $\sigma_{\Psi}
  = \sigma_{\Psi} (x)$ a function such that $\sigma_{\Psi} (f ; x) =
  \sigma_{\Psi} (x) = \sigma_{\Psi} (g ; x)$. The asymptotic relation
  \begin{eqnarray*}
    \mathcal{D}_f (x ; \Delta) & \sim & \mathcal{D}_g (x ; \Delta)
  \end{eqnarray*}
  holds uniformly in the range $(\tmop{loglog} x)^{\varepsilon} \ll \Delta
  \leqslant o (\sigma^{\alpha}_{\Psi})$ if and only if the first $\varrho
  (\alpha) \assign \left\lceil (1 + \alpha) / (1 - \alpha) \right\rceil$
  coefficients of $\mathcal{E}(f ; z) \assign A (f ; \omega (f ; z))$ and
  $\mathcal{E}(g ; z) \assign A (g ; \omega (g ; z))$ agree. 
\end{lemma}

\begin{proof}
  By lemma $6.1$ $\mathcal{D}_f (x ; \Delta) \sim \mathcal{D}_g (x ; \Delta)$
  holds uniformly in $(\tmop{loglog} x)^{\varepsilon} \ll \Delta \leqslant o
  (\sigma_{\Psi}^{\alpha})$ if and only if
  \begin{equation}
    \tmop{loglog} x \cdot \left( \mathcal{E}(f ; \Delta / \sigma_{\Psi})
    -\mathcal{E}(g ; \Delta / \sigma_{\Psi})) = o (1) \right.
  \end{equation}
  throughout $(\tmop{loglog} x)^{\varepsilon} \ll \Delta \leqslant o
  (\sigma_{\Psi}^{\alpha})$. Let $e (z) \assign \mathcal{E}(f ; z)
  -\mathcal{E}(g ; z)$ and denote by $a_n$ the $n$-th coefficient in the
  Taylor expansion of $e (z)$ about $z = 0$.
  
  Suppose to the contrary that $(6.3)$ holds in $(\tmop{loglog}
  x)^{\varepsilon} \leqslant \Delta \leqslant o (\sigma_{\Psi}^{\alpha})$ but
  $a_m \neq 0$ for some integer $m \leqslant \varrho (\alpha)$. Let $m$ be the
  first such integer. Then
  \begin{equation}
    e (\Delta / \sigma_{\Psi}) = a_m \cdot \left( \Delta / \sigma_{\Psi}
    \right)^m \cdot \left( 1 + O \left( \Delta / \sigma_{\Psi} \right) \right)
  \end{equation}
  In $(6.4)$ choose $\Delta = \sigma_{\Psi}^{1 - 2 / m}$. This choice of
  $\Delta$ is allowed (i.e we have $\Delta = o (\sigma^{\alpha}_{\Psi})$)
  because $\varrho (1 - 2 / m) = m - 1 < \varrho (\alpha)$, hence $1 - 2 / m <
  \alpha$ and thus $\Delta = \sigma_{\Psi}^{1 - 2 / m} = o
  (\sigma_{\Psi}^{\alpha})$. With this choice of $\Delta$ by $(6.4)$, equation
  $(6.3)$ becomes \ $a_m \cdot (\tmop{loglog} x / \sigma_{\Psi}^2) = o (1)$.
  Hence $a_m = o (1)$ because $\sigma_{\Psi}^2 \asymp \tmop{loglog} x$.
  Letting $x \rightarrow \infty$ we obtain $a_m = 0$, a contradiction with our
  initial assumption $a_m \neq 0$.
  
  Conversely, suppose that the first $\ell \assign \varrho (\alpha)$
  coefficients of $\mathcal{E}(f ; z)$ and $\mathcal{E}(g ; z)$ are equal.
  Thus
  \begin{equation}
    e (\Delta / \sigma_{\Psi}) =\mathcal{E}(f ; \Delta / \sigma_{\Psi})
    -\mathcal{E}(g ; \Delta / \sigma_{\Psi}) = O ((\Delta /
    \sigma_{\Psi})^{\ell + 1})
  \end{equation}
  uniformly in $1 \leqslant \Delta \leqslant o (\sigma_{\Psi})$. Using $(6.5)$
  and $\sigma_{\Psi}^2 \asymp \tmop{loglog} x$ we obtain for $\Delta \leqslant
  o (\sigma_{\Psi}^{\alpha})$,
  \begin{eqnarray*}
    \tmop{loglog} x \cdot \left( \mathcal{E}(f ; \Delta / \sigma_{\Psi})
    -\mathcal{E}(g ; \Delta / \sigma_{\Psi})) \right. & \ll & \tmop{loglog} x
    \cdot (\Delta / \sigma_{\Psi})^{\ell + 1}\\
    & \leqslant & \sigma^2_{\Psi} \cdot o (\sigma_{\Psi}^{(\alpha - 1) (\ell
    + 1)}) = o (\sigma_{\Psi}^{2 + (\alpha - 1) (\ell + 1)})
  \end{eqnarray*}
  The right hand side is in fact $o (1)$ because $2 + (\alpha - 1) (\ell + 1)
  \leqslant 0$. By the remark right above equation $(6.3)$ this shows 
  that $\mathcal{D}_f(x ; \Delta) \sim \mathcal{D}_g (x ; \Delta)$ 
  in $(\tmop{loglog}
  x)^{\varepsilon} \ll \Delta \leqslant o (\sigma_{\Psi}^{\alpha})$. A quick
  way to check $2 + (\alpha - 1) (\ell + 1) \leqslant 0$ is the following.
  Note that
  \[ \varrho (\alpha) \assign \left\lceil \frac{1 + \alpha}{1 - \alpha}
     \right\rceil \geqslant \frac{1 + \alpha}{1 - \alpha} = \frac{2}{1 -
     \alpha} + \frac{\alpha - 1}{1 - \alpha} = \frac{2}{1 - \alpha} - 1 \]
  Upon rewriting the above we find $2 + (\alpha - 1) (\varrho (\alpha) + 1)
  \leqslant 0$ as desired. 
\end{proof}

The next lemma is crucial.

\begin{lemma}
  Let $f, g \in \mathcal{C}$. Suppose that $\hat{\Psi}'' (f ; 0) =
  \hat{\Psi}'' (g ; 0)$. Let $\alpha \in (1 / 3, 1)$ be given. The first
  $\varrho (\alpha)$ coefficients of $A (f ; \omega (f ; z))$ and $A (g ;
  \omega (g ; z))$ are equal if and only if the $k \um \tmop{th}$ moments $(3
  \leqslant k \leqslant \varrho (\alpha))$ of $\Psi (f ; t)$ and $\Psi (g ;
  t)$ are equal, that is
  \[ \int_{\mathbbm{R}} t^k \mathd \Psi (f ; t) = \int_{\mathbbm{R}} t^k
     \mathd \Psi (g ; t) \text{ for $3 \leqslant k \leqslant \varrho
     (\alpha)$} \]
\end{lemma}

\begin{proof}
  Since $\alpha > 1 / 3$ we have $\varrho (\alpha) \geqslant 3$. We work
  formally with power series and write $O (z^{\ell})$ to indicate terms of
  order $\geqslant \ell$. Denote by $a_k$ and $b_k$ the coefficients in the
  expansion around $0$ of the power series $A (f ; \omega (f ; z))$ and $A (g
  ; \omega (g ; z))$, respectively. Suppose that $a_k = b_k$ for $k \leqslant
  \ell \assign \varrho (\alpha)$. Then
  \begin{equation}
    A (f ; \omega (f ; z)) = A (g ; \omega (g ; z)) + O (z^{\ell + 1})
  \end{equation}
  Differentiating on both sides we obtain $- \hat{\Psi}'' (f ; 0) \omega (f ;
  z) = - \hat{\Psi}'' (g ; 0) \omega (g ; z) + O (z^{\ell})$. Dividing by
  $\hat{\Psi}'' (f ; 0) = \hat{\Psi}'' (g ; 0)$ on both sides, we get
  \[ \omega (f ; z) = \omega (g ; z) + O (z^{\ell}) \]
  Expanding $A (g ; \omega (g ; z))$ into a Taylor series about $\omega (f ;
  z)$, we find that
  \begin{eqnarray*}
    A (g ; \omega (g ; z)) & = & A (g ; \omega (f ; z) + (\omega (g ; z) -
    \omega (f ; z)))\\
    & = & A (g ; \omega (f ; z)) + \sum_{k \geqslant 1} \frac{1}{k!} \cdot
    \left( \omega (g ; z) - \omega (f ; z) \right)^k \cdot A^{(k)} (g ; \omega
    (f ; z))
  \end{eqnarray*}
  Since $\omega (g ; z) - \omega (f ; z) = O (z^{\ell})$ the term $k \geqslant
  2$ contribute $O (z^{2 \ell})$. The term $k = 1$ equals to $- \omega (f ; z)
  \hat{\Psi}'' (g ; \omega (f ; z)) \cdot (\omega (g ; z) - \omega (f ; z))$
  and thus contributes $O (z^{\ell + 1})$ because $\omega (f ; z) = O (z)$. We
  conclude that
  \begin{equation}
    A (g ; \omega (g ; z)) = A (g ; \omega (f ; z)) + O (z^{\ell + 1})
  \end{equation}
  Inserting $(6.7)$ into $(6.6)$ we obtain
  \[ A (f ; \omega (f ; z)) = A (g ; \omega (f ; z)) + O \left( z^{\ell + 1}
     \right) \]
  In this relation we substitute $z \longmapsto \omega^{- 1} (f ; z)$. Since
  $\omega^{- 1} (f ; z)$ is zero at $z = 0$ we have $\omega^{- 1} (f ; z) = O
  (z)$. Therefore, after substitution $A (f ; z) = A (g ; z) + O (z^{\ell +
  1})$. Differentiating on both sides we obtain $z \hat{\Psi}'' (f ; z) = z
  \hat{\Psi}'' (g ; z) + O (z^{\ell})$. Upon division by $z$ we get
  $\hat{\Psi}'' (f ; z) = \hat{\Psi}'' (g ; z) + O \left( z^{\ell - 1}
  \right)$. Since
  \[ \hat{\Psi} (f ; z) = \sum_{k \geqslant 0} \int_{\mathbbm{R}} t^k 
  \mathd \Psi
     (f ; t) \cdot \frac{z^k}{k!} \]
  and $\hat{\Psi}'' (f ; z) = \hat{\Psi}'' (g ; z) + O (z^{\ell - 1})$ with
  $\ell = \varrho (\alpha)$ we conclude that
  \begin{equation}
    \int_{\mathbbm{R}} t^k \mathd \Psi (f ; t) = 
    \int_{\mathbbm{R}} t^k \mathd \Psi (g ; t) \text{ for } k = 2, 3, 
    \ldots, \varrho (\alpha)
  \end{equation}
  Conversely, let us suppose that $\int_{\mathbb{R}} t^k \mathd \Psi(f;t)
  = \int_{\mathbb{R}} t^k \mathd \Psi(g;t)$ holds for all $k = 3, \dots, 
  \varrho(\alpha)$. Since in addition (by assumptions) $\hat{\Psi}''(f;0)
  = \hat{\Psi}''(g;0)$ we obtain
  \[ \hat{\Psi}'' (f ; z) = \hat{\Psi}'' (g ; z) + O \left( z^{\ell - 1}
     \right) \]
  with $\ell \assign \varrho (\alpha)$. 
  Multiplying both sides by $z$ and integrating gives $A (f ; z) = A
  (g ; z) + O (z^{\ell + 1})$. Since $\omega (f ; z) = O (z)$, upon
  substituting $z \longmapsto \omega (f ; z)$ in the last relation, we obtain
  \begin{equation}
    A (f ; \omega (f ; z)) = A (g ; \omega (f ; z)) + O \left( z^{\ell + 1}
    \right)
  \end{equation}
  With this in mind, we evaluate the difference $\omega (f ; z) - \omega (g ;
  z)$. Given any $h \in \mathcal{C}$, by definition $\omega (h ; z)$ equals to
  \begin{equation}
    ( \hat{\Psi}')^{- 1} ( \hat{\Psi}' (h ; 0) + z \cdot \hat{\Psi}'' (h ; 0))
    = \sum_{k \geqslant 0} \frac{z^k \cdot \hat{\Psi}'' (h ; 0)^k}{k!} \left[
    ( \hat{\Psi}')^{- 1} \right]^{(k)} \left( \hat{\Psi}' (h ; 0) \right)
  \end{equation}
  where $( \hat{\Psi}')^{- 1}$ denote the inverse function (under composition)
  to $\hat{\Psi}' (h ; z)$ and $f^{(k)}$ stands for the $k$-th derivative of
  $f$. The term $k = 0$ contributes $0$. The term $k = 1$ contributes $z$,
  since
  \[ [( \hat{\Psi}')^{- 1}]^{(1)} (z) = \frac{1}{\hat{\Psi}'' (h ; (
     \hat{\Psi}')^{- 1} (h ; z))} \]
  so that at $z = \hat{\Psi}' (h ; 0)$ that simplifies to $1 / \hat{\Psi}'' (h
  ; 0)$. However, the important point here, is that the higher derivatives $[(
  \hat{\Psi}')^{- 1}]^{(k)} ( \hat{\Psi}' (h ; 0))$ will involve only the
  terms $\hat{\Psi}^{(k + 1)} (h ; 0), \ldots, \hat{\Psi}'' (h ; 0)$. By
  assumption we have $\hat{\Psi}^{(k)} (f ; 0) = \hat{\Psi}^{(k)} (g ; 0)$ for
  $2 \leqslant k \leqslant \ell \assign \varrho (\alpha)$ therefore the power
  series $(6.10)$ taken respectively for $h = f$ and $h = g$ will agree up to
  the $(\ell - 1)$-th term. This gives
  \begin{equation}
    \omega (f ; z) = \omega (g ; z) + O \left( z^{\ell} \right)
  \end{equation}
  Expanding $A (g ; \omega (f ; z))$ into a Taylor series about $\omega (g ;
  z)$, we find that
  \begin{eqnarray*}
    &  & A (g ; \omega (f ; z))\\
    & = & A (g ; \omega (g ; z) + (\omega (f ; z) - \omega (g ; z))\\
    & = & A (g ; \omega (g ; z)) + A' (g ; \omega (g ; z)) \cdot (\omega (f ;
    z) - \omega (g ; z)) + O \left( (\omega (f ; z) - \omega (g ; z))^2
    \right)
  \end{eqnarray*}
  By $(6.11)$ the third term is bounded by $O (z^{2 \ell})$, while the
  second term is bounded by $O (z^{\ell + 1})$ because $A' (g ; \omega (g ;
  z)) = O (z)$ since $A' (g ; \omega (g ; 0)) = A' (g ; 0) = 0$. It follows
  that $A (g ; \omega (f ; z)) = A (g ; \omega (g ; z)) + O (z^{\ell + 1})$.
  On combining this with $(6.9)$ we conclude that $A (f ; \omega (f ; z)) = A
  (g ; \omega (g ; z)) + O (z^{\ell + 1})$ as desired. 
\end{proof}

\begin{proof}[Proof of Part (1) and Part (2) of Theorem 1.1]
  By part (1) of theorem 2.8 $\mathcal{D}_f (x ; \Delta) \sim (1 / \sqrt[]{2
  \pi}) \int_{\Delta}^{\infty} e^{- u^2 / 2} \mathd u$ for $\Delta$ in the
  range $1 \leqslant \Delta \leqslant o (\sigma (f ; x)^{1 / 3})$. Therefore
  we will always have $\mathcal{D}_f (x ;
  \Delta) \sim \mathcal{D}_g (x ; \Delta)$ uniformly in $1 \leqslant \Delta
  \leqslant o (\sigma^{1 / 3})$. This proves part (1) of theorem 1.1.
  
  By assumptions $\sigma (f ; x) \sim \sigma (g ; x)$. Note that
  \[ \sigma^2 (f ; x) = \hat{\Psi}'' (f ; 0) \cdot \tmop{loglog} x + O (1) \]
  Therefore $\sigma (f ; x) \sim \sigma (g ; x)$ gives $\hat{\Psi}'' (f ; 0) =
  \hat{\Psi}'' (g ; 0)$ and also $\sigma_{\Psi} (f ; x) = \sigma_{\Psi} (g ;
  x)$ because $\sigma_{\Psi} (f ; x)^2 = \hat{\Psi}'' (f ; 0) \tmop{loglog}
  x$. Thus the assumptions of lemma 6.2 and lemma 6.3 are satisfied. Since
  $\mathcal{D}_f (x ; \Delta) \sim (1 / \sqrt[]{2 \pi}) \int_{\Delta}^{\infty}
  e^{- u^2 / 2} \cdot \mathd u$ when $1 \leqslant \Delta \leqslant o
  (\sigma^{1 / 3})$, the relation $\mathcal{D}_f (x ; \Delta) \sim
  \mathcal{D}_g (x ; \Delta)$ holds in the range $1 \leqslant \Delta \leqslant
  o (\sigma^{\alpha})$ if and only if it holds in the range $(\tmop{loglog}
  x)^{\varepsilon} \ll \Delta \leqslant o (\sigma^{\alpha})$ $(0 \leqslant
  \varepsilon < 1 / 6)$. By lemma 6.2, $\mathcal{D}_f (x ; \Delta) \sim
  \mathcal{D}_g (x ; \Delta)$ holds in that range if and only if the first
  $\varrho (\alpha)$ coefficients of the power-series 
  $A (f ; \omega (f ; z))$ and $A (g ;\omega (g ; z))$ coincide. 
  By lemma 6.3 they do coincide if and only if
  \[ \int_{\mathbbm{R}} t^k \mathd \Psi (f ; t) = \int_{\mathbbm{R}} t^k
     \mathd \Psi (g ; t) \]
  for all $k = 3, 4, \ldots, \varrho (\alpha)$. This chain of if and only if's
  proves Part (2) of Theorem 1.1.
\end{proof}

\subsection{The $1 \leqslant \Delta \leqslant o (\sigma$) range}

\begin{proof}[Proof of Part (3) of Theorem 1.1]
  One direction is clear: By Theorem 2.8, when $1 \leqslant \Delta \leqslant o
  (\sigma (f ; x))$ the asymptotic for $\mathcal{D}_f (x ; \Delta)$ depends
  only on $\Psi (f ; t)$. Hence if $\Psi (f ; t) = \Psi (g ; t)$ then
  $\mathcal{D}_f (x ; \Delta) \sim \mathcal{D}_g (x ; \Delta)$ throughout $1
  \leqslant \Delta \leqslant o (\sigma)$.
  
  Now we focus on the converse direction. If $\mathcal{D}_f (x ; \Delta) \sim
  \mathcal{D}_g (x ; \Delta)$ holds throughout $1 \leqslant \Delta \leqslant o
  (\sigma)$ then it also holds in the smaller range $1 \leqslant \Delta
  \leqslant o (\sigma^{\alpha})$ for any $0 < \alpha < 1$. Hence by part 2 of
  Theorem 1.1,
  \begin{eqnarray}
    \int_{\mathbbm{R}} t^k \mathd \Psi (f ; t) & = & \int_{\mathbbm{R}} t^k
    \mathd \Psi (g ; t) 
  \end{eqnarray}
  for all $k = 3, 4, \ldots, \varrho (\alpha) = \left\lceil (1 + \alpha) / (1
  - \alpha) \right\rceil$. Letting $\alpha \rightarrow 1$ it follows that
  $(6.12)$ holds for all $k \geqslant 3$. Recall that
  \begin{eqnarray*}
    \hat{\Psi} (f ; z) & = & 1 + \sum_{k \geqslant 1} \int_{\mathbbm{R}} t^k
    \mathd \Psi (f ; t) \cdot \frac{z^k}{k!}
  \end{eqnarray*}
  Therefore $\hat{\Psi} (f ; z) - \hat{\Psi} (g ; z) = az^2 + bz$ for some $a,
  b \in \mathbbm{R}$. In particular
  \[ a^2 t^4 + b^2 t^2 = | \hat{\Psi} (f ; \mathi t) - \hat{\Psi} (g ; \mathi
     t) |^2 \]
  The right hand side is bounded by $4$ because $| \hat{\Psi} (f ; \mathi t) |
  \leqslant 1$ and $| \hat{\Psi} (g ; \mathi t) | \leqslant 1$. Letting $t
  \rightarrow \infty$ in the above equation it follows that $a = 0 = b$. Hence
  $\hat{\Psi} (f ; \mathi t) = \hat{\Psi} (g ; \mathi t)$. By Fourier
  inversion (or using probabilistic terminology, by ``uniqueness of
  characteristic functions'') $\Psi (f ; t) = \Psi (g ; t)$. \ 
\end{proof}

\subsection{The $1 \leqslant \Delta \leqslant c \sigma$ range}

We prove part 4 of theorem 1.1. We break down the proof into two cases,
depending on whether $\Psi (f ; t)$ is or is not lattice distributed.

\subsubsection{$\Psi (f ; t)$ is not lattice distributed}

\begin{lemma}
  Let $f, g \in \mathcal{C}$. Suppose that $\sigma (f ; x) \sim \sigma (g ;
  x)$. As usual denote by $\sigma = \sigma (x)$ a function such that $\sigma
  (f ; x) \sim \sigma (x) \sim \sigma (g ; x)$. If there is a $\delta > 0$
  such that $\mathcal{D}_f (x ; \Delta) \sim \mathcal{D}_g (x ; \Delta)$
  uniformly in $1 \leqslant \Delta \leqslant \delta \sigma$, then $\mathcal{Z}
  \left( L (f ; z)) =\mathcal{Z}(L (g ; z)) \right.$ where $\mathcal{Z}(h)$
  denote the zero set of $h (\cdot)$ (the zeroes are counted without
  multiplicity). 
\end{lemma}

\begin{proof}
  By assumptions $\mathcal{D}_f (x ; \Delta) \sim \mathcal{D}_g (x ; \Delta)$
  uniformly in $1 \leqslant \Delta \leqslant \delta \sigma$. Hence by part 3 of
  theorem 1.1, $\Psi (f ; t) = \Psi (g ; t)$. Therefore $S_f (x ; \Delta) =
  S_g (x ; \Delta)$ for $x, \Delta \geqslant 0$ and $v_f (x ; \Delta) = v =
  v_g (x ; \Delta)$ since both depend only on $\Psi (f ; t)$ and $\Psi (g ;
  t)$. Thus by part 3 of theorem 2.8 the assumption $\mathcal{D}_f (x ;
  \Delta) \sim \mathcal{D}_g (x ; \Delta)$ simplifies to
  \begin{eqnarray}
    L (f ; v) \cdot e^{- vc (f)} & \sim & L (g ; v) \cdot e^{- vc (g)} \text{
    uniformly in $1 \leqslant \Delta \leqslant \delta \sigma$} 
  \end{eqnarray}
  Pick a $0 < \kappa \leqslant \delta / 2$ and fix $\Delta = \kappa
  \sigma_{\Psi} (f ; x)$ in $(6.13)$ (since $\sigma_{\Psi} (f ; x) = \sigma (f
  ; x) + o (1)$ and $\kappa < \delta$ this is allowed). We have $v = v_f (x ;
  \Delta) = \omega (f ; \Delta / \sigma_{\Psi} (f ; x)) = \omega (f ;
  \kappa)$. Letting $x \rightarrow \infty$ in $(6.13)$ we obtain $L (f ;
  \omega (f ; \kappa)) e^{- \omega (f ; \kappa) c (f)} = L (g ; \omega (f ;
  \kappa)) e^{- \omega (f ; \kappa) c (g)}$. Since $0 < \kappa \leqslant
  \delta / 2$ was arbitrary and $\omega (f ; x)$ is increasing (with $\omega
  (f ; 0) = 0$), the functions $L (f ; z) e^{- zc (f)}$ and $L (g ; z) e^{- zc
  (g)}$ coincide on the interval $[0 ; \omega (f ; \delta / 2)]$. Both
  functions are entire by lemma 4.4. Hence by analytic continuation $L (f ; z)
  e^{- zc (f)} = L (g ; z) e^{- zc (g)}$ for all $z \in \mathbbm{C}$. Since
  exponentials never vanish we obtain $\mathcal{Z}(L (f ; z)) =\mathcal{Z}(L
  (g ; z))$. 
\end{proof}

\begin{lemma}
  Let $f, g \in \mathcal{C}$. If $\mathcal{Z}(L (f ; z)) =\mathcal{Z}(L (g ;
  z))$ then $f = g$, where $\mathcal{Z}(h)$ denotes the zero set of $h
  (\cdot)$ (the zeroes are counted without multiplicity).
\end{lemma}

\begin{proof}
  From the definition of $L (f ; z)$ we find explicitly
  \begin{eqnarray*}
    \mathcal{Z}(L (f ; z)) & = & \left\{ \frac{\left( 2 k + 1) \pi i
    \right.}{f (p)} + \frac{\log (p - 1)}{f (p)} \text{ } : \text{ } k \in
    \mathbbm{Z} \text{ and } p \text{ prime} \right\}
  \end{eqnarray*}
  (note that $f (p) > 0$ because $f \in \mathcal{C}$). Therefore if
  $\mathcal{Z}(L (f ; z)) =\mathcal{Z}(L (g ; z))$ then
  
  \begin{eqnarray}
    \left\{ \frac{\left( 2 k + 1) \pi i \right.}{g (p)} + \frac{\log (p -
    1)}{g (p)} \right\} & = & \left\{ \frac{\left( 2 \ell + 1) \pi i
    \right.}{f (q)} + \frac{\log (q - 1)}{f (q)} \right\} 
  \end{eqnarray}
  
  for $k, \ell \in \mathbbm{Z}$ and $p, q$ going through the set of primes.
  Looking at the common zero of real part 0 and smallest imaginary part we
  conclude that $f (2) = g (2)$. Now, fix $p$ an odd prime. Because of (6.14)
  there is a prime $q$ such that
  \begin{eqnarray*}
    \frac{(2 k + 1) \pi i}{g (p)} + \frac{\log \left( p - 1) \right.}{g (p)} &
    = & \frac{\left( 2 \ell + 1) \pi i \right.}{f (q)} + \frac{\log (q - 1)}{f
    (q)}
  \end{eqnarray*}
  hence
  \begin{equation}
    \frac{f (q)}{g (p)} \text{ } = \text{ } \frac{2 \ell + 1}{2 k + 1} \text{
    } = \text{ } \frac{\log (q - 1)}{\log (p - 1)}
  \end{equation}
  Write $p - 1 = m^r$ with $r \geqslant 1$ maximal and $m$ a positive
  integer. Necessarily $r = 2^a$ with $a \geqslant 0$, otherwise $p$ would
  factorize non-trivially. Further exponentiating $(6.15)$ we get
  \begin{eqnarray*}
    q - 1 & = & \left( p - 1 \right)^{\frac{2 \ell + 1}{2 k + 1}} \text{ } =
    \text{ } m^{r \cdot \frac{2 \ell + 1}{2 k + 1}}
  \end{eqnarray*}
  Note that $r \cdot \frac{2 \ell + 1}{2 k + 1} \in \mathbbm{N}$ since 
  $m^{r (2\ell + 1) / (2k + 1)}$ is an integer and $r \geqslant 1$ was chosen
  maximal. Therefore $r \cdot \frac{2 \ell + 1}{2 k +
  1} = 2^a \cdot \frac{2 \ell + 1}{2 k + 1}$ must be a power of two, otherwise
  $q$ would factorize non-trivially. Therefore the ratio $(2 \ell + 1) / (2 k
  + 1)$ is a power of two, but then $\ell = k$ necessarily. By $(6.15)$ it
  follows that $p = q$ and $g (p) = f (p)$. Therefore $f (p) = g (p)$ for all
  prime $p$. Hence $f = g$ since $f, g$ are strongly additive.
\end{proof}

\begin{proof}[Proof of Part (4) of Theorem 1.1 when $\Psi(f;t)$ is
not lattice distributed]
  One direction is clear if $f = g$ then $\mathcal{D}_f (x ; \Delta)
  =\mathcal{D}_g (x ; \Delta)$. Conversely, suppose that $\mathcal{D}_f (x ;
  \Delta) \sim \mathcal{D}_g (x ; \Delta)$ throughout $1 \leqslant \Delta
  \leqslant \delta \sigma$, then by lemma 6.4 the zero set of $L (f ; z)$ and
  $L (g ; z)$ coincide. Hence by lemma 6.5, $f = g$, as desired. \ 
\end{proof}

\subsubsection{$\Psi (f ; t)$ is lattice distributed}

Suppose that $\Psi (f ; t)$ is lattice distributed on $\alpha \mathbbm{Z}$ for
some $\alpha > 0$. Then $\Psi (f / \alpha ; t)$ is lattice distributed on
$\mathbbm{Z}$. Since $\mathcal{D}_f (x ; \Delta) =\mathcal{D}_{f / \alpha} (x
; \Delta)$ we can assume without loss of generality (for the purpose of
proving Part (4) of Theorem 1.1) that $\Psi (f ; t)$ is lattice distributed on
$\mathbbm{Z}$. To such a $f$ we associate two strongly additive function
$\mathfrak{f}$ and $\mathfrak{h}_f$ defined by
\begin{eqnarray*}
  \mathfrak{f}(p) = \left\{ \begin{array}{l}
    f (p) \text{ if } f (p) \in \mathbbm{Z}\\
    0 \text{ \ \ \ \ otherwise}
  \end{array} \right. & \tmop{and} & \mathfrak{h}_f (p) = \left\{
  \begin{array}{l}
    f (p) \text{ if } f (p) \nin \mathbbm{Z}\\
    0 \text{ \ \ \ \ otherwise}
  \end{array} \right.
\end{eqnarray*}
In particular $f (n) =\mathfrak{f}(n) +\mathfrak{h}_f (n)$. Similarly to an
additive function $g$ we associate $\mathfrak{g}$ and $\mathfrak{h}_g$ with
$\mathfrak{g}$ and $\mathfrak{h}_g$ defined in the same way as $\mathfrak{f}$
and $\mathfrak{h}_f$. \ \ \ \

\begin{lemma}
  Let $f, g \in \mathcal{C}$. Suppose that $\Psi (f ; t) = \Psi (g ; t)$ and
  that $\Psi (f ; t)$ is lattice distributed on $\mathbbm{Z}$. If
  $\mathcal{D}_f (x ; \Delta) \sim \mathcal{D}_g (x ; \Delta)$ throughout $1
  \leqslant \Delta \leqslant \delta \sigma$ for some $\delta > 0$, then,
  \begin{equation}
    L \left( \mathfrak{f}; v \right) e^{- vc (f)} \cdot
    \mathcal{P}_{\mathfrak{h}_f} \left( \xi_f (x ; \Delta) ; v \right) = L
    \left( \mathfrak{g}; v \right) e^{- vc (g)} \cdot
    \mathcal{P}_{\mathfrak{h}_g} \left( \xi_g \left( x ; \Delta \right) ; v
    \right) + o (1)
  \end{equation}
  uniformly throughout $1 \leqslant \Delta \leqslant \delta \sigma (x)$, with
  $v = v_f (x ; \Delta) = v_g (x ; \Delta)$.
\end{lemma}

\begin{proof}
  Since $\Psi(f;t) = \Psi(g;t)$ we have $S_f = S_g$ and $v_f = v = v_g$. 
  Plugging the
  asymptotic of part 4 of theorem 2.8 into $\mathcal{D}_f (x ; \Delta) \sim
  \mathcal{D}_g (x ; \Delta)$ and cancelling $S_f (x ; \Delta) = S_g (x ;
  \Delta)$ on both sides, we obtain $(6.16)$ but with the right hand side
  multiplied by an $1 + o (1)$, instead of an error term of $o (1)$. To obtain
  the $o (1)$ it suffices to prove that $L (\mathfrak{g}; v) e^{- vc (f)}
  \mathcal{P}_{\mathfrak{h}_g} \left( \xi_g \left( x ; \Delta \right) ; v
  \right) = O \left( 1 \right)$. The function $L (\mathfrak{g}; v) e^{- vc
  (f)}$ is continuous and the parameter $v \asymp \Delta / \sigma (f ; x) =
  O_{\delta} (1)$ (because $\Delta \leqslant \delta \sigma (x)$), by 
  lemma 4.7. Therefore $L
  (\mathfrak{g}; v) e^{- vc (f)} = O_{\delta} (1)$. By lemma 4.25,
  $\mathcal{P}_{\mathfrak{h}_g} (\xi_f (x ; \Delta) ; v) = O_{\delta} (1)$.
  The claim $L (\mathfrak{g}; v) e^{- vc (f)} \mathcal{P}_{\mathfrak{h}_g}
  (\xi_f (x ; \Delta) ; v) = O_{\delta} (1)$ follows.
\end{proof}

\begin{lemma}
  Let $f \in \mathcal{C}$. Suppose that $\Psi (f ; t)$ is lattice distributed
  on $\mathbbm{Z}$. Given $C > 0$, uniformly in $0 \leqslant v \leqslant C$ we
  have
  \[ \int_0^1 \mathcal{P}_{\mathfrak{h}_f} \left( a ; v \right) \mathd a =
     \prod_{p : \mathfrak{h}_f \left( p \right) \neq 0} \left( 1 +
     \frac{e^{v\mathfrak{h}_f (p)}}{p - 1} \right) \cdot \left( 1 -
     \frac{1}{p} \right) \]
\end{lemma}

\begin{proof}
  To ease notation let $X (\mathfrak{h}_f) : = \sum_p \mathfrak{h}_f \left( p
  \right) X_p$. By definition of $\mathcal{P}_{\mathfrak{h}_f} \left( a ; v
  \right)$,
  \begin{eqnarray*}
    \int_0^1 \mathcal{P}_{\mathfrak{h}_f} \left( a ; v \right) \mathd a & = &
    v \sum_{k \in \mathbbm{Z}} \int_0^1 e^{v (k + a)} \cdot \mathbbm{P} \left(
    X \left( \mathfrak{h}_f \right) \geqslant k + a \right) \mathd a\\
    & = & v \sum_{k \in \mathbbm{Z}} \int_k^{k + 1} e^{va} \cdot \mathbbm{P}
    \left( X \left( \mathfrak{h}_f \right) \geqslant a \right) \mathd a\\
    & = & v \int_{\mathbbm{R}} e^{va} \cdot \mathbbm{P} \left( X \left(
    \mathfrak{h}_f \right) \geqslant a \right) \mathd a
  \end{eqnarray*}
  where we are allowed to interchange summation and integral because all the
  terms involved are positive. 
  In the above integral write $e^{va} \mathd a = (1 / v) \mathd (e^{va})$ and
  integrate by parts
  \begin{eqnarray*}
    v \int_{\mathbbm{R}} e^{va} \cdot \mathbbm{P} \left( X \left(
    \mathfrak{h}_f \right) \geqslant a \right) \mathd a & = & \left[ e^{va}
    \mathbbm{P} \left( X \left( \mathfrak{h}_f \right) \geqslant a \right)
    \right]^{\infty}_{- \infty} - \int_{\mathbbm{R}} e^{va} \mathd \mathbbm{P}
    \left( X \left( \mathfrak{h}_f \right) \geqslant a \right)
  \end{eqnarray*}
  By lemma 4.19 we have $\mathbbm{E}[e^{AX (\mathfrak{h})}] < + \infty$ for
  any fixed $A > 0$. Therefore by Chernoff's bound $\mathbbm{P}(X
  (\mathfrak{h}) \geqslant t) \leqslant \mathbbm{E}[e^{AX (\mathfrak{h})}]
  e^{- At}$ decays faster than any power of $e^{- t}$. Hence $\left[ e^{va}
  \mathbbm{P}(X (\mathfrak{h}_f) \geqslant a]^{\infty}_{- \infty} \right.$
  vanishes (because $0 \leqslant v \leqslant C$). 
  It remains to note that $- \mathd \mathbbm{P} \left( X
  (\mathfrak{h}_f) \geqslant a) = \mathd \left( 1 -\mathbbm{P} \left( X \left(
  \mathfrak{h}_f \right) < a \right) \right) = \mathd \mathbbm{P} \left( X
  \left( \mathfrak{h}_f \right) < a \right) \right.$. Thus, the second term in
  the above equation equals to
  \[ \int_{\mathbbm{R}} e^{va} \mathd \mathbbm{P} \left( X \left(
     \mathfrak{h}_f \right) < a \right) =\mathbbm{E} \left[ e^{vX \left(
     \mathfrak{h}_f \right)} \right] \]
  the Laplace transform of $X \left( \mathfrak{h}_f \right) = \sum_p
  \mathfrak{h}_f \left( p \right) X_p$ ! By independence of the $X_p$,
  \[ \mathbbm{E} \left[ e^{\left. vX (\mathfrak{h}_f \right)} \right] =
     \prod_p \mathbbm{E} \left[ e^{v\mathfrak{h}_f (p) X_p} \right] = \prod_p
     \left( 1 + \frac{e^{v\mathfrak{h}_f (p)}}{p - 1} \right) \cdot \left( 1 -
     \frac{1}{p} \right) \]
  When $\mathfrak{h}_f (p) = 0$ the relevant term in the product simply equals
  to 1, therefore we can add the condition $\mathfrak{h}_f \left( p \right)
  \neq 0$, in the product, without altering its value.
\end{proof}

\begin{lemma}
  Let $f, g \in \mathcal{C}$. Suppose that $\Psi (f ; t) = \Psi (g ; t)$ and
  that $\Psi (f ; t)$ is lattice distributed on $\mathbbm{Z}$. If $(6.16)$
  holds uniformly throughout $1 \leqslant \Delta \leqslant \delta
  \sigma_{\Psi} (f ; x)$ for some $\delta > 0$, then
  \[ L \left( f ; \kappa \right) e^{- \kappa c (f)} = L \left( g ; \kappa
     \right) e^{- \kappa c (g)} \]
  for all $\kappa > 0$ sufficiently small. \ 
\end{lemma}

\begin{proof}
  Let us start by remarking that since $\Psi (f ; t) = \Psi (g ; t)$ we have
  $\sigma_{\Psi} (f ; x) = \sigma_{\Psi} (g ; x)$. It will be convenient to
  denote the common value by $\sigma_{\Psi} (x)$. As usual denote $v_f (x ;
  \Delta)$ by $v$. Since $\omega (f ; x)$ is increasing and $\omega (f ; 0) =
  0$, for any sufficiently small $\kappa > 0$ we can find a $\lambda$ such
  that $\omega (f ; \lambda) = \kappa$ and $0 < \lambda < \delta$. We restrict
  $\Delta$ to the range $\lambda \sigma_{\Psi} \leqslant \Delta \leqslant
  \lambda \sigma_{\Psi} + 1 / \sigma_{\Psi}$. In this range
  \[ v = \omega (f ; \Delta / \sigma_{\Psi}) = \omega (f ; \lambda) + O (1 /
     \sigma^2_{\Psi}) = \kappa + O (1 / \sigma^2_{\Psi}) \]
  because by lemma 4.7 the function $\omega(f;z)$ is analytic in a neighborhood
  of $\mathbb{R}^{+} \cup \{ 0 \}$. 
  Hence by ``analyticity'' of $L (\mathfrak{f}; v) e^{- vc (f)}$ 
  (see lemma 4.20) and lemma 4.25,
  \begin{eqnarray*}
    L (\mathfrak{f}; v) e^{- vc (f)} & = & L (\mathfrak{f}; \kappa) e^{-
    \kappa c (f)} + O (1 / \sigma^2_{\Psi})\\
    \mathcal{P}_{\mathfrak{h}_f} \left( \xi_f (x ; \Delta) ; v) \right. & = &
    \mathcal{P}_{\mathfrak{h}_f} \left( \xi_f (x ; \Delta) ; \kappa \right) +
    O \left( 1 / \sigma^2_{\Psi} \right)
  \end{eqnarray*}
  Of course the same relations are valid with $f$ replaced by $g$. Multiplying
  the two relations above, we see that when $\Delta$ is confined to $\lambda
  \sigma_{\Psi} \leqslant \Delta \leqslant \lambda \sigma_{\Psi} + 1 /
  \sigma_{\Psi}$ we can rewrite $(6.16)$ in the following equivalent form
  \begin{equation}
    L (\mathfrak{f}; \kappa) e^{- \kappa c (f)} \mathcal{P}_{\mathfrak{h}_f}
    \left( \xi_f (x ; \Delta) ; \kappa \right) = L (\mathfrak{g}; \kappa) e^{-
    \kappa c (g)} \mathcal{P}_{\mathfrak{h}_g} \left( \xi_g \left( x ; \Delta
    \right) ; \kappa \right) + o (1)
  \end{equation}
  If the above relation was true uniformly for a common $0 \leqslant a
  \leqslant 1$ in place of the conceivably distinct 
  $\xi_f (x ; \Delta)$ and $\xi_g (x ;
  \Delta)$ it would be enough to integrate the above over $0 \leqslant a
  \leqslant 1$, use lemma 6.7 and conclude. Unfortunately, such a simplifying
  device is not present, so we have to be slightly more careful. Let $b \in
  \mathbbm{R}$ be arbitrary. Recall that $\{\xi_f (x ; \Delta)\}=\{\mu (f ; x)
  + \Delta \sigma (f ; x)\}$ is $1 / \sigma (f ; x)$ periodic in $\Delta$.
  Hence, by a change of variable and the preceding lemma
  \begin{eqnarray*}
    \int_b^{b + 1 / \sigma (f ; x)} \mathcal{P}_{\mathfrak{h}_f} (\xi_f (x ;
    \Delta) ; \kappa) \mathd \Delta & = & \frac{1}{\sigma (f ; x)} \int_0^1
    \mathcal{P}_{\mathfrak{h}_f} \left( a ; \kappa \right) \mathd a\\
    & = & \frac{1}{\sigma (f ; x)} \prod_{p : \mathfrak{h}_f (p) \neq 0}
    \left( 1 + \frac{e^{\kappa \mathfrak{h}_f(p)}}{p - 1} \right) \cdot 
    \left( 1 - \frac{1}{p} \right)
  \end{eqnarray*}
  By lemma 4.25, $\mathcal{P}_{\mathfrak{h}_f} \left( \xi_f \left( x ; \Delta
  \right) ; \kappa \right) = O_{\delta} \left( 1 \right)$. Therefore
  \begin{eqnarray*}
    \int_b^{b + 1 / \sigma_{\Psi}} \mathcal{P}_{\mathfrak{h}_f} \left( \xi_f
    \left( x ; \Delta \right) ; \kappa \right) \mathd \Delta & = & \left(
    \int_b^{b + 1 / \sigma (f ; x)} + \int_{b + 1 / \sigma (f ; x)}^{b + 1 /
    \sigma_{\Psi}} \right) \mathcal{P}_{\mathfrak{h}_f} \left( \xi_f \left( x
    ; \Delta \right) ; \kappa \right) \mathd \Delta
  \end{eqnarray*}
  We just computed the first integral. The second integral is bounded by $O
  (1)$ times the length of the interval $[b + 1 / \sigma (f ; x) ; b + 1 /
  \sigma_{\Psi}]$. That length being $\ll 1 / \sigma^2_{\Psi}$ the second
  integral is bounded by $O \left( 1 / \sigma^2_{\Psi} \right)$. Now take $b =
  \lambda \sigma_{\Psi}$ and integrate the left hand side of $(6.17)$ over
  $\lambda \sigma_{\Psi} \leqslant \Delta \leqslant \lambda \sigma_{\Psi} + 1
  / \sigma_{\Psi}$. We obtain
  \[ \frac{L (\mathfrak{f}; \kappa) e^{- \kappa c (f)}}{\sigma (f ; x)}
     \prod_{p : \mathfrak{h}_f (p) \neq 0} \left( 1 + \frac{e^{\kappa 
     \mathfrak{h}_f(p)}}{p - 1} \right) \left( 1 - \frac{1}{p} \right) 
     + O \left(\sigma_{\Psi}^{- 2} \right) = \frac{L (f ; \kappa) 
     e^{- \kappa c(f)}}{\sigma (f ; x)} + O \left( \sigma_{\Psi}^{- 2} 
     \right) \]
  The same result is true with $f$ replaced by $g$. Therefore integrating
  $(6.17)$ over $\lambda \sigma_{\Psi} \leqslant \Delta \leqslant \lambda
  \sigma_{\Psi} + 1 / \sigma_{\Psi}$ yields
  \[ \frac{1}{\sigma (f ; x)} \cdot L (f ; \kappa) e^{- \kappa c (f)} \text{}
     = \frac{1}{\sigma (g ; x)} \cdot L (g ; \kappa) e^{- \kappa c (g)} + o
     \left( \frac{1}{\sigma_{\Psi}} \right) \]
  Since $\sigma (f ; x) \sim \sigma (g ; x)$ and $\sigma(f;x) \sim 
  \sigma_{\Psi}(f;x)$ letting $x \rightarrow \infty$ we conclude $L (f ;
  \kappa) e^{- \kappa c (f)} = L (g ; \kappa) e^{- \kappa c (g)}$. Since
  $\kappa > 0$ was an arbitrary, sufficiently small real number, it follows
  that $L (f ; \kappa) e^{- \kappa c (f)} = L (g ; \kappa) e^{- \kappa c (g)}$
  holds for all $\kappa > 0$ sufficiently small.
\end{proof}

\begin{proof}[Proof of Part (4) of Theorem 1.1 when $\Psi(f;t)$ is lattice
distributed]
  If $f = g$ then $\mathcal{D}_f (x ; \Delta) =\mathcal{D}_g (x ; \Delta)$ for
  all $x, \Delta \geqslant 1$. Conversely, suppose that $\Psi (f ; t)$ is
  lattice distributed on $\alpha \mathbbm{Z}$ and that $\mathcal{D}_f (x ;
  \Delta) \sim \mathcal{D}_g (x ; \Delta)$ for $1 \leqslant \Delta \leqslant c
  \sigma (x)$. Since $\mathcal{D}_f (x ; \Delta) \sim \mathcal{D}_g (x ;
  \Delta)$ also holds in $1 \leqslant \Delta \leqslant o (\sigma)$, by part 3
  of theorem 1.1, we have $\Psi (f ; t) = \Psi (g ; t)$. Note that $\Psi (f /
  \alpha ; t) = \Psi (f ; t \alpha) = \Psi (g ; t \alpha) = \Psi (g / \alpha ;
  t)$ is lattice distributed on $\mathbbm{Z}$. In addition $\mathcal{D}_{f /
  \alpha} (x ; \Delta) =\mathcal{D}_f (x ; \Delta) \sim \mathcal{D}_g (x ;
  \Delta) =\mathcal{D}_{g / \alpha} (x ; \Delta)$ holds throughout $1
  \leqslant \Delta \leqslant c \alpha \sigma_{\alpha} (x)$ where
  $\sigma_{\alpha} (x) \assign (1 / \alpha) \sigma (x) \sim \sigma (f / \alpha
  ; x) \sim \sigma (g / \alpha ; x)$. Thus we can assume without loss of
  generality that $\Psi (f ; t) = \Psi (g ; t)$ is lattice distributed on
  $\mathbbm{Z}$. Hence by lemma 6.6 relation $(6.16)$ holds and thus lemma 6.8
  is applicable. By lemma 6.8, $L (f ; \kappa) e^{- \kappa c (f)} = L (g ;
  \kappa) e^{- \kappa c (g)}$ for all $\kappa > 0$ sufficiently small. By
  analytic continuation $L (f ; z) e^{- zc (f)} = L (g ; z) e^{- zc (g)}$ for
  all $z \in \mathbbm{C}$, because $L (f ; z)$ and $L (g ; z)$ are entire by
  lemma 4.4. It follows that the zero set of $L (f ; z)$ and $L (g ; z)$
  coincides. Thus $f = g$ by lemma 6.5. 
\end{proof}

\section{Kubilius model -- Theorems 2.1 and 2.2}

Let $\mathcal{A}(f ; z)$ denote the function defined in theorem 2.2. Since
$\mathcal{A}(f ; z)$ is analytic in $\mathbbm{R}^{+} \cup \{0\}$ and 
$\mathcal{A}(f ; 0) =
1$ we have $\mathcal{A}(f ; \Delta / \sigma) = 1 + o (1)$ for $1 \leqslant
\Delta \leqslant o (\sigma)$. Thus Theorem 2.1 is a consequence of Theorem
2.2. The overall strategy in our proof of theorem 2.2 is to establish an
asymptotic for
\begin{equation}
  \mathbbm{P} \left( \sum_{p \leqslant x} f (p) \left[ X_p - \frac{1}{p}
  \right] \geqslant \Delta \sigma (f ; x) \right)
\end{equation}
and compare it with the asymptotic for $\mathcal{D}_f (x ; \Delta)$ from
theorem 2.8. We will deduce an asymptotic for $(7.1)$ from the three general
propositions established in section 4. Throughout this section the $X_p$'s
will denote independent Bernoulli random variables, distributed according to
\begin{eqnarray*}
  \mathbbm{P}(X_p = 1) = 1 / p & \tmop{and} & \mathbbm{P}(X_p = 0) = 1 - 1 / p
\end{eqnarray*}
We break down the proof of an asymptotic for $(7.1)$ into three cases. The
proof of theorem 2.2 is in section 7.4.

\subsection{$(\tmop{loglog} x)^{\varepsilon} \ll \Delta \ll \sigma (f ; x)$
and $\Psi (f ; t)$ is lattice distributed on $\mathbbm{Z}$}

Throughout write $f =\mathfrak{g}+\mathfrak{h}$ with $\mathfrak{g}$ and
$\mathfrak{h}$ two strongly additive functions defined by
\begin{eqnarray*}
  \mathfrak{g}(p) = \left\{ \begin{array}{l}
    f (p) \text{ if } f (p) \in \mathbbm{Z}\\
    0 \text{ \ \ \ \ otherwise}
  \end{array} \right. & \tmop{and} & \mathfrak{h}(p) = \left\{
  \begin{array}{l}
    f (p) \text{ if } f (p) \nin \mathbbm{Z}\\
    0 \text{ \ \ \ \ otherwise}
  \end{array} \right.
\end{eqnarray*}
\begin{lemma}
  Let $f \in \mathcal{C}$. Suppose that $\Psi (f ; t)$ is lattice distributed
  on $\mathbbm{Z}$. Define the random variable $\Omega (\mathfrak{g}; x) =
  \sum_{p \leqslant x} \mathfrak{g}(p) X_p$. Given $C > 0$, we have uniformly
  in the region $- C \leqslant \kappa \assign \tmop{Re} s \leqslant C$, $|
  \tmop{Im} s| \leqslant 2 \pi$,
  \begin{eqnarray*}
    \mathbbm{E} \left[ e^{s \Omega (\mathfrak{g}; x)} \right] & = & L
    (\mathfrak{g}; s) \cdot e^{\gamma ( \hat{\Psi} (f ; s) - 1)} \cdot \left(
    \log x \right)^{\hat{\Psi} (f ; s) - 1} + O_C \left( (\log x)^{\hat{\Psi}
    (f ; \kappa) - 3 / 2} \right)
  \end{eqnarray*}
  as $x \rightarrow \infty$.
\end{lemma}

\begin{proof}
  Since the $X_p$ are independent Bernoulli random variables, we have
  \begin{eqnarray}
    \mathbbm{E} \left[ e^{s \Omega (\mathfrak{g}; x)} \right] & = & \prod_{p
    \leqslant x} \mathbbm{E} \left[ e^{s\mathfrak{g}(p) X_p} \right] \text{ }
    = \text{ } \prod_{p \leqslant x} \left( 1 - \frac{1}{p} \right) \cdot
    \left( 1 + \frac{e^{s\mathfrak{g}(p)}}{p - 1} \right) \nonumber\\
    & = & \left[ \prod_{p \leqslant x} \left( 1 - \frac{1}{p}
    \right)^{\hat{\Psi} (f ; s)} \left( 1 + \frac{e^{s\mathfrak{g}(p)}}{p - 1}
    \right) \right] \cdot \prod_{p \leqslant x} \left( 1 - \frac{1}{p}
    \right)^{- ( \hat{\Psi} (f ; s) - 1)} 
  \end{eqnarray}
  The product on the right equals $e^{\gamma ( \hat{\Psi} (f ; s) - 1)} \cdot
  (\log x)^{\hat{\Psi} (f ; s) - 1} \cdot \left( 1 + O ((\log x)^{- 1})
  \right)$ by Mertens's formula (we use that $|\hat{\Psi}(f;s)| \leqslant
  \hat{\Psi}(f;C)$). On the other hand since $x \rightarrow
  \infty$, lemma 4.20 is applicable and so the product on the left hand side
  equals $L (\mathfrak{g}; s) \cdot (1 + O ((\log x)^{- 1 / 2})$. Thus
  \[ \mathbbm{E} \left[ e^{s \Omega (\mathfrak{g}; x)} \right] = L
     (\mathfrak{g}; s) e^{\gamma ( \hat{\Psi} (f ; s) - 1)} \cdot \left( \log
     x \right)^{\hat{\Psi} (f ; s) - 1} \cdot \left( 1 + O \left( 1 /
     \sqrt[]{\log x} \right) \right) \]
  By lemma 4.20 the function $L (\mathfrak{g}; s)$ is entire. Therefore $L
  (\mathfrak{g}; s)$ is bounded in the region $| \tmop{Re} s| \leqslant C$, $|
  \tmop{Im} s| \leqslant 2 \pi$ because this region is bounded. The function
  $e^{\gamma ( \hat{\Psi} (f ; s) - 1)}$ is bounded in $\tmop{Re} s \leqslant
  C$ because of the inequality $| \hat{\Psi} (f ; s) | \leqslant \hat{\Psi} (f
  ; \tmop{Re} s)$. It follows that the previous equation simplifies to
  \[ \mathbbm{E} \left[ e^{s \Omega (\mathfrak{g}; x)} \right] = L
     (\mathfrak{g}; s) e^{\gamma ( \hat{\Psi} (f ; s) - 1)} \cdot \left( \log
     x \right)^{\hat{\Psi} (f ; s) - 1} + O \left( (\log x)^{\hat{\Psi} (f ;
     \kappa) - 3 / 2} \right) \] 
   which proves the lemma.
\end{proof}

\begin{lemma}
  Let $f \in \mathcal{C}$. Suppose that $\Psi (f ; t)$ is lattice distributed
  on $\mathbbm{Z}$. Given a $\delta, \varepsilon > 0$ we have, uniformly in
  $(\tmop{loglog} x)^{\varepsilon} \ll \Delta \leqslant \delta \sigma (f ;
  x)$,
  \begin{equation}
    \mathcal{D}_f (x ; \Delta) \sim \frac{e^{- \gamma ( \hat{\Psi} (f ; v) -
    1)}}{\Gamma ( \hat{\Psi} (f ; v))} \cdot \mathbbm{P} \left( \sum_{p
    \leqslant x} f (p) \left[ X_p - \frac{1}{p} \right] \geqslant \Delta
    \sigma (f ; x) \right)
  \end{equation}
  where $v \assign v_f (x ; \Delta)$. 
\end{lemma}

\begin{proof}
  Let $X_p$ be independent Bernoulli random variables, distributed according
  to
  \[ \left. \mathbbm{P}(X_p = 1 \right) = \frac{1}{p} \text{ and } \mathbbm{P}
     \left( X_p = 0 \right) = 1 - \frac{1}{p} \]
  Denote by $(\Omega, \mathcal{F}, \mathbb{P})$ the underlying probability 
  space. Given a
  strongly additive function $g$, define $\Omega (g ; x) = \sum_{p \leqslant
  x} g (p) X_p$. Let $\mathfrak{H}$ be a strongly additive function such that
  $0 \leqslant \mathfrak{H}(p) \leqslant \left\lceil \mathfrak{h}(p)
  \right\rceil$. Note that $\mathfrak{H}(p)$ vanishes when $\mathfrak{h}(p)$
  does. Thus $\mathfrak{g}$ and $\mathfrak{H}$ are ``supported'' on two
  disjoint sets of primes and as a consequence the random variables $\Omega
  (\mathfrak{g}; x)$ and $\Omega (\mathfrak{H}; x)$ are independent. Therefore
  \begin{eqnarray*}
    \mathbbm{E} \left[ e^{s \Omega (\mathfrak{g}; x) + s \Omega (\mathfrak{H};
    x)} \right] & = & \mathbbm{E} \left[ e^{s \Omega (\mathfrak{g}; x)}
    \right] \cdot \mathbbm{E} \left[ e^{s \Omega (\mathfrak{H}; x)} \right]\\
    & = & \mathbbm{E} \left[ e^{s \Omega (\mathfrak{g}; x)} \right] \cdot
    \prod_{p \leqslant x} \left( 1 - \frac{1}{p} +
    \frac{e^{s\mathfrak{H}(p)}}{p} \right)
  \end{eqnarray*}
  Furthermore by the previous lemma for $0 \leqslant \kappa \assign \tmop{Re}
  s \leqslant C$ and $| \tmop{Im} s| \leqslant 2 \pi$,
  \[ \mathbbm{E} \left[ e^{s \Omega (\mathfrak{g}; x)} \right] = L
     (\mathfrak{g}; s) e^{\gamma ( \hat{\Psi} (f ; s) - 1)} \cdot (\log
     x)^{\hat{\Psi} (f ; s) - 1} + O \left( \left( \log x \right)^{\hat{\Psi}
     (f ; \kappa) - 3 / 2} \right) \]
  By lemma 4.20 and 4.2 the functions $L (\mathfrak{g}; s)$ and $e^{\gamma (
  \hat{\Psi} (f; s) - 1)}$ are entire. From the product
  representation it is clear that $L (\mathfrak{g}; x) \neq 0$ for $x
  \geqslant 0$. Therefore the function $L (\mathfrak{g}; s) e^{\gamma (
  \hat{\Psi} (f ; s) - 1)}$ is in addition non-vanishing on the positive real
  axis. Therefore the assumption of proposition 4.17 are satisfied. It follows
  that the expression
  \[ \mathbbm{P} \left( \sum_{p \leqslant x} f (p) \cdot \left[ X_p -
     \frac{1}{p} \right] \geqslant \Delta \sigma (f ; x) \right) \]
  is asymptotic to
  \begin{eqnarray}
    &  & L (\mathfrak{g}; v) e^{\gamma ( \hat{\Psi} (f ; v) - 1)} \cdot (1 /
    v)\mathcal{P}_{\mathfrak{h}} (\xi_f (x ; \Delta) ; v) \cdot \frac{\left(
    \log x \right)^{\hat{\Psi} (f ; v) - 1 - v \hat{\Psi}' (f ; v)}}{(2 \pi
    \hat{\Psi}'' (f ; v) \tmop{loglog} x)^{1 / 2}} \cdot e^{- vc (f)} 
  \end{eqnarray}
  uniformly in $(\tmop{loglog} x)^{\varepsilon} \ll \Delta \leqslant c \sigma
  (f ; x)$ where $v \assign v_f (x ; \Delta)$. Furthermore, by part 4 of
  theorem 2.8, $\mathcal{D}_f (x ; \Delta)$ is asymptotic to
  \begin{equation}
    \frac{L (\mathfrak{g}; v)}{\Gamma ( \hat{\Psi} (f ; v))} \cdot (1 /
    v)\mathcal{P}_{\mathfrak{h}} (\xi_f (x ; \Delta) ; v) \cdot \frac{(\log
    x)^{\hat{\Psi} (f ; v) - 1 - v \hat{\Psi}' (f ; v)}}{(2 \pi \hat{\Psi}''
    (f ; v) \tmop{loglog} x)^{1 / 2}} \cdot e^{- vc (f)}
  \end{equation}
  throughout $(\tmop{loglog} x)^{\varepsilon} \ll \Delta \leqslant c \sigma (f
  ; x)$. Comparing $(7.4)$ and $(7.5)$ proves the lemma.
\end{proof}

\subsection{$(\tmop{loglog} x)^{\varepsilon} \ll \Delta \ll \sigma (f ; x)$
and $\Psi (f ; t)$ is not lattice distributed}

\begin{lemma}
  Let $f \in \mathcal{C}$. Let $\Omega (f ; x) \assign \sum_{p \leqslant x} f
  (p) X_p$. Given $C > 0$, uniformly in $- C \leqslant \kappa \assign
  \tmop{Re} s \leqslant C$ and $| \tmop{Im} s| \leqslant \tmop{loglog} x$,
  \begin{eqnarray*}
    \mathbbm{E} \left[ e^{s \Omega (f ; x)} \right] & = & L (f ; s) e^{\gamma
    ( \hat{\Psi} (f ; s) - 1)} \cdot \left( \log x \right)^{\hat{\Psi} (f ; s)
    - 1} + O \left( (\log x)^{\hat{\Psi} (f ; \kappa) - 3 / 2} \right)
  \end{eqnarray*}
  as $x \rightarrow \infty$.
\end{lemma}

\begin{proof}
  This is the same proof as in lemma 7.1. There is a minor twist because $L (f
  ; s)$ is no more bounded and we use lemma 4.4 instead of lemma 4.20. We give
  the proof anyway. Since the $X_p$ are independent Bernoulli random variable
  \begin{eqnarray*}
    \mathbbm{E} \left[ e^{s \Omega (f ; x)} \right] & = & \prod_{p \leqslant
    x} \mathbbm{E} \left[ e^{sf (p) X_p} \right] \text{ } = \text{ } \prod_{p
    \leqslant x} \left( 1 - \frac{1}{p} \right) \left( 1 + \frac{e^{sf (p)}}{p
    - 1} \right)\\
    & = & \left[ \prod_{p \leqslant x} \left( 1 - \frac{1}{p}
    \right)^{\hat{\Psi} (f ; s)} \left( 1 + \frac{e^{sf (p)}}{p - 1} \right)
    \right] \cdot \prod_{p \leqslant x} \left( 1 - \frac{1}{p} \right)^{- (
    \hat{\Psi} (f ; s) - 1)}
  \end{eqnarray*}
  The product on the right equals $e^{\gamma ( \hat{\Psi} (f ; s) - 1)} \cdot
  (\log x)^{\hat{\Psi} (f ; s) - 1} \cdot (1 + O ((\log x)^{- 1}))$ by
  Mertens's formula. On the other hand since $x \rightarrow \infty$ by lemma
  4.4, the product on the left equals to $L (f ; s) \cdot (1 + O ((\log x)^{-
  1}))$. Thus
  \[ \mathbbm{E} \left[ e^{s \Omega (f ; x)} \right] = L (f ; s) e^{\gamma (
     \hat{\Psi} (f ; s) - 1)} \cdot (\log x)^{\hat{\Psi} (f ; s) - 1} \cdot (1
     + O ((\log x)^{- 1})) \]
  By lemma 4.4 we have $L (f ; s) \ll_C 1 + \tmop{loglog} x$ uniformly in $|
  \tmop{Im} s| \leqslant \tmop{loglog} x$ and $| \tmop{Re} s| \leqslant C$.
  The function $e^{\gamma ( \hat{\Psi} (f ; s) - 1)}$ is bounded in $\tmop{Re}
  s \leqslant C$ because of the inequality $| \hat{\Psi} (f ; s) | \leqslant
  \hat{\Psi} (f ; \kappa)$. Thus the previous equation simplifies to 
  \[ \mathbbm{E} \left[ e^{s \Omega (f ; x)} \right] = L (f ; s) e^{\gamma (
     \hat{\Psi} (f ; s) - 1)} \cdot (\log x)^{\hat{\Psi} (f ; s) - 1} + O
     \left( (\log x)^{\hat{\Psi} (f ; \kappa) - 3 / 2} \right) \]
  (where $\kappa \assign \tmop{Re} s \leqslant C$) which is the claim.
\end{proof}

\begin{lemma}
  Let $f \in \mathcal{C}$. Suppose that $\Psi (f ; t)$ is not lattice
  distributed. Let $\delta, \varepsilon > 0$ be given. We have, uniformly in
  $(\tmop{loglog} x)^{\varepsilon} \ll \Delta \leqslant \delta \sigma (f ;
  x)$,
  \[ \mathcal{D}_f (x ; \Delta) \sim \frac{e^{- \gamma ( \hat{\Psi} (f ; v) -
     1)}}{\Gamma ( \hat{\Psi} (f ; v))} \cdot \mathbbm{P} \left( \sum_{p
     \leqslant x} f (p) \left[ X_p - \frac{1}{p} \right] \geqslant \Delta
     \sigma (f ; x) \right) \]
  Here, as usual $v \assign v_f (x ; \Delta)$.
\end{lemma}

\begin{proof}
  Let $\Omega (f ; x) \assign \sum_{p \leqslant x} f (p) X_p$. By the previous
  lemma for any given $C > 0$, we have uniformly in $0 \leqslant \kappa
  \assign \tmop{Re} s \leqslant C$ and $| \tmop{Im} s| \leqslant \tmop{loglog}
  x$,
  \[ \mathbbm{E} \left[ e^{s \Omega (f ; x)} \right] = L (f ; s) e^{\gamma (
     \hat{\Psi} (f ; s) - 1)} \cdot \left( \log x \right)^{\hat{\Psi} (f ; s)
     - 1} + O \left( (\log x)^{\hat{\Psi} (f ; \kappa) - 3 / 2} \right) \]
  By lemma 4.4 the function $L (f ; s)$ is entire and $L (f ; s) = O_{C,
  \varepsilon} (1 + | \tmop{Im} s|^{\varepsilon})$ throughout $0 \leqslant
  \tmop{Re} s \leqslant C$. From the product representation for $L (f ; x)$ is
  it clear that $L (f ; s)$ doesn't vanish on $\mathbbm{R}^+$. The function
  $\exp (\gamma ( \hat{\Psi} (f ; s) - 1)$ is entire by lemma 4.2, never zero,
  and bounded in $\tmop{Re} s \leqslant C$, because $| \hat{\Psi} (f ; s) |
  \leqslant \hat{\Psi} (f ; \tmop{Re} s)$. It follows that proposition 4.10 is
  applicable. Therefore
  \[ \mathbbm{P} \left( \frac{\Omega (f ; x) - \mu (f ; x)}{\sigma (f ; x)}
     \geqslant \Delta \right) \sim L (f ; v) e^{\gamma ( \hat{\Psi} (f ; v) -
     1)} \cdot \frac{\left( \log x \right)^{\hat{\Psi} (f ; v) - 1 - v
     \hat{\Psi}' (f ; v)}}{v (2 \pi \hat{\Psi}'' (f ; v) \tmop{loglog} x)^{1 /
     2}} \cdot e^{- vc (f)} \]
  with $v \assign v_f (x ; \Delta)$ uniformly in $(\tmop{loglog}
  x)^{\varepsilon} \ll \Delta \leqslant c \sigma (f ; x)$. On the other hand,
  by part 3 of theorem 2.8,
  \[ \mathcal{D}_f (x ; \Delta) \sim \frac{L (f ; v)}{\Gamma ( \hat{\Psi} (f ;
     v))} \cdot \frac{(\log x)^{\hat{\Psi} (f ; v) - 1 - v \hat{\Psi}' (f ;
     v)}}{v (2 \pi \hat{\Psi}'' (f ; v) \tmop{loglog} x)^{1 / 2}} \cdot e^{-
     vc (f)} \]
  with $v \assign v_f (x ; \Delta)$. On comparing the two asymptotics, the
  lemma follows. 
\end{proof}

\subsection{The range $1 \leqslant \Delta \ll (\tmop{loglog} x)^{1 / 12}$}

\begin{lemma}
  Let $f \in \mathcal{C}$. Uniformly in $1 \leqslant \Delta \ll (\tmop{loglog}
  x)^{1 / 12}$,
  \[ \mathbbm{P} \left( \sum_{p \leqslant x} f (p) \left[ X_p - \frac{1}{p}
     \right] \geqslant \Delta \sigma (f ; x) \right) \sim \mathcal{D}_f (x ;
     \Delta) \]
\end{lemma}

\begin{proof}
  Let $\Omega (f ; x) \assign \sum_{p \leqslant x} f (p) X_p$. By lemma 7.3,
  we have
  \begin{eqnarray*}
    \mathbbm{E} \left[ e^{s \Omega (f ; x)} \right] & = & L (f ; s) e^{\gamma
    ( \hat{\Psi} (f ; s) - 1)} \cdot (\log x)^{\hat{\Psi} (f ; s) - 1} + O
    \left( (\log x)^{\hat{\Psi} (f ; \kappa) - 3 / 2} \right)
  \end{eqnarray*}
  uniformly in $|s| \leqslant \varepsilon$, for any given $\varepsilon > 0$.
  Since $L (f ; 0) e^{\gamma ( \hat{\Psi} (f ; 0) - 1)} = 1 \neq 0$, and $L (f
  ; z) e^{\gamma ( \hat{\Psi} (f ; z) - 1)}$ is entire (by lemma 4.4 and 4.2),
  proposition 4.9 is applicable. Therefore
  \begin{eqnarray*}
    \mathbbm{P} \left( \frac{\Omega (f ; x) - \mu (f ; x)}{\sigma (f ; x)}
    \geqslant \Delta \right) & \sim & \frac{1}{\sqrt[]{2 \pi}}
    \int_{\Delta}^{\infty} e^{- u^2 / 2} \cdot \mathd u
  \end{eqnarray*}
  uniformly in $1 \leqslant \Delta \leqslant o ((\tmop{loglog} x)^{1 / 6})$.
  On the other hand, by part 1 of theorem 2.8 we know that $\mathcal{D}_f (x ;
  \Delta) \sim (1 / \sqrt[]{2 \pi}) \int_{\Delta}^{\infty} e^{- u^2 / 2} \cdot
  \mathd u$ for $1 \leqslant \Delta \leqslant o ((\tmop{loglog} x)^{1 / 6})$.
  The lemma follows.
\end{proof}

\subsection{Proof of Theorem 2.2}

\begin{proof}[Proof of Theorem 2.2]
  By lemma 7.2, lemma 7.4 and lemma 7.5, theorem 2.2 holds in all cases except
  when $\Psi (f ; t)$ is lattice distributed on $\alpha \mathbbm{Z}$ with
  $\alpha \neq 1$. So, suppose that $\Psi (f ; t)$ is lattice distributed on
  $\alpha \mathbbm{Z}$ $(\alpha \neq 1)$. Then $\Psi (f / \alpha ; t)$ is
  lattice distributed on $\mathbbm{Z}$. Hence, by our earlier work, uniformly
  in $1 \leqslant \Delta \leqslant c \sigma (f / \alpha ; x) = (c / \alpha)
  \sigma (f ; x)$,
  \begin{eqnarray}
    \mathcal{D}_{f / \alpha} (x ; \Delta) & \sim & \frac{e^{- \gamma (
    \hat{\Psi} (f / \alpha ; v_{\alpha}) - 1)}}{\Gamma ( \hat{\Psi} (f /
    \alpha ; v_{\alpha}))} \cdot \mathbbm{P} \left( \sum_{p \leqslant x}
    \frac{f (p)}{\alpha} \left[ X_p - \frac{1}{p} \right] \geqslant \Delta
    \sigma (f / \alpha ; x) \right) 
  \end{eqnarray}
  where $v_{\alpha} \assign v_{f / \alpha} (x ; \Delta)$. Note that
  $\mathcal{D}_{f / \alpha} (x ; \Delta) =\mathcal{D}_f (x ; \Delta)$ and that
  similarly the probability term in $(7.6)$ is invariant under multiplication
  by $\alpha$. It remains to show that $\hat{\Psi} (f / \alpha ; v_{\alpha}) =
  \hat{\Psi} (f ; v)$ where $v \assign v_f (x ; \Delta)$ but this follows from
  lemma 5.3. 
\end{proof}

\section{Proof of proposition 2.4}

As usual, we denote by $f (n ; y)$ a truncated additive function
\[ f (n ; y) = \sum_{\tmscript{\begin{array}{c}
     p|n\\
     p \leqslant y
   \end{array}}} f (p) \]

The following lemma is due to Barban and Vinogradov (see {\cite{21}}, lemma 3.2, p. 122). It improves the error term obtained by Kubilius in his theorem.

\begin{lemma}
  Let $f$ be a strongly additive function. Let $u = \log x / \log y$. Then, uniformly in $t \in \mathbbm{R}$,
  \begin{equation}
    \frac{1}{x} \cdot \# \left\{ n \leqslant x : f (n ; y) \geqslant
    t^{^{^{}}} \right\} =\mathbbm{P} \left( \sum_{p \leqslant y} f (p) X_p
    \geqslant t \right) + O \left( u^{- u / 8} \right)
  \end{equation}
\end{lemma}

We will use theorem 2.2 and theorem 2.8 to show that when $f \in \mathcal{C}$
and $u \asymp \log\log x$, $t \asymp \mu(f;x)$ the main term on the right 
hand side of $(8.1)$ is dominating.

\begin{proof}[Proof of Proposition 2.4]
  In lemma $8.1$ take $f \in \mathcal{C}$ and in $(8.1)$ choose $t \assign
  \xi_f (y ; \Delta)$ and $u \asymp \tmop{loglog} x$. In the range $1
  \leqslant \Delta \leqslant c \sigma (f ; y)$ we have,
  \begin{equation}
    \mathbbm{P} \left( \sum_{p \leqslant y} f (p) \left[ X_p - \frac{1}{p}
    \right] \geqslant \Delta \sigma (f ; y) \right) \geqslant \mathbbm{P}
    \left( \sum_{p \leqslant y} f (p) \left[ X_p - \frac{1}{p} \right]
    \geqslant c \sigma^2 (f ; y) \right)
  \end{equation}
  Let $w \assign v_f (x ; c \sigma (f ; y))$. First of all note that $w =
  \omega (f ; c) + o (1)$ by definition of $\omega (f ; z)$ and its
  analyticity. Therefore $\mathcal{P}_{\mathfrak{h}} (a ; w) \gg 1$ uniformly
  in $a \geqslant 0$ by lemma 4.25 and the remark right after
  the statement.  Also $L (f ; w) e^{- wc (f)}
  \neq 0$ because $L (f ; x) e^{- xc (f)}$ is never zero on the positive real
  line. It follows by theorem 2.2 and theorem 2.8 that the right hand side of
  $(8.2)$ is
  \[ \gg (\log x)^{\hat{\Psi} (f ; w) - 1 - w \hat{\Psi}' (f ; w)} \cdot
     (\tmop{loglog} x)^{- 1/2} \]
  Hence $(8.2)$ is dominating over the error term in $(8.1)$ when $t \assign
  \xi_f (y ; \Delta), u \asymp \tmop{loglog} x$ and $\Delta$ is allowed to
  vary throughout $1 \leqslant \Delta \leqslant c \sigma (f ; y)$. It follows
  that,
  \[ \frac{1}{x} \cdot \# \left\{ n \leqslant x : \frac{f (n ; y) - \mu (f ;
     y)}{\sigma (f ; y)} \geqslant \Delta \right\} \sim \mathbbm{P} \left(
     \sum_{p \leqslant y} f (p) \left[ X_p - \frac{1}{p} \right] \geqslant
     \Delta \sigma (f ; y) \right) \]
  uniformly in $1 \leqslant \Delta \leqslant c \sigma (f ; y)$ as desired. 
\end{proof}

\section{Integers to primes}

The goal of this section is to prove theorem 2.6. Throughout we will work with
\[ B^2 (f ; x) = \sum_{p \leqslant x} \frac{f (p)^2}{p} \]
rather than with $\sigma^2 (f ; x)$. Of course $B^2 (f ; x) = \sigma^2 (f ; x)
+ O (1)$ so there is little difference between the two. Let us also define
$\mathcal{D}^{\times}_f (x ; \Delta)$ by
\[ \mathcal{D}_f^{\times} (x ; \Delta) \assign \frac{1}{x} \cdot \# \left\{ n
   \leqslant x : \frac{f (n) - \mu (f ; x)}{B (f ; x)} \geqslant \Delta
   \right\} \]
This is simply $\mathcal{D}_f (x ; \Delta)$ with a different normalization.

\subsection{Large deviations for $\mathcal{D}^{\times}_f (x ; \Delta)$ and
$\mathbbm{P}(\mathcal{Z}_{\Psi} (x) \geqslant t)$}

We will usually need to ``adjust'' some of the results taken from the
literature. Our main tool will be Lagrange inversion.

\begin{lemma}
  Let $C > 0$ be given. Let $f (z)$ be analytic in $|z| \leqslant C$. Suppose
  that $f' (z) \neq 0$ for all $|z| \leqslant C$ and that in $|z| \leqslant C$
  the function $f (z)$ vanishes only at the point $z = 0$. Then, the function
  $g$ defined implicitly by $f (g (z)) = z$ is analytic in a neighborhood of
  0 and its n-th coefficient $a_n$ in the Taylor expansion
  about 0 is given by
  \[ a_n = \frac{1}{2 \pi i} \oint_{\gamma} \frac{\zeta f' (\zeta)}{f
     (\zeta)^{n + 1}} \mathd \zeta \]
  where $\gamma$ is a circle about $0$, contained in $| \zeta | \leqslant C$.
  The function $g (z)$ is given by
  \begin{eqnarray*}
    g (z) & = & \frac{1}{2 \pi i} \oint_{\gamma} \frac{\zeta f' (\zeta)}{f
    (\zeta) - z} \cdot \mathd \zeta
  \end{eqnarray*}
  and again $\gamma$ is a circle about $0$, contained in $| \zeta | \leqslant
  C$. 
\end{lemma}

The desired asymptotic for $\mathcal{D}_f^{\times} (x ; \Delta)$ is contained
in Maciulis's paper ({\cite{14}}, lemma 1A).

\begin{lemma}
  Let $f$ be an additive function. Suppose that $0 \leqslant f (p) \leqslant O
  (1)$ and that $B (f ; x) \longrightarrow \infty$ (or equivalently $\sigma (f
  ; x) \rightarrow \infty$). Let $B^2 = B^2 (f ; x)$. Uniformly in the range
  $1 \leqslant \Delta \leqslant o (\sigma (f ; x))$ we have
  \[ \mathcal{D}_f^{\times} (x ; \Delta) \sim \exp \left( - \frac{\Delta^3}{B}
     \sum_{k = 0}^{\infty} \frac{\lambda_f (x ; k + 2)}{k + 3} \cdot \left(
     \Delta / B \right)^k \right) \int_{\Delta}^{\infty} e^{- u^2 / 2} \cdot
     \frac{\mathd u}{\sqrt[]{2 \pi}} \]
  where the coefficients $\lambda_f (x ; k)$ are defined recursively by
  $\lambda_f (x ; 0) = 0, \lambda_f (x ; 1) = 1$ and
  \[ \lambda_f (x ; j) = - \sum_{i = 2}^j \frac{1}{i!} \cdot \left(
     \frac{1}{B^2 (f ; x)} \sum_{p \leqslant x} \frac{f (p)^{i + 1}}{p}
     \right) \sum_{k_1 + \ldots + k_i = j} \lambda_f (x ; k_1) \cdot \ldots
     \cdot \lambda_f (x ; k_i) \]
  Further there is a constant $C = C (f)$ such that $| \lambda_f (x ; k) |
  \leqslant C^k$ for all $k, x \geqslant 1$.
\end{lemma}

\begin{proof}
  Except the bound $| \lambda_f (x ; k) | \leqslant C^k$, the totality of the
  lemma is contained in Maciulis's paper ({\cite{14}}, lemma 1A). Let us prove
  that $| \lambda_f
  (x ; k) | \leqslant C^k$ for a suitable positive constant $C > 0$. To do so,
  we consider the power series
  \begin{eqnarray*}
    \mathcal{G}_f \left( x ; z \right) & = & \sum_{j \geqslant 2} \lambda_f
    \left( x ; j \right) \cdot z^j + z
  \end{eqnarray*}
  Let us look in more detail at the sum over $j \geqslant 2$. By making use of
  the recurrence relation for $\lambda_f (x ; j)$ we see that the sum in
  question equals to
  \begin{eqnarray*}
    & = & - \sum_{j \geqslant 2} \sum_{i = 2}^j \frac{1}{i!} \cdot \left(
    \frac{1}{B^2 \left( f ; x \right)} \sum_{p \leqslant x} \frac{f (p)^{i +
    1}}{p} \right) \sum_{k_1 + \ldots + k_i = j} \lambda_f \left( x ; k_1
    \right) \cdot \ldots \cdot \lambda_f \left( x ; k_i \right) \cdot z^j\\
    & = & - \sum_{i \geqslant 2} \frac{1}{i!} \cdot \left( \frac{1}{B^2
    \left( f ; x \right)} \sum_{p \leqslant x} \frac{f (p)^{i + 1}}{p} \right)
    \cdot \left( \sum_{k \geqslant 0} \lambda_f \left( x ; k \right) z^k
    \right)^i\\
    & = & - \frac{1}{B^2 (f ; x)} \sum_{p \leqslant x} \frac{f (p)}{p}
    \sum_{i \geqslant 2} \frac{1}{i!} \cdot f (p)^i \cdot \mathcal{G}_f \left(
    x ; z \right)^i\\
    & = & - \frac{1}{B^2 \left( f ; x \right)} \sum_{p \leqslant x} \frac{f
    (p)}{p} \cdot \left( e^{f (p)\mathcal{G}_f \left( x ; z) \right.} - f
    (p)\mathcal{G}_f (x ; z) - 1 \right)_{}
  \end{eqnarray*}
  The above calculation reveals that $\mathcal{F}_f (x ; \mathcal{G}_f (x ;
  z)) = z$ where $\mathcal{F}_f (x ; z)$ is defined by
  \begin{eqnarray*}
    \mathcal{F}_f (x ; z) & = & z + \frac{1}{B^2 (f ; x)} \sum_{p \leqslant x}
    \frac{f (p)}{p} \cdot \left( e^{f (p) z} - f (p) z - 1 \right)\\
    & = & \frac{1}{B^2 (f ; x)} \sum_{p \leqslant x} \frac{f (p)}{p} \cdot
    \left( e^{f (p) z} - 1 \right)
  \end{eqnarray*}
  Since all the $f (p)$ are bounded by some $M \geqslant 0$, we have
  $\mathcal{F}_f (x ; z) = z + O \left( Mz^2 \right)$ when $z$ is in a
  neighborhood of $0$, furthermore the implicit constant in the big $O$,
  depends only on $M$. Therefore $\mathcal{F}_f (x ; z) \gg 1$ for $z$ in the
  annulus $B / 2 \leqslant |z| \leqslant B$, where $B$ is a sufficiently small
  constant, depending only on $M$. Let us also note the derivative
  \[ \frac{\mathd}{\mathd z} \cdot \mathcal{F}_f (x ; z) = \frac{1}{B^2 (f ;
     x)} \sum_{p \leqslant x} \frac{f (p)^2}{p} \cdot e^{f (p) z} \]
  doesn't vanish and is bounded uniformly in $|z| \leqslant B$ for $B$
  sufficiently small, depending only on $M$. Hence by Lagrange inversion the
  function $\mathcal{G}_f (x ; z)$ is for each $x \geqslant 1$ analytic in the
  neighborhood $|z| \leqslant B$ of $0$, and in addition, its coefficients
  $\lambda_f (x ; k)$ are given by
  \begin{eqnarray*}
    \lambda_f (x ; k) & = & \frac{1}{2 \pi i} \oint_{| \zeta | = B / 2}
    \frac{\zeta \cdot (\mathd / \mathd \zeta)\mathcal{F}_f (x ;
    \zeta)}{\mathcal{F}_f (x ; \zeta)^{k + 1}} \mathd \zeta
  \end{eqnarray*}
  However we know that $\mathcal{F}_f (x ; \zeta) \gg 1$ and that $(d / d
  \zeta)\mathcal{G}_f (x ; \zeta) \ll 1$ on the boundary $| \zeta | = B / 2$
  with the implicit constant depending only on $M$. Therefore, the integral is
  bounded by $C^k$, for some $C > 0$ depending only on $M$. Hence $| \lambda_f
  (x ; k) | \leqslant C^k$.
\end{proof}

From Hwang's paper {\cite{11}} -- itself heavily based on the same methods as
used by Maciulis {\cite{14}} -- we obtain the next lemma. Since Hwang's lemma
is not exactly what is stated below, we include the deduction.

\begin{lemma}
  Let $\Psi$ be a distribution function. Suppose that there is an $\alpha > 0$
  such that $\Psi (\alpha) - \Psi (0) = 1$. Let $B^2 = B^2 (x) \rightarrow
  \infty$ be some function tending to infinity. Uniformly in the range $1
  \leqslant \Delta \leqslant o (B^{})$ we have
  \[ \mathbbm{P} \left( \mathcal{Z}_{\Psi} \left( B^2 \right) \geqslant \Delta
     B \right) \sim \exp \left( - \frac{\Delta^3}{B} \sum_{k = 0}^{\infty}
     \frac{\Lambda (\Psi ; k + 2)}{k + 3} \cdot \left( \Delta / B \right)^k
     \right) \int_{\Delta}^{\infty} e^{- u^2 / 2} \cdot \frac{\mathd
     u}{\sqrt[]{2 \pi}} \]
  the coefficients $\Lambda (\Psi ; k)$ satisfy $\Lambda (\Psi ; 0) = 0,
  \Lambda (\Psi ; 1) = 1$ and the recurrence relation
  \[ \Lambda (\Psi ; j) = - \sum_{2 \leqslant \ell \leqslant j} \frac{1}{\ell
     !} \int_{\mathbbm{R}} t^{\ell - 1} \mathd \Psi (f ; t) \sum_{k_1 + \ldots
     + k_{\ell} = j} \Lambda (\Psi ; k_1) \cdot \ldots \cdot \Lambda (\Psi ;
     k_{\ell}) \]
  Furthermore there is a constant $C = C (\Psi) > 0$ such that $| \Lambda
  (\Psi ; k) | \leqslant C^k$ for $k \geqslant 1$.
\end{lemma}

\begin{proof}
  Let $u (z) = u (\Psi ; z) = \int_{\mathbbm{R}} (e^{zt} - zt - 1) \cdot t^{-
  2} \mathd \Psi (t)$. Note that $u (z)$ is entire because $\Psi (t)$ is
  supported on a compact interval. By Hwang's theorem 1 (see {\cite{11}})
  \[ \mathbbm{P} \left( \mathcal{Z}_{\Psi} \left( B^2 \right) \geqslant \Delta
     B \right) \sim \exp \left( - B^2 \sum_{k \geqslant 0} \frac{\Lambda (\Psi
     ; k + 2)}{k + 3} \cdot (\Delta / B)^k \right) \int_{\Delta}^{\infty} e^{-
     u^2 / 2} \cdot \frac{\mathd u}{\sqrt[]{2 \pi}} \]
  uniformly in $1 \leqslant \Delta \leqslant o (B (x))$ with the coefficients
  $\Lambda (\Psi ; k)$ given by $\Lambda (\Psi ; 0) = 0$, $\Lambda (\Psi ; 1)
  = 1$ and for $k \geqslant 0$,
  \begin{eqnarray*}
    \frac{\Lambda (\Psi ; k + 2)}{k + 3} & = & - \frac{1}{k + 3} \cdot
    \frac{1}{2 \pi i} \oint_{\gamma} u'' (z) \cdot \left( \frac{u' (z)}{z}
    \right)^{- k - 3} \cdot \frac{\mathd z}{z^{k + 2}}\\
    & = & - \frac{1}{k + 3} \oint_{\gamma} \frac{zu'' (z)}{u' (z)^{k + 3}}
    \cdot \frac{\mathd z}{2 \pi i}
  \end{eqnarray*}
  
  %% Come back here

  (we set $m = k+3$, $k \geqslant 0$, $q_m = \Lambda(\Psi;k+2)/(k+3)$ in
  equation $(7)$ of {\cite{11}} and rewrite equation $(8)$ in {\cite{11}}
  in terms of Cauchy's formula).  
  Here $\gamma$ is a small circle around the origin. First let us show that
  the coefficients $\Lambda (\Psi ; k + 2)$ are bounded by $C^k$ for a
  sufficiently large (but fixed) $C > 0$. Around $z = 0$ we have $u' (z) =
  \int_{\mathbbm{R}} (e^{zt} - 1) \cdot t^{- 1} \mathd \Psi (t) = z + O
  (z^2)$. Therefore if we choose the circle $\gamma$ to have sufficiently
  small radius then $u' (z) \gg 1$ for $z$ on $\gamma$. Hence looking at the
  previous equation, the Cauchy integral defining $\Lambda (\Psi ; k + 2) /
  k + 3$ is bounded in modulus by $\ll C^{k + 3}$ for some constant $C > 0$.
  The bound $| \Lambda (\Psi ; k) | \ll C^k$ ensues (perhaps with a larger $C$
  than earlier). Our goal now is to show that $\Lambda (\Psi ; k)$ satisfies
  the recurrence relation given in the statement of the lemma. Multiplying by
  $\xi^{k + 3}$ and summing over $k \geqslant 0$ we obtain
  \begin{eqnarray*}
    &  & \sum_{k \geqslant 0} \frac{\Lambda (\Psi ; k + 2)}{k + 3} \cdot
    \xi^{k + 3} \text{ } = \text{ } - \sum_{k \geqslant 0} \frac{\xi^{k +
    3}}{k + 3} \oint_{\gamma} \frac{zu'' (z)}{u' (z)^{k + 3}} \cdot
    \frac{\mathd z}{2 \pi i}\\
    & = & - \oint_{\gamma} zu'' (z) \sum_{k \geqslant 0} \frac{1}{k + 3}
    \cdot \left( \frac{\xi}{u' (z)} \right)^{k + 3} \cdot \frac{\mathd z}{2
    \pi i}\\
    & = & \oint_{\gamma} zu'' (z) \cdot \left( - \frac{\xi}{u' (z)} -
    \frac{1}{2} \cdot \frac{\xi^2}{u' (z)^2} - \log \left( 1 - \frac{\xi}{u'
    (z)} \right) \right) \cdot \frac{\mathd z}{2 \pi i}
  \end{eqnarray*}
  Differentiating with respect to $\xi$ on both sides yields
  \begin{eqnarray*}
    \mathcal{G}_{\Psi} (\xi) & = & \sum_{k \geqslant 0} \Lambda (\Psi ; k + 2)
    \xi^{k + 2} \text{ } = \text{ } \oint \left( \frac{zu'' (z)}{u' (z) - \xi}
    - \frac{zu'' (z)}{u' (z)} - \frac{\xi zu'' (z)}{u' (z)^2} \right) 
    \frac{\mathd z}{2 \pi i}
  \end{eqnarray*}
  By Lagrange inversion this last integral is equal to $(u')^{- 1} (\xi) - 0 -
  \xi$ where $(u')^{- 1}$ denotes the inverse function to $u' (z)$. Hence
  \begin{eqnarray}
    \sum_{k \geqslant 0} \Lambda (\Psi ; k) \xi^k & = & (u')^{- 1} (\xi) 
  \end{eqnarray}
  Let us compose this with $u' (\cdot)$ on both sides and compute the
  resulting left hand side. First of all we expand $u' (z)$ in a power series.
  This gives
  \[ u' (z) = \int_{\mathbbm{R}} \frac{e^{zt} - 1}{t} \mathd \Psi (t) =
     \sum_{\ell \geqslant 1} \frac{1}{\ell !} \int_{\mathbbm{R}} t^{\ell - 1}
     \mathd \Psi (t) \cdot z^{\ell} \]
  Therefore, composing $(9.1)$ with $u' (\cdot)$ yields
  \begin{eqnarray*}
    \xi & = & u' \left( \sum_{k \geqslant 0} \Lambda (\Psi ; k) \xi^k \right)
    \text{ } = \text{ } \sum_{\ell \geqslant 1} \frac{1}{\ell !}
    \int_{\mathbbm{R}} t^{\ell - 1} \mathd \Psi (t) \cdot \left( \sum_{k
    \geqslant 0} \Lambda (\Psi ; k) \xi^k \right)^{\ell}\\
    & = & \sum_{\ell \geqslant 1} \frac{1}{\ell !} \int_{\mathbbm{R}} t^{\ell
    - 1} \mathd \Psi (t) \cdot \left( \sum_{k_1, \ldots, k_{\ell} \geqslant 1}
    \Lambda (\Psi ; k_1) \cdot \ldots \cdot \Lambda (\Psi ; k_{\ell}) \cdot
    \xi^{k_1 + \ldots + k_{\ell}} \right)\\
    & = & \sum_{m \geqslant 1} \left( \sum_{1 \leqslant \ell \leqslant m}
    \frac{1}{\ell !} \int_{\mathbbm{R}} t^{\ell - 1} \mathd \Psi (t) \sum_{k_1
    + \ldots + k_{\ell} = m} \Lambda (\Psi ; k_1) \cdot \ldots \cdot \Lambda
    (\Psi ; k_{\ell}) \right) \cdot \xi^m
  \end{eqnarray*}
  Thus the first coefficient $\Lambda (\Psi ; 1)$ is equal to 1, as desired,
  while for the terms $m \geqslant 2$ we have
  \[ \sum_{1 \leqslant \ell \leqslant m} \frac{1}{\ell !} \int_{\mathbbm{R}}
     t^{\ell - 1} \mathd \Psi (t) \sum_{k_1 + \ldots + k_{\ell} = m} \Lambda
     (\Psi ; k_1) \cdot \ldots \cdot \Lambda (\Psi ; k_{\ell}) = 0 \]
  The first term $\ell = 1$ is equal to $\Lambda (\Psi ; m)$. It suffice to
  move it on the right hand side of the equation, to obtain the desired
  recurrence relation. 
\end{proof}

Finally we will need one last result ``from the literature''. Namely a weak
form of the method of moments. For a proof we refer the reader to Gut's book
{\cite{6}}, p. 237. (Note that the next lemma follows from the result in
{\cite{6}} because in our case the random variables are positive, and bounded, 
in particular their distribution is determined uniquely by their moments).

\begin{lemma}
  Let $\Psi$ be a distribution function. Suppose that there is an $a > 0$ such
  that $\Psi (a) - \Psi (0) = 1$. Let $F (x ; t)$ be a sequence of
  distribution functions, one for each $x > 0$. If for each $k \geqslant 0$,
  \[ \int_{\mathbbm{R}} t^k \tmop{dF} (x ; t) \text{ } \longrightarrow
     \text{\, } \int_{\mathbbm{R}} t^k d \Psi (t) \]
  Then $F (x ; t) \longrightarrow \Psi (t)$ at all continuity points $t$ of
  $\Psi (t)$. 
\end{lemma}

\subsection{A transfer lemma}

The following lemma will allow us to transfer any results established with the
$B (f ; x)$ normalization to corresponding results with a $\sigma (f ; x)$
normalization.

\begin{lemma}
  Let $f$ be a strongly additive function such that $0 \leqslant f(p) 
  \leqslant O(1)$ and $B(f;x) \rightarrow \infty$. Let $\Psi$ be a distribution
  function. Suppose that
  there is an $a > 0$ such that $\Psi (a) - \Psi (0) = 1$. We have, uniformly
  in $1 \leqslant \Delta \leqslant o (\sigma (f ; x))$,
  \begin{eqnarray*}
    \mathcal{D}_f^{\times} \left( x ; \Delta \right) & \sim & \mathcal{D}_f
    \left( x ; \Delta \right)\\
    \mathbbm{P} \left( \mathcal{Z}_{\Psi} \left( \sigma^2 (f ; x) \right)
    \geqslant \Delta \sigma (f ; x) \right) & \sim & \mathbbm{P} \left(
    \mathcal{Z}_{\Psi} \left( B^2 (f ; x) \right) \geqslant \Delta B (f ; x)
    \right)
  \end{eqnarray*}
\end{lemma}

\begin{proof}
  Both results are consequences of lemma 9.3 and lemma 9.2 respectively. Let
  us first prove that $\mathcal{D}_f^{\times} \left( x ; \Delta \right) \sim
  \mathcal{D}_f (x ; \Delta)$ holds. Note that
  \begin{eqnarray*}
    \mathcal{D}_f^{\times} \left( x ; \Delta \cdot \sigma (f ; x) / B (f ; x)
    \right) & = & \mathcal{D}_f (x ; \Delta)
  \end{eqnarray*}
  Therefore using the asymptotic of Lemma 9.2 we conclude that
  \begin{equation}
    \mathcal{D}_f (x ; \Delta) \sim \exp \left( - \frac{\sigma^3}{B^3} \cdot
    \frac{\Delta^3}{B^{}} \sum_{k \geqslant 0} \frac{\lambda_f (x ; k + 2)}{k
    + 3} \cdot \left( \Delta \sigma / B^2 \right)^k \right)
    \int_{\Delta}^{\infty} e^{- u^2 / 2} \cdot \frac{\mathd u}{\sqrt[]{2 \pi}}
  \end{equation}
  where we used the abbreviation $B \assign B (f ; x)$ and $\sigma \assign
  \sigma (f ; x)$. Further the coefficients $\lambda_f (x ; k + 2)$ are
  defined in Lemma 9.2, and satisfy $| \lambda_f (x ; k) | \ll C^k$ for some
  fixed $C > 0$ depending only on $f$. Because of that bound on $\lambda_f (x
  ; k)$, the function
  \begin{eqnarray*}
    \mathcal{G}_f (x ; z) & = & \sum_{k \geqslant 0} \frac{\lambda_f (x ; k +
    2)}{k + 3} \cdot z^k
  \end{eqnarray*}
  is analytic in $|z| < 1 / C$ for all $x > 0$. Therefore
  \begin{eqnarray}
    \mathcal{G}_f (x ; \Delta / B^{} \cdot \sigma / B) & = & \mathcal{G}_f
    \left( x ; \Delta / B) + \left( \Delta / B \cdot (\sigma / B - 1)) \cdot
    \mathcal{G}_f' (x ; \xi \right)  \right.
  \end{eqnarray}
  for some $\Delta / B \cdot \sigma / B \leqslant \xi \leqslant \Delta / B$.
  Derivatives are always taken with respect to the second argument -- that is
  $\mathcal{G}_f' (x ; \xi) \assign (\mathd / \mathd \xi)\mathcal{G}_f (x ;
  \xi)$. Upon using the inequality $\lambda_f (x ; k) \ll C^k$ we find the
  bound $\mathcal{G}'_f (x ; \xi) \ll (1 - C \Delta / B)^{- 1}$ which is 
  $O(1)$ because $\Delta \leqslant o (B (f ; x))$. Further $\Delta / B (1 -
  \sigma / B) \ll \Delta / B^3$ because $\sigma^2 = B^2 + O (1)$ hence $\sigma
  / B = 1 + O (B^{- 2})$. It now follows from $(9.3)$ that $\mathcal{G}_f (x ;
  \Delta / B \cdot \sigma / B) =\mathcal{G}_f (\Delta / B) + O \left( \Delta /
  B^3 \right)$. Note also that $\mathcal{G}_f(x;\Delta/B) \ll 1$ 
  uniformly in $1 \leqslant \Delta \leqslant o(B)$. Using these two
  estimates, we find that
  \begin{eqnarray*}
    &  & - (\sigma / B)^3 \cdot (\Delta^3 / B) \cdot \mathcal{G}_f (x ;
    \Delta / B \cdot \sigma / B)\\
    & = & \left. - (\sigma / B)^3 \cdot (\Delta^3 / B) \cdot (\mathcal{G}_f
    (x ; \Delta / B) + O \left( \Delta / B^3 \right) \right)\\
    & = & - \left( 1 + O \left( 1 / B^2 \right) \right) \cdot \Delta^3 / B
    \cdot \mathcal{G}_f (x ; \Delta / B) + O \left( (\Delta / B)^4 \right)\\
    & = & - \Delta^3 / B \cdot \mathcal{G}_f (x ; \Delta / B) + O \left(
    (\Delta / B)^3 + (\Delta / B)^4 \right)
  \end{eqnarray*}
  Since $\Delta \leqslant o (\sigma (f ; x))$ the error term is $o (1)$. The
  previous equation, together with $(9.2)$ leads to
  \begin{eqnarray*}
    \mathcal{D}_f (x ; \Delta) & \sim & \exp \left( - (\sigma / B)^3 \cdot
    (\Delta^3 / B) \cdot \mathcal{G}_f (x ; \Delta / B \cdot \sigma / B)
    \right) \cdot (1 - \Phi (\Delta))\\
    & \sim & \exp \left( - (\Delta^3 / B) \cdot \mathcal{G}_f (x ; \Delta /
    B)) \cdot (1 - \Phi (\Delta)) \text{ } \sim \text{ }
    \mathcal{D}_f^{\times} \left( x ; \Delta \right) \right.
  \end{eqnarray*}
  uniformly in $1 \leqslant \Delta \leqslant o (\sigma (f ; x))$ and where $1
  - \Phi (\Delta) \assign \int_{\Delta}^{\infty} e^{- u^2 / 2} \cdot \mathd
  u$. The above equation establishes the first part of the lemma. The proof of
  the second part is along similar lines, but easier. Denote by
  \[ \mathcal{G}_{\Psi} \left( z \right) \assign \sum_{k \geqslant 0}
     \frac{\Lambda (\Psi ; k + 2)}{k + 3} \cdot z^k \]
  with the coefficients $\Lambda (\Psi ; k)$ defined as in Lemma 9.3. The
  coefficients $\Lambda (\Psi ; k)$ are bounded by $C^k$, for some suitable $C
  > 0$, therefore $\mathcal{G}_{\Psi} (z)$ is analytic in $|z| < 1 / C$. Hence
  $\mathcal{G}_{\Psi} \left( \Delta / B \right) =\mathcal{G}_{\Psi} \left(
  \Delta / \sigma \right) + O \left( \Delta / B - \Delta / \sigma \right)
  =\mathcal{G}_{\Psi} \left( \Delta / \sigma \right) + O \left( \Delta /
  \sigma^3 \right)$, where in the error term we used the estimate $B = \sigma
  + O (1 / \sigma)$. Therefore
  \begin{eqnarray*}
    - (\Delta^3 / B) \cdot \mathcal{G}_{\Psi} \left( \Delta / B \right) & = &
    - (\Delta^3 / B) \cdot \left( \mathcal{G}_{\Psi} \left( \Delta / \sigma
    \right) + O \left( \Delta / \sigma^3 \right) \right)\\
    & = & - (\sigma / B) \cdot (\Delta^3 / \sigma) \cdot \mathcal{G}_{\Psi}
    \left( \Delta / \sigma \right) + O \left( (\Delta / \sigma)^4 \right)\\
    & = & - \left( 1 + O (1 / \sigma^2) \right) \cdot (\Delta^3 / \sigma)
    \cdot \mathcal{G}_{\Psi} \left( \Delta / \sigma \right) + O \left( (\Delta
    / \sigma)^4 \right)\\
    & = & - (\Delta^3 / \sigma) \cdot \mathcal{G}_{\Psi} (\Delta / \sigma) +
    O \left( (\Delta / \sigma)^4 + (\Delta / \sigma)^3 \right)
  \end{eqnarray*}
  and the error term is $o (1)$ because $\Delta \leqslant o (\sigma (f ; x))$.
  Therefore, using lemma 9.3 we conclude that
  \begin{eqnarray*}
    \mathbbm{P} \left( \mathcal{Z}_{\Psi} \left( \sigma) \geqslant \Delta
    \sigma) \right. \right. & \sim & \exp \left( - (\Delta^3 / \sigma) \cdot
    \mathcal{G}_{\Psi} (\Delta / \sigma) \right) \cdot (1 - \Phi (\Delta))\\
    & \sim & \exp \left( - (\Delta^3 / B) \cdot \mathcal{G}_{\Psi} \left(
    \Delta / B \right) \right) \cdot (1 - \Phi (\Delta)) \text{ } \sim \text{
    } \mathbbm{P} \left( \mathcal{Z}_{\Psi} \left( B \right) \geqslant \Delta
    B \right)
  \end{eqnarray*}
  uniformly in $1 \leqslant \Delta \leqslant o (\sigma (f ; x))$ and where $1
  - \Phi (\Delta) = \int_{\Delta}^{\infty} e^{- u^2 / 2} \cdot \mathd u /
  \sqrt[]{2 \pi}$. The above equation establishes the second part of the
  lemma. 
\end{proof}

\subsection{Proof of the ``integers to primes'' theorem}

\begin{proof}[Proof of Theorem 2.6]
  By assumptions $\mathcal{D}_f (x ; \Delta) \sim
  \mathbbm{P}(\mathcal{Z}_{\Psi} (\sigma^2 (f ; x)) \geqslant \Delta \sigma (f
  ; x))$ holds throughout $1 \leqslant \Delta \leqslant o (\sigma (f ; x))$.
  Thus, by lemma 9.5, $\mathcal{D}_f^{\times} (x ; \Delta) \sim
  \mathbbm{P}(\mathcal{Z}_{\Psi} (B^2 (f ; x)) \geqslant \Delta B (f ; x))$
  uniformly in $1 \leqslant \Delta \leqslant o (B (f ; x))$. We are going to
  work with this last condition.
  
  The proof is in three steps. Retaining the notation of Lemma 9.3 and Lemma
  9.2 we first show that $\lambda_f (k ; x) \longrightarrow \Lambda (\Psi ;
  k)$ for all $k \geqslant 2$ (for $k = 1$ this is trivial). Then, we deduce
  from there that
  \begin{equation}
    \frac{1}{B^2 (f ; x)} \sum_{\tmscript{\begin{array}{c}
      p \leqslant x\\
      f (p) \leqslant t
    \end{array}}} \frac{f (p)^{k + 2}}{p} \text{ } \longrightarrow \text{ }
    \int_{\mathbbm{R}} t^k \mathd \Psi (t)
  \end{equation}
  Finally by the method of moments (and an elementary manipulation)
  \begin{eqnarray}
    \frac{1}{\sigma^2 (f ; x)} \sum_{\tmscript{\begin{array}{c}
      p \leqslant x\\
      f (p) \leqslant t
    \end{array}}} \frac{f (p)^2}{p} \cdot \left( 1 - \frac{1}{p} \right) &
    \longrightarrow & \Psi (t) 
  \end{eqnarray}
  The last step being the easy one. To prove our first step we will proceed by
  induction on $k \geqslant 0$. We will prove the stronger claim that
  \[ \lambda_f (x ; k + 2) = \Lambda (\Psi ; k + 2) + O_k \left( B^{- 2^{- (k
     + 1)}} \right) \]
  where we write $B = B (f ; x)$ to simplify notation. By Lemma 9.2 and 9.3,
  our assumption $\mathcal{D}_f^{\times} (x ; \Delta) \sim
  \mathbbm{P}(\mathcal{Z}_{\Psi} (B^2 (f ; x)) \geqslant \Delta B (f ; x))$
  (for $1 \leqslant \Delta \leqslant o (B (f ; x)))$ reduces to
  \begin{equation}
    - \frac{\Delta^3}{B} \sum_{m \geqslant 0} \frac{\lambda_f \left( x ; m +
    2) \right.}{m + 3} \cdot \left( \Delta / B \right)^m = -
    \frac{\Delta^3}{B} \sum_{m \geqslant 0} \frac{\Lambda (\Psi ; m + 2)}{m +
    3} \cdot \left( \Delta / B \right)^m + o (1)
  \end{equation}
  valid throughout the range $1 \leqslant \Delta \leqslant o (\sigma (f ;
  x))$. Let us first establish the base case $\lambda_f (x ; 2) = \Lambda
  (\Psi ; 2) + O (B^{- 1 / 2})$. In $(9.6)$ we choose $\Delta = B^{1 / 2}$.
  Because of the bounds $| \lambda_f (x ; m) | \leqslant C^m$ and $| \Lambda
  (\Psi ; m) | \leqslant C^m$ (see lemma 9.2 and 9.3) the terms 
  $m \geqslant 1$ contribute $O (1)$.
  The $m = 0$ term is $\asymp B^{1 / 2}$. It follows that $\lambda_f (x ;
  2) = \Lambda (\Psi ; 2) + O (B^{- 1 / 2})$ and so the base case follows. Let
  us now suppose that for all $\ell < k$, $(k \geqslant 1)$
  \[ \lambda_f \left( x ; \ell + 2 \right) = \Lambda \left( \Psi ; \ell + 2
     \right) + O_{\ell} \left( B^{- 2^{- (\ell + 1)}} \right) \]
  Note that we can assume (in the above equation) that the implicit constant
  depends on $k$, by taking the max of the implicit constants in $O_{\ell}(
  B^{-2^{-(\ell+1)}})$ for $\ell < k$. 
  In equation $(9.6)$ let's choose $\Delta = B^{1 - 2^{-
  (k + 1)}}$. With this choice of $\Delta$ the terms that are $\geqslant k +
  1$ in $(9.6)$ contribute at most
  \[ \left( \Delta^3 / B \right) \cdot \left( C \cdot \Delta / B \right)^{k +
     1} \ll_k B^2 \cdot B^{- 3 \cdot 2^{- (k + 1)}} \cdot B^{- (k + 1) \cdot
     2^{- (k + 1)}} \]
  on both sides of $(9.6)$. On the other hand, we see (by using the induction
  hypothesis) that the terms $m \leqslant k - 1$ on the left and the right
  hand side of $(9.6)$ differ by no more that
  \begin{eqnarray*}
    &  & - \frac{\Delta^3}{B} \cdot \left( \sum_{\ell < k} \frac{(\Delta /
    B)^{\ell}}{\ell + 3} \cdot \left( \lambda_f \left( x ; \ell + 2 \right) -
    \Lambda \left( \Psi ; \ell + 2 \right) \right) \right)\\
    & = & O_k \left( B^2 \cdot B^{- 3 \cdot 2^{- (k + 1)}} \cdot \left(
    \sum_{\ell < k} \frac{B^{- \ell \cdot 2^{- (k + 1)}}}{\ell + 3} \cdot B^{-
    2^{- (\ell + 1)}} \right) \right)\\
    & = & O_k \left( B^2 \cdot B^{- 3 \cdot 2^{- (k + 1)}} \cdot \sum_{\ell <
    k} \frac{1}{\ell + 3} \cdot B^{- 2^{- (k + 1)} \cdot \left( \ell + 2^{k -
    \ell} \right)} \right)
  \end{eqnarray*}
  Note that for each integer $\ell < k$ we have $\ell + 2^{k - \ell} \geqslant
  k + 1$. Therefore the above error term is bounded by $O_k (B^2 \cdot B^{- 3
  \cdot 2^{- k}} \cdot B^{- (k + 1) 2^{- (k + 1)}})$. With these two
  observations at hand, relation $(9.6)$ reduces to
  \[ - \frac{\Delta^3}{B} \cdot \frac{\left( \Delta / B \right)^k}{k + 3}
     \left[ \lambda_f \left( x ; k \upl 2 \right) \um \Lambda \left( \Psi ; k
     \upl 2 \right)^{^{}} \right] = O_k \left( B^{2 - 3 \cdot 2^{- \left( k
     \upl 1) \right.}} \cdot B^{- \left( k \upl 1 \right) \cdot 2^{- \left( k
     \upl 1) \right.}} \right) \upl o (1) \]
  where $\Delta = B^{1 - 2^{- (k + 1)}}$. Dividing by $\Delta^3 / B \cdot
  \left( \Delta / B \right)^k \asymp B^{2 - 3 \cdot 2^{- (k + 1)}} \cdot B^{-
  k \cdot 2^{- (k + 1)}}$ on both sides, we conclude that $\lambda_f (x ; k +
  2) - \Lambda (\Psi ; k + 2) = O_k (B^{- 2^{- (k + 1)}})$ as desired, thus
  finishing the inductive step. Now, we will prove that $\lambda_f (x ; k)
  \longrightarrow \Lambda (\Psi ; k)$ implies
  \begin{equation}
    \mathcal{M}_f (x ; \ell) \text{ } \assign \text{ } \frac{1}{B^2 (f ; x)}
    \sum_{p \leqslant x} \frac{f (p)^{\ell + 2}}{p} \longrightarrow
    \int_{\mathbbm{R}} t^{\ell} \mathd \Psi (t)
  \end{equation}
  for each fixed $\ell \geqslant 0$. This follows almost immediately from the
  recurrence relation for $\lambda_f (x ; k)$ and $\Lambda (\Psi ; k)$. Indeed
  let us prove $(9.7)$ by induction on $k \geqslant 0$. The base case $k = 0$
  is obvious, for the left hand side and right hand side of $(9.7)$ are both
  equal to $1$. Let us now suppose that $(9.7)$ holds for all $\ell < k$. We
  will prove that convergence also holds for $\ell = k$. \ By definition of
  $\lambda_f (x ; k + 1)$ we have
  \begin{equation}
    \lambda_f (x ; k \upl 1) = \um \sum_{j = 2}^k \frac{\mathcal{M}_f (x ; j
    \um 1)}{j!} \sum_{\ell_1 + \ldots + \ell_j = k + 1} \lambda_f (x ; \ell_1)
    \ldots \lambda_f (x ; \ell_j) \um \frac{\mathcal{M}_f (x ; k)}{(k + 1) !}
  \end{equation}
  (we single out $j = k + 1$ on the right hand side). By induction hypothesis
  $\mathcal{M}_f (x ; j \um 1) \longrightarrow \int t^{j - 1} \mathd 
  \Psi (t)$
  as $x \rightarrow \infty$, for $j \leqslant k$. Further as we've shown
  earlier $\lambda_f (x ; i) \longrightarrow \Lambda (\Psi ; i)$ for all $i
  \geqslant 0$. Therefore the whole double sum on the right hand side of
  $(9.8)$ tends to
  \[ - \sum_{j = 2}^k \frac{1}{j!} \int_{\mathbbm{R}} t^{j - 1} \mathd \Psi (t)
  \sum_{\ell_1 + \ldots + \ell_j = k + 1} \Lambda (\Psi ; \ell_1)
     \cdot \ldots \cdot \Lambda (\Psi ; \ell_j) \]
  which, by definition of $\Lambda (\Psi ; k)$ is equal to $\Lambda (\Psi ; k
  + 1) + 1 / (k + 1) ! \int_{\mathbbm{R}} t^k \mathd \Psi (t)$. But also
  $\lambda_f (x ; k + 1) \longrightarrow \Lambda (\Psi ; k + 1)$ because
  $\lambda_f (x ; i) \longrightarrow \Lambda (\Psi ; i)$ for all $i \geqslant
  0$. Thus the left hand side of $(9.8)$ tends to $\Lambda (\Psi ; k + 1)$
  while the double sum on the right hand side of $(9.8)$ tends to $\Lambda
  (\Psi ; k + 1) + 1 / (k + 1) ! \int_{\mathbbm{R}} t^k \mathd \Psi (t)$.
  Therefore equation $(9.8)$ transforms into
  \[ \Lambda (\Psi ; k + 1) = \Lambda (\Psi ; k + 1) + \frac{1}{(k + 1) !}
     \int_{\mathbbm{R}} t^k \mathd \Psi (t) - \frac{\mathcal{M}_f (x ;
     k)}{(k + 1) !} + o_{x \rightarrow \infty} (1) \]
  and $\mathcal{M}_f (x ; k) \rightarrow \int_{\mathbbm{R}} t^k \mathd \Psi(t)$
  $(x \rightarrow \infty)$ follows. This establishes the induction step
  and thus $(9.7)$ for all fixed $\ell \geqslant 0$. Now we use the method of
  moments to prove that
  \[ F (x ; t) \assign \text{ } \frac{1}{B^2 (f ; x)}
     \sum_{\tmscript{\begin{array}{c}
       p \leqslant x\\
       f (p) \leqslant t
     \end{array}}} \frac{f (p)^2}{p} \text{ } \underset{x \rightarrow
     \infty}{\longrightarrow} \text{ } \Psi (t) \]
  holds at all continuity points $t$ of $\Psi (t)$. Let us note that, the
  $k$-th moment of the distribution function $F (x ; t)$, is given by
  $\mathcal{M}_f (x ; k)$, and as we've just shown this converges to the the
  $k$-th moment of $\Psi (t)$. That is
  \[ \int_{\mathbbm{R}} t^k \mathd F (x ; t) = \frac{1}{B^2 (f ; x)}  \sum_{p
     \leqslant x} \frac{f (p)^{k + 2}}{p} \text{ } \longrightarrow \text{ }
     \int_{\mathbbm{R}} t^k \mathd \Psi (t) \]
  Since $\Psi (a) - \Psi (0) = 1$ for some $a > 0$, the distribution function
  $\Psi$ satisfies the assumption of Lemma 9.4, hence, by Lemma 9.4 (the
  method of moments) we have $F (x ; t) \longrightarrow \Psi (t)$ at all
  continuity points $t$ of $\Psi$. Finally, we deduce that
  \begin{equation}
    \frac{1}{\sigma^2 (f ; x)} \sum_{\tmscript{\begin{array}{c}
      p \leqslant x\\
      f (p) \leqslant t
    \end{array}}} \frac{f (p)^2}{p} \cdot \left( 1 - \frac{1}{p} \right)
    \text{ } \longrightarrow \text{ } \Psi (t)
  \end{equation}
  at all continuity points of $\Psi$. This is almost trivial, because $B^2 (f
  ; x)$ and $\sigma^2 (f ; x)$ differ only by an $O (1)$, and the sum $\sum f
  (p)^2 / p^2 = O (1)$ because the $f (p)$ are $O (1)$. Hence, by a simple
  computation
  \begin{eqnarray*}
    \frac{1}{\sigma^2 (f ; x)} \sum_{\tmscript{\begin{array}{c}
      p \leqslant x\\
      f (p) \leqslant t
    \end{array}}} \frac{f (p)^2}{p} \cdot \left( 1 - \frac{1}{p} \right) & = &
    F (x ; t) \text{ } + \text{ } O \left( \frac{1}{B^2 (f ; x)} \right)
  \end{eqnarray*}
  And since $F (x ; t) \longrightarrow \Psi (t)$ at all continuity points of
  $\Psi$, it follows that $(9.9)$ must be true.
\end{proof}

\section{Primes to integers}

We keep the same notation as in the previous section. Namely we let
\begin{eqnarray*}
  B^2 (f ; x) \assign \sum_{p \leqslant x} \frac{f (p)^2}{p} & \tmop{and} &
  \mathcal{D}_f^{\times} \left( x ; \Delta \right) \assign \frac{1}{x} \cdot
  \# \left\{ n \leqslant x : \frac{f (n) - \mu (f ; x)}{B (f ; x)} \geqslant
  \Delta \right\}
\end{eqnarray*}
We first need to modify a little some of the known large deviations results
for $\mathcal{D}_f^{\times} (x ; \Delta)$ and $\mathbbm{P}(\mathcal{Z}_{\Psi}
(B^2 (f ; x)) \geqslant \Delta B (f ; x))$.

\subsection{Large deviations for $\mathcal{D}^{\times}_f (x ; \Delta)$ and
$\mathbbm{P}(\mathcal{Z}_{\Psi} (x) \geqslant t)$ revisited}

First we require the result of Maciulis ({\cite{14}}, theorem) in a 
``saddle-point'' version.

\begin{lemma}
  Let $g$ be a strongly additive function such that $0 \leqslant g (p)
  \leqslant O (1)$ and $B(g;x) \rightarrow \infty$. 
  We have uniformly in $1 \leqslant \Delta \leqslant o (B (g
  ; x))$,
  \[ \mathcal{D}_g^{\times} (x ; \Delta) \sim \exp \left( \sum_{p \leqslant x}
     \frac{e^{\eta g (p)} \um \eta g (p) \um 1}{p} - \eta \sum_{p \leqslant x}
     \frac{g (p) (e^{\eta g (p)} \um 1)}{p} \right) \frac{e^{\Delta^2 /
     2}}{\sqrt[]{2 \pi}} \int_{\Delta}^{\infty} e^{- t^2 / 2} \mathd t \]
  where $\eta = \eta_g (x ; \Delta)$ is defined as the unique positive
  solution of the equation
  \[ \sum_{p \leqslant x} \frac{g (p) e^{\eta g (p)}}{p} = \mu (g ; x) +
     \Delta B (g ; x) \]
  Furthermore $\eta_g (x ; \Delta) = \Delta / B (g ; x) + O (\Delta^2 / B (g ;
  x)^2)$.
\end{lemma}

\begin{proof}
  Only the last assertion needs to be proved, because it is not stated
  explicitly in Maciulius's paper. Fortunately enough, it's a triviality.
  Indeed, writing $\eta = \eta_g (x ; \Delta)$, we find that
  \[ 0 \leqslant \eta_g (x ; \Delta) \sum_{p \leqslant x} \frac{g (p)^2}{p}
     \leqslant \sum_{p \leqslant x} \frac{g (p) (e^{\eta g (p)} - 1)}{p} =
     \Delta B (g ; x) \]
  Dividing by $B^2 (g ; x)$ on both sides $0 \leqslant \eta_g (x ; \Delta)
  \leqslant \Delta / B (g ; x)$ follows. Now expanding $e^{\eta g (p)} - 1 =
  \eta g (p) + O (\eta^2 g (p)^2)$ and noting that $O (\eta^2 g (p)^2) = O
  (\eta^2 g (p))$ because $g (p) = O (1)$, we find that
  \[ \Delta B (g ; x) = \sum_{p \leqslant x} \frac{g (p) (e^{\eta g (p)} -
     1)}{p} = \eta \sum_{p \leqslant x} \frac{g (p)^2}{p} + O \left( \eta^2
     \sum_{p \leqslant x} \frac{g (p)^2}{p} \right) \]
  Again dividing by $B^2 (g ; x)$ on both sides, and using the bound $\eta = O
  (\Delta / B (g ; x))$ the claim follows.
\end{proof}

Adapting Hwang's {\cite{11}} result we prove the following.

\begin{lemma}
  Let $\Psi$ be a distribution function. Suppose that there is an $\alpha > 0$
  such that $\Psi (\alpha) - \Psi (0) = 1$. Let $B^2 = B^2 (x) \rightarrow
  \infty$ be some function tending to infinity. Then, uniformly in $1
  \leqslant \Delta \leqslant o (B (x))$ the quantity $\mathbbm{P} \left(
  \mathcal{Z}_{\Psi} (B^2) \geqslant \Delta B \right)$ is asymptotic to
  \[ \exp \left( B^2 \int_{\mathbbm{R}} \frac{e^{\rho u} - \rho u - 1}{u^2}
     \mathd \Psi (u) - B^2 \cdot \rho \int_{\mathbbm{R}} \frac{e^{\rho u} -
     1}{u} \mathd \Psi (u) \right) \cdot \frac{e^{\Delta^2 / 2}}{\sqrt[]{2
     \pi}} \int_{\Delta}^{\infty} e^{- t^2 / 2} \tmop{dt} \]
  where $\rho = \rho_{\Psi} (B (x) ; \Delta)$ is defined implicitly, as the
  unique positive solution to
  \[ B^2 (x) \int_{\mathbbm{R}} \frac{e^{\rho u} - 1}{u} \mathd \Psi (u) =
     \Delta \cdot B (x) \]
\end{lemma}

\begin{proof}
  We keep the same notation as in lemma 9.3. While proving lemma 9.3 we
  established the following useful relationship (see $(9.1)$)
  \[ \sum_{k \geqslant 0} \Lambda (\Psi ; k + 2) \xi^{k + 2} = (u')^{- 1}
     (\xi) - \xi \text{ where } u (z) = \int_{\mathbbm{R}} \frac{e^{zt} - zt -
     1}{t^2} \mathd \Psi (t) \]
  Here $(u')^{- 1}$ denotes the inverse function of $u'$. \ Integrating the
  above gives
  \[ \sum_{k \geqslant 0} \frac{\Lambda (\Psi ; k + 2)}{k + 3} \cdot \xi^{k +
     3} = - \frac{\xi^2}{2} + \xi \cdot (u')^{- 1} (\xi) - u ((u')^{- 1}
     (\xi)) \]
  Now choose $\xi = \Delta / B$, then by definition $\rho = \rho_{\Psi} (B (x)
  ; \Delta) = (u')^{- 1} (\xi)$. Thus the above formula becomes
  \[ \sum_{k \geqslant 0} \frac{\Lambda (\Psi ; k + 2)}{k + 3} \cdot (\Delta /
     B)^{k + 3} = - \frac{(\Delta / B)^2}{2} + \frac{\Delta}{B} \cdot \rho -
     \int_{\mathbbm{R}} \frac{e^{\rho t} - \rho t - 1}{t^2} \mathd \Psi (t) \]
  Also, note that by definition $\Delta / B = \int_{\mathbbm{R}} (e^{\rho t} -
  1) / t \cdot \mathd \Psi (t)$. Using the above formula (in which we replace
  $(\Delta / B) \cdot \rho$ by $\rho \int_{\mathbbm{R}} (e^{\rho t} - 1) / t
  \cdot \mathd \Psi (t)$) and lemma 9.3
  \begin{eqnarray*}
    &  & \mathbbm{P}(\mathcal{Z}_{\Psi} (B^2) \geqslant \Delta B) \sim \exp
    \left( - B^2 \sum_{k \geqslant 0} \frac{\Lambda (\Psi ; k + 2)}{k + 3}
    \cdot \left( \Delta / B \right)^{k + 3} \right) \int_{\Delta}^{\infty}
    e^{- u^2 / 2} \cdot \frac{\mathd u}{\sqrt[]{2 \pi}}\\
    & = & \exp \left( B^2 \int_{\mathbbm{R}} \frac{e^{\rho t} - \rho t -
    1}{t^2} \mathd \Psi (t) - B^2 \cdot \rho \int_{\mathbbm{R}} \frac{e^{\rho
    t} - 1}{t} \mathd \Psi (t) \right) \cdot \frac{e^{\Delta^2 / 2}}{\sqrt[]{2
    \pi}} \int_{\Delta}^{\infty} e^{- u^2 / 2} \mathd u
  \end{eqnarray*}
  This is the claim.
\end{proof}

Before we prove Theorem 2.5, we need to show that the parameters $\eta_g (x ;
\Delta)$ and $\rho_{\Psi} (B (f ; x) ; \Delta)$ (as defined respectively in
Lemma 10.1 and Lemma 10.2) are ``close'' when the distribution of the $g (p)$'s
resembles $\Psi (t)$. The ``closeness'' assertion is made precise in the next
lemma.

\begin{lemma}
  Let $\Psi (\cdot)$ be a distribution function. Let $f$ be a positive
  strongly additive function. Suppose that $0 \leqslant f (p) \leqslant O (1)$
  for all primes $p$. Let
  \[ K_f (x ; t) \assign \frac{1}{B^2 (f ; x)}
     \sum_{\tmscript{\begin{array}{c}
       p \leqslant x\\
       f (p) \leqslant t
     \end{array}}} \frac{f (p)^2}{p} \]
  If $K_f (x ; t) - \Psi (t) \ll 1 / B^2 (f ; x)$ uniformly in $t \in
  \mathbbm{R}$, then
  \[ \rho_{\Psi} (B (f ; x) ; \Delta) - \eta_f (x ; \Delta) = o (1 / B^2 (f ;
     x)) \]
  uniformly in $1 \leqslant \Delta \leqslant o (B (f ; x))$. The symbols
  $\rho_{\Psi} (B (x) ; \Delta)$ and $\eta_f (x ; \Delta)$ are defined in
  lemma 10.2 and lemma 10.1 respectively. 
\end{lemma}

\begin{proof}
  Let $\eta = \eta_f (x ; \Delta)$. Recall that by lemma 10.1, 
  $\eta = o (1)$ in the range $1 \leqslant \Delta \leqslant o(B(f;x))$. 
  This will
  justify the numerous Taylor expansions involving the parameter $\eta$. With
  $K_f (x ; t)$ defined as in the statement of the lemma, we have
  \begin{eqnarray}
    \sum_{p \leqslant x} \frac{f (p) e^{\eta f (p)}}{p} - \mu (f ; x) & = &
    \sum_{p \leqslant x} \frac{f (p) (e^{\eta f (p)} - 1)}{p}_{} \nonumber\\
    & = & B^2 (f ; x) \int_{\mathbbm{R}} \frac{e^{\eta t} - 1}{t} \mathd K_f
    (x ; t) 
  \end{eqnarray}
  Let $M > 0$ be a real number such that $0 \leqslant f (p) \leqslant M$ for
  all $p$. Since the $f (p)$ are bounded, for each $x > 0$ the distribution
  function $K_f (x ; t)$ is supported on $[0 ; M]$. Furthermore since $K_f (x
  ; t) \rightarrow \Psi (t)$ the distribution function $\Psi (t)$ is supported
  on exactly the same interval. From these considerations, it follows that
  \begin{eqnarray}
    &  & \int_{\mathbbm{R}} \frac{e^{\eta t} - 1}{t} \mathd K_f (x ; t)
    \text{ } = \text{ } \int_0^M \frac{e^{\eta t} - 1}{t} \mathd K_f (x ; t)
    \nonumber\\
    & = & \int_0^M \frac{e^{\eta t} - 1}{t} \mathd \Psi (t) + \int_0^M \left(
    K_f (x ; t) - \Psi (t)^{^{^{}}} \right) \cdot \left[ \frac{e^{\eta t} -
    \eta t \cdot e^{\eta t} - 1}{t^2} \right] \mathd t 
  \end{eqnarray}
  By a simple Taylor expansion $e^{\eta t} - \eta t \cdot e^{\eta t} - 1 = O
  (\eta^2 t^2)$. Therefore the integral on the right hand side is bounded by
  $O (\eta^2 / B^2 (f ; x))$. We conclude from $(10.1)$ and $(10.2)$ that
  \begin{equation}
    \sum_{p \leqslant x} \frac{f (p) e^{\eta f (p)}}{p} - \mu (f ; x) = B^2 (f
    ; x) \int_{\mathbbm{R}} \frac{e^{\eta t} - 1}{t} \mathd \Psi (t) + O
    \left( \eta^2 \right)
  \end{equation}
  By definition of $\rho_{\Psi}$ and $\eta_f$,
  \begin{equation}
    B^2 (f ; x) \int_{\mathbbm{R}} \frac{e^{\rho_{\Psi} (B ; \Delta) t} -
    1}{t} \mathd \Psi (t) = \Delta B (f ; x) = \sum_{p \leqslant x} \frac{f
    (p) e^{\eta_f (x ; \Delta) f (p)}}{p} - \mu (f ; x)
  \end{equation}
  From $(10.3)$ and $(10.4)$ it follows that
  \begin{equation}
    B^2 (f ; x) \int_{\mathbbm{R}} \frac{e^{\rho_{\Psi} (B ; \Delta) t} -
    e^{\eta_f (x ; \Delta) t}}{t} \mathd \Psi (t) = O \left( \eta^2_f (x ;
    \Delta) \right)
  \end{equation}
  Since $\Psi (t)$ is supported on $[0 ; M]$ we can restrict the above
  integral to $[0 ; M]$. By lemma 10.1, we have $\eta_f (x ; \Delta) \sim
  \Delta / B (f ; x) = o (1)$ in the range $\Delta \leqslant o (B (f ; x))$.
  Also $0 \leqslant \rho_{\Psi} (x ; \Delta) \leqslant \int_{[0 ; M]}
  (e^{\rho_{\Psi} (x ; \Delta) t} - 1) / t \cdot \mathd \Psi (t) = \Delta / B
  (f ; x) = o (1)$ for $\Delta$ in the same range. Write $\rho \assign
  \rho_{\Psi} (B ; \Delta)$ and $\eta \assign \eta_f (x ; \Delta)$. For $0
  \leqslant t \leqslant M$, we have
  \begin{eqnarray*}
    (1 / t) \left( e^{\rho t} - e^{\eta t} \right) & = & (1 / t) e^{\rho t}
    \cdot (1 - e^{(\eta - \rho) t})\\
    & = & e^{\rho t} \cdot (\eta - \rho) + O ((\eta - \rho)^2) \text{ }
    \asymp \text{ } \eta - \rho
  \end{eqnarray*}
  because $\rho = o (1)$, $\eta = o (1)$. Inserting this estimate into
  $(10.5)$ we get $\eta - \rho = O (\eta^2 / B^2 (f ; x)) = o (1 / B^2 (f ;
  x))$ since $\eta^2 = o (1)$. The lemma is proved. 
\end{proof}

\subsection{Proof of the ``primes to integers'' theorem}

\begin{proof}[Proof of Theorem 2.5]
  Note that
  \[ \frac{1}{\sigma^2 (f ; x)} \sum_{\tmscript{\begin{array}{c}
       p \leqslant x\\
       f (p) \leqslant t
     \end{array}}} \frac{f (p)^2}{p} \cdot \left( 1 - \frac{1}{p} \right) =
     \frac{1}{B^2 (f ; x)} \sum_{\tmscript{\begin{array}{c}
       p \leqslant x\\
       f (p) \leqslant t
     \end{array}}} \frac{f (p)^2}{p} + O \left( \frac{1}{B^2 (f ; x)} \right)
  \]
  and denote the main term on the right hand side by $K_f (x ; t)$. By
  assumption, the left hand side in the above equation, differs from $\Psi
  (t)$ by$\ll 1 / \sigma^2 (f ; x) \asymp 1 / B^2 (f ; x)$. Hence $K_f (x ; t)
  = \Psi (t) + O (1 / B^2 (f ; x))$. Let $\eta = \eta_f (x ; \Delta)$ be the
  parameter from lemma 10.1. Proceeding as in the proof of the previous lemma,
  we get
  \begin{eqnarray}
    \sum_{p \leqslant x} \frac{e^{\eta f (p)} - \eta f (p) - 1}{p} & = & B^2
    (f ; x) \int_{\mathbbm{R}} \frac{e^{\eta u} - \eta u - 1}{u^2} \mathd \Psi
    (u) + o (1) \\
    \eta \sum_{p \leqslant x} \frac{f (p) \cdot (e^{\eta f (p)} - 1)}{p} & = &
    B^2 (f ; x) \eta \int_{\mathbbm{R}} \frac{e^{\eta u} - 1}{u} \mathd \Psi
    (u) + o (1) 
  \end{eqnarray}
  throughout the range $1 \leqslant \Delta \leqslant o (B (f ; x))$. Let $\rho
  \assign \rho_{\Psi} (B (f ; x) ; \Delta)$ denote the parameter from Lemma
  10.2. The functions on the right of $(10.6)$ and $(10.7)$ are analytic.
  Therefore, by lemma 10.3,
  \begin{eqnarray}
    B^2 (f ; x) \int_{\mathbbm{R}} \frac{e^{\eta u} - \eta u - 1}{u^2} \mathd
    \Psi (u) & = & B^2 (f ; x) \int_{\mathbbm{R}} \frac{e^{\rho u} - \rho u -
    1}{u^2} \mathd \Psi (u) \upl o (1) \\
    B^2 (f ; x) \eta \int_{\mathbbm{R}} \frac{e^{\eta u} - 1}{u} \mathd \Psi
    (u) & = & B^2 (f ; x) \rho \int_{\mathbbm{R}} \frac{e^{\rho u} - 1}{u} 
    \mathd \Psi (u) + o (1) 
  \end{eqnarray}
  uniformly in $1 \leqslant \Delta \leqslant o (B (f ; x))$. On combining
  $(10.6)$ with $(10.8)$ and $(10.7)$ with $(10.9)$ we obtain
  \begin{eqnarray*}
    &  & \sum_{p \leqslant x} \frac{e^{\eta f (p)} - \eta f (p) - 1}{p} -
    \eta \sum_{p \leqslant x} \frac{f (p) (e^{\eta f (p)} - 1)}{p}\\
    & = & B^2 (f ; x) \int_{\mathbbm{R}} \frac{e^{\rho u} - \rho u - 1}{u^2}
    \mathd \Psi (u) - B^2 (f ; x) \rho \int_{\mathbbm{R}} \frac{e^{\rho u} -
    1}{u} \mathd \Psi (u) + o(1)
  \end{eqnarray*}
  By lemma 10.1, lemma 10.2 and the above equation, we get
  \[ \mathcal{D}_f^{\times} (x ; \Delta) \sim \mathbbm{P} \left(
     \mathcal{Z}_{\Psi} \left( B^2 (f ; x) \right) \geqslant \Delta B (f ; x)
     \right) \]
  uniformly in $1 \leqslant \Delta \leqslant o (B (f ; x))$. By lemma 9.5, it
  follows that
  \[ \mathcal{D}_f (x ; \Delta) \sim \mathbbm{P} \left( \mathcal{Z}_{\Psi}
     (\sigma^2 (f ; x)) \geqslant \Delta \sigma (f ; x) \right) \]
  uniformly in $1 \leqslant \Delta \leqslant o (\sigma (f ; x))$. This proves
  the theorem. 
\end{proof}

{\noindent}\tmtextbf{Acknowledgements.} The author would like to thank first and foremost Andrew Granville. There is too much to thank for, so it is simpler to note that this project would not surface without his constant support. The author would also like to thank Kevin Ford for suggesting Conjecture 2.3, thus throwing light on the (previously obscure) function $\mathcal{A}(f;z)$ and to Philippe Sosoe for proof-reading a substantial part of this paper. 
{\hspace*{\fill}}{\medskip}

\end{document}